\documentclass[11pt,reqno]{amsart}

\usepackage{amsmath, amsthm, amssymb}
\usepackage{amsfonts}
\usepackage{mathtools}

\usepackage[ansinew]{inputenc}
\usepackage[dvips]{epsfig}
\usepackage{graphicx}
\usepackage[english]{babel}

\usepackage{aliascnt}

\theoremstyle{plain}
\newtheorem{theorem}{Theorem}[section]

\newaliascnt{lemma}{theorem}
\newtheorem{lemma}[lemma]{Lemma}
\aliascntresetthe{lemma}

\newaliascnt{proposition}{theorem}
\newtheorem{proposition}[proposition]{Proposition}
\aliascntresetthe{proposition}

\newaliascnt{corollary}{theorem}

\aliascntresetthe{corollary}

\theoremstyle{definition}
\newaliascnt{definition}{theorem}
\newtheorem{definition}[definition]{Definition}
\aliascntresetthe{definition}

\newaliascnt{remark}{theorem}
\newtheorem{remark}[remark]{Remark}
\aliascntresetthe{remark}

\newaliascnt{example}{theorem}
\newtheorem{example}[example]{Example}
\aliascntresetthe{example}

\theoremstyle{definition}
\newtheorem*{remark*}{Remark}
\newtheorem*{assumption*}{Assumption}

\numberwithin{equation}{section}

\usepackage[backgroundcolor=white, bordercolor=blue,
linecolor=blue]{todonotes}

\usepackage{hyperref}
\usepackage{cleveref}

\usepackage{dsfont}
\usepackage{bbm}

\newcommand{\N}{\mathbb{N}}
\newcommand{\R}{\mathbb{R}}

\newcommand{\Z}{\mathbb{Z}}

\newcommand{\cH}{\mathcal{H}}
\newcommand{\cA}{\mathcal{A}}

\newcommand{\cP}{\mathcal{P}}
\newcommand{\cQ}{\mathcal{Q}}

\newcommand{\cN}{\mathcal{N}}

\newcommand{\cM}{\mathcal{M}}
\newcommand{\cI}{\mathcal{I}}

\newcommand{\eps}{\epsilon}

\newcommand{\ddiv}{\mathrm{div}_v}
\newcommand{\ddiw}{\mathrm{div}_w}

\newcommand{\bbR}{\mathbb{R}}

\DeclareMathOperator{\supp}{supp}

\renewcommand{\d}{\textnormal{\,d}}

\newcommand{\average}{{\mathchoice {\kern1ex\vcenter{\hrule height.4pt
width 6pt depth0pt} \kern-9.7pt} {\kern1ex\vcenter{\hrule
height.4pt width 4.3pt depth0pt} \kern-7pt} {} {} }}
\newcommand{\dashint}{\average\int}

\begin{document}
\allowdisplaybreaks
\title{Sharp regularity near the grazing set for kinetic Fokker-Planck equations}

\author{Kyeongbae Kim}
\address{Institute for Applied Mathematics, University of Bonn, Endenicher Allee 60, 53115, Bonn, Germany}
\email{kkim@uni-bonn.de}
\urladdr{https://sites.google.com/view/kyeongbaekim/}

\author{Marvin Weidner}
\address{Institute for Applied Mathematics, University of Bonn, Endenicher Allee 60, 53115, Bonn, Germany}
\email{mweidner@uni-bonn.de}
\urladdr{https://sites.google.com/view/marvinweidner/}



\allowdisplaybreaks

\begin{abstract}
We prove optimal regularity results for solutions to linear kinetic Fokker-Planck equations in bounded domains. Our contributions are two-fold. First, we establish the sharp $C^{1/2}$ regularity for either diffuse reflection or prescribed in-flow boundary conditions. Previously, in this setting, it was only known that solutions are $C^{\alpha}$ for some small $\alpha > 0$. Second, we provide a complete characterization of the solution behavior near the grazing set by proving higher order expansions beyond the critical regularity threshold of $\frac{1}{2}$. These results demonstrate for the first time that solutions maintain higher smoothness up to the grazing set near the incoming boundary.
\end{abstract}

\maketitle

\section{Introduction}  

The goal of this article is to investigate the boundary behavior of solutions to linear kinetic equations in domains, subject to physical boundary conditions. The prototype model of our analysis is the classical Kolmogorov equation 
\begin{align}
\label{eq:Kolmogorov}
    (\partial_t + v \cdot \nabla_x) f - \Delta_v f = F \quad \text{in } (0,T) \times \Omega \times \R^n,
\end{align}
yet all our results hold for the general class of kinetic Fokker-Planck equations given by
\begin{align}
\label{eq:FPK}
    (\partial_t + v \cdot \nabla_x) f - \ddiv(A\nabla_vf) = B \cdot \nabla_v f + F \quad \text{in } (0,T) \times \Omega \times \R^n
\end{align}
with sufficiently smooth coefficients. These are well-known models in kinetic theory, arising naturally as the linearization of the Landau equation from plasma physics.

Establishing the regularity of solutions is a central challenge in kinetic theory. In the context of nonlinear models like the non-cutoff Boltzmann and Landau equations, regularity is directly linked to well-posedness \cite{HeSn20,GKTT17,GuSi25,HST25,ISV26} and convergence to equilibrium \cite{DeVi01,DeVi05}. 
In recent years there has been huge interest in the regularity of linear kinetic equations such as \eqref{eq:FPK} \cite{GIMV19,ImMo21} and their nonlocal counterparts \cite{ImSi20,ImSi21}, since these results provide the necessary estimates to establish conditional regularity results for the Boltzmann and Landau equations \cite{CSS18,GIMV19,HeSn20}, \cite{ImSi20,ImSi22,FRW25}. Yet, despite these advances, the theory remains largely limited to interior regularity results or to $x$-periodic solutions.

Characterizing the behavior of solutions near physical boundaries remains an outstanding challenge in kinetic theory that is currently far from understood, even for the Kolmogorov equation \eqref{eq:Kolmogorov}, which is the simplest canonical model. The difficulty is largely due to the characteristic nature of the boundary, which leads to singular solution behavior near the grazing set \cite{GJW99,HJV14}
\begin{align*}
    \gamma_0 = (0,T) \times \{ (x,v) \in \partial \Omega \times \R^n : v \cdot n_x = 0 \}.
\end{align*}
To date, a complete regularity theory is only available for solutions with specular reflection boundary condition \cite{RoWe25}, namely
\begin{align}
\label{eq:specular}
    f(t,x,v) = f(t,x,R_xv) \quad \text{for } x \in \partial \Omega, \quad \text{where } R_x v = v - 2(v \cdot n_x) n_x,
\end{align}
where $n_x$ denotes the outward unit normal vector at $x$. In \cite{RoWe25}, it was proved that solutions are at most $C^{4,1}$ up to $\partial \Omega$ and that this regularity is optimal as $(t,x,v)$ approaches the grazing set $\gamma_0$, whereas solutions are $C^{\infty}$ away from $\gamma_0$. Previously, it was only known that solutions are $L^{\infty}_t C^{\alpha/3,\alpha}_{x,v}$ and $C^{1,\alpha}_v$ for any $\alpha \in (0,1)$ \cite{DGY22,GHJO20}.

Beyond this, other physically relevant boundary conditions have remained largely unexplored, and available results fail to reach the optimal regularity threshold, as we will discuss below. 

The first main contribution of this article is to resolve this open question by establishing \textit{optimal regularity estimates} for solutions with either \textit{diffuse reflection} boundary condition or prescribed \textit{in-flow}. Both of these conditions induce a boundary behavior that stands in stark contrast to the one of specular reflection, since solutions are in general not better than $C^{1/2}$ up to the grazing set $\gamma_0$.

Second, motivated by the singular behavior of solutions near the grazing set, we establish \textit{higher order expansions at $\gamma_0$} that surpass the regularity threshold of $\frac{1}{2}$. This provides a precise characterization of the asymptotic behavior of solutions at higher order. As an application, we establish the following properties of solutions $f$ to \eqref{eq:FPK} with zero in-flow (i.e. boundaries are absorbing), which, to the best of our knowledge, are the first regularity results \textit{above the critical threshold of $\frac{1}{2}$}:
\begin{itemize}
    \item near the incoming boundary, $f \in C^{3-\eps}$ up to $\gamma_0$, i.e. as $v \cdot n_x \nearrow 0$, $f$ is better than $C^{1/2}$ ,
    \item away from the incoming boundary, $f/\Phi \in C^{0,1}$ up to $\gamma_0$ for an explicit function $\Phi \in C^{1/2}$.
\end{itemize}

In the sequel, we will explain our main results in detail and discuss their novelty.

\subsection{Optimal regularity estimates}

We decompose the kinetic boundary $\gamma := (0,T) \times \partial \Omega \times \R^n$ into its outgoing, incoming, and grazing parts $\gamma = \gamma_+ \cup \gamma_- \cup \gamma_0$, where
\begin{align*}
    \gamma_{\pm} := (0,T) \times \{ (x,v) \in \partial \Omega \times \R^n : \pm v \cdot n_x > 0 \}, \quad \gamma_{0} := (0,T) \times \{ (x,v) \in \partial \Omega \times \R^n : v \cdot n_x = 0 \}.
\end{align*}

There are several natural boundary conditions for kinetic equations (see \cite{Vil02}). The specular reflection boundary condition \eqref{eq:specular} has a clear physical interpretation and is mathematically tractable as it resembles a Neumann condition. However, since boundaries are usually rough on a microscopic scale, particles do not bounce elastically, but interact with the boundary and are re-emitted according to a velocity distribution determined by the boundary temperature \cite{Mis10}. 

This mechanism is encoded in the diffuse reflection boundary condition, which is defined as follows, given a function $\cM$\footnote{Usually, one assumes $\int_{\R^n} \cM (t,x,w) (w \cdot n_x)_- \d w = 1$. A canonical choice would be $\cM(v) = 2 e^{-|v|^2}$.}:
\begin{align}
\label{eq:diffuse}
    f = \cN f \quad \text{in } \gamma_-, \quad \text{where}& \quad \cN f(t,x,v) = \cM(t,x,v) \int_{\R^n} f(t,x,w) (w \cdot n_x)_- \d w.
\end{align}

Our first main result gives a complete characterization of the smoothness of solutions to \eqref{eq:FPK} subject to the diffuse reflection boundary condition.

\begin{theorem}
\label{thm:main-diffuse}
    Let $\Omega \subset \R^n$ be a bounded smooth domain and let $A,B,F \in C^{\infty}((0,T) \times \overline{\Omega} \times \R^n)$ and $A$ be uniformly elliptic. Let $\cM \in C^{\infty}(\gamma_-)$. Let $f$ be a weak solution to \eqref{eq:FPK} with diffuse reflection condition \eqref{eq:diffuse},
    and assume that $f, F$, and $\cM$ have fast decay\footnote{This means that $(1+|v_0|)^k f$, $(1+|v_0|)^k F$, and $(1+|v_0|)^k \partial^{\beta} \cM$ are bounded for any $k \in \N$ and any multi-index $\beta$.} as $|v| \to \infty$. 
    
    Then, for any $[t_0,t_1] \in (0,T)$, it holds
    \begin{itemize}
        \item[(i)] $f$ is smooth away from the grazing set, i.e. $f \in C^{\infty}(( [t_0,t_1] \times \overline{\Omega} \times \R^n) \setminus \gamma_0)$,
        \item[(ii)] $f \in C^{1/2}([t_0,t_1] \times \overline{\Omega} \times \R^n)$ with the estimate
        \begin{align*}
           \qquad \Vert f \Vert_{C^{1/2}([t_0,t_1] \times \overline{\Omega} \times \R^n)} \le C \big( \Vert (1 + |v|)^q f \Vert_{L^{\infty}((0,T) \times \Omega \times \R^n)} + \Vert (1 + |v|)^q F \Vert_{L^{\infty}((0,T) \times \Omega \times \R^n)} \big),
        \end{align*}
        where $q,C > 0$ depend only on $n,t_0,t_1,\Omega,A,B,\cM$.
    \end{itemize}
\end{theorem}

As the following remark clarifies, \autoref{thm:main-diffuse} is a simplified version of our main result on the diffuse reflection condition.

\begin{remark}
 We actually show the following more general statements:
    \begin{itemize}
        \item (ii) remains true when $\partial \Omega \in C^{2,\eps}$, $A \in C^{\eps}$ for some $\eps > 0$ and $B,F \in L^{\infty}$ (see \autoref{thm:hol12.gen.diff}).
        \item We prove Schauder-type estimates of order $k + \alpha$ away from $\gamma_0$ in localized kinetic cylinders under suitable assumptions on the coefficients $A,B,F,\cM$ (see \autoref{thm.diff}(ii)).
    \end{itemize}
\end{remark}

The $C^{1/2}$ regularity is already new for solutions to the Kolmogorov equation \eqref{eq:Kolmogorov}, and it is sharp in the sense that we construct an explicit solution $f$ to \eqref{eq:Kolmogorov} in the half-space with $h, \cM \in C^{\infty}$ and such that $f \in C^{1/2}$, but $f \not\in C^{1/2 + \eps}$ for every $\eps > 0$ (see \autoref{example:diffuse-optimal}). Moreover, note that also the $C^{\infty}$ regularity away from $\gamma_0$ is highly non-trivial due to the nonlocal dependence of the boundary data on the solution itself. We refer to Section \ref{sec:prelim} for the precise definition of the kinetic H\" older spaces.

Prior to our result, it was only known that solutions to \eqref{eq:FPK} with diffuse reflection condition \eqref{eq:diffuse} are $C^{\alpha}$ up to the boundary for some small $\alpha > 0$ (see \cite{Zhu24}), both at $\gamma_0$ and at $\gamma_-$. Proving regularity results up to the optimal order requires completely different techniques compared to \cite{Zhu24}. Moreover, unlike specular reflection, \eqref{eq:diffuse} does not resemble a Neumann boundary condition. Hence, we cannot use the approach from \cite{RoWe25}, and significant new ideas are required, as we explain below.

Instead, in this article we demonstrate that diffuse reflection shares the same singular boundary behavior of solutions near the grazing set $\gamma_0$ as the in-flow boundary condition
\begin{align}
\label{eq:in-flow}
    f = g \quad \text{in } \gamma_-,
\end{align}
where regularity is limited to $C^{1/2}$. Indeed, our second main result establishes the optimal $C^{1/2}$ regularity for prescribed in-flow from which we ultimately derive \autoref{thm:main-diffuse} by setting $g = \cN f$. The in-flow condition, which prescribes the distribution of particles with incoming velocities, is another fundamental and physically relevant boundary condition in kinetic theory (see \cite{Vil02}). 

The following theorem characterizes the boundary regularity of solutions to \eqref{eq:FPK} with prescribed in-flow and it is the second main result of this article.

\begin{theorem}
\label{thm:in-flow}
   Let $\Omega \subset \R^n$ be a bounded smooth domain and let $A,B,F \in C^{\infty}((0,T) \times \overline{\Omega} \times \R^n)$ and $A$ be uniformly elliptic. Let $g \in C^{\infty}(\gamma_-)$. Let $f$ be a weak solution to \eqref{eq:FPK} with in-flow condition \eqref{eq:diffuse},
    and assume that $f,g$, and $F$ have fast decay as $|v| \to \infty$.  
    
    Then, for any $[t_0,t_1] \in (0,T)$, it holds
    \begin{itemize}
        \item[(i)] $f$ is smooth away from the grazing set, i.e. $f \in C^{\infty}(( [t_0,t_1] \times \overline{\Omega} \times \R^n) \setminus \gamma_0)$,
        \item[(ii)] $f \in C^{1/2}([t_0,t_1] \times \overline{\Omega} \times \R^n)$ with the estimate
        \begin{align*}
            \Vert f \Vert_{C^{1/2}([t_0,t_1] \times \overline{\Omega} \times \R^n)} \le C \left( \Vert (1 + |v|)^q f \Vert_{L^{1}((0,T) \times \Omega \times \R^n)} + [g]_{C^{1/2 + \eps}_q(\gamma_-)} + \Vert (1 + |v|)^q F \Vert_{L^{\infty}((0,T) \times \Omega \times \R^n)} \right),
        \end{align*}
        where $q,C > 0$ depend only on $n,t_0,t_1,\Omega,A,B$. \footnote{We write $[g]_{C^{\alpha}_{q}(\gamma_-)} = \sup_{z_0 \in \gamma_-} (1 + |v_0|)^q [g]_{C^{\alpha}(\mathcal{Q}_1(z_0) \cap \gamma_-)}$. See \eqref{eq:kinetic-cylinder} for the definition of $\mathcal{Q}_1(z_0)$.}
    \end{itemize}
\end{theorem}

\begin{remark}
We comment on more precise versions of \autoref{thm:in-flow}.
    \begin{itemize}
        \item (ii) remains true when $\partial \Omega \in C^{2,\eps}$, $A \in C^{\eps}$, and $g \in C^{1/2 + \eps}$ for some $\eps > 0$, and $B,F \in L^{\infty}$ (see \autoref{thm:hol12.gen}). We expect there to be non $C^{1/2}$ solutions if $g \in C^{1/2} \setminus C^{1/2 + \eps}$.
        \item We also establish localized $C^{1/2}$ estimates up to $\gamma_0$ for solutions in the intersection of $(-1,1) \times \Omega \times \R^n$ with a kinetic cylinder (see \autoref{thm.hol12}).
        \item Schauder-type estimates of order $k+\alpha$ away from $\gamma_0$ have already been established in \cite{RoWe25}.
    \end{itemize}
\end{remark}

The $C^{1/2}$ regularity up to the grazing set is already new for the Kolmogorov equation \eqref{eq:Kolmogorov}. Prior to our result, as for diffuse reflection, it was only known that solutions are $C^{\alpha}$ for some small $\alpha > 0$ in case $g \not\equiv 0$ (see \cite{Sil22,Zhu24}). Moreover, in the literature, there are several results on the boundary regularity for absorbing boundary conditions, i.e. $g \equiv 0$. Let us comment on these results and on how \autoref{thm:in-flow} extends them:

\begin{remark}
    \begin{itemize}
        \item In \cite{HJV14,HJJ15,HJJ18}, the authors prove that $f \in C^{1/2 - \eps}_v \cap C^{1/6 -\eps}_x$ for any $\eps > 0$ up to the grazing set. Their results do not establish any regularity in $t$ and they work with H\"older spaces that do not capture the transport structure of \eqref{eq:Kolmogorov}, compared to \autoref{thm:in-flow}.
        \item \cite{Zhu25} establishes \emph{a priori} $C^{1/2}$ estimates for solutions. This means that $C^{1/2}$ regularity of solutions up to the boundary has to be \emph{assumed} in order for estimate to hold (see the absorption argument in \cite[Proposition 4.3, Lemma 4.7]{Zhu25}). In all our results, we merely assume $f$ to be a weak solution, i.e. $f \in H^1_v$ up to the boundary. 
        \item The results in \cite{Zhu25} do not yield global estimates in $v$ for global solutions to \eqref{eq:Kolmogorov}, such as \autoref{thm:in-flow}, since the $|z_0|$-dependence of the constants is not explicit.
           \end{itemize}
\end{remark}

Moreover, we refer to \cite{LitNys21,AAMN24,AveHou26,Hou26} for further results on variational methods and trace theorems for kinetic equations with in-flow boundary conditions \eqref{eq:in-flow}.

The sharpness of our result \autoref{thm:in-flow} follows from well-known explicit $C^{1/2}$ solutions to \eqref{eq:Kolmogorov} with $g \equiv 0$ in 1D, which go back to work by probabilists on first passage times of integrated Brownian motion from the last century (see \cite{GJW99}, and also \cite{McK63,Gol71,Sin92,IsWa94}, and the more recent articles \cite{HJV14,Zhu24}). 

Finally, while our results are new already for absorbing boundaries,  extending these optimal regularity estimates to general non-zero in-flow data is a non-trivial task. This is mainly because standard barrier arguments do not seem to apply so easily in the kinetic setting. We refer to Subsection \ref{subsec:strategy} for a more detailed discussion of our approach.

\subsection{Higher order estimates near the grazing set}

The low $C^{1/2}$ regularity of solutions at the grazing set stands in stark contrast to the smoothness of solutions near all other boundary points. This distinct behavior at $\gamma_0$ arises because the transport operator degenerates for velocities tangent to the boundary, which breaks the hypoelliptic regularization of solutions. This phenomenon motivates a detailed analysis of the exact solution behavior near $\gamma_0$. In this context, the following question arises naturally:
\begin{align*}
    \textit{Can we characterize the higher order asymptotics of solutions near } \gamma_0?
\end{align*}

It is another major goal of this article to resolve this question. Moreover, it turns out that higher order asymptotics are crucial even in order to prove the $C^{1/2}$ regularity, as we will explain below.

As a starting point of our analysis, we recall the explicit solution 
\begin{align*}
\phi_0(x,v) = 
\begin{cases}
    x^{1/6}U\left(-\frac16,\frac23,-\frac{v^3}{9x}\right)&\quad\text{if }v<0,\\
             c_0 x^{1/6}e^{-\frac{v^3}{9x}}U\left(\frac56,\frac23,\frac{v^3}{9x}\right)&\quad\text{if }v\geq0
             \end{cases}
\end{align*}
to the stationary Kolmogorov equation in 1D with absorbing boundaries
\begin{equation}
\label{eq:stat-Kolmo}
\left\{
\begin{alignedat}{3}
    v \partial_x \phi_0 &= \partial_{vv} \phi_0 \qquad &&\text{in } (0,\infty) \times \R ,\\
    \phi_0 &= 0  &&\text{on } \{ 0 \} \times (0,\infty),
\end{alignedat} \right.
\end{equation}
which goes back to \cite{Gol71,GJW99}. Here, $c_0 = \frac{\Gamma(7/6)}{\Gamma(1/6)}$ and $U$ denotes the Tricomi confluent hyper-geometric function. The precise boundary behavior of $\phi_0$ can be computed explicitly \cite{Gol71,Sin92,IsWa94,GJW99}, revealing three characteristic regions with distinct asymptotic profiles:
\begin{align}
\label{eq:phi_0-behavior}
\phi_0(x,v)\eqsim\begin{cases}
       x^{\frac16}\left(\frac{v^3}x\right)^{-\frac56}e^{-\frac{v^3}{9x}}\quad&\text{if } x < v^3,\\
        x^{\frac16}\quad&\text{if } x \ge |v|^3,\\
        |v|^{\frac12}\quad&\text{if } x < - v^3.
        \end{cases}
\end{align}

It is natural to suspect, (and, indeed, this is an adequate reformulation of the aforementioned question) that \textit{all solutions} to the Kolmogorov equation
\begin{equation}
\label{eq:Kolmo-localized}
\left\{
\begin{alignedat}{3}
( \partial_t + v \cdot \nabla_x ) f - \Delta_v f &= F \quad && \text{in } \big((-1,1) \times \Omega \times \R^n \big) \cap \cQ_1(z_0) =: \cH_1(z_0),\\
        f &= g && \text{on } \gamma_- \cap \cQ_1(z_0),
\end{alignedat} \right.
\end{equation}
where $\mathcal{Q}_1(z_0)$ denotes a standard kinetic cylinder (see \eqref{eq:kinetic-cylinder}), behave like the function $\phi_0$ up to a smooth error term locally at any boundary point in $\gamma_0$, at least if $g = 0$. Our main results confirm this conjecture, however they also demonstrate that the situation is significantly more complex in general domains and in the presence of source terms. 

Our findings yield a precise description of the anisotropic nature of the singularity of solutions at $\gamma_0$. In order to state them, we need to decompose the domain near $z_0 \in \gamma_0$ into several different regions, similar to \eqref{eq:phi_0-behavior}. We denote $d_{\Omega}(x) = \mathrm{dist}(x,\partial \Omega)$ and write for $\eps \in (0,1]$:
\begin{align*}
    \mathcal{R}_-^{\eps} := \big\{ d_{\Omega}^{\eps}(x) \le (n_x \cdot v)^{3} \big\}, \qquad \mathcal{R}_+ = \{ d_{\Omega}(x) \le - (n_x \cdot v)^3\}, \qquad \mathcal{R}_0 = \{ d_{\Omega}(x) \ge |n_x \cdot v|^3\}.
\end{align*}
Here, if $x \in \Omega$, we denote by $n_x$ the outer unit normal vector at the projection $\bar{x} \in \partial \Omega$ of $x$ to $\partial \Omega$, and without loss of generality, we assume that the projection is unique\footnote{We can always find a tubular neighborhood of $z_0$ where the projection to the boundary is unique.} in $\mathcal{Q}_1(z_0)$ .

Our third main result shows that solutions to the Kolmogorov equation maintain \textit{higher regularity of order less than three up to the grazing set} $\gamma_0$ near the incoming boundary $\gamma_-$, i.e. as $v \cdot n_x \nearrow 0$.

\begin{theorem}
\label{thm:C3}
    Let $\Omega \subset \R^n$ be a bounded smooth domain and let $z_0 \in \gamma_0$ and $\eps \in (0,1)$. Let $F \in C^{5/2}(\mathcal{H}_1(z_0))$ and $g \in C^{2,1}(\mathcal{H}_1(z_0) \cap \gamma_-)$. Let $f$ be a solution to \eqref{eq:Kolmo-localized} in $\mathcal{H}_1(z_0)$.

    Then, it holds $f \in C^{3 - \eps}(\mathcal{H}_{1/2}(z_0) \cap \mathcal{R}_-^{\eps})$, and 
    \begin{align*}
        \Vert f \Vert_{C^{3-\eps}(\mathcal{H}_{1/2}(z_0) \cap \mathcal{R}_-^{\eps})} \le C \big( \Vert f \Vert_{L^1(\mathcal{H}_1(z_0))} +  [F]_{C^{5/2}(\mathcal{H}_1(z_0))} +  [g]_{C^{2,1}(\mathcal{H}_1(z_0) \cap \gamma_-)} \big),
    \end{align*}
    where $C$ depends only on $n,\eps,|z_0|,\Omega$.
\end{theorem}

This result remains true for solutions to kinetic Fokker-Planck equations \eqref{eq:FPK} with smooth enough coefficients (see \autoref{thm:C3-coeff}).

It is important to emphasize that, in contrast to \autoref{thm:in-flow}(i), the norms in \autoref{thm:C3} remain uniformly bounded up to the grazing set. This fact crucially sets \autoref{thm:C3} apart from previous higher order regularity results in $\gamma_-$ \cite[Theorem 1.3]{Sil22}, \cite[Proposition 4.5]{RoWe25}, and \cite[Theorem 1.2]{Zhu25}. Hence, to the best of our knowledge, \autoref{thm:C3} is the first result confirming higher regularity up to $\gamma_0$ in specific directions.

We point out that the regularity established in \autoref{thm:C3} is sharp in the sense that solutions in general do not exceed $C^{3-\eps}$ regularity. Indeed, we construct an explicit function which is globally $C^{1,1}$, and $C^{3-\eps}$ in $\mathcal{R}_-^{\eps}$, and it solves \eqref{eq:Kolmogorov} with $F \equiv 1$ subject to absorbing boundaries (see \autoref{ex.sharp.3eps}). Note that in the homogeneous case $F \equiv 0$, one could even show that solutions are $C^{\infty}$ up to $\gamma_0$ near $\gamma_-$ by combining \cite[Theorem 1.3]{Sil22} with interior estimates.

Regularity estimates up to $\gamma_0$ above the critical regularity threshold $\frac{1}{2}$ cannot hold in any other region than in $\mathcal{R}_-^{\eps}$ for $\eps \in (0,1)$, which is the region considered in \autoref{thm:C3}. In fact, such behavior is ruled out by the explicit solution $\phi_0$, which is merely $C^{1/2}$ along the trajectory $x = v^3$, i.e. on the boundary of $\mathcal{R}_-^1$, and in $\mathcal{R}^0 \cup \mathcal{R}^+$. However, as our next main result asserts, in these regions, all solutions (subject to absorbing boundary conditions) behave asymptotically like $\phi_0$ up to an error term that is Lipschitz continuous. 

\begin{theorem}
\label{thm:f/phi}
    Let $\Omega \subset \R^n$ be a bounded smooth domain and let $z_0 \in \gamma_0$. Let $F \in L^{\infty}(\mathcal{H}_1(z_0))$ and $g \equiv 0$. Let $f$ be a solution to \eqref{eq:Kolmo-localized} in $\cH_1(z_0)$ and define $\Phi(t,x,v) = \phi_0(\mathrm{dist}(x,\partial \Omega) , v \cdot n_x)$.
    
    Then, it holds
    \begin{align*}
        \frac{f}{\Phi} \in C^{0,1} \big(\mathcal{H}_{1/2}(z_0) \setminus \mathcal{R}^1_- \big).
    \end{align*}
    Moreover, we have the estimate
    \begin{align*}
        \Vert f / \Phi \Vert_{C^{0,1}(\mathcal{H}_{1/2}(z_0) \setminus \mathcal{R}_-^{1})} \le C \big( \Vert f \Vert_{L^1(\mathcal{H}_1(z_0))} +  \Vert F \Vert_{L^{\infty}(\mathcal{H}_1(z_0))} \big) ,
    \end{align*}
    where $C$ depends only on $n,|z_0|,\Omega$.
\end{theorem}

A variant of this theorem remains true for solutions to kinetic Fokker-Planck equations \eqref{eq:FPK} with sufficiently smooth coefficients (see \autoref{thm.h/phi} for a version in the half-space and \autoref{thm:f/phi-coeff} for a result in a smooth convex domain $\Omega \subset \R^n$). In this case, the function $\Phi$ has to be modified appropriately. Note that $\Phi^2$ is always comparable to the kinetic distance to $\gamma_0$, denoted by $\mathrm{dist}_{\text{kin}}$ (see \eqref{eq:kinetic-dist}), up to higher order in the sense that
    \begin{align*}
        \Phi \asymp \mathrm{dist}_{\text{kin}}^{1/2}, \qquad D^{\beta} \Phi \lesssim \mathrm{dist}_{\text{kin}}^{1/2 - |\beta|}\quad\text{in }\mathcal{H}_1(z_0)\setminus \mathcal{R}^1_-.
    \end{align*}

\autoref{thm:f/phi} confirms that the function $\phi_0$ precisely characterizes the boundary behavior of solutions to \eqref{eq:Kolmogorov} near $\gamma_0$ away from the incoming boundary. Note, however, that the asymptotic description of \autoref{thm:f/phi} breaks down near the incoming boundary. Indeed, the function $\phi_{-1}$ constructed in Appendix \ref{sec:appendix} is a solution to \eqref{eq:stat-Kolmo} and $\phi_{-1}/\phi_0$ fails to be bounded on $\gamma_-$ up to $\gamma_0$ by \eqref{phi-1.asymptotic}.

\begin{remark}
In the setting of \autoref{thm:f/phi}, the Lipschitz regularity is sharp in some sense, as we discuss in \autoref{ex.sharp.f/phi}. However, in certain special cases, the optimal regularity exponent is higher and regularity estimates can be achieved by following our technique, as we explain carefully in \autoref{rem:higher-quotient-reg}:
\begin{itemize}
    \item In the half-space, one can show that the optimal regularity is $C^{3/2}$ instead of $C^{0,1}$. This indicates that the presence of coefficients heavily complicates the boundary behavior of solutions.
    \item In the half-space in case $F \equiv 0$, one can even prove that the optimal regularity is even $C^{1,1}$ instead of $C^{0,1}$.
\end{itemize}
\end{remark}

We emphasize that \autoref{thm:C3} and \autoref{thm:f/phi} seem to be the first results establishing regularity above the critical threshold of $\frac{1}{2}$ for solutions to the Kolmogorov equation with absorbing boundaries. The work most closely related to ours is probably \cite{HLW24}, where the authors examine the precise boundary behavior of solutions to the Kolmogorov equation in the half-space in terms of $\phi_0$. However, \cite{HLW24} does not address the smoothness of the remainder terms. Conversely, their results include sharp time decay estimates of solutions.

\subsection{Strategy of the proof}
\label{subsec:strategy}

The common key difficulty in the proofs of all our main results arises from the singular solution behavior at $\gamma_0$. The central novelty of this article is to characterize the higher order asymptotics of the solution near $\gamma_0$ subject to the prescribed in-flow boundary condition. Note that this is more complicated for the in-flow condition than for specular reflection (see \cite{RoWe25}), since solutions exhibit a much more singular behavior.

\subsubsection{Higher order expansions at $\gamma_0$}

The main idea to show \autoref{thm:in-flow}, \autoref{thm:f/phi}, and \autoref{thm:C3} is to establish that all solutions to \eqref{eq:FPK} behave like the explicit solution $\phi_0$ at boundary points in $\gamma_0$ up to error terms which possess higher smoothness. 

Let us explain this result in the simplest case, i.e when $\Omega = \{ x_n > 0 \}$ is a half-space. We consider solutions in the intersection of $(-1,1) \times \Omega \times \R^n$ with a kinetic cylinder 
\begin{align}
\label{eq:kinetic-cylinder}
    \mathcal{Q}_R(z_0) = \big\{ (t,x,v) \in (t_0 - R^2 , t_0 + R^2) \times B_{R^3}(x_0 + (t-t_0) v_0) \times B_R(v_0) \big\}.
\end{align}

Note that in the flat case, it holds $\gamma = (-1,1) \times \{ x_n = 0 \} \times \R^n$, and
\begin{align*}
    \gamma_0 = (-1,1) \times \{ x_n = 0 \} \times \{ v_n = 0 \}, \quad
    \gamma_{\pm} = (-1,1) \times \{ x_n = 0 \} \times \{ \pm v_n < 0 \}.
\end{align*}

Let us assume that $f$ solves the Kolmogorov equation with prescribed in-flow
\begin{equation}
\label{eq:Kolmo-loc-flat}
\left\{
\begin{alignedat}{3}
    \partial_t f + v \cdot \nabla_x f - \Delta_v f &= F \quad &&\text{in } Q_1(z_0) \cap \big( (-1,1) \times \{ x_n > 0 \} \times \R^n \big) =: \mathcal{H}_1(z_0),\\
    f &= g &&\text{in } \mathcal{Q}_1(z_0) \cap \big( (-1,1) \times \{ x_n = 0 \} \times \{ v_n > 0 \} \big)
\end{alignedat} \right.
\end{equation}
for some $z_0 \in \gamma_0$, where $F \in L^{\infty}(\mathcal{H}_1(z_0))$ and $g \in C^{\frac{1}{2} + \eps}(\mathcal{Q}_1(z_0) \cap \gamma_-)$ for some $\eps \in (0,\frac{1}{2})$.

We establish the following expansion of order $\frac{1}{2} + \eps$ at the boundary point $z_0$ for some $c_0 \in \R$:
\begin{align}
\label{eq:expansion-1}
    |f(z) - g(z_0) - c_0 \phi_0(x_n,v_n)| \lesssim r^{\frac{1}{2} + \eps}  ~~ \forall z = (t,x,v) \in \mathcal{H}_{r}(z_0).
\end{align}

Since the function $\phi_0$ is $C^{1/2}$, by combination of \eqref{eq:expansion-1} with interior estimates and boundary estimates away from the grazing set (see \cite{RoWe25}), we manage to deduce that $f \in C^{1/2}$ up to the grazing set. This proves our main result \autoref{thm:in-flow} on the optimal regularity for the in-flow boundary condition.

We emphasize that \eqref{eq:expansion-1} should be interpreted as a regularity estimate of order strictly greater than $\frac{1}{2}$, since it indicates that the function $z \mapsto f(z) - g(z_0) - c_0 \phi_0(z)$ is $C^{\frac{1}{2} + \eps}$ locally at the point $z_0$. 

Although it might be tempting to prove \autoref{thm:in-flow} solely relying on barrier arguments based on the explicit solutions $\phi_0$, we could not reach the optimal $C^{1/2}$ regularity via this approach due to the following reasons:
\begin{itemize}
    \item By perturbation arguments one would naturally lose $\eps$ derivatives for a small $\eps > 0$, when considering equations with coefficients \eqref{eq:FPK}.
    \item It is unclear how to treat non-zero boundary data $g$. Since $g$ is in general not $C^2$, one cannot simply ``subtract'' it from the solution to reduce the problem to a homogeneous one. 
\end{itemize}

The expansion \eqref{eq:expansion-1} allows to surpass all of these obstacles, while giving us higher order information on the solution behavior at the same time. 

In order to prove our main results \autoref{thm:C3} and \autoref{thm:f/phi} on the precise solution behavior near $\gamma_0$, we require asymptotic expansions of even higher order $\beta < 3 + \frac{1}{2}$. In fact, our most general result in this direction is \autoref{lem.exp.nd.a}, where we show that there exist polynomials
$P \in \cP_{\lfloor \beta \rfloor}$, $p_{\phi_0} \in \cP_{\lfloor \beta - \frac{1}{2} \rfloor}$, $p_{\psi_0} \in \cP_{\lfloor \beta - 2 \rfloor}$, and $p_{\phi_1} \in \cP_0$ such that
    \begin{align}
    \label{eq:expansion-2}
        | f(z) - P(z) - p_{\phi_0}(z) \phi_0(x_n,v_n) - p_{\psi_0}(z) \psi_0(x_n,v_n) - p_{\phi_1} \partial_{v_n} \phi_1(x_n,v_n) | \lesssim r^{\beta} ~~ \forall z \in \cH_r(z_0).
    \end{align}

This higher order expansion reveals a key feature of kinetic equations: the explicit 1D function $\phi_0$ alone is \emph{not sufficient to fully capture the boundary behavior} of solutions at the grazing set. Indeed, there are two more explicit solutions $\psi_0$ and $\phi_1$ to the stationary Kolmogorov equation, which have to be included in the expansion as soon as $\beta > 2$ and $\beta > 2 + \frac{1}{2}$, respectively. The function $\phi_1$ satisfies $\partial_x \phi_1 = c_1 \phi_0$ for some $c_1 \in \R$ and is yet another explicit solution of \eqref{eq:stat-Kolmo}, which is homogeneous of degree $3 + \frac{1}{2}$, and has very similar asymptotic behavior as the function $\phi_0$ (see \eqref{phim.asymptotic}). On the other hand, the function $\psi_0$ solves 
\begin{equation}
\label{psi0.sol}
\left\{
\begin{alignedat}{3}
    v \partial_x \psi_0 - \partial_{vv} \psi_0 &= 1 \quad &&\text{in } (0,\infty) \times \R ,\\
    \psi_0 &= 0 &&\text{on } \{ 0 \} \times (0,\infty),
\end{alignedat} \right.
\end{equation}
and its boundary behavior is fundamentally different from the one of $\phi_0$. For instance, it turns out that $\psi_0$ is homogeneous of degree two (instead of $\frac{1}{2}$), which implies that $\psi_0 \in C^{1,1}$ (instead of $C^{1/2}$) up to $\gamma_0$. On the other hand, the function $\psi_0$ does not exhibit exponential decay at $\gamma_-$, but decays only polynomially. Precisely, its sharp behavior as $v \cdot n_x \nearrow 0$ is such that $\psi_0 \in C^{3-\eps}(\{ x_n^{\eps} \le v_n^3 \})$ for $\eps \in (0,1)$ (see \autoref{lemma:psi-reg}). Hence, due to \eqref{eq:expansion-2}, we are able to deduce that any solution to \eqref{eq:Kolmo-loc-flat} exhibits the same regularity properties as $v \cdot n_x \nearrow 0$, which yields \autoref{thm:C3}. 

\begin{remark}
Note that \eqref{eq:expansion-2} needs to be modified in case $\beta > 3 + \frac{1}{2}$. Indeed, one can find a countable family of solutions $\phi_m$ to \eqref{eq:stationary-intro} that are homogeneous of degree $\frac{1}{2} + 3m$ and satisfy $\partial_x \phi_m = c_m \phi_{m-1}$ for some constant $c_m \in \R$. Analogously, one can find higher order versions of the functions $\psi_0$. We compute these solutions explicitly in \autoref{sec:appendix} and establish several properties that might be of independent interest. The higher the value of $\beta$, the more terms involving these functions (and their derivatives) would have to be included in the expansion. Although or approach would allow for it, we decided not to include expansions of order $\beta > 3 + \frac{1}{2}$ in this paper.
\end{remark}

\subsubsection{Liouville theorem}

Already in the simplest case, i.e. when $f$ solves \eqref{eq:Kolmo-loc-flat}, the proof of the expansions \eqref{eq:expansion-1} and \eqref{eq:expansion-2} is far from trivial.
We prove them by a blow-up argument at the boundary point $z_0 \in \gamma_0$, which heavily relies on the following  Liouville-type classification result of solutions in the half-space (see also \autoref{lemma:nD-classification-inflow}):

\begin{theorem}
\label{thm:Liou-intro}
    Let $\beta \in (0,3 + \frac{1}{2})$ and $q$, $Q$ be polynomials. Let $f$ be a solution to 
\begin{equation*}
\left\{
\begin{alignedat}{3}
        (\partial_t + v \cdot \nabla_x ) f - \Delta_v f &= Q \quad &&\text{in } \R \times \{ x_n > 0 \} \times \R^n, \\
        f &= q \quad &&\text{in } \R \times \{ x_n = 0 \} \times \{ v_n > 0 \},
\end{alignedat} \right.
\end{equation*}
such that
    \begin{align}
    \label{eq:growth-Liouville-intro}
        |f(t,x,v)| \le C (1 + |t|^{1/2} + |x|^{1/3} + |v| )^{\beta} \qquad \forall (t,x,v) \in \R \times \{ x_n \ge 0 \} \times \R^n.
    \end{align}
    Then, there exist polynomials $P$, $p_{\phi_0}$, $p_{\psi_0}$, and $p_{\phi_1}$ such that $f$ is of the form 
    \begin{align*}
        f = P + p_{\phi_0} \phi_0 + p_{\psi_0} \psi_0 + p_{\phi_1} \partial_{v_n} \phi_1.
    \end{align*}
\end{theorem}

Note that a blow-up procedure based on a Liouville theorem to show boundary regularity for kinetic equations was recently developed in \cite{RoWe25} for solutions with specular reflection condition. Yet, our proof of the Liouville theorem has to be fundamentally different from the one in \cite{RoWe25}, making \autoref{thm:Liou-intro} one of the main contributions of this article. In fact, the specular reflection condition resembles a Neumann-type boundary condition in the sense that for $Q=0$ solutions in $\{ x_n > 0 \}$ can be extended to global solutions by mirror reflection, and hence they must be polynomials. Such a powerful tool is not available in our setting, so that we have to develop an entirely different approach, already for stationary solutions with $q=Q=0$ in 1D, i.e.
\begin{align}
\label{eq:stationary-intro}
v \partial_x  f - \partial_{vv} f = 0 \quad \text{in } (0,\infty) \times \R, \qquad f = 0 \quad \text{on } \{ 0 \} \times (0,\infty).
\end{align}
In this case, all solutions satisfying \eqref{eq:growth-Liouville-intro} for $\beta < 3 + \frac{1}{2}$ are given by $\phi_0$ (up to a multiplicative constant). To prove it, we proceed as follows: First, by solving an ODE problem, we show that the only \emph{homogeneous} solution (up to a multiplying constant) to \eqref{eq:stationary-intro} is given by $\phi_0$. From here, it remains to prove that all solutions to \eqref{eq:stationary-intro} are homogeneous, however this is far from obvious in our setting. While for elliptic equations or nonlocal problems, this can often be deduced from Sturm-Liouville type theorems (see \cite{RoSe16,FaRo22}), such a tool is unavailable here since $\phi_0$ does not belong to the corresponding weighted $L^2$ space. Instead, our proof relies on the construction of explicit sub- and supersolutions which yield sharp upper and lower bounds for solutions to \eqref{eq:stationary-intro} near the boundary. These bounds imply a boundary Harnack-type principle for solutions to \eqref{eq:stationary-intro} and yield a 1D Liouville theorem of order $\beta < \frac{1}{2} + \eps$, which we can raise subsequently with the help of the maximum principle and a precise analysis of the asymptotic behavior of $\phi_0$.

\subsubsection{Equations with coefficients}

In this paper we develop a unified approach which also allows us to establish expansions such as \eqref{eq:expansion-2} for solutions to general linear kinetic Fokker-Planck equations with coefficients \eqref{eq:FPK}. Note that it is already necessary to consider non-translations invariant equations in order to prove results for the Kolmogorov equation \eqref{eq:Kolmogorov} in general domains $\Omega \subset \R^n$.  

In that case the expansion \eqref{eq:expansion-2} and its proof have to undergo substantial modifications, some of which we comment on below:

\begin{itemize}
    \item At any point $z_0$, the function $\phi_0$ (and also $\psi_0$) has to be rescaled in such a way that $v \partial_x \phi_{0} = A_{n,n}(z_0) \partial_{vv} \phi_0$. Hence, the regularity of the coefficient matrix $A$ is crucial in order to deduce regularity of solutions from expansions at boundary points $z_0 \in \gamma_0$.
        
    \item Due to the $z$-dependence of the coefficients and the non-diagonal structure of $A$, the (rescaled) function $\phi_0$ will solve \eqref{eq:FPK} with unbounded source terms $F$ behaving like $\partial_v \phi_0$ and $\partial_{vv} \phi_0$ at the boundary. In order to compensate for this, we also have to incorporate functions of the following form in the expansion \eqref{eq:expansion-2}, which are not present in the translation invariant case,
\begin{align*}
v_n^i \partial_{v_n}^{(j)} \phi_0, \quad v_n^i \partial_{v_n}^{(j)} \phi_1, \quad v_n^i \partial_{v_n}^{(j)} \psi_0, \qquad \text{ for } i,j \in \{0,1,2\}.
\end{align*}    
Since all of these functions create further error terms, it is a key difficulty in the proof to balance them out appropriately.
\end{itemize}

The aforementioned modifications in \eqref{eq:expansion-2} arising due to the presence of coefficients lead to several non-trivial challenges in the proofs of the expansions and in our main results. For instance, we have to generalize the $C^{\alpha}$ estimates from \cite{Sil22,Zhu24} to equations with source terms that might be in divergence-form or unbounded (see \autoref{lem.bdry.hol}). We discuss all of these complications in full detail in Section \ref{sec:grazing-expansions}. Our most general expansion result is given by \autoref{lem.exp.nd.a}.

\subsubsection{Diffuse reflection}

We close this subsection by making several comments on the proof of our main result on the diffuse reflection condition (see \autoref{thm:main-diffuse}).

As we already mentioned, solutions satisfying the diffuse reflection condition have the same singular boundary behavior as solutions with prescribed in-flow due to \autoref{example:diffuse-optimal}. Hence, in order to prove that solutions are globally $C^{1/2}$ and $C^{\infty}$ locally in $\gamma_-$, we need to show 
\begin{itemize}
    \item[(a)] $\cN f \in C^{1/2+\eps}$ in $\gamma_-$ up to $\gamma_0$ for some $\eps > 0$,
    \item[(b)]  $\cN f \in C^{\infty}$ locally in $\gamma_-$ (away from $\gamma_0$),
\end{itemize}
so that \autoref{thm:main-diffuse} follows by application of \autoref{thm:in-flow} with $g = \cN f$.

Verifying (a) and (b) is a nontrivial task, since the quantity depends on the values of the solution $f$ in $\gamma_+$ in an integrated way, i.e.
\begin{align*}
    \cN f(t,x,v) = \cM(t,x,v) h(t,x,v), \qquad \text{where} \qquad h(t,x,v) = \int_{\R^n} f(t,x,w) (w \cdot n_x)_- \d w.
\end{align*}

In order to prove (a), we apply global regularity estimates for $f$ of order $\alpha \in (0,1)$ as they were established in \cite{Sil22,Zhu24} and combine them with regularity estimates away from $\gamma_0$ (see \cite[Proposition 4.5]{RoWe25}), making use of the weight $(w \cdot n_x)_-$, which vanishes at $\gamma_0$ (see \autoref{lem.reg.nf.diff}(i)). 

The proof of (b) is significantly more involved, since the function $f$ is in general only $C^{1/2}$ up to $\gamma_0$, and therefore, in principle, higher derivatives of order $\beta$ explode as $D^{\beta} f \sim (w \cdot n_x)_-^{- \beta + 1/2}$ as $(w \cdot n_x)_- \to 0$.
We circumvent this issue by flattening the boundary, i.e. considering the problem for $\Omega = \{ x_n > 0 \}$, and observing that in this case $n_x \equiv - e_n$ becomes constant. Hence, the function $h$ only depends on $(t,x')$, which are tangential to the boundary. This means that we only need to prove higher regularity of $f$ in $(t,x')$ on $\gamma_+$ up to $\gamma_0$, in order to get higher regularity of $h$ (and thus of $\cN f$). We achieve this by taking increments $\delta_{h}f$ of $f$ in the tangential directions $(t,x')$, and analyzing the linear kinetic Fokker-Planck equation satisfied by those increments  (see \eqref{defn.diff.quot} for their precise definition), which is (roughly) of the form
\begin{equation*}
\left\{
\begin{alignedat}{3}
(\partial_t + v \cdot \nabla_x f) \delta_{h} f - \mathrm{div}_v(A \nabla \delta_{h} f ) &= \delta_h F + \mathrm{div}_v (A \nabla_v f) \quad && \text{in } \{ x_n > 0 \},\\
    \delta_h f &= \delta_h (\cN f) && \text{on } \gamma_-
\end{alignedat} \right.
\end{equation*}
(see also \eqref{diff1.eq.diffuse.ind}). Here, a key difficulty comes from the source term $\mathrm{div}_v(A \nabla_v f)$, since the function $A \nabla_v f$ is in general unbounded at $\gamma_0$ since $f$ is not better than $C^{1/2}$. We solve this issue by using the expansion \eqref{eq:expansion-2} combined with higher order regularity results away from $\gamma_0$ (see \autoref{lem.gra.bdep} and \autoref{lem.bdry.hol.bdep}) to deduce that $A \nabla_v f$ belongs to a suitable $L^p$ space so that regularity estimates for $\delta_h f$ follow from \autoref{thm:in-flow} by perturbation arguments. This procedure becomes increasingly more involved for higher order increments, and we have to develop a delicate inductive argument to deduce $C^{\infty}$ regularity of $\cN f$ (see \autoref{lem.reg.nf.diff}(ii)).

\subsection{Outline}

This article is structured as follows. In Section \ref{sec:prelim} we introduce the necessary function spaces and weak solutions concepts required for this work. Section \ref{sec:Liouville-1D} is dedicated to the proof of a Liouville theorem for stationary solutions in 1D with prescribed in-flow condition and in Section \ref{sec:Liouville-half-space}, we generalize this result by proving a Liouville theorem in the half-space (see \autoref{thm:Liou-intro} for a simplified version). Next, in Section \ref{sec:grazing-expansions}, we establish expansions at grazing boundary points $z_0 \in \gamma_0$ and prove \autoref{lem.exp.nd.a}, which implies in particular \eqref{eq:expansion-1}, \eqref{eq:expansion-2}. In Sections \ref{sec:inf-low} and \ref{sec:diffuse}, we establish our main results for solutions to \eqref{eq:FPK} with coefficients in the special case $\Omega = \{ x_n > 0 \}$, and discuss their sharpness. In Section \ref{sec:main-proof}, we extend these results to general domains, thereby establishing \autoref{thm:main-diffuse}, \autoref{thm:in-flow}, \autoref{thm:f/phi}, and \autoref{thm:C3}. Finally, this article comes with four appendices, the most notable one being Appendix \ref{sec:appendix}, where we find explicit countable families of homogeneous solutions to \eqref{eq:stat-Kolmo} with absorbing and prescribed in-flow conditions and derive fine estimates on their behavior close to $\gamma_0$. We believe these results to be of independent interest.

\subsection{Acknowledgments}

Kyeongbae Kim and Marvin Weidner were supported by the Deutsche Forschungsgemeinschaft (DFG,
German Research Foundation) under Germany's Excellence Strategy - EXC-2047/1 - 390685813 and through the CRC 1720.

\section{Preliminaries}
\label{sec:prelim}

In this section, we introduce some notation, function spaces, introduce our notion of weak solutions, and provide several auxiliary lemmas.

In what follows, we denote general constants by $c$. When relevant, its dependence on parameters is specified in parentheses; for example, $c=c(n,\Lambda)$ indicates that the constant $c$ depends only on $n$ and $\Lambda$. In addition, we write $a\eqsim b$, when $\frac1c a\leq b\leq cb$ for some universal constant $c\geq1$.

First, let us introduce some geometric notation:
\begin{itemize}
    \item We write points $X_0\in \bbR^{2n}$ and $z_0\in \bbR^{2n+1}$ as
    \begin{align*}
        X_0=(x_0,v_0)\in\bbR^n\times \bbR^n,\quad z_0=(t_0,x_0,v_0)\in \bbR\times \bbR^n\times \bbR^n.
    \end{align*}
    In addition, we denote by $x_0'\in\bbR^{n-1}$ the first $n-1$ components of $x_0=(x_0',x_{0,n})\in \bbR^{n-1}\times \bbR$. In particular, for any $z_0=(t_0,x_0,v_0)\in\bbR^{2n+1}$, we write $z_0'=(t_0,x_0',v_0')\in \bbR^{2(n-1)+1}$, where $x_0=(x_0',x_{0,n}),v_0=(v_0',v_{0,n})$.
    \item We recall the left group action defined by 
    \begin{align*}
        z_0\circ z=(t_0+t,x_0+x+tv_0,v+v_0),\quad z_0^{-1}=(-t_0,-x_0+t_0v_0,-v_0).
    \end{align*}
    In particular, we write 
    \begin{equation*}
        |z_0-z_1|=|z_{1}^{-1}\circ z_0|\coloneqq \max\{|t_0-t_1|^{\frac12},|x_0-(x_1+(t_0-t_1)v_1)|^{\frac13},|v_0-v_1|\}
    \end{equation*}
    and
    \begin{equation*}
        |X_0-X_1|=\max\{|x_0-x_1|^{\frac13},|v_0-v_1|\}.
    \end{equation*}
    \item For any $r>0$ and $z\in\bbR^{2n+1}$, we write $S_rz\coloneqq (r^2t,r^3x,rv)$.
    \item For any $v_0\in\bbR^n$, we write $\langle v_0\rangle\coloneqq 1+|v_0|$.
    \item Kinetic cylinders and their intersection with a domain are defined as follows
    \begin{align*}
    &{\mathcal{Q}}_R(z_0)=\left\{(t,x,v)\,:\,t\in (t_0-R^2,t_0+R^2) =: I_R(t_0), ~~ x\in B_{R^3}(x_0+(t-t_0)v_0), ~~ v\in B_R(v_0)\right\},\\
    &{\mathcal{H}}_R(z_0)\coloneqq{\mathcal{Q}}_R(z_0)\cap (\bbR\times \{x_n>0\}\times \bbR).
    \end{align*}
    \item Throughout the paper, we always consider a coefficient matrix $A(z):\bbR^{2n+1}\to \bbR^{n^2}$ such that $A$ is uniformly elliptic with ellipticity constant $\Lambda \ge 1$ in the sense that
    \begin{align*}
        \Lambda^{-1}I_n\leq A\leq \Lambda I_n.
    \end{align*}
    \item Given a domain $(a,b)\times \Omega\times V\subset \bbR\times \bbR^n\times \bbR^n$ with $\Omega $ and $V$ being bounded sets, we define the kinetic boundary of $(a,b)\times \Omega\times V$ by 
    \begin{align*}
        \partial_{\mathrm{kin}}((a,b)\times \Omega\times V)=(\{t=a\}\times \Omega\times V) \cup((a,b)\times \Omega\times\{v\in\partial V\})\cup \gamma_-.
    \end{align*}
\end{itemize}

We recall the kinetic distance and kinetic H\"older spaces as in \cite{ImSi21}. For any $z_1,z_2$, we write 
\begin{align*}
    \mathrm{dist}(z_1,z_2)\coloneqq \min_{w\in\bbR^n}\left\{\max\left\{|t_1-t_2|^{\frac12}, |x_1-(x_2+(t_1-t_2)w|^{\frac13},|v_1-w|,|v_2-w|\right\}\right\}.
\end{align*}
Then we observe $|z_1-z_2|\eqsim  \mathrm{dist}(z_1,z_2)$.

For any $k\in\N\cup \{0\}$ and $\eps\in(0,1]$, we write
\begin{align*}
    C^{k,\eps}(D) &\coloneqq \{f:D\to\bbR\,:\,\|f\|_{C_{\ell}^{k,\eps}(D)}\coloneqq \|f\|_{L^\infty(D)}+[f]_{C^{k,\eps}(D)}<\infty\},\\
    [f]_{C^{k,\eps}(D)} &\coloneqq \sup_{z_1\in D}\sup_{p\in \cP_k}\sup_{z_1\neq z_2\in D}\frac{|(f-p)(z_2)|}{\mathrm{dist}(z_1,z_2)^{k+\eps}},
\end{align*}
and we denote by $\mathcal{P}_k$ the space of kinetic polynomials of kinetic degree $k$ (see \cite{RoWe25} for the precise definition).
Note that we use the following convention. When we specify $f\in C^{\alpha}(\mathcal{H}_R(z_0))$ with $\alpha\leq0$, it means $f\in L^\infty(\mathcal{H}_R(z_0))$. In the same way, $[f]_{C^{\alpha}(\mathcal{H}_R(z_0))}=\|f\|_{L^\infty(\mathcal{H}_R(z_0))}$ when $\alpha\leq 0$. 

Moreover, we will use the notation $C^{k,\alpha} = C^{k+\alpha}$ interchangeably whenever $k+\alpha \not\in \N$.

In addition, we denote by a $k$-the order differential operator as ${D}^k$. For example ${D}^2$ can be $\partial_{v_i,v_j}$ or $\partial_t+v\cdot\nabla_x$, and $D^3$ can be $\partial_{v_i,v_j,v_k}$, $\partial_{v_i}(\partial_t+v\cdot\nabla_x)$ or $\partial_{x_i}$.

\subsection{Weak solution concept}

We now describe the definition of a weak solution to kinetic Fokker-Planck equations at the boundary.
\begin{definition}
Let $\mathcal{Q}\subset \bbR\times \Omega\times \bbR^n$ be an open set. Let $B\in L^\infty(\mathcal{Q})$, $F \in L^{2}(\mathcal{Q})$ and $g \in L^{\infty}(\gamma_-)$. We say that $f$ is a weak solution to 
\begin{equation}\label{eq.def.para}
\left\{
\begin{alignedat}{3}
(\partial_t+v\cdot\nabla_x)f-\ddiv(A\nabla_vf)&=B\cdot\nabla_vf+F&&\qquad \mbox{in $\mathcal{Q}$}, \\
f&=g&&\qquad  \mbox{in $ \partial_{\mathrm{kin}}\mathcal{Q}$},
\end{alignedat} \right.
\end{equation}
if $ \nabla_v f\in L^2(\mathcal{Q})$ and $f$ satisfies
\begin{align*}
    -\int_{\mathcal{Q}}f(\partial_t+v\cdot\nabla_x)\psi \,dz+\int_{\mathcal{Q}}A\nabla_vf\nabla_v\psi\,dz=\int_{\mathcal{Q}}B\cdot\nabla_vf\psi\,dz-\int_{\gamma_-}g\psi (v\cdot n)\,d\gamma+\int_{\mathcal{Q}}F\psi\,dz
\end{align*}
for any $\psi\in C_c^\infty(\bbR^{1+2n})$ with $\supp \psi\cap \partial Q\subset \gamma_-\cup \gamma_0$. 

When $\mathcal{Q}=\bbR\times \{x_n>0\}\times\bbR^n$, $h$ is a weak solution to \eqref{eq.def.para} if $h\in L^2_{\mathrm{loc}}(\bbR\times \{x_n>0\};H^1_{\mathrm{loc}}(\bbR^n))$, and
\begin{align*}
   -\int_{\mathcal{Q}}f(\partial_t+v\cdot\nabla_x)\psi \,dz+\int_{\mathcal{Q}}A\nabla_vf\nabla_v\psi\,dz=\int_{\mathcal{Q}}B\cdot\nabla_vf\psi\,dz-\int_{\gamma_-}g\psi (-v_n)\,d\gamma+\int_{\mathcal{Q}}F\psi\,dz
\end{align*}
for any $\psi\in C_c^\infty(\bbR^{1+2n})$ with $\supp \psi\cap \partial \mathcal{Q}\subset \{x_n=0\}\times \{v_n\geq0\}$.
\end{definition}

Note that the definition implies that $v \cdot \nabla_x f \in L^2_{t,x} H_v^{-1}(\mathcal{Q})$ by the same argument as in \cite[p.15-16]{Sil22}.

\begin{remark}\label{rmk.equ.weak.str.par}
    Note that if the function $f\in C(\overline{\mathcal{Q}})$ with $\nabla_v f\in L^2(\mathcal{Q})$, $v\cdot\nabla_xf\in L^2(\mathcal{Q})$  and $f=g$ on $\partial_{\mathrm{kin}}\mathcal{Q}$ satisfies in the classical sense
    \begin{equation*}
        (\partial_t+v\cdot\nabla_x)f-\Delta_vf=F\quad\text{in }\mathcal{Q},
    \end{equation*}
     then it is a weak solution to \eqref{eq.def.para}. This was already observed in [Definition 5.1 Sil 22].
\end{remark}

\begin{lemma}\label{lem.classical}
    Let $f\in C(\overline{\mathcal{H}_R})$ be a classical solution to 
    \begin{equation*}
\left\{
\begin{alignedat}{3}
(\partial_t+v\cdot\nabla_x)f-\Delta_vf&=F-\ddiv(G)&&\qquad \mbox{in $\mathcal{H}_R$}, \\
f&=0&&\qquad  \mbox{in $ \gamma_-\cap\mathcal{Q}_R$},
\end{alignedat} \right.
\end{equation*}
where $F,G\in L^2(\mathcal{H}_R)$. If $\nabla_vf\in L^2(\mathcal{H}_R)$, then $f$ is a weak solution.
\end{lemma}
\begin{proof}
    Let us take a function $\psi\in C_c^\infty(\bbR^{1+2n})$ with $\supp \psi\cap \partial\mathcal{H}_R\subset \{x_n=0\}\times \{v_n\geq0\}$. Since $f$ is a classical solution and $F\in L^2(\mathcal{H}_R)$, we have 
    \begin{align}\label{j.weak.classical}
        \int_{\mathcal{H}_R}F\psi+G\nabla_v\psi\,dz=\int_{\mathcal{H}_R}(\partial_t+v\cdot\nabla_x-\Delta_v)f\psi\,dz\eqqcolon J.
    \end{align}
    Now using the integration by parts together with the fact that $f\in C(\overline{H_R})$, $\nabla_vf\in L^2(\mathcal{H}_R)$, $\supp \psi\cap \partial\mathcal{H}_R\subset \{x_n=0\}\times \{v_n\geq0\}$ and $f\equiv 0$ in $ \gamma_-\cap\mathcal{Q}_R$, we estimate $J$ as 
    \begin{align*}
        J=\int_{\mathcal{H}_R}-f(\partial_t+v\cdot\nabla_x)\psi+\nabla_vf\nabla_v\psi\,dz.
    \end{align*}
    Now plugging this into \eqref{j.weak.classical} shows that $f$ is a weak solution. This completes the proof.
\end{proof}

\subsection{Interior regularity} 
Here, we present two interior regularity results for kinetic Fokker-Planck equations.
First, we recall the interior regularity estimates from \cite[Corollary 3.3]{ImMo21}.
\begin{lemma}\label{lem.int.grad}
    Let $g$ be a weak solution to 
    \begin{align*}
        (\partial_t+v\cdot\nabla_x)g-\Delta_vg=0\quad\text{in }\mathcal{Q}_R.
    \end{align*}
    Then for any nonnegative integer $k$, 
    \begin{align*}
        \|D^kg\|_{L^\infty(\mathcal{Q}_{R/2})}\leq cR^{-k}\|g\|_{L^\infty(\mathcal{Q}_R)}
    \end{align*}
    for some constant $c=c(n,k)$.
\end{lemma}
Next, we provide interior Schauder type estimates in divergence form.
\begin{lemma}\label{lem.sch.divergence}
    Let $f$ be a weak solution to 
    \begin{align*}
         (\partial_t+v\cdot\nabla_x)f-\ddiv(A\nabla_vf)=B\cdot\nabla_vf+F-\ddiv(G)\quad\text{in }\mathcal{Q}_1.
    \end{align*}
    If $A\in C^{\eps}(\mathcal{Q}_1)$, $B\in L^\infty(\mathcal{Q}_1)$, $F\in L^\infty(\mathcal{Q}_1)$, and $G\in C^{\eps}(\mathcal{Q}_1)$, then we have 
\begin{align*}
    \|f\|_{C^{1,\eps}(\mathcal{Q}_{\frac12})}\leq c(\|f\|_{L^1(\mathcal{Q}_1)}+\|F\|_{L^\infty(\mathcal{Q}_1)}+[G]_{C^\eps}(\mathcal{Q}_1)),
\end{align*}
where $c=c(n,\Lambda,\eps,\|A\|_{C^\eps(\mathcal{Q}_1)},\|B\|_{L^\infty(\mathcal{Q}_1)})$.
\end{lemma}
Since the proof closely follows the argument in \cite{KLN25}, we defer the details to \autoref{rmk.l8delta}.

\section{A Liouville theorem for stationary solutions in 1D}
\label{sec:Liouville-1D}
The goal of this section is to prove a Liouville theorem in the half-space in 1D (see \autoref{thm.liou.inflow}). Its proof is split into several steps. 
\begin{itemize}
    \item First, in Subsection \ref{subsec:expansion} we establish an expansion of order $\frac{1}{2} + \alpha$ for some small $\alpha > 0$ at $(0,0)$ in terms of $\phi_0$. In particular, it yields H\"older continuity of $h/\phi$ in $\mathcal{R}_0$ and already implies a Liouville theorem for functions that do not grow faster than $1 + |z|^{\frac{1}{2} + \alpha}$.

    \item Second, in Subsection \ref{subsec:higher-Liouville} we improve the order of the Liouville theorem by a delicate argument based on the maximum principle and the asymptotic behavior of $\phi_0$ to functions that do not grow faster than $1 + |z|^{\frac{1}{2} + 3 - \eps}$.

    \item Third, in Subsection \ref{subsec:inflow-Liouville} we find explicit homogeneous solutions to inhomogeneous equations in the half-space with in-flow condition. A combination with the Liouville theorem from the previous step concludes the proof of \autoref{thm.liou.inflow}.
\end{itemize}

We start by introducing stationary kinetic cylinders in dimension one, as follows:
\begin{equation}\label{defn.staty.cylinder}
\begin{aligned}
     &Q_{R}(x_0,v_0)=(x_0-R^3,x_0+R^3)\times (v_0-R,v_0+R),\\
    &H_{R}(x_0,v_0)=(x_0-R^3,x_0+R^3)\times (v_0-R,v_0+R)\cap \{x>0\}\times \bbR,\\
    &H_{R}^{+}\coloneqq ((R/2)^3,R^3)\times (-R,R).
\end{aligned}
\end{equation}

\subsection{Harnack's type inequality for the stationary case}
First, let us introduce a suitable notion of weak solutions to the stationary version of the kinetic Fokker-Planck equation in align with \eqref{eq.def.para}

\begin{definition}
\label{deff.weak.stationary}
Let $Q=(a,b)\times (c,d)$ and $g\in C(\overline{Q})$. We say that $h$ is a weak solution to 
\begin{equation}\label{eq.def}
\left\{
\begin{alignedat}{3}
v\partial_xh-\partial_{vv}h&=F&&\qquad \mbox{in $Q$}, \\
h&=g&&\qquad  \mbox{in $ \partial_{\mathrm{kin}}Q$},
\end{alignedat} \right.
\end{equation}
if $ h\in L^2((a,b);H^1(c,d))$ and $h$ satisfies
\begin{align*}
    -\iint_{Q}h(v\partial_x\psi )\,dx\,dv+\iint_{Q}\partial_vh\partial_v\psi\,dx\,dv=-\int_{\gamma_-}g\psi (v\cdot n)\,d\gamma+\int_{Q}F\psi\,dx\,dv
\end{align*}
for any $\psi\in C_c^\infty(\bbR^{2})$ with $\supp \psi\cap \partial Q\subset \gamma_-\cup \gamma_0$. 

When $Q=(0,\infty)\times \bbR$, we say that $h$ is a weak solution to \eqref{eq.def} if $h\in L^2_{\mathrm{loc}}(0,\infty;H^1_{\mathrm{loc}}(\bbR))$, and
\begin{align}\label{def.weak.stationary}
    -\iint_{Q}h(v\partial_x\psi )\,dx\,dv+\iint_{Q}\partial_vh\partial_v\psi\,dx\,dv=-\int_{\{x=0\}\times \{v>0\}}g\psi (-v)\,d\gamma
\end{align}
for any $\psi\in C_c^\infty(\bbR^{2})$ with $\supp \psi\cap \partial Q\subset \{x=0\}\times \{v\geq0\}$.
\end{definition}
\begin{remark}\label{rmk.exist.stationary}
Now, we give two remarks about the existence and regularity of solutions to \eqref{eq.def}.
\begin{itemize}
\item Note that when $Q$ is a bounded domain and $g\in C^\infty(\overline{Q})$, by \cite[Theorem 1, Theorem 3 and Definition 3.3]{AveHou26}, there is a unique weak solution $h$ to \eqref{eq.def}. Note that, by the remark given after \cite[Definition 5.1]{Sil22} and by \cite[Lemma 3.1, Definition 3.3]{AveHou26}, the notion of a weak solution given in \cite{AveHou26} coincides with the one in \autoref{deff.weak.stationary}. 
\item By considering $h(t,x,v)=h(x,v)$, we observe that $h$ is also a weak solution to 
\begin{align*}
    (\partial_t+v\cdot\nabla_x)h-\partial_{vv}h=F\quad\text{in }(-1,0]\times Q
\end{align*}
with $h=g$ on $I\times \partial_{\mathrm{kin}}Q$. Thus, by \cite{Sil22} and \autoref{lem.int.grad}, if $g\in  C^\alpha(\overline{Q})$ and $F\in C^\infty(Q)$, then  we have $h\in C_{\mathrm{loc}}^\infty(Q)\cap C^\alpha(\overline{Q})$.  
\end{itemize}
\end{remark}
In this subsection, we denote by $\widetilde{\mathcal{Q}}$ the one-sided kinetic cylinder defined by 
\begin{align*}
    \widetilde{\mathcal{Q}}_r(z_0)\coloneqq \left\{z\,:\, t\in (t_0-r^2,t_0],\, v\in B_{r}(v_0),\, x\in B_{r^3}(x_0+(t-t_0)v_0)\right\}.
\end{align*}

First, based on the parabolic version of Harnack's inequality given in \cite{GIMV19}, we provide Harnack's inequality for the stationary case. The cylinders are slanted in $x$ due to the underlying kinetic geometry.
\begin{lemma}\label{lem.har.sta}
    Let $h\geq0$ be a weak solution to 
\begin{equation}\label{eq.har.sta}
v \partial_xh-\partial_{vv}h=0\quad\text{in }Q.
\end{equation}
There are constants $\theta,\gamma\in(0,1/2)$, and $c \ge 1$ such that for any $R \in (0,2]$,
\begin{align*}
\sup_{A_\gamma\times B_{\theta R}(v_0)}h\leq c\inf_{A_0\times B_{\theta R}(v_0)}h,
\end{align*}
where
\begin{align*}
    A_\gamma &\coloneqq \{x\,:\,|x-x_0-v_0(t-R^2\gamma)|<(\theta R)^3,\,t\in(-(R\theta)^2,0]\} \\
    &= \begin{cases}
    \big( x_0 - v_0 R^2 (\gamma + \theta^2) - (\theta R)^3 , x_0 - v_0 R^2 \gamma + (\theta R)^3\big) ~~ \text{ if } v_0 > 0,\\
    \big( x_0 - v_0 R^2 \gamma - (\theta R)^3 , x_0 - v_0 R^2 (\theta^2 + \gamma) + (\theta R)^3 \big) ~~ \text{ if } v_0 < 0
    \end{cases},
\end{align*} whenever $(x_0,v_0)$ is such that $\{x\,:\,|x-(x_0+v_0t)|< R^3\,:t\in(-R^2,0]\}\times B_{R}(v_0)\subset Q$. 
\end{lemma}
\begin{proof}
  Note that if $\{x\,:\,|x-(x_0+v_0t)|< R^3\,:t\in(-R^2,0]\}\times B_{R}(v_0)\subset Q$, then $\widetilde{\mathcal{Q}}_R(0,x_0,v_0)\subset (-R^{2},0]\times Q$.  
    Next, we observe that $h=h(t,x,v)$ is a weak solution to 
    \begin{align*}
        (\partial_t+v\cdot\partial_x)h-\partial_{vv}h=0\quad\text{in }(-R^2,0]\times Q.
    \end{align*}
    Now, applying the interior Harnack inequality given in \cite{GIMV19} that there are constants $\theta,\gamma\in(0,1/2)$, and $c \ge 1$ such that
\begin{align*}
    \sup_{\widetilde{\mathcal{Q}}_{\theta R}\left(-R^2\gamma,x_0,v_0\right)}h\leq c\inf_{\widetilde                                     {\mathcal{Q}}_{\theta R}(0,x_0,v_0)}h
\end{align*}
for some constant $c\geq1$. Therefore, we deduce the desired estimate.
\end{proof}

As usual, the Harnack inequality can be applied iteratively (Harnack chains) to compare the values of the solution also in more general domains. In particular, we can show a comparability result for the values of $h(\cdot,v_0)$ for any fixed $v_0$. Due to the underlying kinetic geometry, the ''direction`` of the comparability depends heavily on the sign of $v_0$.
\begin{lemma}\label{lem.strmax}
    Let $h\geq0$ be a weak solution to 
    \begin{equation*}
v\partial_xh-\partial_{vv}h=0\quad \mbox{in  $H_{\mathcal{R}}$}.
\end{equation*}
If $v_0\in[0,\mathcal{R}/2]$, then for any $0<x_0 \le x_1<\mathcal{R}^3/2$, we have 
\begin{align}\label{v0non.strmax}
    h(x_0,v_0)\leq ch(x_1,v_0),
\end{align}
where the constant $c$ depends only on $x_0$, $v_0$, and $\mathcal{R}$. 

If $v_0\in[-\mathcal{R}/2,0)$, then for any $0\leq x_0\leq x_1<\mathcal{R}^{3}/2$, we have
\begin{align}\label{v0neg.strmax}
    h(x_1,v_0)\leq ch(x_0,v_0),
\end{align}
where the constant $c$ depends only on $|v_0|$ and $\mathcal{R}$.
\end{lemma}

\begin{remark}
    Note that it is possible to choose $x_0 = 0$ in case $v_0 < 0$. In addition, the constant $c$ of \eqref{v0non.strmax} and \eqref{v0neg.strmax} also depends on $\theta ,\gamma\in(0,1/2)$ which are determined in \autoref{lem.har.sta}. However, both constants $\theta,\gamma$ are universal so that we just write the constant $c$ depends only on $x_0,v_0,\mathcal{R}$ or $|v_0|,\mathcal{R}$. 
\end{remark}

\begin{proof}
    We prove the statements for $v_0 > 0$ and $v_0 < 0$ separately.
    \begin{itemize}
        \item Let $v_0\geq0$ and $0<x_0<x_1$. First, note that there is $R_0\leq \mathcal{R}/2$ such that 
        \begin{align}\label{ineq.s0v0.har}
            x_0-v_0(2R_0)^2-(2R_0)^3=0,
        \end{align}
        as $g(r)=x_0-v_0(2r)^2-(2r)^3$ is non-increasing when $r\geq0$ and 
        \begin{align}\label{g.inc.strmax}
            g\left(\left(\frac{x_0}{16v_0}\right)^{\frac12}\right)>0>g\left(\left(\frac{x_0}{4v_0}\right)^{\frac12}\right).
        \end{align}
        Moreover, \eqref{g.inc.strmax} yields
        \begin{align}
        \label{g.R_0-bd}
            \left(\frac{x_0}{16v_0}\right)^{\frac12}<R_0<\left(\frac{x_0}{4v_0}\right)^{\frac12},
        \end{align}
        which will ensure that \eqref{v0non.strmax} only depends on $R_0$ through $x_0$ and $v_0$. We now choose 
        \begin{equation*}
            R_1\coloneqq {\min\left\{R_0,\frac{\sqrt{x_1-x_0}}{2(1+\sqrt{v_0})}\right\}}
        \end{equation*}
        to see that there is a nonnegative integer $k$ such that
        \begin{align*}
            \frac{x_1-x_0}{k+1}\leq v_0R_1^2\gamma<\frac{x_1-x_0}{k},
        \end{align*}
        where the constant $\gamma\in(0,1/2)$ is determined in \autoref{lem.har.sta}. Note that by the definition of $R_1$ and due to \eqref{g.R_0-bd}, there is some $k_0 > 0$, depending only on $x_0,v_0$, and $\mathcal{R}$, such that $k \le k_0$. Thus, there is $R\leq R_1$ such that 
        \begin{align}\label{equl.strmax}
            \frac{x_1-x_0}{{k_0+1}}=v_0R^2\gamma.
        \end{align}
        We now define
        \begin{align}\label{defn.yj.strmax}
            y_j\coloneqq x_1-(\theta R)^3-v_0R^2\gamma j
        \end{align}
        to see that 
        \begin{equation}\label{equl2.strmax}
            \mbox{$x_1=y_0+(\theta R)^3$ and $x_0=y_{k_0+1}+(\theta R)^3$}.
        \end{equation}
        In addition, we observe from \eqref{ineq.s0v0.har} and \eqref{equl.strmax} that
        \begin{align*}
            \{x\,:\,|x-(y_j+v_0t)|<R^3\,:\,t\in(-R^2,0]\}\subset (0,\mathcal{R}^3)
        \end{align*}
        for any $j=0,1,\ldots , k_0+1$.
        Now using \eqref{defn.yj.strmax} and applying \autoref{lem.har.sta} with $x_0$ replaced by $y_j$, we deduce
        \begin{align*}
            \sup_{(y_{j+1}-v_0\theta^2 R^2-(\theta R)^3,y_{j+1}+(\theta R)^3)\times B_{\theta R}(v_0)}h&= \sup_{(y_j-v_0(\theta^2+\gamma) R^2-(\theta R)^3,y_j-v_0R^2\gamma+(\theta R)^3)\times B_{\theta R}(v_0)}h\\
            &\leq c\inf_{(y_j-v_0\theta^2 R^2-(\theta R)^3,y_j+(\theta R)^3)\times B_{\theta R}(v_0)}h
        \end{align*}
        Now using a Harnack chain together with \eqref{equl2.strmax}, we derive 
        \begin{equation}\label{goal1.strmax}
        \begin{aligned}
            h(x_0,v_0)&\leq \sup_{(y_{k+1}-v_0\theta^2 R^2-(\theta R)^3,y_{k+1}+(\theta R)^3)\times B_{\theta R}(v_0)}h\\
            &\leq \cdots
            \leq c^{k+1}\inf_{(y_0-v_0\theta^2 R^2-(\theta R)^3,y_0+(\theta R)^3)\times B_{\theta R}(v_0)}h\leq c^{k_0+1}h(x_1,v_0).
        \end{aligned}
        \end{equation}
       \item Let $v_0<0$ and $0\leq x_0\leq x_1$. Note that $h=h(t,x,v)$ is a weak solution to 
       \begin{align*}
           (\partial_t+v\partial_x)h-\partial_{vv}h=0\quad\text{in }(-\mathcal{R}^2,0]\times H_{\mathcal{R}}.
       \end{align*}
       Since $\{x=0\}\times [-2|v_0|,-|v_0|/4]\subset \gamma_+$, by \cite{Sil22} and \cite{RoWe25}, $h$ is smooth in $[0,7\mathcal{R}^3/8]\times [-2|v_0|,-|v_0|/4]$. Thus, there is a smooth extension $H\geq0$ of $h$  that defined in $(-3\mathcal{R}^3/4,3\mathcal{R}^3/4)\times (-3|v_0|/2,-|v_0|/2)$ such that $H=h$ in $[0,3\mathcal{R}^3/4]\times [-3|v_0|/2,-|v_0|/2]$. By \autoref{rmk.exist.stationary}, there is a unique weak solution $F\geq0$ such that 
       \begin{equation*}
\left\{
\begin{alignedat}{3}
v\partial_xF-\partial_{vv}F&=0&&\qquad \mbox{in  $(-3\mathcal{R}^3/4,3\mathcal{R}^3/4)\times (-3|v_0|/2,-|v_0|/2)$}, \\
F&=H&&\qquad  \mbox{on $ \{x=3\mathcal{R}^3/4\}\times[-3|v_0|/2,-|v_0|/2]\cup [-3\mathcal{R}^3/4,3\mathcal{R}^3/4]\times\{v=-|v_0|\pm |v_0|/2\}$}.
\end{alignedat} \right.
\end{equation*}
Since $h$ is a smooth solution to 
\begin{equation*}
\left\{
\begin{alignedat}{3}
v\partial_xh-\partial_{vv}h&=0&&\qquad \mbox{in  $(0,3\mathcal{R}^3/4)\times (-3|v_0|/2,-|v_0|/2)$}, \\
h&=H&&\qquad  \mbox{on $ \{x=3\mathcal{R}^3/4\}\times[-3|v_0|/2,-|v_0|/2]\cup [0,3\mathcal{R}^3/4]\times\{v=-|v_0|\pm |v_0|/2\}$},
\end{alignedat} \right.
\end{equation*}
by \autoref{rmk.exist.stationary}, $F=h$ in $[0,3\mathcal{R}^3/4]\times [-3|v_0|/2,-|v_0|/2]$. As $|v_0|\leq \mathcal{R}/2$ and $x_1\leq\mathcal{R}^3/2$, the constant 
\begin{equation*}
    R_0\coloneq\frac{\mathcal{R}}{4(\sqrt{|v_0|}+1)}
\end{equation*}satisfies
\begin{equation*}
    x_1-v_0(2R_0)^2+(2R_0)^3\leq 3\mathcal{R}^3/4.
\end{equation*}
We now proceed in a similar way as in the first step to find a suitable sequence as in \eqref{defn.yj.strmax}. 

First, we select
        \begin{equation*}
            R_1\coloneqq \min\left\{R_0,\frac{\sqrt{x_1-x_0}}{2(1+\sqrt{|v_0|})}\right\}
        \end{equation*}
        to see that there is a nonnegative integer $k$ such that
        \begin{align*}
            \frac{x_1-x_0}{{k+1}}\leq -v_0R_1^2\gamma<\frac{x_1-x_0}{k}.
        \end{align*}
        As before, there is $k_0 > 0$, depending only on $v_0$ and $\mathcal{R}$ such that $k \le k_0$.
        Moreover, there is $R$ such that
        \begin{align*}
            \frac{x_1-x_0}{{k_0+1}}=-v_0R^2\gamma.
        \end{align*}
        By taking $y_j\coloneqq x_0+(\theta R)^3 - v_0R^2\gamma j$, we get
        \begin{align}\label{equl3.strmax}
            x_0=y_{0}-(\theta R)^3\quad\text{and}\quad x_1=y_{k_0+1}-(\theta R)^3
        \end{align}
        and
        \begin{align*}
    \{x\,:\,|x-(y_j+v_0t)|<R^3\,:\,t\in(-R^2,0]\}\times B_{\theta R}(v_0)\subset Q
\end{align*}
for any $j=0,1,\ldots, k_0$, where we write $Q = (-3\mathcal{R}^3/4,3\mathcal{R}^3/4)\times (-3v_0/2,-v_0/2)$.

We are now going to apply \autoref{lem.har.sta} with $x_0$ replaced by $y_j$ to the function $F$ to see that 
\begin{equation}\label{ineq2.strmax}
\begin{aligned}
   \sup_{(y_{j+1}-(\theta R)^3,y_{j+1}-v_0R^2\theta^2+(\theta R)^3)\times B_{\theta R}(v_0)}F&= \sup_{(y_j-v_0R^2\gamma-(\theta R)^3,y_j-v_0R^2(\theta^2+\gamma)+(\theta R)^3)\times B_{\theta R}(v_0)}F\\
    &\leq c\inf_{(y_j-(\theta R)^3,y_j-v_0R^2\theta^2+(\theta R)^3)\times B_{\theta R}(v_0)}F,
\end{aligned}
\end{equation}
where we have used \eqref{equl3.strmax} for the first equality.

As in \eqref{goal1.strmax}, along with \eqref{equl3.strmax} and \eqref{ineq2.strmax}, we deduce
        \begin{equation}\label{goal2.strmax}
        \begin{aligned}
            h(x_1,v_0)=F(x_1,v_0)&\leq \sup_{(y_{k_0+1}-(\theta R)^3,y_{k_0+1}-v_0R^2\theta^2+(\theta R)^3)\times B_{\theta R}(v_0)}F\\
            &\leq \cdots
            \leq c^{k_0+1}\inf_{(y_0-(\theta R)^3,y_0-v_0R^2\theta^2+(\theta R)^3)\times B_{\theta R}(v_0)}F\\
            &\leq c^{k_0+1}F(x_0,v_0)=c^{k_0+1}h(x_0,v_0).
        \end{aligned}
        \end{equation}
    \end{itemize}
    By \eqref{goal1.strmax} and \eqref{goal2.strmax}, we get the desired result.
\end{proof}

\subsection{An expansion in 1D}
\label{subsec:expansion}

Throughout this subsection, we denote $\phi := \phi_0$, where $\phi_0$ is defined as
\begin{align*}
        \phi_0(x,v) = \begin{cases}
            x^{\frac16}U\left(-\frac16,\frac23,-\frac{v^3}{9x}\right)&\quad\text{if }v<0,\\
            \frac16x^{\frac16}e^{-\frac{v^3}{9x}}U\left(\frac56,\frac23,\frac{v^3}{9x}\right)&\quad\text{if }v\geq0.
        \end{cases}
\end{align*}
See also \eqref{rel.phiandU}. The goal of this subsection is to prove the following oscillation estimate for $h/\phi$ at $(0,0)$, where $h$ is a solution to the stationary Kolmogorov's equation with absorbing boundary condition in 1D:

\begin{proposition}
\label{lemma:1D.expansion.gamma0}
    Let $R \ge 1$ and $h$ be a weak solution to 
    \begin{equation*}
\left\{
\begin{alignedat}{3}
v \partial_xh-\partial_{vv}h&=0&&\qquad \mbox{in  $H_R$}, \\
h&=0&&\qquad  \mbox{in $ \{x=0\}\times (0,R)$}.
\end{alignedat} \right.
\end{equation*}
Then we have for some constants $c\geq1$ and $R_0\geq1$,
\begin{align}\label{bdd.h/phi}
    |h/\phi|\leq cR^{-\frac12}\|h\|_{L^\infty(H_R)}\quad\text{in }H_{R/R_0}
\end{align}
and
\begin{align}\label{hol.h/phi}
    \sup_{H_{rR}}\frac{h}{\phi}-\inf_{H_{rR}}\frac{h}{\phi}\leq cr^{\alpha}R^{-\frac12}\|h\|_{L^\infty(H_R)}\quad\text{for any }r\leq 1/R_0.
\end{align}
\end{proposition}

The proof of \autoref{lemma:1D.expansion.gamma0} goes in several steps.

First, we prove the following lemma which yields a pointwise lower bound for any nonnegative solution $h$ in the half-space in terms of $\phi$.

\begin{lemma}\label{lemma:1D-expansion:stationary}
    Let $h\geq0$ be a weak solution to 
    \begin{equation}\label{eq.1d.expansion}
\left\{
\begin{alignedat}{3}
v \partial_xh-\partial_{vv}h&=0&&\qquad \mbox{in  $(0,4^{3})\times (-4,4)$}, \\
h&=0&&\qquad  \mbox{on $ \{x=0\}\times (0,4)$}.
\end{alignedat} \right.
\end{equation}
Then we have 
\begin{align}\label{goal.lem.expansion.sta}
    \inf_{H^+_{2}}\frac{h}{\phi}\leq c\inf_{H_1}\frac{h}{\phi}.
\end{align}
\end{lemma}
\begin{proof}
Let us denote 
\begin{align*}
    m_0\coloneqq \inf_{H_2^+}\frac{h}{\phi}
\end{align*}
and assume $m_0>0$, otherwise the proof is done. Then we have $ h\geq m_0\phi$ in $H_2^+$, which implies 
\begin{align}\label{fir.lem.expansion.sta}
   h\geq \frac{m_0}{c_0}\quad\text{in }H_2^{+}
\end{align}
for some constant $c_0\geq1$, where we have also used that $\phi \ge c_0^{-1}$ in $H_2^+$ by \eqref{phi.asymptotic}.

We now divide the remaining proof into three steps.

\textbf{Step 1.} First, we are going to construct a suitable barrier. To do so, let us define a function ${\psi}$ by 
    \begin{align*}
        {\psi}(x,v)\coloneqq (v-2)^2,
    \end{align*}
    and observe that
    \begin{align}\label{propsi.expansion.sta}
        {\psi(v)}\in[1,16] \quad\text{on }[-2,1],\quad {\psi}'(v)\leq0\quad\text{on }(-\infty,2],\quad {\psi}''(v)\geq0,\quad\text{and}\quad {\psi}(2)=0.
    \end{align}
    We now want to prove that the function $\varphi\coloneqq \phi \psi$ is a subsolution to 
    \begin{align}
    \label{psisol.lem.expansion.sta}
        v\partial_x \varphi-\partial_{vv} \varphi=0\quad\text{in }(0,2^3]\times (-2,2).
    \end{align}
    To this end, we use the terms $J_1$ and $J_2$ given in \eqref{appen.defn.j1}, and \eqref{appen.defn.j2} with $\lambda_m=\frac16$, $b=\frac23$ and $a_m=\frac56$, along with \eqref{eq:multiple-U} and \eqref{explicit.U} to see that
    \begin{align*}
        \partial_v \phi(x,v)= \begin{cases}
            \frac16x^{\frac16} \left(-\frac{v^2}{3x}\right) U\left(-\frac16+1,\frac23+1,-\frac{v^3}{9x}\right)<0&\quad\text{if }v<0,\\
            \frac16x^{\frac16}e^{-\frac{v^3}{9x}}\left(-\frac{v^2}{3x}\right) \left[U\left(\frac56,\frac23,\frac{v^3}{9x}\right) + \frac56 U\left(\frac56+1,\frac23+1,\frac{v^3}{9x}\right)\right]<0&\quad\text{if }v\geq0.
        \end{cases}
    \end{align*}
    In light of the fact that $\partial_v \phi \leq0$ and $\phi\geq0$, which follows from \eqref{phi.asymptotic}, we derive
    \begin{align*}
        v\partial_x\varphi-\partial_{vv}\varphi=(v\partial_x\phi-\partial_{vv}\phi)\psi-2\partial_{v}\psi\partial_{v}\phi-\partial_{vv}\psi\phi= -2\partial_{v}\psi\partial_{v}\phi-\partial_{vv}\psi\phi\leq0,
    \end{align*}
    when $v\leq 2$. This implies \eqref{psisol.lem.expansion.sta}.

\textbf{Step 2.} We next prove that for any $x\in[0,1]$,  
\begin{align}\label{secgoal.lem.expansion.sta}
    {h(1,-2)}\leq ch(x,-2).
\end{align}
This follows from \eqref{v0non.strmax} in \autoref{lem.strmax} with $x_1=1$, $v_0=-2$, and $x_0=x\in[0,1]$. In addition, by following the proof of \eqref{v0non.strmax}, we deduce that the constant $c$ is universal.

\textbf{Step 3.} We now use the supersolution $\varphi$ given in \eqref{psisol.lem.expansion.sta} and the estimate \eqref{secgoal.lem.expansion.sta} to derive the desired result with the help of the maximum principle. By \eqref{fir.lem.expansion.sta} and \eqref{secgoal.lem.expansion.sta} we get 
\begin{align*}
    h\geq \frac{m_0}{c_1}\quad\text{on }\left([0,1]\times\{v=-2\}\right)\cup \left(\{x=1\}\cup[-2,0]\right)\eqqcolon \Gamma_1
\end{align*}
for some constant $c_1\geq1$. By \eqref{phi.asymptotic} there is a constant $c_2\geq1$ such that $\frac{\phi}{c_2}\leq \frac{1}{16c_1}$ on $\Gamma_1$. Therefore, we derive 
\begin{align*}
    \frac{m_0\varphi}{c_2}=\frac{m_0\phi\psi}{c_2}\leq \frac{16m_0\phi}{c_2}\leq \frac{m_0}{c_1}\leq h \quad\text{on }\Gamma_1,
\end{align*}where we have also used the first observation in \eqref{propsi.expansion.sta}. In addition, since $\psi(2)=0$ and $\phi = 0$ on $\{ x = 0 \} \times (0,2]$, we get
\begin{align*}
   \frac{m_0\varphi}{c_2}=\frac{m_0\phi\psi}{c_2}=0\leq h \quad\text{on } ([0,1]\times \{v=2\}) \cup (\{ x = 0 \} \times [0,2]) = \Gamma_2.
\end{align*} Using this and \eqref{psisol.lem.expansion.sta}, we have 
\begin{equation*}
\left\{
\begin{alignedat}{3}
(v\partial_x-\partial_{vv})(m_0\varphi/c_2-h)&\leq0&&\qquad \mbox{in  $(0,1)\times(-2,2)$}, \\
m_0\varphi/c_2&\leq h&&\qquad  \mbox{in $ \partial_{\mathrm{kin}}((0,1)\times(-2,2))=\Gamma_1\cup\Gamma_2$},
\end{alignedat} \right.
\end{equation*}
which implies $m_0\varphi/c_2\leq h$ in $H_1$ by the maximum principle (see \cite[Theorem 3]{AveHou26}). Using the fact that $\psi(v)\in [1,9]$ for $v \in [-1,1]$ by \eqref{propsi.expansion.sta},  the desired estimate follows. This completes the proof.
\end{proof}
We next provide a variant of the interior Harnack estimate.
\begin{lemma}\label{lemma:1D.harnack.barrier}
    Let $h\geq0$ be a weak solution to \eqref{eq.1d.expansion}. Then we have
    \begin{align}\label{goal.lem.varhar}
        \sup_{H_{2\theta}^{+}}\frac{h}{\phi}\leq c\inf_{H_{2\theta}^{+}}\frac{h}{\phi},
    \end{align}
    where the constant $\theta\in(0,1)$ is determined in \autoref{lem.har.sta}.
\end{lemma}
\begin{proof}
    By \autoref{lem.har.sta} with $v_0=0$, $R=2$, and $x_0=9\theta^3$, we have 
    \begin{align*}
        \sup_{(\theta^3,8\theta^3)\times B_{2\theta}} h\leq \sup_{\{|x-9\theta^3|<(2\theta)^3\}\times B_{2\theta}} h\leq c\inf_{\{|x-9\theta^3|<(2\theta)^3\}\times B_{2\theta}}h \leq   c\inf_{(\theta^3,8\theta^3)\times B_{2\theta}} h,
    \end{align*}
    as $\{|x-9\theta^3|<(2\theta)^3\}\times B_{2\theta}\subset H_4$. Since $\theta$ is a universal constant, we get $\frac1c\leq \phi\leq c$ in $H_{2\theta}^+$ by \eqref{phi.asymptotic} for some constant $c\geq1$. Therefore, the desired estimate follows immediately.
\end{proof}

Our next goal is to prove a suitable upper bound for solutions $h$ to \eqref{eq.1d.expansion} in terms of $\phi$ (see \autoref{lemma:1D-supsolution}). As a first step, we show that $h$ vanishes to infinite order at $\gamma_-$. The proof follows the same arguments as the ones given in \cite[Section 3]{Zhu25}.
\begin{lemma}\label{lem.quasi}
    Let $\mathcal{R} > 0$ and $h$ be a weak solution to 
   \begin{equation}\label{eq.inf.van}
\left\{
\begin{alignedat}{3}
v \partial_xh-\partial_{vv}h&=0&&\qquad \mbox{in  $(0,(2\mathcal{R})^3)\times (-2\mathcal{R},2\mathcal{R})$}, \\
h&=0&&\qquad  \mbox{on $ \{x=0\}\times (0,(2\mathcal{R})^3)$}.
\end{alignedat} \right.
\end{equation}
Fix $v_0,R>0$ such that $v_0\geq R$. Then there is a constant $C_1\geq1$ such that for any $x\in[0,R^3/2^{10}]$,
\begin{align}\label{goal.quasi}
    |h(x,v_0)|\leq 2e^{-\frac{1}{C_1}\frac{v_0R^2}{x}}\|h\|_{L^\infty(Q)},
\end{align}
whenever $Q\coloneqq [0,x]\times [v_0-\sqrt{2}R/8,v_0+\sqrt{2}R/8]\subset (0,(2\mathcal{R})^3)\times (-2\mathcal{R},2\mathcal{R})$.

\end{lemma}
\begin{proof}
Let us fix $r\leq R^3/C_0$, where the constant $C_0$ will be determined later (see \eqref{c_0.cond.barr}). Define a quasi-distance function by
    \begin{align*}
        \rho(x,v)\coloneqq \left(\frac{R^2}{4r}(x+r)^2-R(x+r)(v-v_0)+64r(v-v_0)^2\right)^{\frac12}.
    \end{align*}
    Then we observe
    \begin{align*}
        \rho(0,v_0)=\frac{Rr^{\frac12}}2\quad\text{and}\quad \rho(r,v_0)=Rr^{\frac12},
    \end{align*}
    and
    \begin{align}\label{bdry.quasi}
         [0,r]\times \left\{v=v_0\pm \frac{\sqrt{2}}{8}R\right\}\subset \left\{(x,v)\,:\,2R^2r\leq\rho^2(x,v)\leq 4R^2r\quad\text{and}\quad x\geq0\right\}.
    \end{align}
    We are now going to prove that
    \begin{align}\label{goal.quasi.exponential}
    h(x,v)\leq \left(\frac{v_0}{4c_2r^2}+\frac1{R^2r}\right)\rho^2(x,v)e^{-\frac{v_0R^4}{c_2\rho^2}+\frac{v_0R^2}{2c_2r}}\|h\|_{L^\infty(Q_r)}\quad\text{in }Q_r
\end{align}
for some constant $c_2\geq 1$, where
\begin{equation}\label{qr.quasi}
    Q_r\coloneqq [0,r]\times \left[v_0-\frac{\sqrt{2}}{8}R,v_0+\frac{\sqrt{2}}{8}R\right].
\end{equation}
To do this, we divide the proof into three steps.

    \textbf{Step 1. }Construct a suitable increasing convex function.
    To do this, we fix a constant $c_2\geq1$ which will be determined later (see at the end of \textbf{Step 2}). Let us define for $\xi \ge 0$,
    \begin{align}\label{defn.Phi.quasi}
        \Phi(\xi)\coloneqq \int_{0}^{\xi}e^{-\frac{v_0R^4}{c_2\tau}}\,d\tau
    \end{align}
to see that 
\begin{align}\label{eq.quasi}
    \tau^2\Phi''(\tau)-\frac{v_0R^4}{c_2}\Phi'(\tau)=0.
\end{align}
Since 
\begin{align*}
    (\tau^2\Phi'(\tau))'=\tau^2\Phi''(\tau)+2\tau\Phi'(\tau)=(c_2^{-1}v_0R^4+2\tau)\Phi'(\tau),
\end{align*}
we get
\begin{align}\label{ineq211.quasi}
    \tau^2\Phi'(\tau)=\int_{0}^{\tau}(c_2^{-1}v_0R^4+2s)\Phi'(s)\,ds\leq (c_2^{-1}v_0R^4+2\tau)\Phi(\tau),
\end{align}
where we have used the fundamental theorem of calculus together with the fact that $\Phi'(s)\geq0$ and $\Phi(0)=0$. Thus, by setting 
\begin{align}\label{defn.psi.quasi}
    \Psi(\xi)\coloneqq\frac{\Phi(\xi)}{\Phi(2R^2r)},
\end{align} we derive 
\begin{equation}\label{ineq21.quasi}
\begin{aligned}
    \Psi(\xi)=\frac{\Phi(\xi)-\Phi(0)}{\Phi(2R^2r)}\leq \frac{c_2^{-1}v_0R^4+4R^2r}{(2R^2r)^2}\frac{\xi\Phi'(\xi)}{\Phi'(2R^2r)}&=\frac{c_2^{-1}v_0R^4+4R^2r}{(2R^2r)^2}\frac{\xi e^{-\frac{v_0R^4}{c_2\xi}}}{\Phi'(2R^2r)}\\
    &\leq \left(\frac{v_0}{4c_2r^2}+\frac{1}{R^2r}\right)\xi e^{-\frac{v_0R^4}{c_2\xi}+\frac{v_0R^2}{2c_2r}},
\end{aligned}
\end{equation}
where we have also used the mean value theorem, \eqref{ineq211.quasi} with $\tau=2R^2r$ and the fact that $\Phi'\geq0$ is non-decreasing.

\textbf{Step 2.} Construct a suitable barrier function. More precisely, we want to prove that if $C_0=2^{10}$ and $c_2$ is sufficiently large, then
\begin{align}\label{eq2.quasi}
    v\partial_x (\Phi(\rho^2))-\partial_{vv}(\Phi(\rho^2))\geq 0\quad\text{in }Q_r,
\end{align}
where the domain $Q_r$ is given in \eqref{qr.quasi}.

To do so, we note
\begin{align*}
    \partial_x (\Phi(\rho^2))=\left(\frac{R^2}{2r}(x+r)-R(v-v_0)\right)\Phi'(\rho^2)
\end{align*}
and
\begin{align*}
    \partial_{vv}(\Phi(\rho^2))=128r \Phi'(\rho^2)+\Phi''(\rho^2)\left(-R(x+r)+128r(v-v_0)\right)^2.
\end{align*}
Thus, we get for any $(x,v)\in Q_r$,
\begin{align*}
    v\partial_x (\Phi(\rho^2))-\partial_{vv}(\Phi(\rho^2))\geq \left[vR^2\left(\frac12-\frac{\sqrt{2}}8\right)-128r\right]\Phi'(\rho^2)-c(Rr)^2\Phi''(\rho^2)
\end{align*}
for some constant $c\geq1$. We now choose $C_0=2^{10}$ so that 
\begin{equation}\label{c_0.cond.barr}
    \left[vR^2\left(\frac12-\frac{\sqrt{2}}8\right)-128r\right]\geq \left[v_0R^2\left(1-\frac{\sqrt{2}}8\right)\left(\frac{1}{2}-\frac{\sqrt{2}}{8}\right)-128r\right]\geq   \frac{v_0R^2}{4}-128r\geq \frac{v_0R^2}{8},
\end{equation}
which implies, 
\begin{align*}
    v\partial_x (\Phi(\rho^2))-\partial_{vv}(\Phi(\rho^2))\geq \frac{v_0R^2}{8}\Phi'(\rho^2)-c(Rr)^2\Phi''(\rho^2),
\end{align*}
in $Q_r$ for some constant $c$. Indeed, we have used the fact that $v_0\geq R$.
Moreover, we get
\begin{equation}\label{ineq.quasi}
\begin{aligned}
    \frac{R^2r}{16}\leq R^2(x+r)\left(\frac14-\frac{\sqrt{2}}{8}\right)&\leq R(x+r)\left(\frac{R(x+r)}{4r}-(v-v_0)\right)\\
    &\leq\rho^2(x,v) \leq \rho^2\left(r,v_0-\sqrt{2}R/8\right)\leq 4rR^2,
\end{aligned}
\end{equation}
i.e. it holds $r\eqsim R^{-2}\rho^2$ in $Q_r$. Thus, we obtain
\begin{align*}
    v\partial_x (\Phi(\rho^2))-\partial_{vv}(\Phi(\rho^2))\geq \frac{v_0R^2}{8}\Phi'(\rho^2)-c\rho^2r\Phi''(\rho^2)\geq cR^{-2}\left(\frac{v_0R^4}{8c}\Phi'(\rho^2)-\rho^4\Phi''(\rho^2)\right).
\end{align*}
We now take $c_2=8c\geq1$ and use \eqref{eq.quasi} to deduce \eqref{eq2.quasi}.

\textbf{Step 3. } We now want to prove for any $r\leq R^3/2^{10}$, we have 
\begin{align}\label{ste4.goal.quasi}
    h(x,v)\leq\|h\|_{L^\infty(Q_r)} \Psi(\rho^2)\quad\text{in }Q_r.
\end{align}
To this end, we note from \eqref{bdry.quasi} and \eqref{defn.psi.quasi} that
\begin{align*}
    \Psi(\rho^2)\geq 1\quad\text{on }[0,r]\times \left\{v=v_0\pm \frac{\sqrt{2}}{8}R\right\},
\end{align*}
and
\begin{align*}
    \Psi(\rho^2)\geq 0\quad\text{on }\{x=0\}\times \left[v_0-\frac{\sqrt{2}}{8}R,v_0+\frac{\sqrt{2}}{8}R\right],
\end{align*}
as the function $\Phi\geq0$ given in \eqref{defn.Phi.quasi} is non-decreasing.
Using these observations and \eqref{eq2.quasi}, we obtain
\begin{equation*}
\left\{
\begin{alignedat}{3}
v \partial_x(\Psi(\rho^2))-\partial_{vv}(\Psi(\rho^2))&\geq0&&\qquad \mbox{in  $Q_r$}, \\
\Psi(\rho^2)&\geq \frac{h}{\|h\|_{L^\infty(Q_r)}}&&\qquad  \mbox{in $ \partial_{\mathrm{kin}}Q_r$}.
\end{alignedat} \right.
\end{equation*}
By the maximum principle, we derive \eqref{ste4.goal.quasi} for any $r\leq R^3/2^{10}$. Due to \eqref{ste4.goal.quasi}, \eqref{ineq21.quasi}, we get \eqref{goal.quasi.exponential}.

Finally, for any $r\leq R^3/2^{10}$ and $C_1= 4c_2$ together with the fact that $\rho^2(r,v_0)=R^2r$, where the constant $c_2\geq1$ is determined in \textbf{Step 2}, we use \eqref{goal.quasi.exponential} to get 
\begin{align*}
    h(r,v_0)\leq \left(\frac{v_0}{4c_2r^2}+\frac{1}{R^2r}\right)R^2re^{-\frac{v_0R^2}{2c_2 r}}\|h\|_{L^\infty(Q_r)}\leq \left(\frac{v_0R^2}{4c_2r}+1\right)e^{-\frac{v_0R^2}{2c_2 r}}\|h\|_{L^\infty(Q)}\leq 2e^{-\frac{v_0R^2}{C_1r}}\|h\|_{L^\infty(Q)},
\end{align*}
where we have used the fact that $\frac{\xi}{2}e^{-\xi}\leq e^{-\frac{\xi}{2}}$ for any $\xi\in\bbR$. Note that we can establish the same estimate for $(-h)$. Altogether, by choosing $x=r$, we derive \eqref{goal.quasi} whenever $x\leq R^3/2^{10}$.
\end{proof}

Now, we are in a position to prove boundedness of the function $h/\phi$ at $(0,0)$, where $h$ is a solution to \eqref{eq.1d.expansion}. Here, the key difficulty is that $\phi$ vanishes exponentially near $\gamma_-$, whereas \autoref{lem.quasi} yields exponential decay at a rate that is suboptimal for large $v > 0$. We circumvent this issue by application of \autoref{lem.quasi} on a suitably small scale and introducing another barrier function.

\begin{lemma}\label{lemma:1D-supsolution}
    Let $h$ be a weak solution to \eqref{eq.1d.expansion}. Then there are constants $R_0 , c \geq1$ such that
    \begin{align*}
        |h/\phi|\leq c\|h\|_{L^\infty(H_{{1}})}\quad\text{in }H_{1/R_0}.
    \end{align*}
\end{lemma}
\begin{proof}
First, fix a constant $c\geq1$, which will be determined later (see \eqref{ineq3.1d-supsol}) and define 
    \begin{align}\label{defn.psi.1d-supsol}
        \psi(x,v)\coloneqq c\phi(x,v)+2ve^{-\frac1{C_1}\frac{1}{8x}},
    \end{align}
    where the constant $C_1$ is given in \autoref{lem.quasi}. Then we observe 
    \begin{align}\label{ineq1.1d-supsol}
        v\partial_x\psi-\partial_{vv}\psi\geq v^2\frac1{C_1}\frac{1}{8x^2}e^{-\frac1{C_1}\frac{1}{8x}}\geq0,
    \end{align}
    and 
    \begin{align}\label{ineq2.1d-supsol}
        \psi(x,v)\geq e^{-\frac1{C_1}\frac{1}{8x}} \quad\text{on }[0,2^{-13}]\times\{v=1/2\}\quad\text{and}\quad \psi(0,v)=0\quad\text{for }v\in(0,1/2].
    \end{align}
    Moreover, we get by taking $c\geq1$ sufficiently large,
    \begin{equation}\label{ineq3.1d-supsol}
    \begin{aligned}
    \psi(x,v)\geq  1\quad\text{on }([2^{-13} , 1/8]\times \{v=1/2\})\cup(\{x=1/8\}\times [-1/2,0])\cup([0,1/8]\times \{v=-1/2\}),
    \end{aligned}
    \end{equation}
    where we have used \eqref{phi.asymptotic}. 
    Now note from \autoref{lem.quasi} with $\mathcal{R}=1$, $v_0=1/2$, and $R=v_0$ that
\begin{align}\label{ineq.1d-supsol}
    h(x,1/2)\leq e^{-\frac1{C_1}\frac{1}{8x}}\|h\|_{L^\infty(H_{1})}\quad\text{for any }x\in[0,2^{-13}].
\end{align}
Thus by \eqref{ineq.1d-supsol}, \eqref{ineq2.1d-supsol}, and \eqref{ineq3.1d-supsol}, we have 
    \begin{align}\label{ineq5.1d-supsol}
        \psi\geq \frac{h}{\|h\|_{L^\infty(H_{1})}}\quad\text{on }\partial_{\mathrm{kin}}H_{1/2}.
    \end{align}
Thus, by applying maximum principle to $\psi$ and $h$, which is justified by \eqref{ineq1.1d-supsol} and \eqref{ineq5.1d-supsol}, we deduce 
    \begin{align}\label{ineq.lem.supsol}
        h\leq \|h\|_{L^\infty(H_{1})}\psi\quad\text{in }H_{1/2}.
    \end{align}
We are now going to prove that when $R_0\geq2$ is sufficiently large, then 
\begin{align}
\label{eq:phi-est-at-zero}
    e^{-\frac{1}{C_1}\frac1{8x}}\leq c\phi(x,v) \quad \text{in } H_{1/R_0}.
\end{align}
To this end, note that there is a sufficiently large $R_0\geq1$ such that 
\begin{equation*}
\begin{aligned}
     e^{-\frac1{C_1}\frac{1}{8x}}\leq xR_0^{\frac52}e^{-\frac1{9R_0^3x}}\leq xv^{-\frac52}e^{-\frac{v^3}{9x}}\quad\text{for any }(x,v)\in H_{1/R_0}\cap \mathcal{R}^-,\\
     e^{-\frac1{C_1}\frac{1}{8x}}\leq x^{\frac16}\quad\text{for any } (x,v)\in H_{1/R_0}\cap \mathcal{R}^0,\\
    e^{-\frac1{C_1}\frac{1}{8x}}\leq |v|^{\frac12}\quad\text{for any }(x,v)\in H_{1/R_0}\cap \mathcal{R}^+.
\end{aligned}
\end{equation*}
Hence, we deduce \eqref{eq:phi-est-at-zero} from \eqref{phi.asymptotic}. By \eqref{eq:phi-est-at-zero} along with \eqref{defn.psi.1d-supsol}, we further estimate \eqref{ineq.lem.supsol} by 
\begin{align}\label{ineq.lem.supsol-2}
        h\leq c\|h\|_{L^\infty(H_{1})}\phi\quad\text{in }H_{1/R_0}
    \end{align}
for some constant $c\geq1$. Similarly, we prove
    \begin{align}\label{ineq7.1d-supsol}
        (-h)\leq c\|h\|_{L^\infty(H_{1})}\phi\quad\text{in }H_{1/R_0}
    \end{align}
    by considering the function $(-h)$ instead of $h$. Since $\phi$ is nonnegative, a combination of \eqref{ineq.lem.supsol-2} and \eqref{ineq7.1d-supsol} yields the desired estimate. 
\end{proof}

We are now ready to prove \autoref{lemma:1D.expansion.gamma0}.

\begin{proof}[Proof of \autoref{lemma:1D.expansion.gamma0}]
First, note that it suffices to prove the result on scale one, namely that any weak solution $h$ to \eqref{eq.1d.expansion} satisfies
    \begin{align}
    \label{eq:1D.expansion.gamma0-scale1}
        \sup_{H_{r}}\frac{h}{\phi}-\inf_{H_{r}}\frac{h}{\phi}\leq cr^\alpha\|h\|_{L^\infty(H_1)}
    \end{align}
for any $r\leq 1/R_0$, where $c, R_0 \geq 1$ are determined by \autoref{lemma:1D-supsolution}.
Indeed, by taking $h_R(x,v)\coloneqq h(R^3x,Rv)$ and $\phi_R(x,v)\coloneqq \phi(R^3x,Rv)/R^{\frac12}$ to ensure that $\phi_R(x,v)=\phi(x,v)$, we have \eqref{goal.lem.expansion.sta} and \eqref{goal.lem.varhar} with $h$ and $\phi$ replaced by $h_R$ and $\phi_R$, respectively. By scaling back, we get the desired estimates Moreover, \eqref{bdd.h/phi} follows by rescaling \autoref{lemma:1D-supsolution}.

We prove the estimate \eqref{eq:1D.expansion.gamma0-scale1} on scale one by an induction argument based on \autoref{lemma:1D-supsolution}, \autoref{lemma:1D-expansion:stationary}, and \autoref{lemma:1D.harnack.barrier}. More precisely, we find constants $M_i$, $m_i$, $K\geq1$, and $\alpha\in(0,1)$ such that 
\begin{align*}
    m_i\leq \frac{h}{\phi}\leq M_i\quad\text{in }H_{\mathcal{R}_i},
\end{align*}
and
\begin{align}\label{ineq.lem.expansion.1dgamma0}
    M_i-m_i={K\|h\|_{L^\infty(H_1)}}{\mathcal{R}_i^{\alpha}},
\end{align}
where $\mathcal{R}_i\coloneqq (1/R_0)^{i+1}$. Moreover, we may assume $R_0\geq \frac{8}{\theta}$, where the constant $\theta\in(0,1/2)$ is given in \autoref{lemma:1D.harnack.barrier}.
By \autoref{lemma:1D-supsolution}, we have $|h/\phi|\leq c\|h\|_{L^\infty(H_1)}$ in $H_{\mathcal{R}_0}$. Thus, we take
\begin{align}\label{ineq1.lem.expansion.1dgamma0}
    K := 2c, \quad\text{and}\quad m_0\coloneqq \inf_{H_{\mathcal{R}_0}}\frac{h}{\phi},\quad\text{and}\quad M_0\coloneqq \sup_{H_{\mathcal{R}_0}}\frac{h}{\phi}
\end{align}
to see that \eqref{ineq.lem.expansion.1dgamma0} and \eqref{ineq1.lem.expansion.1dgamma0} hold for $i=0$. We now assume that \eqref{ineq.lem.expansion.1dgamma0} and \eqref{ineq1.lem.expansion.1dgamma0} hold for $i=0,1,\ldots , j$ for some integer $j\geq0$. Let us consider a function $h_j\coloneqq h-m_j\phi\geq0$ on $H_{\mathcal{R}_j}$ to see that 
\begin{equation*}
\left\{
\begin{alignedat}{3}
v \partial_xh_j-\partial_{vv}h_j&=0&&\qquad \mbox{in  $(0,\mathcal{R}_j^{3})\times (-\mathcal{R}_{j},\mathcal{R}_j)$}, \\
h_j &=0&&\qquad  \mbox{in $ \{x=0\}\times (0,\mathcal{R}_j)$}.
\end{alignedat} \right.
\end{equation*}
By using a rescaled version of \autoref{lemma:1D.harnack.barrier} and \autoref{lemma:1D-expansion:stationary}, we derive
\begin{align}\label{ineq2.lem.expansion.1dgamma0}
    \sup_{H_{\mathcal{R}_j^+\theta/4}} \left[ h/\phi - m_j \right] = \sup_{H_{\mathcal{R}_j^+\theta/4}}h_j/\phi\leq c\inf_{H_{\mathcal{R}_j^+\theta/4}}h_j/\phi\leq C \inf_{H_{\mathcal{R}_j\theta/8}}h_j/\phi = C\inf_{H_{\mathcal{R}_j\theta/8}} [h/\phi - m_j]
\end{align}
for some constants $c,C\geq1$. Similarly, by following the same lines as in the proof of \eqref{ineq2.lem.expansion.1dgamma0} with $h_j$ replaced by $H_j\coloneqq M_j\phi-h\geq0$, we derive
\begin{align}\label{ineq3.lem.expansion.1dgamma0}
    \sup_{H_{\mathcal{R}_j^+\theta/4}} \left[ M_j -  h/\phi \right]  = \sup_{H_{\mathcal{R}_j^+\theta/4}}H_j/\phi\leq c\inf_{H_{\mathcal{R}_j^+\theta/4}}H_j/\phi\leq C\inf_{H_{\mathcal{R}_j\theta/8}}H_j/\phi = C\inf_{H_{\mathcal{R}_j\theta/8}}[M_j - h/\phi].
\end{align}
We now combine \eqref{ineq2.lem.expansion.1dgamma0} and \eqref{ineq3.lem.expansion.1dgamma0} to deduce that
\begin{align*}
    M_{j}-m_j&\le \sup_{H_{\mathcal{R}_j^+\theta/4}} \left[ M_j - \frac{h}{\phi} \right] + \sup_{H_{\mathcal{R}_j^+\theta/4}} \left[ \frac{h}{\phi} - m_j \right] \\
    &\leq C\left(\inf_{H_{\mathcal{R}_j\theta/8}} \left[M_j-\frac{h}{\phi}\right] + \inf_{H_{\mathcal{R}_j\theta/8}}\left[\frac{h}{\phi}-m_j\right] \right)\\
    &= C\left(\inf_{H_{\mathcal{R}_j\theta/8}}\frac{h}{\phi}-\sup_{H_{\mathcal{R}_j\theta/8}}\frac{h}{\phi}-m_j+M_j\right).
\end{align*}
By the choice of $\mathcal{R}_j=(1/R_0)^{j+1}$ and the inductive assumption \eqref{ineq.lem.expansion.1dgamma0}, we obtain
\begin{align*}
    \sup_{H_{\mathcal{R}_{j+1}}}\frac{h}{\phi}-\inf_{H_{\mathcal{R}_{j+1}}}\frac{h}{\phi}\leq \sup_{H_{\mathcal{R}_j\theta/8}}\frac{h}{\phi}-\inf_{H_{\mathcal{R}_j\theta/8}}\frac{h}{\phi}\leq \frac{C-1}{C}(M_j-m_j)\leq \frac{C-1}{C}K\|h\|_{L^\infty(H_1)}\mathcal{R}_j^{\alpha}.
\end{align*}
We now choose $\alpha\in(0,1)$ sufficiently small so that 
\begin{align*}
     \sup_{H_{\mathcal{R}_{j+1}}}\frac{h}{\phi}-\inf_{H_{\mathcal{R}_{j+1}}}\frac{h}{\phi}\leq K\|h\|_{L^\infty(H_1)}\mathcal{R}_{j+1}^{\alpha}.
\end{align*}
By taking 
\begin{align*}
    m_{j+1}\coloneqq \inf_{H_{\mathcal{R}_{j+1}}}\frac{h}{\phi} \quad\text{and}\quad M_{j+1}\coloneqq \sup_{H_{\mathcal{R}_{j+1}}}\frac{h}{\phi},
\end{align*}
we get \eqref{ineq.lem.expansion.1dgamma0} and \eqref{ineq1.lem.expansion.1dgamma0} with $i=j+1$. Hence, by induction, \eqref{ineq.lem.expansion.1dgamma0} and \eqref{ineq1.lem.expansion.1dgamma0} hold for any $j\geq0$, which implies the desired result.
\end{proof}

\subsection{Liouville theorem involving $\phi_0$}
\label{subsec:higher-Liouville}
In this subsection, we prove a preliminary Liouville theorem in the half-space in 1D with absorbing condition. It allows solutions to grow up to order $3+\frac12-\epsilon$. 

\begin{lemma}\label{lem.liou2}
    Let $h$ be a weak solution to 
    \begin{equation}\label{eq.liou2}
\left\{
\begin{alignedat}{3}
v \partial_xh-\partial_{vv}h&=0&&\qquad \mbox{in  $\{x>0\}\times \bbR$}, \\
h&=0&&\qquad  \mbox{in $ \{x=0\}\times \{v>0\}$},\\
|h(x,v)|&\leq M\left(1+(|x|^{\frac13}+|v|)^{\frac12+3-\epsilon}\right)&&\qquad\mbox{in $\{x>0\}\times \bbR$},
\end{alignedat} \right.
\end{equation}
where $M\geq1$ and $\epsilon\in(0,3)$. Then we have $h=c\phi_0$ for some $c \in \R$.
\end{lemma}

\begin{remark}\label{rmk.liou2}
    Note from \autoref{lemma:1D.expansion.gamma0} that 
    \begin{align*}
        \frac{h(0)}{\phi_0(0)}=\lim_{(x,v)\to0}\frac{h(x,v)}{\phi_0(x,v)}
    \end{align*}
    is well defined. Thus, we obtain that the constant $c$ given in \autoref{lem.liou2} coincides with $h(0)/\phi_0(0)$.
\end{remark}

Note that by using the expansion from \autoref{lemma:1D-expansion:stationary}, it is possible to derive a Liouville theorem of order $\frac{1}{2} + \alpha$, where $\alpha > 0$ is as in \autoref{lemma:1D-expansion:stationary}. We are able to upgrade the order to $\frac{1}{2} + 3 - \eps$ by tracking carefully the asymptotic behavior of $h - \frac{h(0)}{\phi_0(0)} \phi_0$ and the $x$-derivative of this function near $\{0\}\times \R$ and as $|(x,v)| \to \infty$. Then, we can compare it to the asymptotic behavior of $\phi_0$ and $\partial_x \phi_0$ and conclude the proof by a contradiction argument based on the maximum principle.

As explained before, our proof is based on boundary regularity estimates of solutions to \eqref{eq.liou2} away from $\gamma_0$. Near $\gamma_+$ we use the regularity results in \cite{RoWe25}, and near $\gamma_-$, we combine suitable interior estimates with the upper bound in \autoref{lem.quasi}.

The first lemma is about the boundary behavior of solutions near $\gamma_-$. The exponential factor $e^{-\frac{1}{C} \frac{v_0}{R}}$ shows that all derivatives vanish to infinite order.

\begin{lemma}\label{lemma.1D.grax.est}
Let $h$ be a weak solution to 
    \begin{equation*}
v \partial_xh-\partial_{vv}h=0\quad \mbox{in  $H_{2R}(0,v_0)$}
\end{equation*}
where $v_0 > 0$ and $H_{2R}(0,v_0)\subset \{x>0\}\times \{v>0\}$.  Let $k$ be a nonnegative integer. Then, there are constants $c=c(k)$ and $C\geq1$ such that
\begin{align}\label{main.grax.est}
    \|D^k h\|_{L^\infty(H_R(0,v_0))}\leq cR^{-k}e^{-\frac1C\frac{v_0}{R}}\|h\|_{L^\infty(H_{2R}(0,v_0))}.
\end{align}
\end{lemma}

\begin{proof}
By a standard covering argument it suffices to prove
\begin{align}\label{goal1.ineq1.grax}
    \Vert D^k h \Vert_{L^{\infty}(Q_{R/2^5}(0,v_0))} \le c R^{-k}e^{-\frac1c\frac{v_0}{R}} \|h\|_{L^\infty(Q_{2R}(0,v_0))}
\end{align}
for some constant $c\geq1$. Let us fix $(x_1,v_1)\in Q_{R/2^5}(0,v_0)$. First, we see that 
\begin{equation}\label{rel.x1v1.grax}
    x_1<(R/2^5)^3 \quad\text{and}\quad v_1>v_0-R/2^5>R.
\end{equation}In addition, as in \eqref{g.inc.strmax}, we find $r_1$ such that
\begin{equation}\label{rel.x1r1.grax}
    x_1-(2r_1)^2v_1=(2r_1)^3 \quad\text{and}\quad \left(\frac{x_1}{16v_1}\right)^{\frac12}<r_1< \left(\frac{x_1}{4v_1}\right)^{\frac12}\leq \frac{R}{8}.
\end{equation}
The value $r_1$ is the largest radius $r$ for which $\mathcal{Q}_{r}(0,x_1,v_1) \subset \R \times \{ x > 0 \} \times \R$.
Note that $\widetilde{h}(t,x,v)=h(x,v)$ is a weak solution to 
\begin{align*}
    (\partial_t+v\partial_x-\partial_{vv})\widetilde{h}=0\quad\text{in }\bbR\times \{x>0\}\times \bbR.
\end{align*}
Thus, by the interior estimates in \autoref{lem.int.grad}, and since $\widetilde{h}$ is independent of $t$, we have
\begin{align*}
    \sup_{t\in(-(r_1/2)^2,(r_1/2)^2)} & \sup_{x\in B_{(r_1/2)^3}(x_1+tv_1)}\|D^kh(x,\cdot)\|_{L^\infty(B_{r_1/2}(v_1))} = \| D^k\widetilde{h}\|_{L^\infty(\mathcal{Q}_{r_1/2}(0,x_1,v_1))}\\
    &\leq cr_1^{-k}\|\widetilde{h}\|_{L^\infty(\mathcal{Q}_{r_1}(0,x_1,v_1))} = cr_1^{-k}\sup_{t\in(-r_1^2,r_1^2)}\sup_{x\in B_{r_1^3}(x_1+tv_1)}\|h(x,\cdot)\|_{L^\infty(B_{r_1}(v_1))}.
\end{align*}
In light of this and the fact that $2x_1>x_1+r_1^2v_1+r_1^3$, which follows from \eqref{rel.x1r1.grax}, we get 
\begin{align}\label{ineq.j.grax}
    \|D^kh\|_{L^\infty(Q_{r_1/2}(x_1,v_1))}\leq cr_1^{-k}\|h\|_{L^\infty\left((0,2x_1)\times B_{r_1}(v_1)\right)}\eqqcolon J
\end{align}
for some constant $c=c(k)$.

Now we employ \autoref{lem.quasi} to further estimate the term $J$. First, fix $(x_2,v_2)\in (0,2x_1)\times B_{r_1}(v_1)$. Then we have 
\begin{align*}
    v_2>v_1-r_1>R-(x_1/(4v_1))^{\frac12}>R/2,
\end{align*}
where we have used \eqref{rel.x1v1.grax} and \eqref{rel.x1r1.grax}.
We now apply \autoref{lem.quasi} with $v_0$ and $R$ replaced by $v_2$ and $R/2$, respectively, to see that 
\begin{align*}
    |h(x,v_2)|\leq 2e^{-\frac{1}{C_1}\frac{v_2(R/2)^2}{x}}\|h\|_{L^\infty(Q)}\quad\text{for any }x\in\left[0,R^3/2^{13}\right]
\end{align*}
for some constant $C_1\geq1$, where 
\begin{align*}
    Q\coloneqq\left[0,R^3/2^{13}\right]\times\left[v_2-\frac{\sqrt{2}}8\frac{R}{2},v_2+\frac{\sqrt{2}}8\frac{R}{2}\right].
\end{align*}
Since $Q\subset Q_{2R}(0,v_0)$ and $v_2< v_1< 4v_2$, we have 
\begin{align*}
    |h(x,v_2)|\leq 2e^{-\frac{1}{4C_1}\frac{v_1R^2}{x}}\|h\|_{L^\infty(Q_{2R}(0,v_0))}\quad\text{for any }x\in\left[0,R^3/2^{13}\right].
\end{align*}
Using \eqref{rel.x1v1.grax}, we get $x_2\leq 2x_1\leq 2\left({R}/{2^5}\right)^{3}\leq R^3/2^{13}$ and
\begin{align}\label{pt.ineq1.grax}
    |h(x_2,v_2)|\leq 2e^{-\frac{1}{4C_1}\frac{v_1R^2}{x_2}}\|h\|_{L^\infty(Q_{2R}(0,v_0))}\leq 2e^{-\frac{1}{8C_1}\frac{v_1R^2}{x_1}}\|h\|_{L^\infty(Q_{2R}(0,v_0))}.
\end{align}
As \eqref{pt.ineq1.grax} holds for any $(x_1,v_1)\in (0,2x_1)\times B_{r_1}(v_1)$, we further estimate $J$ given in \eqref{ineq.j.grax} as 
\begin{align*}
    \|D^kh\|_{L^\infty(Q_{r_1/2}(x_1,v_1))}&\leq c R^{-k} e^{-\frac{1}{8C_1}\frac{v_1R^2}{x_1}}\|h\|_{L^\infty(Q_{2R}(0,v_0))} \leq cR^{-k}e^{-\frac1{32C_1}\frac{v_0R^2}{R^3}}\|h\|_{L^\infty(Q_{2R}(0,v_0))},
\end{align*}
for some constant $c=c(k)$, where we have used \eqref{rel.x1r1.grax} and the fact that $x_1<R^3$, as well as $v_0/2<v_1$.
Thus, \eqref{goal1.ineq1.grax} follows, as the previous estimate holds for any $(x_1,v_1)\in H_{R/2^5}(0,v_0)$, as desired.
\end{proof}

Next, we prove gradient estimates near $\gamma_+$. Note that the factor $(v_0/R)^{k/2}$ can in general not be improved since it arises from the underlying kinetic scaling, when applied to a stationary equation.

\begin{lemma}\label{lemma.1D.grax.est2}
    Let $h$ be a weak solution to 
    \begin{equation*}
v \partial_xh-\partial_{vv} h=0 \quad \text{in } H_{2R}(0,v_0),
\end{equation*}
where $v_0 < 0$ and $R > 0$ are such that $H_{2R}(0,v_0)\subset \{x>0\}\times \{v<0\}$. 
Let $k$ be a nonnegative integer. Then, there is a constant $c=c(k)$ such that 
\begin{align}\label{main.grax.est2}
    \|D^k h\|_{L^\infty(H_R(0,v_0))}\leq cR^{-k}\left(\frac{v_0}{R}\right)^{\frac{k}2}\|h\|_{L^\infty(H_{2R}(0,v_0))},
\end{align}

\end{lemma}
\begin{proof}
We fix $v_1\in (v_0-R,v_0+R)$ to see that $|v_1|>R$. Note that there is $r \in (0,R)$ such that 
    \begin{align}\label{eqr2v0.grax.est2}
        -r^2v_1+r^3=R^3,
    \end{align}
    as $g(\xi)\coloneq -\xi^2v_1+\xi^3$ is increasing when $\xi\geq0$ and 
    \begin{align*}
        g\left( \frac{R^{\frac32}}{2|v_1|^{\frac12}}\right)<R^3<g\left(\frac{R^{\frac32}}{|v_1|^{\frac12}}\right).
    \end{align*}
    Thus, $r$ satisfies
    \begin{align}
    \label{ineqr2v12.grax.est2}
        \frac{R^{\frac32}}{2|v_1|^{\frac12}}<r<\frac{R^{\frac32}}{|v_1|^{\frac12}} \quad\text{and} \quad -(3r/2)^2v_1+(3r/2)^3<(2R)^3.
    \end{align}
    Next, we observe that by following the lines of the proof of \cite[Lemma 4.4]{RoWe25} but using only the expansion from \cite[Lemma 4.1]{RoWe25}, we deduce for any $\epsilon\in(0,1)$ and $l\geq3$,
    \begin{align}\label{ineq1.par.grax2}
        [D^l h]_{C^{\epsilon}({\mathcal{H}}_r(0,0,v_1))}\leq c(k,\epsilon)r^{-l+\epsilon}\|h\|_{L^\infty({\mathcal{H}}_{3r/2}(0,0,v_1))}
    \end{align}
    for some constant $c=c(l,\epsilon)$, where we write 
    \begin{align*}
        {\mathcal{H}}_r(0,0,v_1)\coloneqq \left\{(t,x,v)\,:\,t\in (-r^2,r^2],\, v\in B_r(v_1),\, x\in B_{r^3}(tv_1)\right\}\cap \{x\geq0\}.
    \end{align*}
    Here, the only difference compared to the statement of \cite[Lemma 4.4]{RoWe25} is that the constant in \eqref{ineq1.par.grax2} is independent of $|v_0|$.
    Moreover, by \cite[Lemma 2.5]{RoWe25} we get 
   \begin{align}\label{ineq2.par.grax2}
        [ h]_{C^{\alpha}(\widetilde{\mathcal{H}}_r(0,0,v_1))}\leq cr^{-\alpha}\|h\|_{L^\infty(\widetilde{\mathcal{H}}_{3r/2}(0,0,v_1))},
    \end{align}
    for some $\alpha\in(0,1)$.
    We now use a rescaled version of the interpolation inequality given in \cite[Proposition 2.10]{ImSi21} to get that 
    \begin{align*}
[r^{k}D^{k-1} h]_{C^{0,1}(\widetilde{\mathcal{H}}_r(0,0,v_1))} \leq c\left([r^{k+2+\alpha}D^{k+2} h]^{\theta}_{C^{\alpha}(\widetilde{\mathcal{H}}_r(0,0,v_1))}[r^{\alpha}h]^{1-\theta}_{C^\alpha(\widetilde{\mathcal{H}}_r(0,0,v_1))}+r^{\alpha}[h]_{C^\alpha(\widetilde{\mathcal{H}}_r(0,0,v_1))}\right)
    \end{align*}
    for some constant $c=c(k)$, where $\theta\in(0,1)$ satisfies
    \begin{align*}
        k=\theta (k+2+\alpha)+(1-\theta)\alpha.
    \end{align*}
    Using Young's inequality, \eqref{ineq1.par.grax2} with $\epsilon=\alpha$ and \eqref{ineq2.par.grax2}, we deduce
    \begin{align*}
        \|r^kD^{k} h\|_{L^\infty(\widetilde{\mathcal{H}}_r(0,0,v_1))}=[r^{k}D^{k-1} h]_{C^{0,1}(\widetilde{\mathcal{H}}_r(0,0,v_1))} \leq c\|h\|_{L^\infty(\widetilde{\mathcal{H}}_{3r/2}(0,0,v_1))}.
    \end{align*}
    From this, we get
    \begin{align*}
        \sup_{x\in (0,-r^{2}v_1+r^3)}\|D^kh(x,\cdot)\|_{L^\infty(B_{r}(v_1))}&=\|D^k h\|_{L^\infty(\widetilde{\mathcal{H}}_r(0,0,v_1))}\\
        &\leq cr^{-k}\|h\|_{L^\infty(\widetilde{\mathcal{H}}_{3r/2}(0,0,v_1))}\\
        &=cr^{-k}\sup_{x\in(0,-(3r/2)^2v_1+(3r/2)^3)}\|h\|_{L^\infty(B_{3r/2}(v_1))}
    \end{align*}
    for some constant $c=c(k)$. Using this together with \eqref{eqr2v0.grax.est2} and \eqref{ineqr2v12.grax.est2}, we derive
    \begin{align*}
        \|D^kh\|_{L^\infty((0,R^3)\times (v_1-r,v_1+r))}\leq cr^{-k}\|h\|_{L^\infty((0,(2R)^3)\times (v_1-2r,v_1+2r))}\leq cR^{-k}\left(\frac{v_1}{R}\right)^{\frac{k}2}\|h\|_{L^\infty(H_{2R}(0,v_0))}
    \end{align*}
    for some constant $c=c(k)$. Since this holds for any $v_1\in (v_0-R,v_0+R)$, we derive
    \begin{align*}
    \|D^k h\|_{L^\infty(H_R(0,v_0))}\leq cR^{-k}\left(\frac{v_0}{R}\right)^{\frac{k}2}\|h\|_{L^\infty(H_{2R}(0,v_0))},
\end{align*}
where we have used the fact that $|v_1|\leq 2|v_0|$. This completes the proof.
\end{proof}
We are now ready to prove \autoref{lem.liou2}.
\begin{proof}[Proof of \autoref{lem.liou2}.]
   First, we observe that
\begin{align}
\label{kappa0.boundedness}
    \kappa_0\coloneqq \lim_{(x,v)\to0}\frac{h(x,v)}{\phi_0(x,v)} \quad\text{satisfies}\quad  \left|\kappa_0\right|\leq c\|h\|_{L^\infty(H_1)}\leq c(M),
\end{align}
where $\kappa_0$ is well defined by \autoref{rmk.liou2} and we have used the growth condition in \eqref{eq.liou2} for the last inequality.
Next, we define
    \begin{align*}
        H\coloneqq \partial_x(h-\kappa_0\phi_0) \quad \text{and} \quad
        \Phi\coloneqq \phi_0+\partial_x\phi_0=\phi_0+\frac1{36}\phi_{-1},
    \end{align*}
where the last equality follows from \eqref{lemma:derx.phim}. 

Note that the proof is complete, once we establish that $H \equiv 0$. To deduce this, we will first prove that for any $c_0\in\bbR$, there is a large constant $\mathcal{M}_0=\mathcal{M}_0(c_0,M,\epsilon)\geq4$ such that
\begin{align}\label{goal0.liou2}
    (c_0H-\Phi)(x_0,v_0)\leq \frac12\quad\text{if }(x_0,v_0)\notin (A_1\cup A_2),
\end{align}
where we write
\begin{align*}
    A_1\cup A_2 := \Big( \left[(\mathcal{M}_0/4)^{-3},(\mathcal{M}_0/4)^3\right]\times[-\mathcal{M}_0\times \mathcal{M}_0] \Big) \cup \Big(  \left[0,(\mathcal{M}_0/4)^{-3}\right]\times \left[-\mathcal{M}_0,-\mathcal{M}_0^{-1}\right] \Big).
\end{align*}
To this end, first, we fix $c_0\in\bbR$. We now divide the proof of \eqref{goal0.liou2} into several cases depending on the location of the point $(x_0,v_0)$.
\begin{itemize}
    \item Suppose $4x_0^{\frac13}\leq v_0$. Since $h-\kappa_0\phi_0$ solves \eqref{eq.liou2}, by \autoref{lemma.1D.grax.est}, we derive
    \begin{equation*}
    \begin{aligned}
        |H(x_0,v_0)|&\leq {\|\partial_{x}(h-\kappa_0\phi_0)\|_{L^\infty({H}_{{R}}(0,v_0))}}\leq cR^{-3}e^{-\frac1c\frac{v_0}{R}}\|h-\kappa_0\phi_0\|_{L^\infty({Q}_{2R}(0,v_0))},
    \end{aligned}
    \end{equation*}
     where we set $R=x_0^{\frac13}$. If $x_0, v_0\leq 1$, then we obtain
    \begin{equation}\label{ineq1.liou2}
    \begin{aligned}
        |H(x_0,v_0)|&\leq cR^{-3}e^{-\frac1c\frac{v_0}{R}}\|\phi_0\|_{L^\infty([0,8x_0]\times [v_0/2,3v_0/2])}\left\|\frac{h}{\phi_0}-\kappa_0\right\|_{L^\infty(H_{2v_0})}\\
        &\leq cx_0^{-1}x_0v_0^{-\frac52}e^{-\frac{v_0^3}{9x_0}}\left(v_0^{\alpha}\|h\|_{L^\infty(H_{2R_0})}\right)\\
        &\leq c(M)v_0^{-\frac52+\alpha}e^{-\frac{v_0^3}{9x_0}},
    \end{aligned}
    \end{equation}
    where we have used \eqref{phi.asymptotic} and \eqref{hol.h/phi}, and the constant $R_0\geq1$ is determined in \autoref{lemma:1D-supsolution}. Moreover, we have from \eqref{phi-1.asymptotic} 
    \begin{align}\label{ineq11.liou2}
        \Phi(x_0,v_0) \geq c^{-1} x_0^{-1}v_0^{\frac12}e^{-\frac{v_0^3}{9x_0}}.
    \end{align}
    Using this and \eqref{ineq1.liou2}, we obtain that if $x_0,v_0\leq 1$ and $v_0\geq 4x_0^{\frac13}$, then
    \begin{align}\label{ineq20.liou2}
        c_0H(x_0,v_0)-\Phi(x_0,v_0)\leq e^{-\frac{v_0^3}{9x_0}}v_0^{-\frac52+\alpha}\left(c_0 C-c^{-1}\frac{v_0^{3-\alpha}}{x_0}\right)\leq e^{-\frac{v_0^3}{9x_0}}v_0^{-\frac52+\alpha}\left(c_0 C -c^{-1}x_0^{-\frac{\alpha}3}\right).
    \end{align}
    We now consider the case when $v_0 \ge 1$.
    Then using the first line given in \eqref{ineq1.liou2} together with \eqref{phi.asymptotic}, \eqref{hol.h/phi}, and the third condition in \eqref{eq.liou2}, we get for any nonnegative integer $k$,
    \begin{equation}\label{ineq2.liou2}
    \begin{aligned}
        |H(x_0,v_0)|&\leq cx_0^{-1}e^{-\frac{v_0x_0^{-\frac13}}c}x_0v_0^{-\frac52}e^{-\frac{v_0^3}{9x_0}}(v_0^{-\frac12}\|h\|_{L^\infty(H_{2R_0v_0})})\\
        &\leq c(M,\epsilon)e^{-\frac{v_0x_0^{-\frac13}}c}v_0^{\frac12-\epsilon}e^{-\frac{v_0^3}{9x_0}}\\
        &\leq c(k,M,\epsilon)\left(v_0x_0^{-\frac13}\right)^{-k}v_0^{\frac12-\epsilon}e^{-\frac{v_0^3}{9x_0}},
    \end{aligned}
    \end{equation}
    where we have also used the fact that for any $k$,
    \begin{align}
    \label{eq:exp-help}
        e^{-\frac{v_0x_0^{-\frac13}}c}\leq c(k)\left(v_0x_0^{-\frac13}\right)^{-k}.
    \end{align} 
    In case $x_0 \le 1 \le v_0$, we apply this fact with $k=0$ and \eqref{ineq11.liou2}. This yields 
    \begin{align*}
        c_0H(x_0,v_0)-\Phi(x_0,v_0)\leq v_0^{\frac12-\epsilon}e^{-\frac{v_0^3}{9x}}\left(c_0 C-\frac{v_0^\epsilon}{cx_0}\right)\leq v_0^{\frac12-\epsilon}e^{-\frac{v_0^3}{9x_0}}\left(c_0 C-\frac{x_0^{-1+\frac{\epsilon}{3}}}{c}\right).
    \end{align*}
    By this and \eqref{ineq20.liou2}, we have
    \begin{align}\label{ineq21.liou2} 
        c_0H(x_0,v_0)-\Phi(x_0,v_0)\leq \frac12\quad\text{in }x_0\in[0,(\mathcal{M}_0/4)^{-3}],
    \end{align}
    whenever $\mathcal{M}_0=\mathcal{M}_0(M,\epsilon)$ is sufficiently large.
    
    If $1\leq x_0,v_0$, then by \eqref{phi.asymptotic} and \eqref{phi-1.asymptotic}, we have 
    \begin{equation}\label{case.Phi}
    \begin{aligned}
        \Phi(x_0,v_0)\eqsim\begin{cases}
            x_0^{-1}v_0^{\frac12}e^{-\frac{v_0^3}{9x_0}}&\quad\text{if }v_0^3\geq x_0^2,\\
            x_0v_0^{-\frac52}e^{-\frac{v_0^3}{9x_0}}&\quad\text{if }v_0^3\leq x_0^2.
        \end{cases}
    \end{aligned}
    \end{equation}
    Note from \eqref{ineq2.liou2} and \eqref{eq:exp-help} with $k=3$ that 
    \begin{align*}
        |H(x_0,v_0)|\leq c(M,\epsilon)  x_0 v_0^{-\frac52-\epsilon}e^{-\frac{v_0^3}{9x_0}}.
    \end{align*}
    Using this and \eqref{case.Phi} leads to 
    \begin{align*}
        (c_0H-\Phi)(x_0,v_0)\leq x_0v_0^{-\frac52-\epsilon}e^{-\frac{v_0^3}{9x_0}}\left(c_0 C-\frac{v_0^{\epsilon}}{c}\right)\leq \frac12,
    \end{align*}
    if $v_0\geq\mathcal{M}_0$, where $\mathcal{M}_0=\mathcal{M}_0(c_0,M,\epsilon)\geq1$ is sufficiently large. Using this and \eqref{ineq21.liou2}, we obtain by taking $\mathcal{M}_0$ sufficiently large, 
    \begin{align}\label{goal1.liou2}
        c_0H(x_0,v_0)-\Phi(x_0,v_0)\leq \frac12 \quad\text{in }\mathcal{A}^-,
    \end{align}
    where 
    \begin{align*}
        \mathcal{A}^-\coloneqq\{v\geq 4x^{\frac13}\,:\,\,x\in[0,(\mathcal{M}_0/4)^{-3}]\}\cup \{v\geq 4x^{\frac13}\,:\,\,v\in[\mathcal{M}_0,\infty)\}.
    \end{align*}
    This concludes the proof of \eqref{goal0.liou2} in case $v_0 \ge 4 x_0^{\frac{1}{3}}$.
\item Suppose $4x_0^{\frac13}\geq |v_0|$. By \autoref{lem.int.grad}, we deduce 
\begin{align}\label{ineq3.liou2}
    |H(x_0,v_0)|\leq cR^{-3}\|h-\kappa_0\phi_0\|_{L^\infty(\mathcal{Q}_{R/4}(0,x_0,v_0))}\leq cR^{-3}\|h-\kappa_0\phi_0\|_{L^\infty(H_{R}(x_0,v_0))},
\end{align}
where $R=\frac{2x_0^{\frac13}}{3}$. If $x_0\leq 1$, then we have 
\begin{align*}
    |H(x_0,v_0)|\leq cx_0^{-1}\|\phi_0\|_{L^\infty([0,8x_0]\times [-5x_0^{\frac13},5x_0^{\frac13}])}\left\|\frac{h}{\phi_0}-\kappa_0\right\|_{L^\infty(H_{5x_0^{1/3}})}&\leq cx_0^{-\frac56}\left(x_0^{\frac{\alpha}{3}}\|h\|_{L^\infty(H_{5R_0})}\right)\\
        &\leq c(M)x_0^{-\frac56+\frac\alpha3},
\end{align*}
where we have used \eqref{phi.asymptotic}, \eqref{hol.h/phi} and the third condition given in \eqref{eq.liou2}.
Using \eqref{phi-1.asymptotic} and \eqref{phi.asymptotic}, we get that if $x_0\leq 1$, then 
\begin{align}\label{ineq31.liou2}
        (c_0H-\Phi)(x_0,v_0)\leq x_0^{-\frac56+\frac{\alpha}3}\left(c_0C-\frac{x_0^{-\frac{\alpha}3}}{c}\right).
    \end{align}
    If $1\leq x_0$, then using \eqref{ineq3.liou2}, \eqref{phi.asymptotic}, \eqref{hol.h/phi}, and the third condition in \eqref{eq.liou2}, we get
\begin{align*}
    |H(x_0,v_0)|\leq cx_0^{-1}\|\phi_0\|_{L^\infty([0,8x_0]\times [-5x_0^{\frac13},5x_0^{\frac13}])}\left\|\frac{h}{\phi_0}-\kappa_0\right\|_{L^\infty(H_{5x_0^{1/3}})} \leq cx_0^{-\frac56}x_0^{-\frac16}x_0^{1+\frac16-\frac{\epsilon}3}\leq cx_0^{\frac16-\frac\epsilon3}.
\end{align*}
Using \eqref{phi-1.asymptotic} and \eqref{phi.asymptotic} we get $\Phi(x_0,v_0)\geq\frac1c{x_0^{\frac16}}$, which gives
\begin{align*}
        (c_0H-\Phi)(x_0,v_0)\leq x_0^{\frac16-\frac{\epsilon}3}\left(c_0C-\frac{x_0^{\frac{\epsilon}3}}{c}\right).
    \end{align*}
With this and \eqref{ineq31.liou2}, when $\mathcal{M}_0$ is sufficiently large, then we have
 \begin{align}\label{goal2.liou2}
        c_0H(x_0,v_0)-\Phi(x_0,v_0)\leq \frac12 \quad\text{in }\mathcal{A}^0,
    \end{align}
    where 
    \begin{align*}
        \mathcal{A}^0\coloneqq\{|v|\leq 4x^{\frac13}\,:\,\,x\in[0,(\mathcal{M}_0/4)^{-3}]\}\cup \{|v|\leq 4x^{\frac13}\,:\,\,x\in[(\mathcal{M}_0/4)^{3},\infty)\}.
    \end{align*}
    This concludes the proof of \eqref{goal0.liou2} in case $|v_0| \le 4 x_0^{\frac{1}{3}}$.
\item Suppose $-v_0\geq 4x_0^{\frac13}$. By \autoref{lemma.1D.grax.est2} with $R=|v_0|/4$, we have 
\begin{align}\label{ineq4.liou2}
    |H(x_0,v_0)|\leq \|H\|_{L^\infty({H}_R(0,v_0))}\leq c|v_0|^{-3}\|h-\kappa_0\phi_0\|_{L^\infty(H_{2R}(0,v_0))}.
\end{align}
If $|v_0|\leq 1$, then we further estimate
\begin{align*}
    |H(x_0,v_0)|&\leq c|v_0|^{-3}\left\|\frac{h}{\phi_0}-\kappa_0\right\|_{L^\infty(H_{2R}(0,v_0))}\|\phi_0\|_{L^\infty([0,|v_0|^{3}/2]\times [-|v_0|/2,-3|v_0|/2])}\\
    &\leq c|v_0|^{-3}\left\|\frac{h}{\phi_0}-\kappa_0\right\|_{L^\infty(H_{2|v_0|})}|v_0|^{\frac12}\\
    &\leq c(M)|v_0|^{-\frac52+\alpha},
\end{align*}
where we have used \eqref{phi.asymptotic} and \eqref{hol.h/phi}. Using \eqref{phi-1.asymptotic}, we have 
\begin{align}\label{ineq5.liou2}
    c_0H(x_0,v_0)-\Phi(x_0,v_0)\leq |v_0|^{-\frac52+\alpha}\left(c_0 C-\frac{|v_0|^{-\alpha}}{c}\right).
\end{align}On the other hand, when $|v_0|\geq1$, we get 
\begin{align*}
    |H(x_0,v_0)|&\leq c|v_0|^{-3}\|\phi_0\|_{L^\infty([|v_0|^3/2]\times [3v_0/2,v_0/2])}\left\|\frac{h}{\phi_0}-\kappa_0\right\|_{L^\infty(H_{4|v_0|})}\\
    &\leq c|v_0|^{-\frac52}|v_0|^{-\frac12}\|h\|_{L^\infty(H_{4R_0|v_0|})}\\
    &\leq c|v_0|^{\frac12-\epsilon},
\end{align*}
where we have used \eqref{phi.asymptotic} and \eqref{hol.h/phi}. 
Using this, \eqref{phi.asymptotic} and \eqref{phi-1.asymptotic}, we derive
\begin{align*}
    c_0H(x_0,v_0)-\Phi(x_0,v_0)\leq |v_0|^{\frac12-\epsilon}\left(c_0C-\frac{|v_0|^{\epsilon}}{c}\right).
\end{align*}
This and \eqref{ineq5.liou2} imply
\begin{align}\label{goal3.liou2}
        c_0H(x_0,v_0)-\Phi(x_0,v_0)\leq \frac12 \quad\text{in }\mathcal{A}^+,
    \end{align}
    if $\mathcal{M}_0$ is sufficiently large, where 
    \begin{align*}
        \mathcal{A}^+\coloneqq\{-v\geq 4x^{\frac13}\,:\,\,v\in[0,\mathcal{M}_0^{-1}]\}\cup \{-v\geq 4x^{\frac13}\,:\,\,v\in[\mathcal{M}_0,\infty)\}.
    \end{align*}
    This concludes the proof of \eqref{goal0.liou2} in case $-v_0 \ge 4 x_0^{\frac{1}{3}}$.
\end{itemize}
Therefore, combining \eqref{goal1.liou2}, \eqref{goal2.liou2}, and \eqref{goal3.liou2} leads to \eqref{goal0.liou2}.

Now using \eqref{goal0.liou2}, we prove $H\equiv 0$. Suppose $H$ is nonzero, then there is a point $H(x_0,v_0)\neq 0$ with $(x_0,v_0)\in \{x>0\}\times \bbR$. Thus there is a constant $c_0\in\bbR$ such that $(c_0H-\Phi)(x_0,v_0)\geq1$. By \eqref{goal0.liou2}, there must be a strict maximum point $(x_1,v_1)\in A_1\cup A_2$ such that $(c_0H-\Phi)(x_1,v_1)\geq1$ and $(c_0H-\Phi)(x_1,v_1)\geq (c_0H-\Phi)(x,v)$ for any $(x,v)\in \{x>0\}\times \bbR$. 

We prove that this leads to a contradiction. Indeed, consider the function 
\begin{align*}
    F\coloneqq (c_0H-\Phi)(x_1,v_1)-(c_0 H-\Phi)(x,v)\geq0,
\end{align*}
which is a weak solution to 
\begin{align*}
    v\partial_xF-\partial_{vv}F=0\quad\text{in } \mathcal{D},
\end{align*}
where we write
\begin{align*}
    \mathcal{D}\coloneqq \Big( \left[(\mathcal{M}_0/2)^{-3},(\mathcal{M}_0/2)^3\right]\times[-2\mathcal{M}_0\times 2\mathcal{M}_0] \Big) \cup \Big(  \left[0,(\mathcal{M}_0/2)^{-3}\right]\times \left[-2\mathcal{M}_0,-(2\mathcal{M}_0)^{-1}\right] \Big).
\end{align*}
\begin{itemize}
    \item Let $v_1\geq0$. Indeed, by \eqref{v0non.strmax} in \autoref{lem.strmax}, which remains valid (with the same proof, using only interior Harnack) if we replace $x_0$, $v_0$, and ${H}_{\mathcal{R}}$ by $(\mathcal{M}_0/4)^{-3}$, $v_1$, and $\left[(\mathcal{M}_0/2)^{-3},(\mathcal{M}_0/2)^3\right]\times[-2\mathcal{M}_0\times 2\mathcal{M}_0]$, we deduce
\begin{align*}
    \frac12\leq(c_0H-\Phi)(x_1,v_1)-(c_0 H-\Phi)\left((\mathcal{M}_0/4)^{-3},v_1\right)=F\left( (\mathcal{M}_0/4)^{-3},v_1\right)\leq cF(x_1,v_1)=0,
\end{align*}
where we have also used \eqref{goal0.liou2}. This is a contradiction. 
\item Let $v_1<0$. Similarly, by modifying the proof of \eqref{v0neg.strmax} of \autoref{lem.strmax} with $x_0$, $x_1$, $v_0$, and $H_{\mathcal{R}}$ replaced by $x_1$, ${(\mathcal{M}_0/4)^3}$, $v_1$, and $\left[0,(\mathcal{M}_0/2)^{-3}\right]\times \left[-2\mathcal{M}_0,-(2\mathcal{M}_0)^{-1}\right]$,  we deduce the following contradiction
\begin{align*}
    \frac12\leq(c_0H-\Phi)(x_1,v_1)-(c_0 H-\Phi)\left((\mathcal{M}_0/4)^3,v_1\right)=F\left((\mathcal{M}_0/4)^3,v_1\right)\leq cF(x_1,v_1)=0.
\end{align*}
\end{itemize}
{Therefore, it must be $H\equiv 0$, which gives $\partial_{vv}(h-\kappa_0\phi_0)=0$ by $vH-\partial_{vv}(h-\kappa_0\phi_0)=0$. This yields $h=\kappa_0\phi_0+c_1+c_2v$ for some constants $c_1,c_2$.
Lastly, by the boundary condition given in \eqref{eq.liou2}, we get $c_1=c_2=0$.} Hence, $h = \kappa_0 \phi_0$, as desired. This completes the proof.
\end{proof}

\subsection{Liouville theorem with in-flow condition}
\label{subsec:inflow-Liouville}
As a last step, it remains to derive a 1D Liouville theorem for solutions to inhomogeneous equations from \autoref{lem.liou2}, where all solutions to homogeneous equations were classified.
In this subsection we consider solutions satisfying the in-flow condition with a polynomial source term.
\begin{theorem}\label{thm.liou.inflow}
    Let $l\in\{-1,0\}$ and let $p\in \mathcal{P}_{3l+1}$. Assume that $h$ is a weak solution to 
\begin{equation}\label{eq.liou.inflow1}
\left\{
\begin{alignedat}{3}
v \partial_xh-\partial_{vv}h&=p&&\qquad \mbox{in  $\{x>0\}\times \bbR$}, \\
h&=\sum_{i=0}^{3(l+1)}b_iv^{i}&&\qquad  \mbox{in $ \{x=0\}\times \{v>0\}$},\\
|h(x,v)|&\leq c\left(1+(|x|^{\frac13}+|v|)^{\frac12+3(l+1)-\epsilon}\right)&&\qquad\mbox{in $\{x>0\}\times \bbR$}
\end{alignedat} \right.
\end{equation}
for some $c\geq1$ and $\epsilon\in(0,1/2)$. Then there exist $P\in\mathcal{P}_{3(l+1)}$, $\mu_j \in \R$, and $\kappa_i\in\bbR$, such that
\begin{equation}\label{form.thm}
    h=P+\sum_{j=0}^{l}\mu_j\psi_{j}+\sum_{i=0}^{l}\kappa_i\phi_i,
\end{equation}
where the functions $\phi_i$ and $\psi_j$ are defined in \eqref{rel.phiandU} and \eqref{defn.psilambda}, respectively.
\end{theorem}

In order to prove \autoref{thm.liou.inflow}, we have to find weak solutions satisfying the in-flow boundary condition with a right-hand side being a homogeneous polynomial.
\begin{lemma}\label{lem.exiinflow}
Let $\lambda \in \N \cup \{0\}$, $p \in \cP_{\lambda-2}$ be a homogeneous polynomial of degree $\lambda-2$ and $a \in \R$ be such that either $p\neq0$ or $a\neq0$. Then there is a weak solution $h$ to
    \begin{align*}
    \begin{cases}
        v \partial_x h - \partial_{vv} h &= p ~~ \text{ in } (0,\infty) \times \R, \\
        h(x,v) &= a v^{\lambda} ~~ \text{ on } \{ 0 \} \times \{ v > 0 \}.
    \end{cases}
    \end{align*}
Indeed, there exist $P\in \cP_\lambda$ and $\mu=\mu(\lambda,a,p) \in \R$ such that
\begin{align*}
    h=P+\mu\psi_{\left[\frac\lambda3\right]}.
\end{align*}
Moreover, if $\lambda=3l$ or $3l+1$ for some nonnegative integer $l$, then $\mu=0$.
\end{lemma}

The proof of this lemma will be provided in Section \ref{sec:appendix}.

With \autoref{lem.exiinflow} and \autoref{lem.liou2}, we are able to prove \autoref{thm.liou.inflow}.
\begin{proof}[Proof of \autoref{thm.liou.inflow}]
First, we consider the case $l=0$. We observe
    \begin{equation*}
\left\{
\begin{alignedat}{3}
v \partial_x(h-G)-\partial_{vv}(h-G)&=p-\partial_{vv}G\eqqcolon q&&\qquad \mbox{in  $\{x>0\}\times \bbR$}, \\
(h-G)&=0&&\qquad  \mbox{in $ \{x=0\}\times \{v>0\}$}, \\
|(h-G)(x,v)|&\leq M\left(1+(|x|^{\frac13}+|v|)^{\frac12+3-\epsilon}\right)&&\qquad\mbox{in $\{x>0\}\times \bbR$},
\end{alignedat} \right.
\end{equation*}
where we write 
\begin{align}\label{eq.G.liou.inflow1} 
    G\coloneqq \sum_{i=0}^{3}b_iv^i \quad \text{and} \quad q=\sum_{i=2}^{3}c_iq_i,
\end{align}
    where $c_i\in \bbR$ and the $q_i$ are homogeneous polynomials of degree $i-2$.
    Now by using \autoref{lem.exiinflow}, there is a weak solution $h_i$ to 
    \begin{equation*}
\left\{
\begin{alignedat}{3}
v \partial_xh_i-\partial_{vv}h_i&=c_iq_i&&\qquad \mbox{in  $\{x>0\}\times \bbR$}, \\
h_i&=0&&\qquad  \mbox{in $ \{x=0\}\times \{v>0\}$}
\end{alignedat} \right.
\end{equation*}
with $h_i$ being of the form 
\begin{align*}
    h_i=p_{i}+\nu_{i}\psi_{[i/3]},
\end{align*}
where $p_{i}\in \cP_{i}$ and $\nu_i\in\bbR$ with $\nu_{i}=0$ when $i= 3l$ or $i=3l+1$ for some nonnegative integer $l$. Therefore, we get that
\begin{align}\label{eq.H.liou.inflo1}
    H\coloneqq h-G-\sum_{i=0}^{3}(p_i+\nu_i\psi_{[i/3]})
\end{align}
is a weak solution to 
\begin{equation}\label{eq2.liou.inflow1}
\left\{
\begin{alignedat}{3}
v \partial_xH-\partial_{vv}H&=0&&\qquad \mbox{in  $\{x>0\}\times \bbR$}, \\
H&=0&&\qquad  \mbox{in $ \{x=0\}\times \{v>0\}$},\\
|H(x,v)|&\leq M\left(1+(|x|^{\frac13}+|v|)^{\frac12+3-\epsilon}\right)&&\qquad\mbox{in $\{x>0\}\times \bbR$},
\end{alignedat} \right.
\end{equation}
where we have used the fact that 
\begin{align*}
    |p_i(x,v)|+|\nu_i\psi_{[i/3]}(x,v)|\leq c\left(1+(|x|^{\frac13}+|v|)^{\frac12+3-\epsilon}\right)
\end{align*}
for some constant $c$, as $i\leq 3<\frac12+3-\epsilon$. Therefore, we can apply \autoref{lem.liou2} to \eqref{eq2.liou.inflow1} and see that
\begin{align*}
    H=\kappa_0\phi_0
\end{align*}
for some $\kappa_0\in\bbR$. By this, \eqref{eq.G.liou.inflow1}, and \eqref{eq.H.liou.inflo1}, we have
\begin{align*}
    h=\sum_{i=0}^{3}b_iv_i+\sum_{i=0}^{3}(p_i+\nu_i\psi_{[i/3]})+\kappa_0\phi_i=P+\mu_{0}\psi_{0}+\kappa_0\phi_0
\end{align*}
for some $P\in \cP_{3}$, $\mu_0\in\bbR$.
This gives the desired result. In case $l = -1$, we observe that $h$ also satisfies the third condition given in \eqref{eq.liou.inflow1} with $l=0$ so that $h$ is of the form \eqref{form.thm}. By the growth condition, we get $h=c$ for some constant $c$. This completes the proof.
\end{proof}

\subsection{Liouville theorem with unbounded source terms}
In this subsection, we classify solutions in the half-space whose equation involves unbounded source terms, such as
\begin{align*}
    v \partial_x h - \partial_{vv} h = \partial_{vv} \phi_0.
\end{align*}
We will actually consider more general source terms of the form \eqref{zero.liou.inflow2.Q}.
Such a classification result is needed in order to prove expansions for kinetic equations with \emph{non-constant coefficients} via a blow up argument (see \autoref{lem.exp.nd.a}).

Solutions to such equations are built from functions of the following form
\begin{align}\label{defn.Phii}
    {\pmb{\Phi}}_{i}=v^{i}\partial_{v}^{(i)}\phi_0,\,
    {\pmb{\Phi}}_{j+3}=v^{j}\partial_{v}^{(j)}\psi_0,\,
    {\pmb{\Phi}}_{5}=\partial_v\phi_1
\end{align}
for $i\in\{0,1,2\}$ and $j\in\{0,1\}$. Given $\lambda\in\{0,1,2,3\}$ and $\epsilon\in(0,1)$ with $\lambda+\epsilon< 3+\frac12$, we define the following space, which contains all solutions to equations involving unbounded source terms of the form \eqref{zero.liou.inflow2.Q} with suitable growth at infinity:
\begin{align*}
    \overline{\mathcal{T}}_{\lambda,\epsilon}\coloneqq\mathrm{span}\left\{p,p_i{\pmb{\Phi}_{i}},p_{j+3}{\pmb{\Phi}_{j+3}},p_5{\pmb{\Phi}}_{5}\,\Bigg|\begin{array}{l} p\in \cP_{\lambda}\,,p_i\in \cP_{[\lambda+\epsilon-\frac12]},\, p_{j+3}\in\cP_{\lambda-2},\,p_5\in\cP_{[\lambda+\epsilon-\frac52]}\\
    p_1(0)=p_{2}(0)= Dp_{2}(0)=p_4(0)=0
    \end{array}\right\},
\end{align*}
where $i\in\{0,1,2\}$ and $j\in\{0,1\}$.

When $\lambda+\epsilon<\frac32$, then the space $\overline{\mathcal{T}}_{\lambda,\epsilon}$ simplifies significantly, since any $P\in \overline{\mathcal{T}}_{\lambda,\epsilon}$ is of the form $P=p+c\phi_{0}$ for some $p\in\cP_1$ and $c\in\bbR$.

Next, we define the subspace ${\mathcal{T}}_{\lambda,\epsilon}$ of the vector space
$\overline{\mathcal{T}}_{\lambda,\epsilon}$ by
\begin{align}\label{defn.vectorT}
    {\mathcal{T}}_{\lambda,\epsilon}\coloneqq\left\{P\in \overline{T}_{\lambda,\epsilon}\,\Bigg|\, P= 0\text{ on }\{x=0\}\times \{v>0\}\right\}.
\end{align}

Note that $\pmb{\Phi}_{0},\pmb{\Phi}_{1},\pmb{\Phi}_{2}$ have homogeneity $\frac{1}{2}$, while $\pmb{\Phi}_{3},\pmb{\Phi}_{4}$ have homogeneity $2$, and $\pmb{\Phi}_{5}$ has homogeneity $\frac{5}{2}$. Hence, the polynomials $p_k$ are chosen precisely such that the products $p_k \pmb{\Phi}_k$ do not grow faster than $|z|^{\lambda + \eps}$ at infinity for any $k \in \{ 0 , \dots, 5\}$. Moreover, the assumptions $p_1(0)=p_{2}(0)= Dp_{2}(0)=p_4(0)=0$ (together with \eqref{zero.liou.inflow2.a}) guarantee that the space $\overline{\mathcal{T}}_{\lambda,\epsilon}$ remains ''invariant under the blow-up proof of \autoref{lem.exp.nd.a}`` in the sense that applying rescaled elements from $\overline{\mathcal{T}}_{\lambda,\epsilon}$ to an equation with suitably regular coefficients leads to a limiting equation that is again solved by an element from $\overline{\mathcal{T}}_{\lambda,\epsilon}$.

Now we provide the Liouville theorem with possibly unbounded source terms.
\begin{theorem}\label{thm.liou.inflow2}
    Let $\lambda \in \{0,1,2,3\}$, $\epsilon\in(0,1)$ with $\lambda+\epsilon<3+\frac12$, and $p \in \cP_{\lambda-2}$. Let $h$ be a weak solution to
    \begin{equation*}
\left\{
\begin{alignedat}{3}
v \partial_xh-\partial_{vv}h&=p+Q&&\qquad \mbox{in  $\{x>0\}\times \bbR$}, \\
h&=0&&\qquad\mbox{in $\{x=0\}\times \{v>0\}$},\\
|h(x,v)|&\leq c\left(1+(|x|^{\frac13}+|v|)^{\lambda+\epsilon}\right)&&\qquad\mbox{in $\{x>0\}\times \bbR$},
\end{alignedat} \right.
\end{equation*}where
    \begin{align}\label{zero.liou.inflow2.Q}
    {Q}&\coloneqq \sum_{i=0}^{2}{q}_{0,i}{\pmb{\Phi}}_{i}+\sum_{k=0}^{4}{q}_{1,k}\partial_{v}{\pmb{\Phi}}_{k}+\sum_{k=0,1,3}{q}_{2,k}\partial_{vv}{\pmb{\Phi}}_{k},
\end{align}
and  for $i\in\{0,1,2\}$ and $j\in\{0,1\}$,
\begin{equation}\label{zero.liou.inflow2.a}
    \begin{aligned}
        &{q}_{0,i}\in\cP_{[\lambda+\epsilon-\frac52]},\, {q}_{1,i}\in \cP_{[\lambda+\epsilon-\frac32]},\, {q}_{1,j+3}\in\cP_{\lambda-3},\, {q}_{2,i}\in \cP_{[\lambda+\epsilon-\frac12]}, \,{q}_{2,3}\in \cP_{\lambda-2},\\
        &{q}_{1,2}(0)={{q}}_{2,0}(0)={{q}}_{2,1}(0)={{q}}_{2,3}(0)=D{q}_{2,1}(0)=0.
    \end{aligned}
\end{equation}
Then, $h\in \mathcal{T}_{\lambda,\epsilon}$.
\end{theorem}

\begin{proof}
We claim that for any $Q$ as in \eqref{zero.liou.inflow2.Q}, \eqref{zero.liou.inflow2.a}, we can find $H \in \overline{\mathcal{T}}_{\lambda,\epsilon}$ such that
 \begin{equation}\label{eq.liou.inflow2}
\left\{
\begin{alignedat}{3}
v \partial_x H-\partial_{vv}H&=Q&&\qquad \mbox{in  $\{x>0\}\times \bbR$}, \\
H&=b_0v^2+b_1v^3&&\qquad\mbox{in $\{x=0\}\times \{v>0\}$},\\
|H(x,v)|&\leq c\left(1+(|x|^{\frac13}+|v|)^{\lambda+\epsilon}\right)&&\qquad\mbox{in $\{x>0\}\times \bbR$}
\end{alignedat}\right\}
\end{equation}
for some constant $b_0,b_1,c\in\bbR$. Note that once this claim is proved, we can deduce the desired result by applying \autoref{thm.liou.inflow} to $h-H$, which yields 
\begin{align*}
    h-H=P+c_0\phi_0+c_1\psi_0
\end{align*}
for some constants $c_0,c_1\in \bbR$ and $P\in \cP_{\lambda}$. Since $P,\phi_0,\psi_0\in \overline{\mathcal{T}}_{\lambda,\epsilon}$ and $h=H+P+c_0\phi_0+c_1\psi_0\equiv0$ on $\{x=0\}\times \{v>0\}$, we get $h\in \mathcal{T}_{\lambda,\epsilon}$, and conclude the proof.

Hence, it remains to prove the claim, i.e. to find $H \in \overline{\mathcal{T}_{\lambda,\epsilon}}$ such that \eqref{eq.liou.inflow2} holds true. To prove it, we note that by elementary computations using the fact that  \eqref{lemma:derx.phim} and \eqref{pde.phim}, it holds
    \begin{equation}\label{basic.formula.liou.inflow2}
    \begin{aligned}
        &(v\partial_x-\partial_{vv})(v{\pmb{\Phi}}_{0})=-2\partial_{v}{\pmb{\Phi}}_{0},\,(v\partial_x-\partial_{vv})(v^2{\pmb{\Phi}_{0}})=-2{\pmb{\Phi}}_{0}-4{\pmb{\Phi}}_{1},\,(v\partial_x-\partial_{vv})(v^2{\pmb{\Phi}}_{1})={\pmb{\Phi}}_{2},\\
        &(v\partial_x-\partial_{vv})(v{\pmb{\Phi}_{1}})=3\partial_v{\pmb{\Phi}}_{0}-5\partial_v{\pmb{\Phi}}_{1},\,(v\partial_x-\partial_{vv})(v^2{\pmb{\Phi}}_{1})=-6{\pmb{\Phi}}_{1}-7{\pmb{\Phi}}_{2},\\
        &(v\partial_x-\partial_{vv})(v^2{\pmb{\Phi}}_{2})=-10{\pmb{\Phi}}_{2}-10v\partial_v{\pmb{\Phi}}_{2},\,(v\partial_x-\partial_{vv}){\pmb{\Phi}}_{5}=c_0{\pmb{\Phi}}_{0},\\
        &v\partial_{v}{\pmb{\Phi}}_0={\pmb{\Phi}}_1,\,v\partial_{v}{\pmb{\Phi}}_1={\pmb{\Phi}}_1+{\pmb{\Phi}}_2,\, v\partial^2_v{\pmb{\Phi}}_0=\partial_v{\pmb{\Phi}}_1-\partial_v{\pmb{\Phi}}_0,\,v^2\partial_{v}^2{\pmb{\Phi}}_0={\pmb{\Phi}}_2,\,v^2\partial_v^2{\pmb{\Phi}}_1=-{\pmb{\Phi}}_2+v\partial_v{\pmb{\Phi}}_2.
    \end{aligned}
    \end{equation}
    Similarly, the same observations given in \eqref{basic.formula.liou.inflow2} hold with ${\pmb{\Phi}}_{0}$ and ${\pmb{\Phi}}_{1}$ replaced by ${\pmb{\Phi}}_{3}$ and ${\pmb{\Phi}}_{4}$, respectively, by \eqref{eq.der.psilambda} and \eqref{eq.reg.psilambda}. By using these formulas \eqref{basic.formula.liou.inflow2}, we can easily construct $H \in \overline{\mathcal{T}}_{\lambda,\epsilon}$ such that \eqref{eq.liou.inflow2} holds true, given any $Q$ as in \eqref{zero.liou.inflow2.Q}, \eqref{zero.liou.inflow2.a}. The proof is complete.
\end{proof}
\section{Liouville theorem in the half-space for in-flow}
\label{sec:Liouville-half-space}
In this section, we prove the higher dimensional Liouville theorem in the half-space for in-flow. 
\subsection{Liouville theorem with polynomials}
Let us recall the space $\mathcal{P}_\lambda$ containing all polynomials of degree less or equal than $\lambda$. Next for any $a>0$, $\epsilon\in(0,1)$, and $\lambda\in\{0,1,2,3\}$, we define vector spaces
\begin{align*}
    \mathcal{T}_{\lambda,\epsilon,a}\coloneqq  \left\{ \begin{array}{l}p_0+p_{0,1}\phi_{0,a}+p_{0,2}\partial_{v_n}\phi_{1,a}+q_{0,1}\psi_{0,a}\mbox{ with} \\
   p_0\in\cP_\lambda,\, p_{0,1}\in \cP_{[\lambda+\epsilon-\frac12]},\,q_{0,1}\in \cP_{\lambda-2},\,
   p_{0,2}\in \cP_{[\lambda+\epsilon-\frac52]}\end{array}\right\},
\end{align*}
where the functions $\phi_{m,a}(x_n,v_n)\coloneqq\phi_{m}(x_n/a,v_n/a)$ for $m\in\{0,1\}$ and $\psi_{0,a}\coloneqq\psi_{0}(x_n/a,v_n/a)$ solve the equation $(v_n\partial_{x_n,x_n}-a^2\partial_{v_n,v_n})f=0$.

We are now ready to prove a general Liouville theorem in the half-space with in-flow boundary condition when the data are given by polynomials.
\begin{theorem}\label{lemma:nD-classification-inflow}
    Let $\lambda\in \{0,1,2,3\}$, $\eps\in(0,1)$ with $\lambda+\eps<3+\frac12$, $p\in \cP_{\lambda-2}$, and $q\in \mathcal{P}_{\lambda}$. Let $f$ be a weak solution to
    \begin{equation}\label{eq.nd.inflow}
\left\{
\begin{alignedat}{3}
(\partial_t+v\cdot\nabla_x)f-a^{i,j}\partial_{v_i,v_j}f&=p&&\qquad \mbox{in  $\bbR\times\{x_n>0\}\times \bbR^n$}, \\
f&=q&&\qquad  \mbox{in $\bbR\times\{x_n=0\}\times \{v_n>0\}$},
\end{alignedat} \right.
\end{equation}
where $a^{i,j}$ is a constant matrix satisfying an uniformly ellipticity condition. Suppose that there is a constant $c\geq1$ such that for any $R>0$,
\begin{align}\label{cond.nd-inflow-poly}
    \|f\|_{L^\infty(\mathcal{H}_R)}\leq c(1+R)^{\lambda+\epsilon}.
\end{align}
Then, $f=P$ for some $P\in {\mathcal{T}}_{\lambda,\epsilon,a}$ with $a\coloneqq \sqrt{a^{n,n}}$. 
\end{theorem}

\begin{proof}
We may assume $a^{n,n}=1$. Next since $f-q$ solves 
\begin{equation*}
\left\{
\begin{alignedat}{3}
(\partial_t+v\cdot\nabla_x)(f-q)-a^{i,j}\partial_{v_i,v_j}(f-q)&=p+\widetilde{p}&&\qquad \mbox{in  $\bbR\times\{x_n>0\}\times \bbR^n$}, \\
f-q&=0&&\qquad  \mbox{in $\bbR\times\{x_n=0\}\times \{v_n>0\}$}
\end{alignedat} \right.
\end{equation*}
for some $\widetilde{p}\in \cP_{\lambda-2}$ and satisfies \eqref{cond.nd-inflow-poly} with $f$ replaced by $f-q$, we may assume $q=0$.

Now let us fix $h>0$ and $y=hy_i$ with $i\neq n$, where $y_i$ denotes the $i$-th unit vector in the $x$-coordinates.  Let us consider a function  
    \begin{align*}
        \delta_{h,y_i}f(t,x,v)\coloneqq f(t,x+y,v)-f(t,x,v)
    \end{align*}
    to see that $\delta_{h,y_i}f$ is a solution to 
    \begin{equation*}
\left\{
\begin{alignedat}{3}
(\partial_t+v\cdot\nabla_x)\delta_{h,y_i}f-a^{i,j}\partial_{v_i,v_j}\delta_{h,y_i}f&=\delta_{h,y_i}p&&\qquad \mbox{in  $\bbR\times\{x_n>0\}\times \bbR^n$}, \\
\delta_{h,y_i}f&=0&&\qquad  \mbox{in $\bbR\times\{x_n=0\}\times \{v_n>0\}$}.
\end{alignedat} \right.
\end{equation*}
 By $C^\alpha$-estimates for in-flow given in \cite[Lemma 2.25]{RoWe25}, we have 
\begin{align*}
    \|\delta_{h,y_i}f\|_{L^\infty(\mathcal{H}_R)}\leq c|h|^\alpha\left[[f]_{C^{{\alpha}}(\mathcal{H}_{2R})}\right]
    &\leq c|h|^{\alpha}R^{-\alpha}\left(\|f\|_{L^\infty(\mathcal{H}_{4R})}+\|R^2\delta_{h,y_i}p\|_{L^\infty(\mathcal{H}_{4R})}\right)\leq c|h|^\alpha R^{\lambda+\epsilon-\alpha},
\end{align*}
where we have chosen $R^3>2|h|$ and used the fact that $p\in \cP_{\lambda-2}$. By iterating this procedure $j$-times, we deduce
\begin{align*}
    \|\delta^j_{h,y_i}f\|_{L^\infty(\mathcal{H}_R)}\leq c|h|^{j\alpha}R^{\lambda+\epsilon-j\alpha},
\end{align*}
where we write $\delta^j_{h,y_i}f=\delta_{h,y_i}\circ \delta_{h,y_i}^{j-1}f$. Taking $j$ sufficiently large and $R\to\infty$, we deduce that $f$ is a polynomial with respect to the $x_i$ variables for any $i=1,\ldots,n-1$.  Similarly, applying the same argument to $\delta_{h,\tau}f(t,x,v)\coloneqq f(t+\tau,x,v)-f(t,x,v)$ yields that $f$ is also a polynomial with respect to the $t$-variable, and this gives 
\begin{align*}
    f=
    \sum_{\substack{\beta=(\beta_t,\beta_{x'},0)\in\mathcal{I}\\
    |\beta|\leq \lambda}}f_\beta(x_n,v)t^{\beta_t}\prod_{j=1}^{n-1}x_j^{\beta_{x_j}}\coloneqq \sum_{\beta\in \mathcal{I}}g_\beta(t,x,v),
\end{align*}
where we denote by $\mathcal{I} = (\N \cup \{ 0 \})^{1+2n}$ with $|\beta| = 2|\beta_t| + 3|\beta_x| + |\beta_v|$ being the kinetic degree of a multi-index  $\beta \in \cI$, and we used \eqref{cond.nd-inflow-poly} to ensure $|\beta|\leq \lambda$. 

\textbf{Step 1. }We aim to prove that for each $\beta\in\mathcal{I}$, 
\begin{equation}\label{goal1.nd.inflow}
    g_\beta(t,x,v)\coloneqq t^{\beta_t}\prod_{j=1}^{n-1}x_j^{\beta_{x_j}}\sum_{\gamma_\beta\in\mathcal{I}_\beta}h_{\gamma_\beta}(x_n,v_n)v_1^{\gamma_{\beta_1}}\cdots v_{n-1}^{\gamma_{\beta_{n-1}}}.
\end{equation} 

Let $M_0$ be the largest value $|\beta|$ in $\mathcal{I}$ so that $\delta_{h,\tau}^{\beta_t}\delta^{\beta_1}_{h,y_1}\cdots \delta^{\beta_{n-1}}_{h,y_{n-1}}f=c_\beta |h|^{\beta}f_\beta(x_n,v)$ for some constant $c_\beta\in\bbR$, and 
\begin{equation}\label{eq2.nd.inflow}
\left\{
\begin{alignedat}{3}
v_n\partial_{x_n}f_\beta(x_n,v)-a^{i,j}\partial_{v_i,v_j}f_\beta(x_n,v)&=\widetilde{p}_\beta&&\qquad \mbox{in  $\{x_n>0\}\times \bbR^n$}, \\
f_\beta(x_n,v)&=0&&\qquad  \mbox{in $\{x_n=0\}\times \{v_n>0\}$}
\end{alignedat} \right.
\end{equation}
for some polynomial $\widetilde{p}_\beta\in \cP_{\lambda-2-\beta}$. This follows by applying the operator $\delta_{h,\tau}^{\beta_t}\delta^{\beta_1}_{h,y_1}\cdots \delta^{\beta_{n-1}}_{h,y_{n-1}}$ to \eqref{eq.nd.inflow}. Since we can take $\delta^j_{h,w_i}$ for any $i=1,2,\ldots, n-1$ in \eqref{eq2.nd.inflow}, we derive 
\begin{equation*}
\left\{
\begin{alignedat}{3}
v_n\partial_{x_n}\delta^j_{h,w_i}f_\beta(x_n,v)-a^{i,j}\partial_{v_i,v_j}\delta^j_{h,w_i}f_\beta(x_n,v)&=\delta^j_{h,w_i}\widetilde{p}_\beta&&\qquad \mbox{in  $\{x_n>0\}\times \bbR^n$}, \\
\delta^j_{h,w_i}f_\beta(x_n,v)&=0&&\qquad  \mbox{in $\{x_n=0\}\times \{v_n>0\}$},
\end{alignedat} \right.
\end{equation*}
where we write
\begin{align*}
    \delta_{h,w_i}f(t,x,v)\coloneqq f(t,x,v+hw_i)-f(t,x,v)\quad\text{and}\quad \delta_{h,w_i}^jf=\delta_{h,w_i}^{j-1}\circ \delta_{h,w_i}f.
\end{align*}
Since $\delta^j_{h,w_i}f_\beta$ is also a weak solution to \eqref{eq.nd.inflow} with $p=\delta^j_{h,w_i}\widetilde{p}_\beta$ and $q=0$, by applying the $C^\alpha$- estimates for in-flow, we get that $f_{\beta}$ is a polynomial with respect to the $v_i$ variables. This gives \eqref{goal1.nd.inflow}. Inductively, we assume for any $\beta\in \mathcal{I}$ such that $|\beta|\geq M_0-j$ that $g_\beta$ is of the form given by \eqref{goal1.nd.inflow}. Let us choose $\beta'\in\mathcal{I}$ such that $|\beta'|=M_0-(j+1)$. Then we get 
\begin{align*}
    \frac{\delta_{h,\tau}^{\beta'_t}\delta^{\beta'_1}_{h,y_1}\cdots \delta^{\beta'_{n-1}}_{h,y_{n-1}}f}{|h|^{\beta'}}&=c_{\beta'}f_{\beta'}(x_n,v)\\
    &\quad+\underbrace{\sum_{\substack{\beta\in\mathcal{I}\\
    |\beta|>|\beta'|}}c_\beta t^{\beta_t-\beta'_t}x_1^{\beta_{x_1}-\beta'_{x_1}}\cdots x_{n-1}^{\beta_{x_{n-1}}-\beta'_{x_{n-1}}}\sum_{\gamma_\beta\in\mathcal{I}_\beta}c_{\gamma\beta}h_{\gamma_\beta}(x_n,v_n)v_1^{\gamma_{\beta_1}}\cdots v_{n-1}^{\gamma_{\beta_{n-1}}}}_{\eqqcolon \overline{f}(t,x,v)}
\end{align*}
and
\begin{align}\label{eq31.nd.inflow}
    v_n\partial_{x_n}f_{\beta'}-a^{i,j}\partial_{v_i,v_j}f_{\beta'}=\widetilde{p}_{\beta'}-{\underbrace{(\partial_t+v\cdot\nabla_x-a^{i,j}\partial_{i,j})\overline{f}}_{\eqqcolon F}}\quad\text{in }\bbR\times\{x_n>0\}\times \bbR^n,
\end{align}
where $\widetilde{p}_{\beta'}\in \cP_{\lambda-2-\beta'}$. Note that $F$ is a polynomial with respect to the $v_i$ variables for each $i=1,\ldots, n-1$ by the induction assumption. Thus, applying $\delta_{h,w_i}^{j}$ to \eqref{eq31.nd.inflow} for sufficiently large $j$, we get 
\begin{align*}
    v_n\partial_{x_n}\delta_{h,w_i}^jf_{\beta'}-a^{i,j}\partial_{v_i,v_j}\delta_{h,w_i}^jf_{\beta'}=\delta_{h,w_i}^j\widetilde{p}_{\beta'}\quad\text{in }\{x_n>0\}\times \bbR^n.
\end{align*}
Similarly, we deduce that $f_{\beta'}(x_n,v)$ is polynomial with respect to the $v_i$ variables for any $i=1,\ldots, n-1$. This implies that $g_{\beta'}$ is also of the form as in \eqref{goal1.nd.inflow}, and by induction, we get \eqref{goal1.nd.inflow} for any $\beta\in\mathcal{I}$.

Hence, we may write 
\begin{align*}
    f(t,x,v)\coloneq \sum_{\substack{\beta=(\beta_t,\beta_{x'},\beta_{v'})\in\mathcal{I},\\
    |\beta|\leq \lambda}}t^{\beta_t}\prod_{j=1}^{n-1}x_j^{\beta_{x_j}}\prod_{j=1}^{n-1}v_j^{\beta_{v_j}}h_{\beta}(x_n,v_n)=\sum_{\beta}F_\beta(t,x,v).
\end{align*}

\textbf{Step 2. }
Next, we want to prove that for each $\beta\in\mathcal{I}$, $h_\beta$ is of the form
\begin{equation}\label{goal2.nd.inflow}
\begin{aligned}
    h_\beta&=\widetilde{p}_{\beta}+\sum_{i=0}^2a_{\beta,i}v_n^{i}\phi_{m}+b_{\beta}\partial_{v_n}\phi_{1}+\sum_{i=0}^1c_{\beta,i}v_n^{i}\psi_{0},
\end{aligned}
\end{equation}
where $\widetilde{p}_\beta(x_n,v_n)\in \cP_\lambda$, $a_{\beta,m,i},b_{\beta,m,i},c_{\beta,l,j},d_{\beta,l,j}\in\bbR$, and by the growth condition given in \eqref{cond.nd-inflow-poly},
\begin{align}\label{cond.const.nd.inflow}
    \begin{cases}
        a_{\beta,i}=0\quad&\text{if }\frac12+i>\lambda+\epsilon,\\
        b_{\beta}=0\quad&\text{if }\frac52+j>\lambda+\epsilon,\\
        c_{\beta,i}=0\quad&\text{if }2+i>\lambda.
    \end{cases}
\end{align}
Let $M_1$ be the largest value of $|\beta|$, where $\beta\in\mathcal{I}$. Then we get
$\partial_t^{\beta_t}\partial_{x_1}^{\beta_{x_1}}\cdots \partial^{\beta_{x_{n-1}}}_{x_{n-1}}\partial_{v_1}^{{\beta_{v_1}}}\cdots \partial_{v_{n-1}}^{{\beta_{v_{n-1}}}}f=c_\beta h_{\beta}$,
\begin{equation}\label{eq3.nd.inflow}
\left\{
\begin{alignedat}{3}
v_n\partial_{x_n}h_{\beta}-a^{n,n}\partial_{v_n,v_n}h_{\beta}&=\widetilde{p}&&\qquad \mbox{in  $\{x_n>0\}\times \bbR$}, \\
f_\beta&=0&&\qquad  \mbox{in $\{x_n=0\}\times \{v_n>0\}$}
\end{alignedat} \right.
\end{equation}
for some polynomial $\widetilde{p} = \widetilde{p}(x_n,v_n)\in \cP_{\lambda}$. This follows by applying $\delta_{h,t}^{\beta_t}\delta_{h,x_1}^{\beta_{x_1}}\cdots \delta^{\beta_{x_{n-1}}}_{h,x_{n-1}}\delta_{h,v_1}^{{\beta_{v_1}}}\cdots \delta_{h,v_{n-1}}^{{\beta_{v_{n-1}}}}$ to the equation and then taking $h\to0$. By \autoref{thm.liou.inflow}, we obtain
\begin{align}\label{hbeta.nd.inflow}
    h_\beta=\widetilde{p}_\beta+a_{\beta,0}\phi_0+c_{\beta,0}\psi_{0}
\end{align}
for some constants $a_{\beta,j},c_{\beta,i}\in\bbR$ and $\widetilde{p}_\beta(x_n,v_n)\in \cP_{\lambda-2}$. This proves \eqref{goal2.nd.inflow} with $|\beta|=M_1$. Now we assume that \eqref{goal2.nd.inflow} with $|\beta|=M_1-k$ holds for some nonnegative integer $k$, and then we want to show \eqref{goal2.nd.inflow} with $|\beta|=M_1-(k+1)$. To do so, we observe
\begin{align*}
    p&=\sum_{\beta\in\mathcal{I}}(\partial_t+v\cdot\nabla_x-a^{i,j}\partial_{v_i,v_j})F_\beta\\
    &=\sum_{\beta\in\mathcal{I}}(\partial_t+v\cdot\nabla_x-a^{i,j}\partial_{v_i,v_j})\left(t^{\beta_t}\prod_{j=1}^{n-1}x_j^{\beta_{x_j}}\prod_{j=1}^{n-1}v_j^{\beta_{v_j}}h_{\beta}(x_n,v_n)\right)\\
    &=\sum_{\beta\in\mathcal{I}}h_{\beta}(x_n,v_n)\left(\partial_t+v'\cdot\nabla_{x'}-\sum_{i,j\neq n}a^{i,j}\partial_{v_i,v_j}\right)\left(t^{\beta_t}\prod_{j=1}^{n-1}x_j^{\beta_{x_j}}\prod_{j=1}^{n-1}v_j^{\beta_{v_j}}\right)\\
    &\quad-\sum_{\beta\in\mathcal{I}}\left(t^{\beta_t}\prod_{j=1}^{n-1}x_j^{\beta_{x_j}}\right)\sum_{\substack{i=n\text{ or }j=n\\
    (i,j)\neq (n,n)}}a^{i,j}\partial_{v_i,v_j}\left(\prod_{j=1}^{n-1}v_j^{\beta_{v_j}}h_\beta\right)\\
    &\quad+\sum_{\beta\in \mathcal{I}}\left(t^{\beta_t}\prod_{j=1}^{n-1}x_j^{\beta_{x_j}}\prod_{j=1}^{n-1}v_j^{\beta_{v_j}}\right) (v_n\partial_{x_n}-a^{n,n}\partial_{v_n,v_n})h_{\beta}.
\end{align*}
Let us now fix $\beta' \in \cI$ with $|\beta'| = M_1 - (k+1)$. As before, by applying the differential operator $\partial_{t}^{\beta'_t}\partial^{\beta'_1}_{x_1}\cdots \partial^{\beta'_{n-1}}_{x_{n-1}}\partial_{v_1}^{\gamma_{\beta'_1}}\cdots \partial_{v_{n-1}}^{\gamma_{\beta'_{n-1}}}$ to the previous display, we deduce for some $\overline{p} \in \cP_{\lambda - 2}$.
\begin{align*}
    \overline{p}&=\sum_{\substack{\beta\in\mathcal{I}\\
    |\beta|\geq M_1-k+1}}c_{\beta,1}h_{\beta}\left(\partial_t+v'\cdot\nabla_{x'}-\sum_{i,j\neq n}a^{i,j}\partial_{v_i,v_j}\right)\left(t^{\beta_t-\beta_t'}\prod_{j=1}^{n-1}x_j^{\beta_{x_j}-\beta'_{x_j}}\prod_{j=1}^{n-1}v_j^{\beta_{v_j}-\beta'_{v_j}}\right)\\
    &\quad-\sum_{\substack{\beta\in\mathcal{I}\\
    |\beta|\geq M_1-k}}c_{\beta,2}\left(t^{\beta_t-\beta_t'}x_1^{\beta_{x_1}-\beta'_{x_1}}\cdots x_{n-1}^{\beta_{x_{n-1}}-\beta'_{x_{n-1}}}\right)\sum_{\substack{i=n\text{ or }j=n\\
    (i,j)\neq (n,n)}}a^{i,j}\partial_{v_i,v_j}\left(v_1^{{\beta_{v_1}-\beta'_{v_1}}}\cdots v_{n-1}^{{\beta_{v_{n-1}}-\beta'_{v_n-1}}}h_\beta\right)\\
    &\quad+\sum_{\substack{\beta\in\mathcal{I}\\
    |\beta|\geq M_1-k}}c_{\beta,3}\left(t^{\beta_t-\beta_t'}x_1^{\beta_{x_1}-\beta'_{x_1}}\cdots x_{n-1}^{\beta_{x_{n-1}}-\beta'_{x_{n-1}}}v_1^{{\beta_{v_1}-\beta'_{v_1}}}\cdots v_{n-1}^{{\beta_{v_{n-1}}-\beta'_{v_n-1}}}\right)(v_n\partial_{x_n}-\partial_{v_n,v_n})h_\beta\\
    &\quad+c_4(v_n\partial_{x_n}-a^{n,n}\partial_{v_n,v_n})h_{\beta'}.
\end{align*}
Moreover, we observe that $h_{\beta'}(0,v_n)\equiv0$ when $v_n>0$. Then using this and the fact that $h_{\beta'}$ is independent of $t,x',v'$, we derive that for some $\widetilde{p} = \widetilde{p}(x_n,v_n)\in \cP_{\lambda-2}$,
\begin{align}\label{ind.step}
    (v_n\partial_{x_n}-a^{n,n}\partial_{v_n,v_n})h_{\beta'}&=\widetilde{p}+\sum_{\substack{\beta\in \mathcal{I}\\
    |\beta|=M_1-k+1}}c'_{\beta,1}h_\beta+\sum_{\substack{\beta\in \mathcal{I}\\
    |\beta|=M_1-k}}c'_{\beta,2}\partial_{v_n}h_\beta.
\end{align}
Before proceeding further, as in \eqref{basic.formula.liou.inflow2}, we observe
\begin{equation}\label{basic.formula.liou}
\begin{aligned}
    &(v_n\partial_{x_n}-\partial_{v_n,v_n})(c_0v_n^2\partial_{v_n}\phi_1+c_1v_n\phi_1)=v_n^2\phi_0,\,(v_n\partial_{x_n}-\partial_{v_n,v_n})(c_3v_n\partial_{v_n}\phi_1)=v_n\phi_0,\\
    &\,(v_n\partial_{x_n}-\partial_{v_n,v_n})(c_4v^3\phi_0+c_5v\partial_{v}\phi_1)=v_n^2\phi_1,\, \partial^2_{v_n}\phi_0=v_n\partial_{x_n}\phi_1=c_6v_n\phi_0.
\end{aligned}
\end{equation}
By the induction assumption \eqref{goal2.nd.inflow}, as well as by the observation \eqref{basic.formula.liou} and \eqref{basic.formula.liou.inflow2} together with the growth condition \eqref{cond.nd-inflow-poly}, we can deduce
\begin{equation*}
\left\{
\begin{alignedat}{3}
(v_n\partial_{x_n}-a^{n,n}\partial_{v_n,v_n})(h_{\beta'}-g)&=\widetilde{p}&&\qquad \mbox{in  $\{x_n>0\}\times \bbR$}, \\
(h_{\beta'}-g)&=\sum_{j\geq0}b_jv^j&&\qquad  \mbox{in $\{x_n=0\}\times \{v_n>0\}$},
\end{alignedat} \right.
\end{equation*}
where $b_j=0$ when $j>\lambda+\epsilon$ and
\begin{align*}
    g&=\widetilde{p}_{\beta}+\sum_{i=0}^2a_{i}v_n^{i}\phi_{m}+b_{0}\partial_{v_n}\phi_{1}+\sum_{i=0}^1c_{i}v_n^{i}\psi_{0}.
\end{align*}
In particular, the constants $a_{i},b_{i},c_{i}\in\bbR$ satisfy \eqref{cond.const.nd.inflow} with $a_{\beta,i},b_{\beta},c_{\beta,i}$ replaced by $a_{i},b_{0},c_{i}$, respectively. Applying \autoref{thm.liou.inflow}, we obtain \eqref{goal2.nd.inflow} for $\beta' \in \cI$. The inductive argument leads to \eqref{goal2.nd.inflow}.

A combination with \eqref{goal1.nd.inflow} and \eqref{goal2.nd.inflow} produces the desired result, which completes the proof.
\end{proof}

\subsection{Liouville theorem with unbounded right-hand sides}
In this subsection, we prove the Liouville theorem under the assumption that the source term is given by a confluent hypergeometric function and that the boundary data is zero. We define the following $n$-dimensional analogs of the spaces $\overline{\mathcal{T}}_{\lambda,\epsilon}$ and ${\mathcal{T}}_{\lambda,\epsilon}$ from the previous section for $\lambda\in\{0,1,2,3\}$ and $\epsilon\in(0,1)$ with $\lambda+\epsilon< 3+\frac12$,
\begin{align*}
    \overline{\mathcal{T}}_{\lambda,\epsilon}\coloneqq\mathrm{span}\left\{p,p_i{\pmb{\Phi}_{i}},p_{j+3}{\pmb{\Phi}_{j+3}},p_5{\pmb{\Phi}}_{5},\,\Bigg|\begin{array}{l} p\in \cP_{\lambda}\,,p_i\in \cP_{[\lambda+\epsilon-\frac12]},\, p_{j+3}\in\cP_{\lambda-2},\,p_5\in\cP_{[\lambda+\epsilon-\frac52]}\\
    p_1(0)=p_{2}(0)= Dp_{2}(0)=p_4(0)=0
    \end{array}\right\},
\end{align*}
and
\begin{align}\label{defn.vector.ndim}
    {\mathcal{T}}_{\lambda,\epsilon}\coloneqq\left\{P\in \overline{T}_{\lambda,\epsilon}\,\Bigg|\, P(z)\equiv 0\text{ on }\{x_n=0\}\times \{v_n>0\}\right\},
\end{align}
where the functions ${\pmb{\Phi}}_k(z)={\pmb{\Phi}}_k(x_n,v_n)$ are defined in \eqref{defn.Phii} for $k \in\{0,1,2,3,4,5\}$. 

In light of \autoref{lemma:nD-classification-inflow} and \autoref{thm.liou.inflow2}, we get the following.
\begin{lemma}\label{lem.liou.inflow2.nd}
    Let $\lambda\in \{0,1,2,3\}$ and $\epsilon\in(0,1)$ with $\lambda+\eps<3+\frac12$. Let $f$ be a weak solution to
    \begin{equation*}
\left\{
\begin{alignedat}{3}
(\partial_t+v\cdot\nabla_x)f-a^{i,j}\partial_{v_i,v_j}f&=p+Q&&\qquad \mbox{in  $\bbR\times\{x_n>0\}\times \bbR^n$}, \\
f&=0&&\qquad  \mbox{in $\bbR\times\{x_n=0\}\times \{v_n>0\}$},
\end{alignedat} \right.
\end{equation*}
where $a^{i,j}$ is a constant uniformly elliptic matrix with $a^{n,n}=1$ and where
\begin{align}
\label{eq.Q.half-space}
    {Q}&\coloneqq \sum_{i=0}^{2}{q}_{0,i}{\pmb{\Phi}}_{i}+\sum_{k=0}^{4}q_{1,k}\partial_{v_n}{\pmb{\Phi}}_{k}+\sum_{k=0,1,3}{q}_{2,k}\partial_{v_n,v_n}{\pmb{\Phi}}_{k},
\end{align}
and for $i\in\{0,1,2\}$ and $j\in\{0,1\}$,
\begin{equation}\label{zero.liou.inflow2-hs}
    \begin{aligned}
        &{q}_{0,i}\in\cP_{[\lambda+\epsilon-\frac52]},\, {q}_{1,i}\in \cP_{[\lambda+\epsilon-\frac32]},\, {q}_{1,j+3}\in\cP_{\lambda-3},\, {q}_{2,i}\in \cP_{[\lambda+\epsilon-\frac12]}, \,{q}_{2,3}\in \cP_{\lambda-2},\\
        &{q}_{1,2}(0)={q}_{2,0}(0)={q}_{2,1}(0)={q}_{2,3}(0)=D{q}_{2,1}(0)=0.
    \end{aligned}
\end{equation}
Suppose that is $c\geq1$ such that for any $R>0$,
\begin{align*}
    \|f\|_{L^\infty(\mathcal{H}_R)}\leq c(1+R)^{\lambda+\epsilon}.
\end{align*}
Then, $f \in {\mathcal{T}}_{\lambda,\epsilon}$.
\end{lemma}
\begin{proof}
First, we claim that for any $Q$ as in \eqref{eq.Q.half-space}, \eqref{zero.liou.inflow2-hs}, there is a function $H\in \overline{T}_{\lambda,\epsilon}$ such that $(\partial_t+v\cdot\nabla_x)H-a^{i,j}\partial_{v_i,v_j}H=Q$ and $H = q$ on $\bbR \times \{ x_n = 0 \} \times \{ v_n > 0 \}$ for some $q \in \cP_{\lambda}$. Once the claim is proved, we deduce the desired result by applying \autoref{lemma:nD-classification-inflow} to $f-H$.

By \autoref{thm.liou.inflow2}, the claim is already proved whenever $Q = Q(x_n,v_n)$. Due to \eqref{zero.liou.inflow2-hs} and the linearity of the equation, it remains to find $H\in \overline{\mathcal{T}}_{\lambda,\epsilon}$ in case 
\begin{align*}
    Q \in \{ v_i \partial_{v_n} \pmb{\Phi}_0 , v_i \partial_{v_n} \pmb{\Phi}_1, {v_i \partial_{v_n} \pmb{\Phi}_2},{v_i\partial_{v_n,v_n} \pmb{\Phi}_0} , v_iv_j \partial_{v_n,v_n} \pmb{\Phi}_0 , t \partial_{v_n,v_n} \pmb{\Phi}_0 ,  v_iv_j \partial_{v_n,v_n} \pmb{\Phi}_1 , t \partial_{v_n,v_n} \pmb{\Phi}_1 , v_i \partial_{v_n,v_n} \pmb{\Phi}_3\},
\end{align*}
where $i, j \in \{1 , \dots , n\}$ and $i \not= n$. Before proceeding further, as in \eqref{basic.formula.liou.inflow2}, we observe 
\begin{equation}\label{basic.formula.liou.inflow4}
\begin{aligned}
    &(v_n\partial_{x_n}-\partial_{v_n,v_n})(v_n{\pmb{\Phi}}_2)=-12\partial_{v_n}{\pmb{\Phi}}_0+12\partial_{v_n}{\pmb{\Phi}}_1-8\partial_{v_n}{\pmb{\Phi}}_2,\\
    &(v_n\partial_{x_n}-\partial_{v_n,v_n}){\pmb{\Phi}}_1=-4\partial_{v_n,v_n}{\pmb{\Phi}}_0,\,(v_n\partial_{x_n}-\partial_{v_n,v_n}){\pmb{\Phi}}_2=-24\partial_{v_n,v_n}{\pmb{\Phi}}_0-6\partial_{v_n,v_n}{\pmb{\Phi}}_1.
\end{aligned}
\end{equation}
In case $Q \in \{ v_i \partial_{v_n} \pmb{\Phi}_k \}$ for some $k \in \{0,1,2\}$, we can deduce the claim from \eqref{basic.formula.liou.inflow2} and \eqref{basic.formula.liou.inflow4} , since
\begin{align}\label{basic.formula.liou.inflow5}
    (\partial_t+v\cdot\nabla_x-a^{i,j}\partial_{v_i,v_j}) \left(v_i \sum_{l = 0}^2 c_{k,l} v_n  \pmb{\Phi}_l \right)-v_i\partial_{v_n}\pmb{\Phi}_k = \left(\sum_{l=0}^2a^{i,n}c_{k,l}(\pmb{\Phi}_l+v_n\partial_{v_n}\pmb{\Phi}_l)\right).
\end{align}
Moreover, we get from \eqref{basic.formula.liou.inflow4} that
\begin{align*}
    (\partial_t+v\cdot\nabla_x-a^{i,j}\partial_{v_i,v_j})(c_0v_i{\pmb{\Phi}}_1)-v_i\partial_{v_n,v_n}\pmb{\Phi}_0=c_1\partial_{v_n}\pmb{\Phi}_1,
\end{align*}
which proves the claim for the case $Q=v_i\partial_{v_n,v_n}\pmb{\Phi}_0$.
To prove the claim for $Q = q_{2,1}\partial_{v_n,v_n}{\pmb{\Phi}}_1$, where $q_{2,1} \in \{t , v_iv_j \}$, we observe from \eqref{basic.formula.liou.inflow2} that
    \begin{align*}
        (\partial_t+v\cdot\nabla_x-a^{i,j}\partial_{v_i,v_j})\left(q_{2,1}\sum_{k=1}^{2}c_{7,k}{\pmb{\Phi}_k}\right)-q_{2,1}\partial_{v_n,v_n}{\pmb{\Phi}}_1&=\partial_tq_{2,1}\sum_{k=1}^{2}c_{7,k}{\pmb{\Phi}_k}-a^{i,j}\partial_{v_i,v_j}q_{2,1}\left(\sum_{k=1}^{2}c_{7,k}{\pmb{\Phi}_k}\right)\\
        &\quad-a^{i,n} \partial_{v_i}q_{2,1} \sum_{k=1}^{2}\partial_{v_n}{\pmb{\Phi}_k} \eqqcolon \sum_{l=0}^{2}Q_l.
    \end{align*}
    Since $\mathrm{deg}(q_{2,1})\leq 2$, we note $Q_{l}=Q_{l}(x_n,v_n)$ for $l \in \{0,1\}$. To treat $Q_2$, we use \eqref{basic.formula.liou.inflow5} to see
    \begin{align*}
        &(\partial_t+v\cdot\nabla_x-a^{i,j}\partial_{v_i,v_j})\left(\partial_{v_m}q_{2,1}\sum_{l=0}^{2}c_{5,k}v_n{\pmb{\Phi}_l}\right)-\partial_{v_m}q_{2,1}\sum_{k=1}^2\partial_{v_n}{\pmb{\Phi}}_k\\
        &=-a^{i,n}\partial_{v_i,v_m}q_{2,1}\sum_{l=0}^{2}c_{5,k}({\pmb{\Phi}_k}+v_n\partial_{v_n}{\pmb{\Phi}_k})\eqqcolon Q_3(x_n,v_n).
    \end{align*}
    Since $Q_0+Q_1+ Q_3$ is of the form given in \eqref{zero.liou.inflow2.Q}, by \autoref{thm.liou.inflow2}, there is a function $H\in\overline{\mathcal{T}}_{\lambda,\epsilon}$ such that $(\partial_t+v\cdot\nabla_x-a^{i,j}\partial_{v_i,v_j})H=q_{2,1}\partial_{v_n v_n}{\pmb{\Phi}}_1$.
    The proof of the claim in the remaining cases $Q \in \{v_iv_j \partial_{v_n,v_n} \pmb{\Phi}_0 , t \partial_{v_n,v_n} \pmb{\Phi}_0, v_i \partial_{v_n,v_n} \pmb{\Phi}_3 \}$ goes in an analogous way. The proof is complete.
\end{proof}
\section{Expansion at grazing boundary points}
\label{sec:grazing-expansions}
In this section, we prove expansion estimates at the grazing set. To this end, for any  $z_0$ with $(x_{0})_n=(v_{0})_n=0$ and $a>0$, we introduce two vector spaces 
\begin{align*}
    \overline{\mathcal{T}}_{\lambda,\epsilon,z_0,a}\coloneqq\mathrm{span}\left\{p,p_i{\pmb{\Phi}_{i,a}},p_{j+3}{\pmb{\Phi}_{j+3,a}},p_5{\pmb{\Phi}}_{5,a}\,\Bigg|\begin{array}{l} p\in \cP_{\lambda}\,,p_i\in \cP_{[\lambda+\epsilon-\frac12]},\, p_{j+3}\in\cP_{\lambda-2},\,p_5\in\cP_{[\lambda+\epsilon-\frac52]}\\
    p_1(z_0)=p_{2}(z_0)= Dp_{2}(z_0)=p_4(z_0)=0
    \end{array}\right\},
\end{align*}
and
\begin{align}\label{defn.vector.ndimen.z_0.a}
    {\mathcal{T}}_{\lambda,\epsilon,z_0,a}\coloneqq\left\{P\in \overline{\mathcal{T}}_{\lambda,\epsilon,z_0,a}\,\Bigg|\, P(z)\equiv 0\text{ on }\{x_n=0\}\times \{v_n>0\}\right\},
\end{align}
where $i\in\{0,1,2\}$, $j\in\{0,1\}$ and  ${\pmb{\Phi}_{i,a}}(x_n,v_n)\coloneqq{\pmb{\Phi}}_i(x_n/a,v_n/a)$ are determined in \eqref{defn.Phii}. In case $a=1$, we denote ${\mathcal{T}}_{\lambda,\epsilon,z_0,a} = {\mathcal{T}}_{\lambda,\epsilon,z_0}$ and $\overline{{\mathcal{T}}}_{\lambda,\epsilon,z_0,a} = \overline{{\mathcal{T}}}_{\lambda,\epsilon,z_0}$. The reason to consider these rescaled functions is that they satisfy the equation $(v_n\partial_{x_n}-a^{2}\partial_{v_n,v_n}) \pmb{\Phi}_{0,a}=0$.

Now we are ready to provide the main result of this section.
\begin{proposition}\label{lem.exp.nd.a}
Let $\lambda\in\{0,1,2,3\}$ and $\epsilon\in(0,1)$ with $\lambda+\epsilon<3+\frac12$ and $\lambda+\epsilon - \frac12 \not \in \N \cup \{ 0 \}$.  Let $f$ be a weak solution to 
   \begin{equation*}
\left\{
\begin{alignedat}{3}
(\partial_t+v\cdot\nabla_x)f-\ddiv(A\nabla_vf)&=B\cdot\nabla_vf+F&&\qquad \mbox{in  $\mathcal{H}_R(z_0)$}, \\
f&=g&&\qquad  \mbox{in $\gamma_-\cap \mathcal{Q}_R(z_0)$},
\end{alignedat} \right.
\end{equation*}
where $(x_{0})_n=(v_{0})_n=0$, $R\leq1$, $A\in C^{\lambda-\frac12+\epsilon}(\mathcal{H}_R(z_0))$, $B\in C^{\lambda-\frac32 + \epsilon}(\mathcal{H}_R(z_0))$, $ F\in C^{\lambda-2 + \epsilon}(\mathcal{H}_R(z_0))$. Then there exists $P_{z_0,a}\in \overline{\mathcal{T}}_{\lambda,\epsilon,z_0,a}$ such that for any $r\in(0,R/2]$ and $z\in \mathcal{H}_r(z_0)$,
\begin{equation}\label{first.exp.nd}
\begin{aligned}
    |f(z)-P_{z_0,a}(z)|&\leq c\left(\frac{r}R\right)^{\lambda+\epsilon}\bigg(\|f\|_{L^\infty(\mathcal{H}_R(z_0))}+R^{\lambda+\epsilon}[F]_{C^{\lambda-2 + \epsilon}(\mathcal{H}_R(z_0))}+R^{\lambda+\eps}[g]_{C^{\lambda + \epsilon}(\gamma_-\cap \mathcal{Q}_R(z_0))}\bigg),
\end{aligned}
\end{equation}
where $a\coloneqq \sqrt{(A(z_0))_{n,n}}$ and $c=c(n,\Lambda,\lambda,\epsilon,\|A\|_{C^{\lambda-\frac12+\epsilon}(\mathcal{H}_R(z_0))},\|{B}\|_{C^{\lambda-\frac32+\epsilon}(\mathcal{H}_R(z_0))},{|v_0|{\mbox{\Large $\chi$}}_{\lambda\geq2}})$. In particular, when $\lambda\leq 1$, then 
\begin{equation}\label{estw/oB.exp.nd}
\begin{aligned}
    |f(z)-P_{z_0,a}(z)|&\leq c\left(1+ \|B\|_{C^{\lambda-\frac32+\epsilon}(\mathcal{H}_R(z_0))}\right)^{\lambda+\eps}\left(\frac{r}R\right)^{\lambda+\epsilon}\Bigg(\|f\|_{L^\infty(\mathcal{H}_R(z_0))}+R^2\|F\|_{L^{\infty}(\mathcal{H}_R(z_0))}\\
    &\qquad\qquad\qquad\qquad\qquad\qquad\qquad\qquad\quad\quad+R^{\lambda+\eps}[g]_{C^{\lambda + \epsilon}(\gamma_-\cap \mathcal{Q}_R(z_0))}\Bigg)
\end{aligned}
\end{equation}
for some constant $c=c(n,\Lambda,\epsilon,\|A\|_{C^{\epsilon}(\mathcal{H}_R(z_0))})$.
\end{proposition}

We provide the proof only for the case $(A(z_0))_{n,n}=1$ since the result in the general case can easily be deduced using scaling arguments. Moreover, we will explain below how to deduce \eqref{estw/oB.exp.nd} from \eqref{first.exp.nd}.

To see that it is sufficient to consider $(A(z_0))_{n,n}= 1$, take $a\coloneqq \sqrt{(A(z_0))_{n,n}}\in[\sqrt{\Lambda^{-1}},\sqrt{\Lambda}]$ and define
\begin{align*}
    &f_a(z)\coloneqq f(t,x',ax_n,v',av_n), \,F_a(z)\coloneqq F(t,x',ax_n,v',av_n),\,
    g_a(z)\coloneqq g(t,x',ax_n,v',av_n)
\end{align*}
to see that 
\begin{equation*}
\left\{
\begin{alignedat}{3}
(\partial_t+v\cdot\nabla_x)f_a-\ddiv(\widetilde{A}\nabla_vf_a)&=\widetilde{B}\cdot\nabla_vf+F_a&&\qquad \mbox{in  $\widetilde{{H}}_{R/a}(z_0)$}, \\
f_a&=g_a&&\qquad  \mbox{in $\gamma_-\cap \widetilde{{Q}}_{R/a}(z_0)$},
\end{alignedat} \right.
\end{equation*}
where 
\begin{align*}
    &\widetilde{Q}_{R/a}(z_0)\coloneqq\{(t,x',x_n,v',v_n)\,:\, t\in I_{R}(t_0),\, |(v',av_n)|<R,\, |(x'-x_0'-v_0'(t-t_0),ax_n)|<R^3\}, \\
    &\widetilde{H}_{R/a}(z_0)\coloneqq\widetilde{Q}_{R/a}(z_0)\cap (\bbR\times \{x_n>0\}\times \bbR^n),\\
    &\widetilde{A}_{i,j}(z)\coloneqq\begin{cases}
        A_{i,j}(t,x',ax_n,v',av_n)&\quad\text{if }i,j\neq n,\\
        A_{i,n}(t,x',ax_n,v',av_n)/a&\quad\text{if }i\neq n,\\
        A_{n,n}(t,x',ax_n,v',av_n)/a^2,
    \end{cases}  
    \quad \widetilde{B}_{i}(z)\coloneqq\begin{cases}B(t,x',ax_n,v',av_n)\quad ~~~\text{if }i\neq n,  \\
     B(t,x',ax_n,v',av_n)/a\quad\text{if }i=n.
    \end{cases}
\end{align*}
Now, we consider two cases depending on the range of $a$.
\begin{itemize}
    \item $a\geq1$. Then by the estimates given in \autoref{lem.exp.nd.a} with $(A(z_0))_{n,n}=1$ together with the fact that $\mathcal{Q}_{R/a}(z_0)\subset \widetilde{{Q}}_{R/a}(z_0)$ and $\mathcal{H}_{R/a}(z_0)\subset \widetilde{{H}}_{R/a}(z_0)$, there is a function $P\in \mathcal{\overline{T}}_{\lambda,\epsilon,z_0}$ such that for any $ra\leq R/(2a)$,
\begin{align*}
    \|f_a-P\|_{L^\infty(\mathcal{H}_{ra}(z_0))}&\leq c\left(\frac{2ra^2}{R}\right)^{\lambda+\epsilon}\Bigg(\|f_a\|_{L^\infty(\mathcal{H}_{R/a}(z_0))}+R^{\lambda+\epsilon}[F_a]_{C^{\lambda-2 + \epsilon}(\mathcal{H}_{R/a}(z_0))}\\
    &\qquad\qquad\qquad\qquad+R^{\lambda+\eps}[g_a]_{C^{\lambda+\epsilon}(\gamma_-\cap \mathcal{Q}_{R/a}(z_0))}\Bigg),
\end{align*}
where $c=c(n,\Lambda,\epsilon,\|\widetilde{A}\|_{C^{\lambda-\frac12+\epsilon}(\mathcal{H}_R(z_0))},\|\widetilde{B}\|_{C^{\lambda-1+\epsilon}(\mathcal{H}_R(z_0))},{|v_0|{\mbox{\Large $\chi$}}_{\lambda\geq2}})$. Therefore, by the change of variables, we have for any $r\leq R/(2a^2)$
\begin{align*}
    \|f-P_a\|_{L^\infty(\mathcal{H}_r(z_0))}&\leq c\left(\frac{r}{R}\right)^{\lambda+\epsilon}\bigg(\|f\|_{L^\infty(\mathcal{H}_{R}(z_0))}+R^{\lambda+\epsilon}[F]_{C^{\lambda-2 + \epsilon}(\mathcal{H}_{R}(z_0))}\\
    &\qquad\qquad\qquad\qquad+R^{\lambda+\eps}[g]_{C^{\lambda + \epsilon}(\gamma_-\cap \mathcal{Q}_{R}(z_0))}\bigg)\eqqcolon L_1(r),
\end{align*}
where $c=c(n,\Lambda,\epsilon,\|{A}\|_{C^{\lambda-\frac12+\epsilon}(\mathcal{H}_R(z_0))},\|{B}\|_{C^{\lambda-1+\epsilon}(\mathcal{H}_R(z_0))},{|v_0|{\mbox{\Large $\chi$}}_{\lambda\geq2}})$, as $a\in [\sqrt{\Lambda^{-1}},\sqrt{\Lambda}]$, and
\begin{align*}
    P_a(t,x',x_n,v',v_n) \coloneqq P(t,x',x_n/a,v',v_n/a)\in \mathcal{\overline{T}}_{\lambda,\eps,z_0,a}.
\end{align*}
To prove the expansion for $r \in ( R/(2a^2) , R/2)$, we observe that by taking $r=R/(2a^2)$, we get 
\begin{align*}
    \|P_a\|_{L^\infty(\mathcal{H}_{R/(2a)^2}(z_0))}\leq cL_1(R)
\end{align*}
for some constant $c=c(n,\Lambda,\epsilon,\|{A}\|_{C^{\lambda-\frac12+\epsilon}(\mathcal{H}_R(z_0))},\|{B}\|_{C^{\lambda-1+\epsilon}(\mathcal{H}_R(z_0))},{|v_0|{\mbox{\Large $\chi$}}_{\lambda\geq2}})$.
Since the vector space $\mathcal{\overline{T}}_{\lambda,\eps,z_0,a}$ is finite dimensional and $a \in [\Lambda^{-1} , \Lambda]$, we have for any $r \in (R/(2a^2) , R/2)$,
\begin{align*}
    \|P_a\|_{L^\infty(\mathcal{H}_{r}(z_0))}\leq c \|P_a\|_{L^\infty(\mathcal{H}_{R/(2a^2)}(z_0))},
\end{align*}
where $c = c(\Lambda)$. These two estimates imply the desired result when $r>R/(2a)^2$.

\item $a<1$. Then we have $\mathcal{Q}_{R}(z_0)\subset \widetilde{{Q}}_{R/a}(z_0)$ and $\mathcal{H}_{R}(z_0)\subset \widetilde{{H}}_{R/a}(z_0)$, and there is a function $P\in \mathcal{\overline{T}}_{\lambda,\epsilon,z_0}$ such that for any $r/a\leq R/2$,
\begin{align*}
    \|f_a-P\|_{L^\infty(\mathcal{H}_{r/a}(z_0))}&\leq c\left(\frac{2r}{R/a}\right)^{\lambda+\epsilon}\Bigg(\|f_a\|_{L^\infty(\mathcal{H}_{R}(z_0))}+R^{\lambda+\epsilon}[F_a]_{C^{\lambda-2 + \epsilon}(\mathcal{H}_{R}(z_0))}\\
    &\qquad\qquad\qquad\qquad+R^{\lambda+\eps}[g_a]_{C^{\lambda + \epsilon}(\gamma_-\cap \mathcal{Q}_{R}(z_0))}\Bigg),
\end{align*}
where $c=c(n,\Lambda,\epsilon,\|\widetilde{A}\|_{C^{\lambda-\frac12+\epsilon}(\mathcal{H}_R(z_0))},\|\widetilde{B}\|_{C^{\lambda-1+\epsilon}(\mathcal{H}_R(z_0))},{|v_0|{\mbox{\Large $\chi$}}_{\lambda\geq2}})$. Similarly, by the change of variables when $r\leq aR/2$ and by taking the constant $c$ sufficiently large when $r>aR/2$, we get the desired estimate.
\end{itemize}
Next, we derive the estimate \eqref{estw/oB.exp.nd} from \eqref{first.exp.nd} when $\lambda\leq1$. To this end, for any $\mathcal{B}>0$, we define
\begin{equation}\label{scaling.argument}
\begin{aligned}
    &f_{\mathcal{B}}(z)\coloneqq f( S_{\mathcal{B}}z),\, F_{\mathcal{B}}(z)\coloneqq \mathcal{B}^2F(S_{\mathcal{B}}z),\, g_{\mathcal{B}}(z)\coloneqq g( S_{\mathcal{B}}z),\\
    &\widetilde{A}\coloneqq A( S_{\mathcal{B}}z),\,\widetilde{B}(z)\coloneqq\mathcal{B}{B( S_{\mathcal{B}}z)},\, z_{\mathcal{B}}\coloneqq S_{\mathcal{B}^{-1}}z_0
\end{aligned}
\end{equation}
to see that
\begin{equation*}
\left\{
\begin{alignedat}{3}
(\partial_t+v\cdot\nabla_x)f_{\mathcal{B}}-\ddiv(\widetilde{A}\nabla_vf_{\mathcal{B}})&=\widetilde{B}\cdot\nabla_vf_{\mathcal{B}}+F_{\mathcal{B}}&&\qquad \mbox{in  ${\mathcal{H}}_{R/{\mathcal{B}}}(z_{\mathcal{B}})$}, \\
f_{\mathcal{B}}&=g_{\mathcal{B}}&&\qquad  \mbox{in $\gamma_-\cap {\mathcal{Q}}_{R/\mathcal{B}}(z_{\mathcal{B}})$}
\end{alignedat} \right.
\end{equation*}
with $z_{\mathcal{B}}\in \gamma_0$. Since $\Lambda^{-1}I\leq \widetilde{A}\leq \Lambda I$, by \eqref{first.exp.nd}, there is a function $\widetilde{P}_{0,a}\in\mathcal{\overline{T}}_{\lambda,\eps,0,a}$ with $a=\sqrt{(\widetilde{A}(z_{\mathcal{B}}))_{n,n}}$ such that for any $\rho\leq \min\left\{\frac{R}{\mathcal{B}},1\right\}\eqqcolon R_0\leq 1$ and $z\in\mathcal{H}_\rho(z_{\mathcal{B}})$,
\begin{align*}
    |f_{\mathcal{B}}(z)-\widetilde{P}_{0,a}(z)|&\leq c\left(\frac{\rho}{R_0}\right)^{\lambda+\epsilon}\Bigg(\|f_{\mathcal{B}}\|_{L^\infty(\mathcal{H}_{R_0}(z_{\mathcal{B}}))}+R_0^{\lambda+\epsilon}[F_{\mathcal{B}}]_{C^{\lambda-2 + \epsilon}(\mathcal{H}_{R_0}(z_{\mathcal{B}}))}\\
    &\qquad\qquad\qquad\qquad+R_0^{\lambda+\epsilon}[g_{\mathcal{B}}]_{C^{\lambda + \epsilon}(\gamma_-\cap \mathcal{Q}_{R}(z_{\mathcal{B}}))}\Bigg),
\end{align*}
where $c=c(n,\Lambda,\lambda,\epsilon,\|\widetilde{A}\|_{C^{\lambda-\frac12+\epsilon}(\mathcal{H}_{R_0}(z_{\mathcal{B}}))},\|\widetilde{B}\|_{C^{\lambda-\frac32+\epsilon}(\mathcal{H}_{R_0}(z_\mathcal{B}))})$. By choosing 
\begin{equation*}
\mathcal{B}\coloneqq \left(1+\|B\|_{C^{\lambda-\frac32+\eps(\mathcal{H}_R(z_0))}}\right)^{-1}\leq 1,
\end{equation*} we have
\begin{equation}\label{widea.wideb.exp.nd}
\begin{aligned}
    &\|\widetilde{A}\|_{C^{\lambda-\frac12+\epsilon}(\mathcal{H}_{R_0}(z_{\mathcal{B}}))}\leq\|{A}\|_{C^{\lambda-\frac12+\epsilon}(\mathcal{H}_{R_0\mathcal{B}}(z_{0}))}\leq \|A\|_{C^{\lambda-\frac12+\epsilon}(\mathcal{H}_{R}(z_0))},\\
    &\|\widetilde{B}\|_{C^{\lambda-\frac32+\epsilon}(\mathcal{H}_{R_0}(z_{\mathcal{B}}))}\leq\mathcal{B}\|{B}\|_{C^{\lambda-\frac32+\epsilon}(\mathcal{H}_{R_0\mathcal{B}}(z_{0}))}\leq \mathcal{B}\|B\|_{C^{\lambda-\frac32+\epsilon}(\mathcal{H}_{R}(z_{0}))}\leq c_B.
\end{aligned}
\end{equation}
Next, by scaling back together with the fact that $(\widetilde{A}(0))_{n,n}=(A(z_0))_{n,n}$ and by the homogeneity of the functions $\pmb{\Phi}_k$, there is ${P}_{0,a}\in\mathcal{\overline{T}}_{\lambda,\eps,0,a}$ such that for any $\rho\leq R_0 $ and $z\in\mathcal{H}_{\rho\mathcal{B}}(z_0)$, we have 
\begin{equation*}
\begin{aligned}
    |f(z)-{P}_{0,a}(z)|&\leq c\left(\frac{\rho}{R_0}\right)^{\lambda+\epsilon}\Bigg(\|f\|_{L^\infty(\mathcal{H}_{R}(z_0))}+{(R_0\mathcal{B})}^{\lambda+\epsilon}[F]_{C^{\lambda-2 + \epsilon}(\mathcal{H}_{R}(z_0))}\\
    &\qquad\qquad\qquad\qquad+{(R_0\mathcal{B})}^{\lambda+\epsilon}[g]_{C^{\lambda + \epsilon}(\gamma_-\cap \mathcal{Q}_{R}(z_0))}\Bigg),
\end{aligned}
\end{equation*}
where $c=c(n,\Lambda,\lambda,\epsilon,\|{A}\|_{C^{\lambda-\frac12+\epsilon}(\mathcal{H}_{R}(z_0))},c_B)$ by \eqref{widea.wideb.exp.nd}. Thus, for any $z\in \mathcal{H}_\rho(z_0)$, we have 
\begin{equation}\label{scaling.gamma0}
\begin{aligned}
    |f(z)-{P}_{0,a}(z)|&\leq c\left(\frac{\rho}{R}\right)^{\lambda+\epsilon}\Bigg(\|f\|_{L^\infty(\mathcal{H}_{R}(z_0))}+{R}^{\lambda+\epsilon}[F]_{C^{\lambda-2 + \epsilon}(\mathcal{H}_{R}(z_0))}+{R}^{\lambda+\epsilon}[g]_{C^{\lambda + \epsilon}(\gamma_-\cap \mathcal{Q}_{R}(z_0))}\Bigg),
\end{aligned}
\end{equation}
where $c=c(n,\Lambda,\lambda,\epsilon,\|{A}\|_{C^{\lambda-\frac12+\epsilon}(\mathcal{H}_{R}(z_0))},c_B)$, whenever $R\leq \mathcal{B}$.

 On the other hand, if $\frac{R}{\mathcal{B}}\geq1$, then we have for any $\rho \leq \mathcal{B}$ and $z\in\mathcal{H}_{\rho}(z_0)$,
\begin{equation}\label{scaling.gamma1}
\begin{aligned}
    |f(z)-{P}_{0,a}(z)|&\leq c\left(\frac{\rho}{\mathcal{B}}\right)^{\lambda+\epsilon}\Bigg(\|f\|_{L^\infty(\mathcal{H}_{R}(z_0))}+{\mathcal{B}}^{\lambda+\epsilon}[F]_{C^{\lambda-2 + \epsilon}(\mathcal{H}_{R}(z_0))}\\
    &\qquad\qquad\qquad\qquad+{\mathcal{B}}^{\lambda+\epsilon}[g]_{C^{\lambda + \epsilon}(\gamma_-\cap \mathcal{Q}_{R}(z_0))}\Bigg)\eqqcolon L_2(\rho).
\end{aligned}
\end{equation}
Since $P_{0,a}(z)=\sum\limits_{\beta\in\mathcal{I}}p_{0,a,\beta}$, where the $p_{0,a,\beta}$ are homogeneous functions of degree $\beta$ with $\beta\leq 3$, we have
\begin{align*}
    \|p_{0,a,\beta}(z)\|_{L^\infty(\mathcal{H}_{r}(z_0))}\leq c\left(\frac{r}{\mathcal{B}}\right)^{\beta}L_2(\mathcal{B}),
\end{align*}
whenever $r\geq \mathcal{B}$.
Therefore, using this and the fact that the set $\mathcal{I}$ is finite, and that $\mathcal{B}\leq r$, $R\leq 1$, we derive for $\mathcal{B}\leq\ r \leq \frac{R}2$,
\begin{equation*}
\begin{aligned}
    |f(z)-{P}_{0,a}(z)|&\leq \sum_{\beta\in\mathcal{I},|\beta|\leq 3}c\left(\frac{r}{\mathcal{B}}\right)^{\beta}L_2(\mathcal{B})\\
    &\leq c\left(\frac{\rho}{R}\right)^{\lambda+\eps}\frac{1}{\mathcal{B}^{\lambda+\eps}}\Bigg(\|f\|_{L^\infty(\mathcal{H}_{R}(z_0))}+{R}^{\lambda+\epsilon}[F]_{C^{\lambda-2 + \epsilon}(\mathcal{H}_{R}(z_0))}+{R}^{\lambda+\epsilon}[g]_{C^{\lambda + \epsilon}(\gamma_-\cap \mathcal{Q}_{R}(z_0))}\Bigg).
\end{aligned}
\end{equation*}
Therefore, using this, \eqref{scaling.gamma1}, and \eqref{scaling.gamma0}, we get \eqref{estw/oB.exp.nd}. Thus, it remains to prove only \eqref{first.exp.nd} with $(A(z_0))_{n,n}=1$.

\subsection{Basic lemmas} Before proving \autoref{lem.exp.nd.a}, we prove several lemmas which provide global uniform bounds of solutions when the source term is given by a multiple of a smooth function and another function from the following set
\begin{align}\label{defn.mathcalg3}
    \mathcal{G} \coloneqq\{1,\pmb{\Phi}_{k},\partial_{v_n}\pmb{\Phi}_{k},\partial_{v_n,v_n}\pmb{\Phi}_{k}\,:\,k\in\{0,1,2,3,4,5\} \}.
\end{align}

Recall that the functions $\pmb{\Phi}_k$ are defined in \eqref{defn.Phii} and source terms of this form occur when applying for instance $\phi_0$ to an equation with coefficients. Hence, the following results will be crucial in the proof of \autoref{lem.exp.nd.a} which is based on a blow-up argument.

The following lemma yields a H\"older regularity estimate. We will prove it by Campanato-type perturbation arguments. We will use crucially that for any $\pmb{\Phi}\in\mathcal{G}$, it holds $\pmb{\Phi} = \pmb{\Phi}(x_n,v_n) \in L^{2+\epsilon_0}(H_R)$ and $\partial_{v_n} \pmb{\Phi}_0 \in L^{4+\epsilon_0}(H_R)$ for any $R>0$ with $\epsilon_0=\frac13$ (see \autoref{rmk.integrable.derv.phim}).

\begin{lemma}\label{lem.bdry.hol}
    Let $f$ be a weak solution to 
    \begin{equation}\label{eq.bdry.hol}
\left\{
\begin{alignedat}{3}
(\partial_t+v\cdot\nabla_x)f-\ddiv(A\nabla_vf)&=B\cdot\nabla_vf+F\pmb{\Phi}-\ddiv(G\partial_{v_n}{\pmb{\Phi}}_0)&&\qquad \mbox{in  $\mathcal{H}_1$}, \\
f&=g&&\qquad  \mbox{in $\gamma_-\cap \mathcal{Q}_1$},
\end{alignedat} \right.
\end{equation}
where $B\in L^{\infty}(\mathcal{H}_1;\bbR^n)$, $F\in L^{\infty}(\mathcal{H}_1)$, $G\in L^\infty(\mathcal{H}_1;\bbR^n)$, $g\in C^{\alpha}(\gamma_-\cap \mathcal{Q}_1)$, and $\pmb{\Phi} \in \mathcal{G}$. Then there is a small constant $\beta=\beta(n,\Lambda,\alpha)$ such that 
\begin{align*}
    [f]_{C^{\beta}(\mathcal{H}_{1/2})}&\leq c\left(\|f\|_{L^2(\mathcal{H}_{1})}+[g]_{C^{\alpha}(\gamma_-\cap \mathcal{Q}_{1})}+\|F\|_{L^\infty(\mathcal{H}_1)}+\|G\|_{L^\infty(\mathcal{H}_1)}\right)
\end{align*}
for some constant $c=c(n,\Lambda,\alpha,\|B\|_{L^\infty(\mathcal{H}_1)})$.
\end{lemma}
\begin{proof}
Let $z_0\in \gamma$ and $r\leq 1/16$. We want to prove that
\begin{align}\label{goal1.bdry.hol}
    \dashint_{\mathcal{H}_{r}(z_0)}|f-(f)_{\mathcal{H}_r(z_0)}|\leq cr^{\beta}(M_0+[g]_{C^{\alpha}(\gamma_-\cap\mathcal{Q}_{r}(z_0))})
\end{align}
for some $\beta=\beta(n,\Lambda)$, where $c=c(n,\Lambda,\|B\|_{L^\infty(\mathcal{H}_1)})$ and $M_0\coloneqq \|f\|_{L^2(\mathcal{H}_1)}+\|F\|_{L^\infty(\mathcal{H}_1)} +\|G\|_{L^\infty(\mathcal{H}_1)}$.
Let us assume $z_0\in \gamma_0$.
First, by \cite[Corollary 2.9]{Zhu24}, there is a unique weak solution $h$ to 
\begin{equation}\label{comp.eq.bdry.hol}
\left\{
\begin{alignedat}{3}
(\partial_t+v\cdot\nabla_x)h-\ddiv(A\nabla_v h)&=B\nabla_v h&&\qquad \mbox{in  $\mathcal{H}_r(z_0)$}, \\
h&=f&&\qquad  \mbox{in $\partial_{\mathrm{kin}}\mathcal{H}_r(z_0)$}.
\end{alignedat} \right.
\end{equation}
Let us write $\widetilde{F}\coloneqq -F\pmb{\Phi}$ and $\widetilde{G}\coloneqq G{\pmb{\Phi}}_0$, and denote $x_0=(x_0',0)$ and $v_0=(v_0',0)$ to see that
\begin{equation}\label{ineq1.bdry.hol}
\begin{aligned}
    \dashint_{\mathcal{H}_r(z_0)}|r^2\widetilde{F}|^2 &+|r\widetilde{G}|^2\,dz
    \\
    &\leq r^4\|F\|^2_{L^\infty(\mathcal{H}_1)}\left(\sup_{t\in I_r(t_0)}\dashint_{B_{r^3}}\dashint_{{B}_r}|\pmb{\Phi}|^{2+\epsilon_0}\,dx_n\,dv_n\right)^{\frac2{2+\epsilon_0}}\\
    &\,+r^2\|G\|^2_{L^\infty(\mathcal{H}_1)}\left(\sup_{t\in I_r(t_0)}\dashint_{B_{r^3}}\dashint_{{B}_r}|{\pmb{\Phi}}_0|^{4+\epsilon_0}\,dx_n\,dv_n\right)^{\frac2{4+\epsilon_0}}\\
    &\leq cr^{\frac{2\epsilon_0}{4+\epsilon_0}}(M_0)^2,
\end{aligned}
\end{equation}
where we have used H\"older's inequality,  $\|\pmb{\Phi}\|_{L^{2+\epsilon_0}(H_1)},\|\partial_{v_n}{\pmb{\Phi}}_0\|_{L^{4+\epsilon_0}(H_1)}\leq c$ with $\epsilon_0=\frac13$.
Since $f-h$ solves the following equation
\begin{equation*}
\left\{
\begin{alignedat}{3}
(\partial_t+v\cdot\nabla_x)(f-h)-\ddiv(A\nabla_v(f-h))&=B\nabla_v(f-h)+\widetilde{F}-\ddiv(\widetilde{G})&&\qquad \mbox{in  $\mathcal{H}_r(z_0)$}, \\
f-h&=0&&\qquad  \mbox{in $\partial_{\mathrm{kin}}\mathcal{H}_r(z_0)$}
\end{alignedat} \right.
\end{equation*}
we can test with $f-h$, and deduce
\begin{equation}\label{ineq2.bdry.hol}
\begin{aligned}
    \frac{1}{\Lambda}\dashint_{\mathcal{H}_r(z_0)}|\nabla_v(f-h)|^2\,dz& \leq \|B\|_{L^\infty(\mathcal{H}_r(z_0))}\dashint_{\mathcal{H}_r(z_0)}|\nabla_v(f-h)|f-h|\\
    &\quad+\dashint_{\mathcal{H}_r(z_0)}|\widetilde{F}||f-h|^2+\dashint_{\mathcal{H}_r(z_0)}|\widetilde{G}||\nabla_v(f-h)|.
\end{aligned}
\end{equation}
Next note from the Poincar\'e inequality with respect to the $v$ variable and using also \eqref{ineq1.bdry.hol}, \eqref{ineq2.bdry.hol}, and Young's inequality, we get 
\begin{align*}
   \dashint_{\mathcal{H}_r(z_0)}|f-h|^2\,dz\leq c \dashint_{\mathcal{H}_r(z_0)}r^2|\nabla_v(f-h)|^2\,dz\leq c\left(\|B\|_{L^\infty(\mathcal{H}_1)}^2 r^2\dashint_{\mathcal{H}_r(z_0)}|f-h|^2\,dz+cr^{\frac{2\epsilon_0}{8+\epsilon_0}}M_0^2\right)
\end{align*}
for some constant $c=c(n,\Lambda)$. Hence, there is a small constant $r_0=r_0(n,\Lambda,\|B\|_{L^\infty(\mathcal{H}_1)})$ such that when $r\leq r_0$, then it holds
\begin{align*}
     \dashint_{\mathcal{H}_r(z_0)}|f-h|^2\,dz\leq cr^{\frac{2\epsilon_0}{8+\epsilon_0}}M_0^2. 
\end{align*}
Note that when $r>r_0$, then \eqref{goal1.bdry.hol} follows directly. Thus, we assume $r\leq r_0$. Next, by \cite[Lemma 2.25]{RoWe25} together with \autoref{lem.rep.hol2}, we derive
\begin{align*}
    [h]_{C^{\alpha}(\mathcal{H}_{r/16}(z_0))}\leq c\left( r^{-(2n+1)}\|r^{-\alpha}h\|_{L^2(\mathcal{H}_{r/8}(z_0))}+[g]_{C^{\alpha}(\gamma_-\cap \mathcal{Q}_{r/8}(z_0))}\right),
\end{align*}
as $f=g$ on $\gamma_-\cap \mathcal{Q}_{r/8}(z_0)$. This together with the fact that $h-(h)_{\mathcal{H}_{r/8}(z_0)}$ solves \eqref{comp.eq.bdry.hol} with $f$ replaced by $f-(h)_{\mathcal{H}_{r/8}(z_0)}$ implies that for any $\rho\in (0,1/16]$,
\begin{align*}
    \dashint_{\mathcal{H}_{\rho r}(z_0)}|h-(h)_{\mathcal{H}_{\rho r}(z_0)}|^2 \leq c \rho^{2\alpha}\dashint_{\mathcal{H}_{ r}(z_0)}|h-(h)_{\mathcal{H}_{r/8}(z_0)}|^2\,dz+ c r^{2\alpha}[g]_{C^{\alpha}(\gamma_-\cap \mathcal{Q}_{r/8}(z_0))}^2.
\end{align*}

Therefore, we have 
\begin{align*}
    \dashint_{\mathcal{H}_{\rho r}(z_0)}|f-(f)_{\mathcal{H}_{\rho r}(z_0)}|^2\,dz&\leq c \dashint_{\mathcal{H}_{\rho r}(z_0)}|h-(h)_{\mathcal{H}_{\rho r}(z_0)}|^2\,dz+ c\dashint_{\mathcal{H}_{\rho r}(z_0)}|f-h|^2\,dz \\
    &\leq c \rho^{2\alpha}\dashint_{\mathcal{H}_{ r/8}(z_0)}|h-(h)_{\mathcal{H}_{ r/8}(z_0)}|^2+cr^{2\alpha}[g]^2_{C^\alpha(\gamma_-\cap \mathcal{Q}_{r/8}(z_0))} \\
    &\quad +c\rho^{-(4n+2)}\dashint_{\mathcal{H}_{ r/8}(z_0)}|f-h|^2\,dz\\
    &\leq c \rho^{2\alpha}\dashint_{\mathcal{H}_{ r/8}(z_0)}|f-(f)_{\mathcal{H}_{ r/8}(z_0)}|^2+cr^{2\alpha}[g]^2_{C^\alpha(\gamma_-\cap \mathcal{Q}_{r/8}(z_0))} \\
    &\quad +c\rho^{-(4n+2)}r^{\frac{2\epsilon_0}{8+\epsilon_0}}M_0^2
\end{align*}
for some constant $c=c(n,\Lambda)$.
Now, taking $\rho^{2(\alpha-\beta)}\leq \frac1c$ with $\beta = \frac12\min\left\{\alpha,\frac{ \epsilon_0}{8 + \epsilon_0}\right\}$ and following the standard iteration as in \cite[Lemma 5.6]{KLN25}, we get \eqref{goal1.bdry.hol} when $z_0\in \gamma_0$. 

We now assume $z_0\in \gamma\setminus \gamma_0$. If $\mathcal{H}_{2r}(z_0)\cap \gamma_0\neq\emptyset$, then there is a point $z_1\in\gamma_0$ such that $\mathcal{H}_{2r}(z_0)\subset\mathcal{H}_{8r}(z_1)$. Indeed, we can choose $z_1=(t_0,x_1',0,v_1',0)$, where $(t_1,x_1',0,v_1',0)\in\mathcal{H}_{2r}(z_0)\cap \gamma_0$. In addition, in this case, $|v_{0,n}|\leq 2r$, and we have $|\mathcal{H}_{r}(z_0)|\eqsim r^{4n+2}$. Therefore, we get 
\begin{align}\label{ineq.almost.bdry.hol}
    \dashint_{\mathcal{H}_r(z_0)}|f-(f)_{\mathcal{H}_r(z_0)}|^2\,dz\leq c\dashint_{\mathcal{H}_{8r}(z_1)}|f-(f)_{\mathcal{H}_{8r}(z_1)}|^2\,dz\leq cr^\beta,
\end{align}
where we have also used \eqref{goal1.bdry.hol} together with the fact that $z_1\in\gamma_0$. We assume $\mathcal{H}_{2r}(z_0)\cap \gamma_0=\emptyset$ so that $2r<|v_{0,n}|$. Then by \cite[Lemma 2.25]{RoWe25}, we have for $r_0\coloneqq |v_{0,n}|/2$,
\begin{align*}
    [f]_{C^{\alpha}(\mathcal{H}_{r_0}(z_0))}\leq c\left(|\mathcal{H}_{r_0}(z_0)|^{-\frac12}r_0^{-\alpha}\|f-(f)_{\mathcal{H}_{3r_0/2}(z_0)}\|_{L^2(\mathcal{H}_{3r_0/2}(z_0))}+r_0v_0^{-\frac12}M_0+[g]_{C^\alpha(\gamma_-\cap\mathcal{Q}_{3r_0/2}(z_0))}\right),
\end{align*}
where we have also used the fact that $ |\pmb{\Phi}(z)|,|\partial_{v_n} \pmb{\Phi}_0(z)|\leq cv_0^{-\frac12}$ when $z\in \mathcal{H}_{3r_0/2}(z_0)$, as $|v_{0,n}|-3r_0/2\geq |v_{0,n}|/4$.
Thus, we deduce 
\begin{align}\label{ineq2.almost.bdry.hol}
    \dashint_{\mathcal{H}_r(z_0)}|f-(f)_{\mathcal{H}_{r}(z_0)}|^2\leq c\left(\frac{r}{r_0}\right)^{\beta}\dashint_{\mathcal{H}_{3r_0/2}(z_0)}|f-(f)_{\mathcal{H}_{3r_0/2}(z_0)}|^2+cr^\beta(M_0+[g]_{C^\alpha(\gamma_-\cap\mathcal{Q}_1)}).
\end{align}
Since $\mathcal{H}_{3r_0(z_0)}\cap \gamma_0\neq\emptyset$, plugging  \eqref{ineq.almost.bdry.hol} into \eqref{ineq2.almost.bdry.hol} leads to \eqref{goal1.bdry.hol} when $z_0\in\gamma\setminus \gamma_0$.
Now, combining \eqref{goal1.bdry.hol} and the interior regularity results from \cite{GIMV19} leads to the desired estimate.
\end{proof}

Based on \cite[Lemma 2.27]{RoWe25}, we now use a compactness argument to replace  the $L^\infty$-norm  by the H\"older seminorm. As we will use this result only for equations without source terms in divergence form, we will prove the result only in this setting.
\begin{lemma}\label{lem.bdry.hol.rep}
    Let $f$ be a weak solution to 
    \begin{equation*}
(\partial_t+v\cdot\nabla_x)f-\ddiv(A\nabla_vf)=B\cdot\nabla_vf+F\pmb{\Phi}\quad \mbox{in  $\mathcal{H}_1$},
\end{equation*}
where $A\in C^{\alpha}(\mathcal{H}_1)$ and $B\in C^{\alpha}(\mathcal{H}_1;\bbR^n)$ for some $\alpha>0$,  $F\in C^{\lambda, \eps}(\mathcal{H}_1)$ for some nonnegative integer $\lambda \geq 0$ and $\epsilon \in (0,1)$, and $\pmb{\Phi} \in \mathcal{G}$. Then, 
\begin{align*}
    \|F\|_{L^\infty(\mathcal{H}_1)}
    &\leq c\left(\|f\|_{L^1(\mathcal{H}_1)}+[F]_{C^{\lambda , \epsilon}(\mathcal{H}_1)}\right),
\end{align*}
where $c$ depends only on $n,\Lambda,\lambda,\epsilon,\alpha,\|A\|_{C^{\alpha}(\mathcal{H}_1)}, \|B\|_{C^{\alpha}(\mathcal{H}_1)}$. 
\end{lemma}

\begin{proof}
    Let us fix any polynomial $p\in \cP_{\lambda}$ and write $F=(F-p)+p$.
    We want to show that 
    \begin{equation}\label{goal.bdry.hol.rep}
    \begin{aligned}
        \|p\|_{L^\infty(\mathcal{H}_1)}
       &\leq c\|f\|_{L^1(\mathcal{H}_1)}+\|F-p\|_{C^\alpha(\mathcal{H}_1)}.
    \end{aligned}
    \end{equation}
    Now by following the argument given in the proof of \cite[Lemma 2.7]{RoWe25}, we derive that if \eqref{goal.bdry.hol.rep} is not true, then there are sequences $(A_l)_l$, $(B_l)_l$, $(F_l)_l$, $(f_l)_l$, $(p_l)_{l}$ such that 
\begin{equation*}
(\partial_t+v\cdot\nabla_x)f_l-\ddiv(A_l\nabla_vf_l)=B_l\nabla_vf_l+F_l\pmb{\Phi}\qquad \mbox{in  $\mathcal{H}_1$},
\end{equation*}
where 
\begin{align*}
    \|A_l\|_{C^{\alpha}(\mathcal{H}_1)} + \|B_l\|_{C^\alpha(\mathcal{H}_1)}\leq \Lambda\quad\text{and}\quad p_l\in \cP_\lambda,
\end{align*}
as well as
\begin{align}\label{contr.assm.bdry.hol.rep}
        \|p_l\|_{L^\infty(\mathcal{H}_1)}=1,
    \end{align}
and
\begin{align}\label{ineq.bdry.hol.rep}
    \|f_l\|_{L^1(\mathcal{H}_1)}+\|F_l-p_l\|_{C^\alpha(\mathcal{H}_1)}\to 0.
\end{align}
By \eqref{defn.mathcalg3} and the explicit formulae of ${\pmb{\Phi}} \in \mathcal{G}$ given in the appendix, we observe $[\pmb{\Phi}]_{C^{\alpha}(Q_r(z_0))}\leq c(\alpha,r,z_0)$ for any $Q_r(z_0)\Subset \{x_n>0\}\times \bbR$ for each $i=1,2$. Thus, using \cite[Theorem 1.15]{KLN25}, we get
\begin{align*}
    \|f_l\|_{C^{1,\alpha}(\mathcal{Q}_r(z_0))}\leq c(\alpha,r,z_0)
\end{align*}
for any $\mathcal{Q}_r(z_0)\Subset \mathcal{H}_1$, where the constant $c$ is independent of $l$. This implies that there is a function $f_\infty$ such that $f_l\to f_\infty$ in $C^{1,\alpha}(\mathcal{Q}_r(z_0))$, whenever $\mathcal{Q}_r(z_0)\Subset \mathcal{H}_1$. Using this and \eqref{ineq.bdry.hol.rep}, we get
\begin{align*}
    \int_{\mathcal{Q}_r(z_0)}-f_\infty(\partial_t+v\cdot\nabla_x)\psi&=-\int_{\mathcal{Q}_r(z_0)}A_\infty\nabla_vf_\infty\nabla_v\psi+\int_{\mathcal{Q}_r(z_0)}B_\infty\cdot\nabla_v\psi\,dz
    +\int_{\mathcal{Q}_r(z_0)}p_\infty\pmb{\Phi}\psi
\end{align*}
for any $\psi\in C^\infty_c(\mathcal{Q}_r(z_0))$. 
At the same time, by \eqref{ineq.bdry.hol.rep} it holds $f_\infty=0$, but this is a contradiction with \eqref{contr.assm.bdry.hol.rep}.

Hence, we have shown \eqref{goal.bdry.hol.rep}. From the characterization of kinetic H\"older spaces via polynomials, we deduce the desired result.
\end{proof}

By combining \autoref{lem.bdry.hol} and \autoref{lem.bdry.hol.rep} with a scaling argument, we derive the following lemma, which will be used to obtain compactness of the blow-up sequence in the proof of \autoref{lem.exp.nd.a}. Recall that the functions ${\pmb{\Phi}}_k$ are determined in \eqref{defn.Phii}.
\begin{lemma}\label{lem.bdry.hol.final}
   Let $f$ be a weak solution to 
    \begin{equation*}
\left\{
\begin{alignedat}{3}
(\partial_t+v\cdot\nabla_x)f-\ddiv(A\nabla_vf)&=B\cdot\nabla_vf+F_0+\sum_{k=0}^{5}\sum_{j=0}^{2}F_{k,j}\partial_{v_n}^{(j)}{\pmb{\Phi}_k}&&\qquad \mbox{in  $\mathcal{H}_R$}, \\
f&=g&&\qquad  \mbox{in $\gamma_-\cap \mathcal{Q}_R$},
\end{alignedat} \right.
\end{equation*}
where $A\in C^{\alpha}(\mathcal{H}_R)$, $F_0\in C^{\lambda,\alpha}(\mathcal{H}_R)$, $F_{i,j}\in C^{\lambda_{i,j},\alpha}(\mathcal{H}_R)$, $g\in C^{\alpha}(\gamma_-\cap \mathcal{Q}_R)$ for some $\alpha\in(0,1)$ and $\lambda,\lambda_{k,j}\in \N\cup \{0\}$. Then there is a small constant $\beta=\beta(n,\Lambda,\alpha)$ such that 
\begin{align}\label{est.bdry.hol.final}
    [f]_{C^\beta(\mathcal{H}_{R/2})}&\leq c\left(\|f\|_{L^\infty(\mathcal{H}_R)}+[F_0]_{C^{\lambda,\alpha}(\mathcal{H}_R)}+\sum_{k=0}^{5}\sum_{j=0}^{2}[F_{k,j}]_{C^{\lambda_{k,j},\alpha}(\mathcal{H}_R)}+[g]_{C^{\alpha}(\gamma_-\cap\mathcal{Q}_R)}\right)
\end{align}
for some constant $c=c(n,\Lambda,\lambda,\lambda_{k,j},\alpha,\|A\|_{C^\alpha(\mathcal{H}_R)},\|B\|_{C^{\alpha}(\mathcal{H}_R)},R)$.
\end{lemma}

Note that it is possible to compute the precise dependence of the constant $c$ on $R$ in \eqref{est.bdry.hol.final}. However, for our purpose it will be sufficient to use the fact that $c$ depends only on $n,\Lambda,\|A\|_{C^\alpha(\mathcal{H}_R)},\|B\|_{C^{\alpha}(\mathcal{H}_R)},R$.

\subsection{Proof of expansion estimates}
Now we are going to prove \autoref{lem.exp.nd.a} when $(A(z_0))_{n,n}=1$.

\begin{proof}[Proof of \autoref{lem.exp.nd.a} with $(A(z_0))_{n,n}=1$.]
We may assume $R=1$ and 
    \begin{align}\label{ass.g.bdry}
        \|g\|_{L^\infty(\gamma_-\cap\mathcal{Q}_\rho(z_0))}+\rho^{\epsilon}[g]_{C^{\eps}(\gamma_-\cap\mathcal{Q}_\rho(z_0))}\leq \rho^{\lambda+\epsilon}[g]_{C^{\lambda,\epsilon}(\gamma_-\cap\mathcal{Q}_\rho(z_0))}\quad\text{for any }\rho\leq 1
    \end{align}
    for some constant $c=c(n,\lambda,\epsilon$), by considering $(f-P[g]_{z_0,\lambda})$ as a solution instead of $f$ itself, where $P[g]_{z_0,\lambda}$ is the Taylor expansion of the function $g$ of order $\lambda$ at $z_0$. {This creates an explicit $|v_0|$ dependence of the constant in the estimate when $\lambda\geq2$, since the source term now depends on $v_0$.}
    
    By contradiction, we suppose that there are sequences $(f_l)_l$,  $(z_{0,l})_l$, $(A_l)_l$, $(B_l)$, $(F_l)_l$, and $(g_l)_l$ with
     \begin{equation*}
\left\{
\begin{alignedat}{3}
(\partial_t+v\cdot\nabla_x)f_{l}-\ddiv(A_l\nabla_vf_l)&=B_l\cdot\nabla_vf_l+F_l&&\qquad \mbox{in  $\mathcal{H}_1(z_{0,l})$}, \\
f_l&=g_l&&\qquad  \mbox{in $\gamma_-\cap \mathcal{Q}_1(z_{0,l})$}
\end{alignedat} \right.
\end{equation*}
with $(A_{l}(z_{0,l}))_{n,n}=1$, as well as $(x_{0,l})_n=(v_{0,l})_n=0$ such that
\begin{equation}\label{ass.Fl.exp.nd}
\begin{aligned}
&    \|f_l\|_{L^\infty(\mathcal{H}_1(z_{0,l}))},\quad[F_l]_{C^{\lambda-2,\epsilon}(\mathcal{H}_1(z_{0,l}))}\leq 1,\quad \|A\|_{C^{\lambda-\frac12+\epsilon}(\mathcal{H}_1(z_{0,l}))}, \|B\|_{C^{\lambda-\frac32+\epsilon}(\mathcal{H}_1(z_{0,l}))}\leq\Lambda,\\
   & \|g_l\|_{L^\infty(\mathcal{H}_\rho(z_{0,l}))}+\rho^{\epsilon}[g_l]_{C^\epsilon(\mathcal{H}_{\rho}(z_{0,l}))}\leq \rho^{\lambda+\epsilon}\quad\text{for any }\rho\leq1,
\end{aligned}
\end{equation}
but
\begin{align*}
     \sup_{l\in\N}\inf_{P\in \mathcal{T}_{\lambda,\epsilon,l} }\sup_{\rho\in(0,1/2]}\rho^{-(\lambda+\epsilon)}\|f_l-P\|_{L^\infty(H_\rho(z_{0,l}))}=\infty,
\end{align*}
where we denote from now on $\mathcal{T}_{\lambda,\epsilon,l} \coloneq \mathcal{T}_{\lambda,\epsilon,z_{0,l}}$. 

Since $\mathcal{T}_{\lambda,\epsilon,l}$ is a finite dimensional vector space, for any $l\in\N$ and $\rho\in(0,1/2]$, there is a unique $P_{l,\rho}\in \mathcal{T}_{\lambda,\epsilon,l}$ such that 
\begin{align*}
    \|f_l-P_{l,\rho}\|_{L^2(H_{\rho}(z_{0,l}))}\leq \|f_l-P\|_{L^2(H_{\rho}(z_{0,l}))}\quad\text{for any }P\in \mathcal{T}_{\lambda,\epsilon,l},
\end{align*}
and
\begin{align}\label{ineq0.cont.exp.nd}
    \int_{\mathcal{H}_\rho(z_{0,l})}(f_{l}-P_{l,\rho})P\,dz=0\quad\text{for any }P\in \mathcal{T}_{\lambda,\epsilon,l}.
\end{align}
Let us define
\begin{align*}
    \theta(r)\coloneqq \sup_{l\in\N}\sup_{\rho\in[r,1/2]}\rho^{-(\lambda+\epsilon)}\|f_l-P_{l,\rho}\|_{L^\infty(\mathcal{H}_\rho(z_{0,l}))}.
\end{align*}
Then we have $\theta(r)\to\infty$ as $r\to0$ by \cite[Lemma 2.14]{RoWe25}, using the same argument as in the proof of \cite[Lemma 4.1]{RoWe25}. 
Since $\theta(r)$ is non-increasing, there are sequences $(l_m)_m$ and $(r_{l_m})_{m}$ such that $r_{l_m}\to 0$, as $m\to\infty $, and 
\begin{align}\label{linfty.comp.exp.nd}
    \theta(r_{l_m})\geq  r_{l_m}^{-(\lambda +\epsilon)}\|f_{l_m}-P_{l_m,r_{l_m}}\|_{L^\infty(\mathcal{H}_{r_{l_m}}(z_{0,l_m}))}\geq \theta(r_{l_m})/2.
\end{align}
We may write from now on $l_m=m$. We deduce from \cite[Lemma 2.15]{RoWe25} that
\begin{align}\label{lbdd.fl.exp.nd}
    \|f_m-P_{m,r_m}\|_{L^\infty(\mathcal{H}_{Rr_m}(z_{0,m}))}\leq c(Rr_m)^{\lambda+\eps}\theta(r_m)
\end{align}
and
\begin{align}\label{limit.coeffpoly.exp.nd}
    \frac{C_m}{\theta(r_m)}\coloneqq\frac{(\|p_m\|_{L^\infty(\mathcal{Q}_1(z_{0,m}))} + \|p_{m,k}\|_{L^\infty(\mathcal{Q}_1(z_{0,m}))})}{\theta(r_{m})}\to0
\end{align}
for each $k\in\{0,\ldots,5\}$. Now define 
\begin{align}\label{defn.gm}
    \widetilde{f}_m(z)\coloneqq \frac{f_m(z_{0,m}\circ S_{r_m}\,z)-P_{m,r_m}(z_{0,m}\circ S_{r_m}\,z)}{r_m^{\lambda+\epsilon}\theta(r_m)}
\end{align}
and
\begin{align}\label{widetildeab.exp.nd}
\widetilde{A}_m(z)\coloneqq A_m(z_{0,m}\circ S_{r_m}z),\quad \widetilde{B}_m(z)\coloneqq r_mB_m(z_{0,m}\circ S_{r_m}z).
\end{align}
Then, we recall from \eqref{defn.vector.ndimen.z_0.a} that $P_{m,r_m}$ is of the form 
\begin{align*}
    P_{m,r_m}(z)\coloneqq p_{m}+\sum_{k=0}^{5}p_{m,k}{\pmb{\Phi}}_k
\end{align*}
to see that $\widetilde{f}_m$ solves 
 \begin{equation}\label{eq.gm.exp.nd}
\left\{
\begin{alignedat}{3}
(\partial_t+v\cdot\nabla_x)\widetilde{f}_{m}-\ddiv(\widetilde{A}_m\nabla_v\widetilde{f}_m)&=\widetilde{B}_m \cdot\nabla_v\widetilde{f}_m+\widetilde{F}_m+\widetilde{Q}_m&&\qquad \mbox{in  $\mathcal{H}_{1/r_m}$}, \\
\widetilde{f}_m&=\widetilde{g}_m&&\qquad  \mbox{in $\gamma_-\cap \mathcal{Q}_{1/r_m}$},
\end{alignedat} \right.
\end{equation}
where 
\begin{align*}
    &\widetilde{F}_m(z)\coloneqq \frac{F_m(z_{0,m}\circ S_{r_m}\,z)-\widetilde{p}_m(z)}{r_m^{\lambda-2+\epsilon}\theta(r_m)}, ~~ \widetilde{g}_m(z)\coloneqq \frac{g_m(z_{0,m}\circ S_{r_m}\,z)}{r_m^{\lambda+\epsilon}\theta(r_m)},\\
    &\widetilde{p}_m(z)\coloneqq (\partial_t+v\cdot\nabla_x)(\overline{p}_{m}(z))+\ddiv(\widetilde{A}_m\nabla_v(\overline{p}_m(z)))-\widetilde{B}_m\cdot\nabla_v(\overline{p}_m(z)),\\
    & P_m(z)\coloneqq \sum_{k= 0}^5 \frac{(p_{m,k}\pmb{\Phi}_k)(z_{0,m}\circ S_{r_m}z)}{r_m^{\lambda+\epsilon}\theta(r_m)},\\
    &\widetilde{Q}_m(z)\coloneqq (\partial_t+v\cdot\nabla_x)(P_m)(z) -\ddiv(\widetilde{A}_m\nabla_v P_{m}(z))-\widetilde{B}_m \cdot\nabla_v P_m(z)
\end{align*}
and $\overline{p}_m(z)\coloneqq {p}_{m}(z_{0,m}\circ S_{r_m}\,z)$.
In order to conclude the proof, our goal is to establish compactness of the sequence $(\widetilde{f}_m)$ so that we can pass the equation for $\widetilde{f}_m$ to the limit $m \to \infty$. Then, we will apply the Liouville theorem from \autoref{lem.liou.inflow2.nd} to obtain a contradiction with \eqref{linfty.comp.exp.nd}. To do so, the main work is to carefully track the structural properties of the terms $\widetilde{Q}_m$. In fact, we can already observe that $p_m\in \cP_\lambda$, $p_{m,i}\in\cP_{[\lambda+\epsilon-\frac12]}$, $p_{m,j+3}\in\cP_{\lambda-2}$ , $p_{m,5}\in\cP_{[\lambda+\epsilon-\frac52]}$ for $i\in\{0,1,2\}$ and $j\in\{0,1\}$ with
\begin{equation}\label{ass.poly.exp.nd}
    p_{m,1}(z_0)=p_{m,2}(z_0)= Dp_{m,2}(z_0)=p_{m,4}(z_0)=0.
\end{equation}

Hence, after a few computations, using  \eqref{ass.poly.exp.nd}, we derive
\begin{align*}
    \widetilde{Q}_m(z)&= \sum_{i=0}^{2}\left({\pmb{\Phi}}_{i}(r_mX_n)\left[\frac{c_{m,i}}{r_m^{\lambda+\epsilon-2}\theta(r_m)}+\frac{\widetilde{B}_m\cdot\widetilde{P}_{m,i}}{r_m^{\lambda+\epsilon-1}\theta(r_m)}\right]-\frac{\ddiv(\widetilde{A}_m\widetilde{P}_{m,i}{\pmb{\Phi}}_{i}(r_mX_n))}{r_m^{\lambda+\epsilon-1}\theta(r_m)}\right)\\
    &\quad+\sum_{j=0}^{1}\left({\pmb{\Phi}}_{j+3}(r_mX_n)\frac{\widetilde{B}_m\cdot \widetilde{P}_{m,j+3}}{r_m^{\lambda+\epsilon-1}\theta(r_m)}-\frac{\ddiv(\widetilde{A}_m\widetilde{P}_{m,j+3}{\pmb{\Phi}}_{j+3}(r_mX_n))}{r_m^{\lambda+\epsilon-1}\theta(r_m)}\right)\\
    &\quad+\sum_{i=0}^{2}{r_mv_n(\partial_{x_n}{\pmb{\Phi}}_{i})(r_mX_n)}\frac{\overline{p}_{m,i}}{r_m^{\lambda+\eps-2}\theta(r_m)}+\sum_{j=0}^{1}{r_mv_n(\partial_{x_n}{\pmb{\Phi}}_{j+3})(r_mX_n)}\frac{\overline{p}_{m,i}}{r_m^{\lambda+\eps-2}\theta(r_m)}\\
    &\quad-\sum_{i=0}^{2}\frac{\ddiv(\widetilde{A}_m\overline{p}_{m,i}(\nabla_v{\pmb{\Phi}}_{i})(r_mX_n) )}{r_m^{\lambda+\eps-1}\theta(r_m)}-\sum_{j=0}^{1}\frac{\ddiv(\widetilde{A}_m\overline{p}_{m,j}(\nabla_v{\pmb{\Phi}}_{j+3})(r_mX_n) )}{r_m^{\lambda+\eps-1}\theta(r_m)}\\
    &\quad-\sum_{i=0}^{2}{(\partial_{v_n}{\pmb{\Phi}}_{i})(r_mX_n)}\frac{(\widetilde{B}_m)_n\overline{p}_{m,i}}{r_m^{\lambda+\epsilon-1}\theta(r_m)}-\sum_{j=0}^{1}{(\partial_{v_n}{\pmb{\Phi}}_{j+3})(r_mX_n)}\frac{(\widetilde{B}_m)_n\overline{p}_{m,j}}{r_m^{\lambda+\epsilon-1}\theta(r_m)}\\
    &\quad+{\overline{p}_{m,5}}\left(\frac{r_mv_n(\partial_{x_n}{\pmb{\Phi}}_{5})(r_mX_n)}{r_m^{\lambda+\epsilon-2}\theta(r_m)}-\frac{\ddiv(\widetilde{A}_m(\nabla_v{\pmb{\Phi}}_{5})(r_mX_n))}{r_m^{\lambda+\epsilon-1}\theta(r_m)}\right)\eqqcolon \sum_{N=1}^{6}\widetilde{Q}_{m,N}(z),
\end{align*}
where we write $X_n = (x_n,v_n)$ and $r_m X_n = (r_m^3 x_n , r_m v_n)$, and
\begin{align*}
    &c_{m,i}\coloneqq \partial_tp_{m,i}, \quad \widetilde{P}_{m,i}(z)\coloneqq (\nabla_vp_{m,i})(z_{0,m}\circ S_{r_m}z), \quad \overline{p}_{m,i}(z)\coloneqq p_{m,i}(z_{0,m}\circ S_{r_m}z).
\end{align*}
Before estimating the terms $\widetilde{Q}_{m,N}$, by the homogeneity of the functions $\pmb{\Phi}_k$, we note 
\begin{align*}
    &{\pmb{\Phi}}_{k}(r_mX_n)=r_m^{\alpha_k}{\pmb{\Phi}}_{k}(X_n),\, (\partial_{v_n}{\pmb{\Phi}}_{k})(r_mX_n)=r_m^{-1+\alpha_k}(\partial_{v_n}{\pmb{\Phi}}_{k})(X_n),\\
    &r_mv_n\partial_{x_n}{\pmb{\Phi}}_{k}(r_mX_n)=r_m^{-2+\alpha_k}v_n\partial_{x_n}{\pmb{\Phi}}_{k}(X_n)
\end{align*}
with $\alpha_i=\frac12$, $\alpha_{j+3}=2$ for $i\in\{0,1,2\},j\in\{0,1\}$ and $\alpha_5=\frac52$.
\begin{itemize}
    \item The case $N=1,2$. Using \eqref{ass.Fl.exp.nd}, \eqref{widetildeab.exp.nd}, and \eqref{limit.coeffpoly.exp.nd}, we get
    \begin{align*}
        \widetilde{Q}_{m,1}+\widetilde{Q}_{m,2}=\sum_{i=0}^{4}\left({\pmb{\Phi}}_{i}F_{m,0,i}-\partial_{v_n}{\pmb{\Phi}}_{i}F_{m,1,i}\right),
    \end{align*}
    where for any $R>0$,
    \begin{align*}
        [F_{m,0,i}]_{C^{\lambda+\epsilon-\frac52}(\mathcal{H}_R)} + [F_{m,1,i}]_{C^{\lambda+\epsilon-\frac32}(\mathcal{H}_R)} + \|F_{m,0,j+3}\|_{L^{\infty}(\mathcal{H}_R)} + [F_{m,1,j+3}]_{C^{\lambda+\epsilon-3}(\mathcal{H}_R)}\leq \frac{c(R)C_m}{\theta(r_m)},
    \end{align*}
    which goes to 0, for $i\in\{0,1,2\}$ and $j\in\{0,1\}$.
    In particular, by \eqref{ass.poly.exp.nd}, we get
    \begin{align}\label{zero0.exp.nd}
        F_{m,1,2}(0)=0.
    \end{align}
    \item The case $N=3,4,5$.
    Using that the functions $\pmb{\Phi}_k$ solve a translation-invariant equation (see \eqref{basic.formula.liou.inflow4}) and the following observation
\begin{align}\label{minusdiv.exp.nd}
    \ddiv(\widetilde{A}_m\overline{p}_{m,i}(\nabla_v{\pmb{\Phi}}_{i})(r_mX_n) )=\ddiv((\widetilde{A}_m-I)\overline{p}_{m,i}(\nabla_v{\pmb{\Phi}}_{i})(r_mX_n))+\ddiv(\overline{p}_{m,i}(\nabla_v{\pmb{\Phi}}_{i})(r_mX_n)),
\end{align}
we derive when $\lambda+\eps<\frac32$,
\begin{align*}
    \widetilde{Q}_{m,3}+\widetilde{Q}_{m,4}+\widetilde{Q}_{m,5}=-\ddiv(G_{m,0,0}\partial_{v_n}{\pmb{\Phi}}_0)+F_{m,1,0}\partial_{v_n}{\pmb{\Phi}}_0
\end{align*}
with 
\begin{align*}
    \|G_{m,0,0}\|_{L^\infty(\mathcal{H}_R)}+\|F_{m,1,0}\|_{L^\infty(\mathcal{H}_R)}\leq \frac{c(R)C_m}{\theta(r_m)} \to 0,
\end{align*}
where we used that $\overline{p}_{m,0}$ is constant and $\overline{p}_{m,k} = 0$ for any $k \ge 1$ in that case. 

Moreover, when $\lambda+\eps>\frac32$, after several elementary computations we can write
\begin{align*}
    \widetilde{Q}_{m,3}+\widetilde{Q}_{m,4}+\widetilde{Q}_{m,5}=\sum_{k=0}^{4}F_{m,2,k}\partial_{v_n,v_n}{\pmb{\Phi}}_{k}+\sum_{k=0}^{4}F_{m,1,k}\partial_{v_n}{\pmb{\Phi}}_{k},
\end{align*}
where 
\begin{align}
\label{zero1.exp.nd.2}
    [F_{m,l,i}]_{C^{\lambda+\epsilon-(2-l)-\frac12}(\mathcal{H}_R)} + [F_{m,l,j+3}]_{C^{\lambda+\epsilon-(2-l)-2}(\mathcal{H}_R)}\leq \frac{c(R)C_m}{C_{v_{0,m}}\theta(r_m)} \to 0
\end{align}
for $l\in\{1,2\}$, $i\in\{0,1,2\}$, and $j\in\{0,1\}$. Moreover, in light of \eqref{ass.poly.exp.nd} together with the fact that $D(fg)(0)=0$ whenever $f(0)=g(0)=0$, we observe that for any $k\in\{0,1,2,3,4\}$,
    \begin{equation}\label{zero1.exp.nd}
    \begin{aligned}
        &F_{m,1,2}(0)=F_{m,2,k}(0)=DF_{m,2,1}(0)=DF_{m,2,2}(0)=DF_{m,2,4}(0)=D^2F_{m,2,2}(0)=0.
    \end{aligned}
    \end{equation}
\item The case $N=6$. Similarly, by \eqref{minusdiv.exp.nd} and the fact that  $\partial_{v_n,v_n}\phi_1=v_n\partial_{x_n}\phi_1=c_0v_n\phi_0$ for some constant $c_0$ by \eqref{lemma:derx.phim}, we get 
\begin{align*}
    \widetilde{Q}_{m,6}=F_{m,0,5}{\pmb{\Phi}_{0}}(X_n)-F_{m,1,5}\partial_{v_n}{\pmb{\Phi}}_{1},
\end{align*}
where 
\begin{align}
\label{zero1.exp.nd.3}
    [F_{m,0,5}]_{C^{\lambda+\epsilon-\frac52}(\mathcal{H}_R)} + \|F_{m,1,5}\|_{L^{\infty}(\mathcal{H}_R)}\leq \frac{c(R)c_m}{\theta(r_m)} \to 0.
\end{align}
\end{itemize}
Therefore, combining all the estimates for the terms $\widetilde{Q}_{m,N}$ together with the fact that 
\begin{align}
\label{conv.gm.exp.nd}
    [\widetilde{F}_m]_{C^{\lambda-2+\epsilon}(\mathcal{H}_R)} + \|\widetilde{g}_m\|_{L^\infty(\mathcal{H}_R)}\leq \frac{c(R)}{\theta(r_m)} \to 0,
\end{align} 
which can be seen from the same arguments as in \cite[Lemma 4.1]{RoWe25} by using \eqref{ass.Fl.exp.nd}, we can apply \autoref{lem.bdry.hol} and \autoref{lem.bdry.hol.final} when $\lambda+\epsilon<\frac32$ and $\lambda+\epsilon>\frac32$, respectively, to get that
\begin{align*}
    \|\widetilde{f}_m\|_{C_{\ell}^{\beta}(\mathcal{H}_R)}\leq c(R).
\end{align*}
In addition, provided that $z_0$ is such that $\mathcal{Q}_1(z_0)\Subset \mathcal{H}_{1/r_m}$, by the interior regularity given in \autoref{lem.sch.divergence} together with a rescaled version of \autoref{lem.bdry.hol.rep}, we deduce
\begin{align*}
    \|\widetilde{f}_m\|_{C_{\ell}^{1,\epsilon}(\mathcal{Q}_1(z_0))}\leq c(\lambda,\epsilon,z_0),
\end{align*}
where the constant $c$ is independent of $m$. Therefore, we get 
\begin{align}\label{conv.grad.exp.nd}
    \widetilde{f}_m\to\widetilde{f}_\infty\quad\text{in }C_{\ell}^{1,\epsilon}(\mathcal{Q}_1(z_0))
\end{align}
for any $\mathcal{Q}_1(z_0)\Subset \{x_n>0\}$.
Moreover, by \cite[Lemma 5.9]{Sil22}, we get 
\begin{align*}
    \int_{\mathcal{H}_{R/2}}|\nabla_v(\widetilde{f}_m-\max_{\mathcal{H}_{R/2}}\widetilde{g}_m)_+|^2+\int_{\mathcal{H}_{R/2}}|\nabla_v(\widetilde{f}_m-\min_{\mathcal{H}_{R/2}}\widetilde{g}_m)_-|^2\leq c(R,\lambda,\epsilon).
\end{align*}
Since $\nabla_v\widetilde{f}_m\to\nabla_v\widetilde{f}_\infty$ pointwise (see \eqref{conv.grad.exp.nd}) and $\min_{\mathcal{H}_{R/2}}\widetilde{g}_m \to 0$, $\max_{\mathcal{H}_{R/2}}\widetilde{g}_m \to 0$ (see \eqref{conv.gm.exp.nd}), by Fatou's lemma, we have 
\begin{align}\label{bdd.grad.exp.nd}
    \int_{\mathcal{H}_R}|\nabla_v\widetilde{f}_\infty|^2=\int_{\mathcal{H}_R}|\nabla_v(\widetilde{f}_\infty)_+|^2+|\nabla_v(\widetilde{f}_\infty)_-|^2\leq c(R,\lambda,\epsilon).
\end{align}

We are now ready to apply \autoref{lem.limit.app} in order to derive that $\widetilde{f}_\infty$ is a weak solution to 
\begin{equation}\label{eq2.gm.exp.nd}
\left\{
\begin{alignedat}{3}
(\partial_t+v\cdot\nabla_x)\widetilde{f}_{\infty}-\ddiv(\widetilde{A}_\infty\nabla_v\widetilde{f}_\infty)&=\widetilde{p}_\infty+\widetilde{Q}_\infty&&\qquad \mbox{in  $\bbR\times\{x_n>0\}\times \bbR$}, \\
\widetilde{f}_\infty&=0&&\qquad  \mbox{in $\{x_n=0\}\times \{v_n>0\}$},\\
|\widetilde{f}_\infty(z)|&\leq c(1+|z|)^{\lambda+\epsilon}&&,
\end{alignedat} \right.
\end{equation}
where $\widetilde{p}_\infty\in \cP_{\lambda-2}$ and by \eqref{zero0.exp.nd} and \eqref{zero1.exp.nd}, we also know
\begin{align*}
    \widetilde{Q}_\infty&\coloneqq \sum_{i=0}^{2}{p}_{\infty,0,i}{\pmb{\Phi}}_{i}+\sum_{k=0}^{4}p_{\infty,1,k}\partial_{v_n}{\pmb{\Phi}}_{k}+\sum_{k=0,1,3}{p}_{\infty,2,k}\partial_{v_n,v_n}{\pmb{\Phi}}_{k},
\end{align*}
with 
\begin{equation}\label{zero.liou.inflow2}
    \begin{aligned}
        &{p}_{\infty,0,i}\in\cP_{[\lambda+\epsilon-\frac52]},\, p_{\infty,1,i}\in \cP_{[\lambda+\epsilon-\frac32]},\, p_{1,j+3}\in\cP_{\lambda-3},\, p_{\infty,2,i}\in \cP_{[\lambda+\epsilon-\frac12]}, \,p_{2,3}\in \cP_{\lambda-2},\\
        &p_{\infty,1,2}(0)={p}_{\infty,2,0}(0)={p}_{\infty,2,1}(0)={p}_{\infty,2,3}(0)=Dp_{\infty,2,1}(0)=0
    \end{aligned}
\end{equation}
for $i\in\{0,1,2\}$ and $j\in\{0,1\}$. Here, we have also used \autoref{lem.bdry.hol.rep} to prove that for all functions $F_{m,l,k}$ it holds $F_{m,l,k} \to p_{\infty,l,k}$ locally uniformly. Moreover, we used that $F_{m,2,2}, F_{m,2,4} \to 0$ by \eqref{zero1.exp.nd.2} and \eqref{zero1.exp.nd}, and that $F_{m,1,5} \to 0$ by \eqref{zero1.exp.nd.3}.

By \autoref{lem.liou.inflow2.nd}, we get $\widetilde{f}_\infty=P_\infty$ for some $P_\infty\in\mathcal{T}_{\lambda,\epsilon}$. However we observe from \eqref{ineq0.cont.exp.nd} after a change of variables that
\begin{align*}
\lim_{m\to\infty} \int_{\mathcal{H}_1}|\widetilde{f}_m |^2= \lim_{m\to\infty} \int_{\mathcal{H}_1}\widetilde{f}_mP_\infty = 0.
\end{align*}
This contradicts the fact that $\|\widetilde{f}_m\|_{L^\infty(\mathcal{H}_1)}\geq 1/2$, which is an immediate consequence of \eqref{linfty.comp.exp.nd} and the definition of $\widetilde{f}_m$. This completes the proof.
\end{proof}

\section{Optimal boundary regularity with in-flow condition}
\label{sec:inf-low}
In this section, using \autoref{lem.exp.nd.a}, we derive several fine boundary regularity results for kinetic equations with in-flow condition.

First, we provide the boundary regularity result in $\gamma_-$ and $\gamma_+$, away from $\gamma_0$, which was essentially already proved in \cite{RoWe25}.

\begin{lemma}\label{lem.reg.gamma+}
    Let $\lambda\in\N\cup\{0\}$ and $\epsilon\in(0,1)$, $z_0\in\gamma_+\cup\gamma_-$ and $\mathcal{Q}_{2R}(z_0)\cap \gamma_0=\emptyset$ with $R\leq1$. Let $f$ be a weak solution to 
    \begin{equation}
\left\{
\begin{alignedat}{3}
(\partial_t+v\cdot\nabla_x)f-\ddiv(A\nabla_vf)&=B\cdot\nabla_vf+F&&\qquad \mbox{in  $\mathcal{H}_R(z_0)$}, \\
f&=g&&\qquad  \mbox{in $\gamma_-\cap\mathcal{Q}_R(z_0)$}.
\end{alignedat} \right.
\end{equation}
Let $A\in C^{\lambda-1,\epsilon}(\mathcal{H}_R(z_0))$, $B\in C^{\lambda-2,\epsilon}(\mathcal{H}_R(z_0);\bbR^n)$, $F\in C^{\lambda-2,\epsilon}(\mathcal{H}_R(z_0))$. Then there is a polynomial $p_{z_0}\in\cP_{[\lambda+\epsilon]}$ such that for any $r \le R/2$
\begin{align*}
    \Vert f-p_{z_0} \Vert_{L^{\infty}(\mathcal{H}_r(z_0))}&\leq c\left(\frac{r}{R}\right)^{\lambda+\epsilon}\Bigg(\|f\|_{L^\infty(\mathcal{H}_R(z_0))}+R^{\lambda + \epsilon}([F]_{C^{\lambda-2,\epsilon}(\mathcal{H}_R(z_0))}+[g]_{C^{\lambda,\eps}(\gamma_-\cap\mathcal{Q}_R(z_0))})\Bigg),
\end{align*}
where $c=c(n,\lambda,\epsilon,\Lambda,\|A\|_{C^{\lambda-1,\epsilon}(\mathcal{H}_R(z_0))},\|B\|_{C^{\lambda - 2,\epsilon}(\mathcal{H}_R(z_0))})$.
\end{lemma}
\begin{proof}
    First note that the regularity at $\gamma_+$ and $\gamma_-$ of kinetic equations in  non-divergence form was established in \cite[Lemma 4.1 and Lemma 4.3]{RoWe25} by using the blow-up argument as in the proof of \autoref{lem.exp.nd.a}. Therefore, just a few simple modifications of the proof are required in order to deal with divergence operators (see \autoref{lem.exp.nd.a}).
\end{proof}

\subsection{Optimal regularity up to the grazing set}

Now we prove the optimal $C^{1/2}$ global regularity for solutions with in-flow boundary condition.
\begin{theorem}\label{thm.hol12}
    Let $f$ be a weak solution to 
    \begin{equation}\label{eq.hol12}
\left\{
\begin{alignedat}{3}
(\partial_t+v\cdot\nabla_x)f-\ddiv(A\nabla_vf)&=B\cdot\nabla_vf+F&&\qquad \mbox{in  $\mathcal{H}_R(z_0)$}, \\
f&=g&&\qquad  \mbox{in $\gamma_-\cap \mathcal{Q}_R(z_0)$}.
\end{alignedat} \right.
\end{equation}
If $A\in C^{\epsilon}(\mathcal{H}_{R}(z_0))$, $F\in L^{\infty}(\mathcal{H}_R(z_0))$, and $g\in C^{\frac12+\epsilon}(\gamma_-\cap\mathcal{Q}_R(z_0))$ for some $\epsilon>0$, then we have 
\begin{align*}
    R^{\frac12}[f]_{C^{\frac12}(\mathcal{H}_{R/2}(z_0))}&\leq c\mathbf{B}(z_0,R)^{\frac12}\Bigg(\|f\|_{L^\infty({\mathcal{H}}_R(z_0))}+R^2\|F\|_{L^{\infty}({\mathcal{H}}_R(z_0))}+{R^{\frac12+\epsilon}[g]_{C^{\frac12+\epsilon}(\gamma_-\cap {\mathcal{Q}}_R(z_0))}}\Bigg),
\end{align*}
where $c=c(n,\epsilon, \Lambda, \|A\|_{C^\epsilon(\mathcal{H}_R(z_0)})$ and we write $\mathbf{B}(z_0,R)\coloneqq 1+\|B\|_{L^\infty(\mathcal{H}_R(z_0))}$.
\end{theorem}
\begin{proof}
By the same scaling argument as in \eqref{scaling.argument}, we may assume $\|B\|_{L^\infty(\mathcal{H}_R(z_0))}\leq 1$.  Now we assume
    \begin{align*}
        \|f\|_{L^\infty({\mathcal{H}}_R(z_0))}+R^2\|F\|_{L^{\infty}({\mathcal{H}}_R(z_0))}+R^{\frac12+\eps}[g]_{C^{\frac12+\epsilon}(\gamma_-\cap {\mathcal{Q}}_R(z_0))}\leq 1.
    \end{align*}
    Suppose $\mathcal{Q}_{3R/4}(z_0)\cap \gamma_0\neq\emptyset$, otherwise the results follow from a combination of interior estimates and \autoref{lem.reg.gamma+}.
First, we want to prove that for any $z_1\in \mathcal{Q}_{\frac{R}4}(z_0)$ and $r \le \frac{R}{32}$, there is a constant $c_{z_1}$ such that
    \begin{align}\label{first.goal.hol12}
        J\coloneqq\|f-c_{z_1}\|_{L^\infty(\mathcal{H}_r(z_1))}\leq cr^{\frac12}
    \end{align}
    for some constant $c=c(n,\epsilon,\Lambda,\|A\|_{C^\epsilon(\mathcal{H}_R(z_0))},c_B)$.
    
     To do so, we divide the proof into three cases
     \begin{itemize}
         \item Let $\left(\frac{|v_{1,n}|}{4}\right)^3\leq x_{1,n}$. When $x_{1,n}=0$, then using \autoref{lem.exp.nd.a}, we observe that there are constants $c_0 ,c_1$ such that
\begin{align*}
    \|f-c_0-c_1{\pmb{\Phi}}_{0,a}\|_{L^\infty(\mathcal{H}_r({z}_1))}\leq cr^{\frac12}
\end{align*}
for any $r\leq\frac12$, where $a\coloneqq \sqrt{(A(z_1))_{n,n}}$ and $c=c(n,\eps,\Lambda,\|A\|_{C^\eps(\mathcal{H}_R(z_0))},\|B\|_{L^\infty(\mathcal{H}_R(z_0))})$. Thus, by taking $c_{z_1}\coloneqq c_0$, we derive
    \begin{equation}\label{ineq11.hol12}
    \begin{aligned}
       \|f-c_0\|_{L^\infty(\mathcal{H}_r(z_1))}&\leq \|f-c_0-c_1{\pmb{\Phi}}_{0,a}\|_{L^\infty(\mathcal{H}_{r}(z_1))}+c_1\|{\pmb{\Phi}}_{0,a}-{\pmb{\Phi}}_{0,a}(z_1)\|_{L^\infty(\mathcal{H}_r(z_1))}\leq cr^{\frac12}
    \end{aligned}
    \end{equation}
    for some constant $c=c(n,\epsilon,\|A\|_{C^\epsilon(\mathcal{H}_1(z_0))},c_B)$, where we have used the fact that ${\pmb{\Phi}}_{0,a}(z_1)=0$, $\|{\pmb{\Phi}}_{0,a}\|_{C^{\frac12}(\mathcal{H}_r(z_1))}=\|{\pmb{\Phi}}_{0,a}\|_{C^{\frac12}({H}_r)}\leq cr^{\frac12}$.

Now we assume $x_{1,n}>0$. Then there are constants $r_1\coloneqq \frac{|x_{1,n}|^{\frac13}}{8}$, $r_2=8r_1$ and $\widetilde{z}_1=(t_1,x_1',0,v_1',0)\in\gamma_0\cap \mathcal{Q}_{\frac{3R}4}(z_0)$ such that 
         \begin{align*}
             \mathcal{Q}_{2r_1}({z}_1)\cap\gamma=\emptyset,\quad \mathcal{Q}_{2r_1}(z_1)\subset \mathcal{H}_{r_2}(\widetilde{z_1})\quad\text{and}\quad\mathcal{H}_{2r_2}(\widetilde{z_1})\subset\mathcal{H}_{R}(z_0).
        \end{align*}
    We further assume $r\leq r_1$ to see that for any $c_0\in\bbR$,
    \begin{align*}
        \|f-c_0\|_{L^\infty(\mathcal{H}_{r}(z_1))}\leq c\left(\frac{r}{r_1}\right)^{\frac12}\|f-c_0\|_{L^\infty(\mathcal{Q}_{2r_1}(z_1))}+cr^{\frac12},
    \end{align*}
        where we have used interior regularity results given in \cite[Lemma 5.7]{KLN25} with $G=0$. We now choose $c_0\coloneqq c_{\widetilde{z}_1}$ determined in \eqref{ineq11.hol12} to derive that 
        \begin{align}\label{ineq.int2.hol12}
            \|f-c_0\|_{L^\infty(\mathcal{H}_{r}(z_1))}\leq c\left(\frac{r}{r_1}\right)^{\frac12}r_1^{\frac12}\leq cr^{\frac12}.
        \end{align}
   On the other hand, if $r\geq r_1$, then we have $\mathcal{H}_{r}(z_1)\subset \mathcal{H}_{8r}(\widetilde{z}_1)$ and 
    \begin{align}\label{ineq1.rbig.hol12}
        \|f-c_0\|_{L^\infty(\mathcal{H}_{r}(z_1))}\leq \|f-c_0\|_{L^\infty(\mathcal{H}_{8r}(\widetilde{z}_1))}\leq cr^{\frac12},
    \end{align}
where $c_0\coloneqq c_{\widetilde{z}_1,4r}$ is determined in \eqref{ineq11.hol12}. Therefore, we have proved \eqref{first.goal.hol12} when $\left(\frac{|v_{1,n}|}{4}\right)^3\leq x_{1,n}$.
\item Let $\left(\frac{v_{1,n}}{4}\right)^3\geq x_{1,n}$. When $v_{1,n}=0$, \eqref{first.goal.hol12} follows from the previous case, so that we may assume $v_{1,n}>0$. If $r\geq \frac{v_{1,n}}4$, then 
\begin{align*}
    \mathcal{H}_r(z_1)\subset \mathcal{H}_{8r}(\widetilde{z}_1),
\end{align*}
where $\widetilde{z}_1\coloneqq (t_1,x_1',0,v_1',0)$. Thus, we get the desired estimate as in \eqref{ineq1.rbig.hol12}. Now suppose $r\leq \frac{v_{1,n}}4$. When $r\leq \frac12\left(\frac{x_{1,n}}{v_{1,n}}\right)^{\frac12}\eqqcolon r_1$, then we observe
\begin{align*}
    \mathcal{H}_{r_1}(z_1)\cap\gamma =\emptyset,\quad \mathcal{H}_{2r_1}(z_1)\subset \mathcal{H}_{8r_1}(\widehat{z}_1)\subset \mathcal{H}_{\frac{v_{1,n}}4}(\widehat{z}_1),\quad \mathcal{H}_{\frac{v_{1,n}}4}(\widehat{z}_1)\subset \mathcal{H}_{v_{1,n}}(\widetilde{z}_1),
\end{align*}
where we write 
\begin{align*}
    \widehat{z}_1\coloneqq \left(t_1-\frac{x_{1,n}}{v_{1,n}},x_1'-\frac{x_{1,n}}{v_{1,n}}v_1',0,v_1\right)\quad\text{and}\quad \widetilde{z}_1\coloneqq \left(t_1-\frac{x_{1,n}}{v_{1,n}},x_1'-\frac{x_{1,n}}{v_{1,n}}v_1',0,v_1',0\right).
\end{align*}
Note that $\widehat{z}_1=z_1\circ (-\frac{x_{1,n}}{v_{1,n}},0,0)$ is a natural projection of $z_1$ to $\gamma_-$ in the kinetic geometry, and similarly $\widetilde{z}_1=\widehat{z}_1\circ (0,0,-v_{1,n})$ is the projection of $\widehat{z}_1$ to $\gamma_0$.

By \autoref{lem.reg.gamma+}, there is a constant $c_1$ such that 
\begin{align}\label{ineq2.v11.12}
    \|f-c_0-c_1\|_{L^\infty(\mathcal{H}_{8r_1}(\widehat{z}_1))}\leq c\left(\left(\frac{r_1}{v_{1,n}}\right)^{\frac12}\|f-c_0\|_{L^\infty(\mathcal{H}_{\frac{v_{1,n}}4}(\widehat{z}_1))}+r_1^{\frac12}\right),
\end{align}
where the constant $c_0=c_{\widetilde{z}_1}$ is determined as in \eqref{ineq.int2.hol12}. Therefore, we have 
\begin{align}\label{ineq3.v11.12}
    \|f-c_0-c_1\|_{L^\infty(\mathcal{H}_{8r_1}(\widehat{z}_1))}\leq cr_1^{\frac12}.
\end{align}
On the other hand, by \cite{KLN25}, there is a constant $c_2$ such that 
\begin{align}\label{ineq4.v11.12}
    \|f-c_0-c_1-c_2\|_{L^\infty(\mathcal{H}_{r}(z_1))}\leq c\left(\left(\frac{r}{r_{1}}\right)^{\frac12}\|f-c_0-c_1\|_{L^\infty(\mathcal{H}_{r_1}({z}_1))}+r^{\frac12}\right).
\end{align}
Thus, combining two estimates \eqref{ineq2.v11.12} and \eqref{ineq4.v11.12} leads to the desired estimate if $r\leq  r_1$. On the other hand,  when $r> r_1$, then $\mathcal{H}_{r}(z_1)\subset\mathcal{H}_{4r}(\widehat{z}_1)$. Therefore, by \eqref{ineq2.v11.12} and \eqref{ineq3.v11.12} with $r_1$ replaced by $r$, we get 
\begin{align*}
    \|f-c_0-c_1\|_{L^\infty(\mathcal{H}_{r}({z}_1))}\leq \|f-c_0-c_1\|_{L^\infty(\mathcal{H}_{4r}(\widehat{z}_1))}\leq cr^{\frac12},
\end{align*}
which completes the proof of \eqref{first.goal.hol12} when $\left(\frac{v_{1,n}}{4}\right)^3\geq x_{1,n}$.
    \item Let $\left(\frac{-v_{1,n}}{4}\right)^3\geq x_{1,n}$. Similarly, we may assume $v_{1,n}<0$ and $r\leq v_{1,n}/4$. Now using \autoref{lem.reg.gamma+} and interior regularity results by \cite{KLN25}, we derive \eqref{first.goal.hol12} when $\left(\frac{-v_{1,n}}{4}\right)^3\geq x_{1,n}$, as in \eqref{ineq2.v11.12}-\eqref{ineq4.v11.12}.
     \end{itemize}
Therefore, we have proved \eqref{first.goal.hol12}, which completes the proof.

\end{proof}

\subsection{Higher order asymptotics in $\mathcal{R}^0\cup\mathcal{R}^+$}

Next we prove the Lipschitz estimate of the function $f/{\pmb{\Phi}}$, where ${\pmb{\Phi}}\eqsim \phi_0$.
To this end, for any $z_0\in\gamma_0$ with $(x_0)_n=(v_0)_n=0$ and $M\geq1$, let us define
\begin{equation}\label{defn.h0m}
\begin{aligned}
    &\mathcal{H}^{0}_{\rho,M}(z_0)\coloneqq \left\{z=(t,x',x_n,v',v_n)\in \mathcal{H}_\rho(z_0)\,:\, \frac{v_n^{3}}{x_n}\leq M^3\right\},\\ &\mathcal{H}^{+}_{\rho}(z_0)\coloneqq \left\{z=(t,x',x_n,v',v_n)\in \mathcal{H}_\rho(z_0)\,:\, v_n<0\right\}.
\end{aligned}
\end{equation}

\begin{theorem}\label{thm.h/phi}
   Let $M\geq1$ and let $f$ be a weak solution to 
    \begin{equation}\label{eq.h/phi}
\left\{
\begin{alignedat}{3}
(\partial_t+v\cdot\nabla_x)f-\ddiv(A\nabla_vf)&=B\cdot\nabla_vf+F&&\qquad \mbox{in  $\mathcal{H}_R(z_0)$}, \\
f&=0&&\qquad  \mbox{in $\gamma_-\cup \mathcal{Q}_R(z_0)$},
\end{alignedat} \right.
\end{equation}
where $(x_0)_n=(v_0)_n=0$ and $R\leq1$.
If $A\in C^{1,\epsilon}(\mathcal{H}_{R}(z_0))$, $F\in L^{\infty}(\mathcal{H}_R(z_0))$, $B\in C^{\epsilon}(\mathcal{H}_R(z_0))$ for some $\epsilon\in(0,1)$, then for any $z,{z'}\in \mathcal{H}^{0}_{\frac{R}2,M}(z_0)\cap\mathcal{H}_{\frac{R}2}^+(z_0)$, 
\begin{align*}
   \left|\frac{f}{{\pmb{\Phi}}_{0,a}}(z)-\frac{f}{{\pmb{\Phi}}_{0,a'}}(z')\right| &\leq c\frac{|z-z'|}{R}\Bigg(\|R^{-\frac12}f\|_{L^\infty({\mathcal{H}}_R(z_0))}+R^{\frac32}\|F\|_{L^{\infty}({\mathcal{H}}_R(z_0))}\Bigg)
\end{align*}
for some constant $c=c(n,\epsilon,\Lambda,\|A\|_{C^{1,\epsilon}(\mathcal{H}_R(z_0))},\|B\|_{C^{\epsilon}(\mathcal{H}_R(z_0))},M)$, where $a\coloneqq \sqrt{(A(z))_{n,n}}$ and $a'\coloneqq \sqrt{(A(z'))_{n,n}}$.
\end{theorem}

\begin{proof}
We may assume $R=1$ and 
\begin{align*}
    \|f\|_{L^\infty({\mathcal{H}}_1(z_0))}+\|F\|_{L^{\infty}({\mathcal{H}}_1(z_0))} \leq 1.
\end{align*}
In addition, to simplify the notation, we omit the dependence of the constant when it depends only on $n,\epsilon,\Lambda,\|A\|_{C^{1,\epsilon}(\mathcal{H}_1(z_0))},\|B\|_{C^{\epsilon}(\mathcal{H}_1(z_0))},M$. Furthermore, we write
\begin{align}\label{defn.overline.h}
    \overline{\mathcal{H}}_{\frac12}(z_0)\coloneqq \mathcal{H}^{0}_{\frac{1}{2},M}(z_0)\cap \mathcal{H}_{\frac12}^+(z_0).
\end{align}

First, we show that for each $z_1\in \overline{\mathcal{H}}_{\frac12}(z_0)$ and $\rho\leq \frac{1}{16}$, there is a constant $c_{z_1}$ such that
\begin{align}\label{goal0.h/phi}
    J \coloneqq \dashint_{\mathcal{H}_\rho(z_1)\cap\overline{\mathcal{H}}_{\frac12}(z_0)}\left|\frac{f}{{\pmb{\Phi}}_{0,a_1}}(z)-c_{z_1}\right|\,dz\leq c\rho,
\end{align}
where $a_1\coloneqq\sqrt{(A(z_1)_{n,n})}$.
Note that for any $z_1\in \gamma_0\cap\mathcal{Q}_{\frac58}(z_0)$, by \autoref{lem.exp.nd.a}, there are constants $c_{0},c_1,c_2$ such that for any $\rho<\frac1{32}$,
\begin{align}\label{ineq0.h/phi}
    \sup_{\mathcal{H}_{\rho}(z_1)}|f-(c_{0}+\mathbf{c}_{1}\cdot (v-v_{1})){\pmb{\Phi}}_{0,a_1}-\mathbf{c}_{2}\cdot(v-v_{1}){\pmb{\Phi}}_{1,a_1}|\leq c\rho^{\frac32+\epsilon}
\end{align}
for some $i,j\in\{1,2,\ldots,n\}$, where $\mathbf{c}_1,\mathbf{c}_2\in \bbR^n$ and we have also used that $\mathcal{H}_{\frac1{32}}(z_1)\subset\mathcal{H}_{1}(z_0)$.

To this end, we split the proof into three cases.
\begin{itemize}
    \item Let $z_1\in\gamma_0$, i.e., $z_1=(t_1,x_1',0,v_1',0)$. Take $z\in \mathcal{H}_{\rho}(z_1)\cap\overline{\mathcal{H}}_{\frac12}(z_0)$ and $c_{z_1}\coloneqq c_0$, where the constant $c_0$ is given in \eqref{ineq0.h/phi}. Then by the fact that ${\pmb{\Phi}_{0,a_1}}(z)\geq \frac1c|x_n|^{\frac16}$ in $\overline{\mathcal{H}}_{\frac12}(z_0)$, by \eqref{phi.asymptotic} and \eqref{defn.h0m}, we have
\begin{equation}\label{ineq1.h/phi}
\begin{aligned}
    J &=\dashint_{\mathcal{H}_\rho(z_1)\cap\overline{\mathcal{H}}_{\frac12}(z_0)}|{\pmb{\Phi}}_{0,a_1}(z)|^{-1}\left|f(z)-c_{0}{\pmb{\Phi}}_{0,a_1}(z)\right|\,dz\\
    &\leq c\dashint_{\mathcal{H}_\rho(z_1)\cap\overline{\mathcal{H}}_{\frac12}(z_0)}|x_n|^{-\frac16}\,dz\left(\sup_{\mathcal{H}_\rho(z_1)}\left|f(z)-c_{0}{\pmb{\Phi}}_{0,a_1}(z)\right|\right)\\
    &\leq c\rho^{-\frac12}\sup_{\mathcal{H}_\rho(z_1)}\left|f(z)-c_{0}{\pmb{\Phi}}_{0,a_1}(z)\right|,
\end{aligned}
\end{equation}
where we have also used the fact that 
\begin{align}\label{ineq11.h/phi}
    \dashint_{\mathcal{H}_\rho(z_1)\cap\overline{\mathcal{H}}_{\frac12}(z_0)}|x_n|^{-\frac16}\,dz\leq \dashint_{H_\rho}|x_n|^{-\frac16}\,dx_n\,dv_n\leq c\rho^{-\frac12}.
\end{align}
With the help of \eqref{ineq0.h/phi} and the fact that $|{\pmb{\Phi}}_{i,a}|\leq c|z|^{\frac12}$, we further estimate the term $J$ as 
\begin{equation}\label{ineq111.h/phi}
\begin{aligned}
    J&\leq c\rho^{-\frac12}\sup_{\mathcal{H}_{\rho}(z_1)}\Bigg(|f-(c_{0}+\mathbf{c}_{1}\cdot(v-v_{1})){\pmb{\Phi}}_{0,a_1}-\mathbf{c}_{2}\cdot(v_j-v_{1}){\pmb{\Phi}}_{1,a_1}|\\
    &\qquad\qquad\qquad\quad+|\mathbf{c}_{1}\cdot(v-v_{1}){\pmb{\Phi}}_{0,a_1}-\mathbf{c}_{2}\cdot(v_j-v_{1}){\pmb{\Phi}}_{1,a_1}|\Bigg)\leq c\rho.
\end{aligned}
\end{equation}
Thus, we have proved \eqref{goal0.h/phi} when $z_1\in \gamma_0$.
\item Let $z_1\in (\mathcal{H}^{0}_{\frac12,M}\setminus \gamma_0)$ and $z_1=(t_1,x_1',x_{1,n},v_1',v_{1,n})$ with $x_{1,n}\neq0$.  By \eqref{defn.h0m}, $z_1\notin \gamma_0\cup \gamma_-$. Next, we have for $\widetilde{z_1}\coloneqq (t_1,x_1',0,v_1',0)$ and $r_1\coloneqq \frac{(x_{1,n})^{\frac13}}{8\sqrt{M}}$ and $r_2=4(x_{1,n})^{\frac13}+|v_{1,n}|$,
\begin{align*}
    \mathcal{Q}_{2r_1}(z_1)\cap\gamma=\emptyset,\quad \mathcal{Q}_{2r_1}(z_1)\subset \mathcal{H}_{r_2}(\widetilde{z_1})\quad\text{and}\quad \mathcal{H}_{2r_2}(\widetilde{z_1})\subset\mathcal{H}_1(z_0).
\end{align*}
By \eqref{ineq0.h/phi}, there is a constant $c_{i}$ such that 
\begin{align}\label{ineq3.h/phi}
    \left\|f-(c_{0}+\mathbf{c}_1\cdot(v-\widetilde{v}_{1}){\pmb{\Phi}}_{0,\widetilde{a}}-\mathbf{c}_2\cdot(v-\widetilde{v}_{1}){\pmb{\Phi}}_{1,\widetilde{a}}\right\|_{L^\infty(\mathcal{H}_{r}(\widetilde{z_1}))}\leq cr^{\frac32+\epsilon}
\end{align}
for any $r\leq \frac1{32}$, where $\widetilde{v_1}=(\widetilde{v}_{1,1}\ldots,\widetilde{v}_{1,n})=(v_1',0)$ and $\widetilde{a}\coloneqq\sqrt{(A(\widetilde{z_1}))_{n,n}}$. On the other hand, we employ  ${\pmb{\Phi}}_{0,a}(x_n,v_n)=\phi_0(x_n/a,v_n/a)$ and the fundamental theorem of the calculus, \eqref{phi.asymptotic}, and $A\in C^{1+\epsilon}(\mathcal{H}_1(z_0))$ to see that for any $z\in \mathcal{H}_{\rho}(\widetilde{z}_1)$,
\begin{equation}\label{ineq4.h/phi}
\begin{aligned}
|\pmb{\Phi}_{0,a_1}(z)-\pmb{\Phi}_{0,\widetilde{a}}(z)|&\leq \left|\int_{0}^{1}x_n\partial_{x_n}\phi_0(x_n(\xi /a_1+(1-\xi)/\widetilde{a}),v_n/a_1)\,d\xi (a_1-\widetilde{a})\right|\\
&\quad +c\left|\int_{0}^{1}v_n\partial_v\phi_0(x_n/\widetilde{a},v_n(\xi/ a_1+(1-\xi)/\widetilde{a}))\,d\xi(a_1-\widetilde{a})\right|\\
&\leq c\|\phi_0\|_{L^\infty({H}_{\rho})}|z_1-\widetilde{z}_1|.
\end{aligned}
\end{equation}
To proceed further, first we consider the case $\rho\geq r_1$. In this case, we have $\mathcal{H}_\rho(z_1)\subset\mathcal{H}_{\rho_0}(\widetilde{z}_1)$ with $\rho_0\coloneqq 16M^{\frac32}\rho$ and 
\begin{align*}
    J\coloneqq\dashint_{\mathcal{H}_\rho(z_1)\cap\overline{\mathcal{H}}_{\frac12}(z_0)}\left|\frac{f}{{\pmb{\Phi}}_{0,a_1}}(z)-c_{z_1}\right|\,dz\leq c\rho^{-\frac12}\sup_{\mathcal{H}_{\rho}({z}_1)}|f(z)-c_{z_1}{\pmb{\Phi}_{0,a_1}}(z)|,
\end{align*}
where we have used \eqref{ineq11.h/phi}. By taking $c_{z_1}=c_0$ and following the same lines as in \eqref{ineq111.h/phi} and \eqref{ineq4.h/phi}, we have
\begin{align}\label{ineq30.h/phi}
    J\leq c\rho^{-\frac12}\left(\sup_{\mathcal{H}_{\rho_0}(\widetilde{z}_1)}|f(z)-c_0{\pmb{\Phi}_{0,\widetilde{a}}}(z)|+\sup_{\mathcal{H}_{\rho_0}(\widetilde{z}_1)}|{\pmb{\Phi}_{0,a_1}}(z)-{\pmb{\Phi}_{0,\widetilde{a}}}(z)|\right)\leq c\rho.
\end{align}
Thus, we now assume $\rho\leq r_1$. We fix $c_{z_1}\coloneqq \frac{f(z_1)}{\pmb{\Phi}_{0,a_1}(z_1)}$ to observe
\begin{equation}\label{defn.j1j2.h/phi}
\begin{aligned}
   J&\coloneqq\dashint_{\mathcal{H}_\rho(z_1)\cap\overline{\mathcal{H}}_{\frac12}(z_0)}\left|\frac{f}{{\pmb{\Phi}}_{0,a_1}}(z)-c_{z_1}\right|\,dz\\
   &=\dashint_{\mathcal{H}_\rho(z_1)\cap\overline{\mathcal{H}}_{\frac12}(z_0)}\left|\frac{f(z)-c_{0}{\pmb{\Phi}}_{0,a_1}(z)}{\pmb{\Phi}_{0,a_1}(z)}-\frac{f(z_1)-c_{0}\pmb{\Phi}_{0,a_1}(z_1)}{\pmb{\Phi}_{0,a_1}(z_1)}\right|\,dz\\
   &\leq c\rho\left[\frac{f-c_{0}{\pmb{\Phi}}_{0,a_1}}{\pmb{\Phi}_{0,a_1}}\right]_{C^{0,1}(\mathcal{H}_\rho(z_1)\cap\overline{\mathcal{H}}_{\frac12}(z_0))}\\
   &\leq c\rho[f-c_{0}\pmb{\Phi}_{0,a_1}]_{C^{0,1}(\mathcal{H}_\rho(z_1)\cap\overline{\mathcal{H}}_{\frac12}(z_0))}\|\pmb{\Phi}_{0,a_1}^{-1}\|_{L^\infty(\mathcal{H}_\rho(z_1)\cap\overline{\mathcal{H}}_{\frac12}(z_0))}\\
   &\quad+c\rho[\pmb{\Phi}_{0,a_1}^{-1}]_{C^{0,1}(\mathcal{H}_\rho(z_1)\cap\overline{\mathcal{H}}_{\frac12}(z_0))}\|f-c_{0}\pmb{\Phi}_{0,a_1}\|_{L^\infty(\mathcal{H}_\rho(z_1)\cap\overline{\mathcal{H}}_{\frac12}(z_0))}\eqqcolon J_1+J_2,
\end{aligned}
\end{equation}
where we have used the product rule of the H\"older space for the last inequality.
Next, we observe that  $\mathcal{Q}_{2r_1}(z_1)\cap \gamma_-=\emptyset$ and $\widetilde{f}(z)\coloneqq f(z)-c_{0}{\pmb{\Phi}}_{0,{a}_1}(z)$ solves
\begin{equation*}
(\partial_t+v\cdot\nabla_x)\widetilde{f}-\ddiv(A\nabla_v\widetilde{f})=B\cdot \nabla_v\widetilde{f}+\widetilde{F}\quad\text{in }\mathcal{Q}_{r_1}(z_1),
\end{equation*}
where
\begin{align*}
    \widetilde{F}&= F+c_0\left(v_n\partial_{x_n}{\pmb{\Phi}}_{0,{a}_1}-\partial_{v_i}A_{i,n}\partial_{v_n}{\pmb{\Phi}}_{0,{a}_1}-B_n\partial_{v_n}{\pmb{\Phi}}_{0,{a}_1}+A_{n,n}\partial_{v_n,v_n}\partial_{v_n,v_n}{\pmb{\Phi}}_{0,{a}_1}\right)\\
    &=F-c_0\left(\partial_{v_i}A_{i,n}\partial_{v_n}{\pmb{\Phi}}_{0,{a}_1}+B_n\partial_{v_n}{\pmb{\Phi}}_{0,{a}_1}+(A_{n,n}({z}_1)-A_{n,n}(z))\partial_{v_n,v_n}{\pmb{\Phi}}_{0,{a}_1}\right),
\end{align*}
which follows from the fact that ${\pmb{\Phi}}_{0,{a}_1}$ solves the equation 
\begin{align*}
    (v_n\partial_{x_n}-{a}_1^2\partial_{v_n,v_n})f=0.
\end{align*}
Therefore, by \autoref{lem.sch.divergence}, we have 
\begin{equation}\label{ineq31.h/phi}
\begin{aligned}
    [\widetilde{f}]_{C^{0,1}(\mathcal{Q}_\rho(z_1))}&\leq cr_1^{-1}(\|\widetilde{f}\|_{L^\infty(\mathcal{Q}_{r_1}(z_1))}+r_1^2\|\widetilde{F}\|_{L^\infty(\mathcal{Q}_{r_1}(z_1))})\\
    &\leq cr_1^{-1}(\|\widetilde{f}\|_{L^\infty(\mathcal{Q}_{r_1}(z_1))}+r_1^{\frac32}),
\end{aligned}
\end{equation}
where we have used the fact that for any $z\in \mathcal{Q}_{r_1}(z_1)$,
\begin{align}\label{ineq321.h/phi}
    \partial_{v_n}{\pmb{\Phi}}_{0,{a}_1}(z)\leq c|X_n|^{-\frac12}\leq cr_1^{-\frac12},\quad |A_{n,n}(z_1)-A_{n,n}(z)||\partial_{v_n,v_n}{\pmb{\Phi}}_{0,{a}_1}|\leq c|z_1-z||X_n|^{-\frac32}\leq cr_1^{-\frac12}
\end{align}
with $X_n=(x_n,v_n)$, which follows from \eqref{phi.asymptotic}. 

Thus, using \eqref{ineq4.h/phi} with $\rho $ replaced by $ 2r_2$ and the fact that $\mathcal{Q}_{r_1}(z_1)\subset \mathcal{H}_{2_r2}(\widetilde{z}_1)$ and $|v_{1,n}|^3\leq Mx_{1,n}$, we get
\begin{equation}\label{ineq32.h/phi}
\begin{aligned}
    \|\widetilde{f}\|_{L^{\infty}(\mathcal{Q}_{r_1}(z_1))}
    &\leq c\left(\|f-c_{0}\pmb{\Phi}_{0,\widetilde{a}}\|_{L^{\infty}(\mathcal{Q}_{r_1}(z_1))}+\|\pmb{\Phi}_{0,a_1}-\pmb{\Phi}_{0,\widetilde{a}}\|_{L^\infty(\mathcal{Q}_{r_1}(z_1))}\right) \\
    &\leq c\left(\|f-c_{0}\pmb{\Phi}_{0,\widetilde{a}}\|_{L^\infty(\mathcal{Q}_{r_1}(z_1))}+r_1^{\frac32}\right).
\end{aligned}
\end{equation}
Now by \eqref{ineq3.h/phi} together with the fact that $\mathcal{Q}_{2r_1}(z_1)\subset \mathcal{H}_{r_2}(\widetilde{z_1})$ and $|v_{1,n}|^3\leq Mx_{1,n}$, we get 
\begin{equation}\label{ineq5.h/phi}
\begin{aligned}
    &\|f-c_0\pmb{\Phi}_{0,\widetilde{a}}\|_{L^\infty(\mathcal{H}_{r_1}(z_1))}\\
    &\leq \|f-(c_{0}+\mathbf{c}_2\cdot(v-\widetilde{v}_{1})){\pmb{\Phi}}_{0,\widetilde{a}}-\mathbf{c}_3\cdot(v-\widetilde{v}_{1}){\pmb{\Phi}}_{1,\widetilde{a}}\|_{L^\infty(\mathcal{H}_{r_2}(\widetilde{z_1}))}\\
    &\quad+\|\mathbf{c}_2\cdot(v-\widetilde{v}_{1})){\pmb{\Phi}}_{0,\widetilde{a}}-\mathbf{c}_3\cdot(v-\widetilde{v}_{1}){\pmb{\Phi}}_{1,\widetilde{a}}\|_{L^\infty(\mathcal{H}_{r_2}(\widetilde{z_1}))}\leq cr_1^{\frac32}.
\end{aligned}
\end{equation}

Combining \eqref{ineq31.h/phi}, \eqref{ineq32.h/phi}, \eqref{ineq5.h/phi} and
\begin{equation}\label{ineq6.h/phi}
\|\pmb{\Phi}_{0,a_1}\|_{L^\infty(\mathcal{H}_\rho(z_1)\cap\overline{\mathcal{H}}_{\frac12}(z_0))}\geq \|\pmb{\Phi}_{0,a_1}\|_{L^\infty(\mathcal{H}_{r_1}(z_1)\cap\overline{\mathcal{H}}_{\frac12}(z_0))}\geq c^{-1}{r_1^{\frac12}}
\end{equation}leads to 
\begin{align*}
    J_1\leq c\rho.
\end{align*}
On the other hand, for the estimate of $J_2$, in light of \eqref{phi.asymptotic}, \eqref{ineq4.h/phi} and \eqref{ineq5.h/phi}, we deduce 
\begin{align*}
    J_2&\leq \frac{c\rho}{r_1}  \sup_{z,z'\in\mathcal{H}_{\rho}(z_1)\cap\overline{\mathcal{H}}_{\frac12}(z_0)}\frac{|\pmb{\Phi}_{0,{a}_1}(z)-\pmb{\Phi}_{0,{a}_1}(z')|}{|z-z'|}\left(\|f-c_{0}\pmb{\Phi}_{0,\widetilde{a}}\|_{L^\infty(\mathcal{H}_{\rho}(z_1))}+\|\pmb{\Phi}_{0,{a}_1}-\pmb{\Phi}_{0,\widetilde{a}}\|_{L^\infty(\mathcal{H}_{\rho}(z_1))}\right)\\
    &\leq c\rho r_1^{-\frac32}\left(\|f-c_{0}\pmb{\Phi}_{0,\widetilde{a}}\|_{L^\infty(\mathcal{H}_{\rho}(z_1))}+r_1^{\frac32}\right)\leq c\rho.
\end{align*}
In particular, we have used the fact that 
\begin{align}\label{ineq70.h/phi}
    \sup_{z,z'\in\mathcal{H}_{\rho}(z_1)\cap\overline{\mathcal{H}}_{\frac12}(z_0)}\frac{|\pmb{\Phi}_{0,{a}_1}(z)-\pmb{\Phi}_{0,{a}_1}(z')|}{|z-z'|}\leq c\sup_{|X_n|,|X_n'|\geq \frac{r_1}{16c_M}}\frac{|\pmb{\Phi}_{0}(X_n)-\pmb{\Phi}_{0}(X_n')|}{|X_n-X_n'|}\leq cr_n^{-\frac12}
\end{align}
Since we have proved $J_1+J_2\leq c\rho$, this implies \eqref{goal0.h/phi}.
\item Let $z_1\in \mathcal{H}^+_{\frac12}(z_0)$ and $z_1=(t_1,x_1',x_{1,n},v_1',v_{1,n})$ with $v_{1,n}<0$ and $x_{1,n}\leq \left(\frac{|v_{1,n}|}{4}\right)^{3}$. Then we have for $\widetilde{z_1}\coloneqq (t_1,x_1',0,v_1',0)$, $r_1=|v_{1,n}|/2$ and $r_2\coloneqq 2|v_{1,n}|$,
\begin{align*}
    \mathcal{H}_{r_1}(z_1)\subset\mathcal{H}_{2r_2}(\widetilde{z_1})\subset\mathcal{H}_1(z_0).
\end{align*}
By \eqref{ineq0.h/phi}, there is a constant $c_{i}$ such that 
\begin{align}\label{ineq7.h/phi}
    \left\|f-(c_{0}+\mathbf{c}_1\cdot(v-\widetilde{v}_{1}){\pmb{\Phi}}_{0,\widetilde{a}}-\mathbf{c}_2\cdot(v-\widetilde{v}_{1}){\pmb{\Phi}}_{1,\widetilde{a}}\right\|_{L^\infty(\mathcal{H}_{r}(\widetilde{z_1}))}\leq cr^{\frac32+\epsilon}
\end{align}
for any $r\leq \frac1{32}$, where $\widetilde{v_1}=(\widetilde{v}_{1,1}\ldots,\widetilde{v}_{1,n})$ and $\widetilde{a}\coloneqq\sqrt{(A(\widetilde{z_1}))_{n,n}}$. As in \eqref{ineq30.h/phi}, we get the desired estimate \eqref{goal0.h/phi} when $\rho\geq r_1$, so that we assume $\rho\leq r_1$. Combining interior regularity results given in \autoref{lem.sch.divergence} and \autoref{lem.reg.gamma+} together with a standard covering argument, we have 
\begin{equation*}
    [\widetilde{f}]_{C^{0,1}(\mathcal{Q}_\rho(z_1))}
    \leq cr_1^{-1}(\|\widetilde{f}\|_{L^\infty(\mathcal{Q}_{r_1}(z_1))}+r_1^{\frac32}),
\end{equation*}
where we have also used \eqref{ineq321.h/phi} and $\widetilde{f}\coloneqq f-c_0{\pmb{\Phi}}_{0,a_1}$. Similarly, we get \eqref{ineq6.h/phi} and \eqref{ineq70.h/phi}, as for any $z\in\mathcal{H}_{r_1}(z_1)$, $|v_n|\geq \frac{|v_{1,n}|}{2}$, and $|\phi_0(z)|\eqsim |X_n|^{\frac12}$ on $\overline{\mathcal{H}}^{\frac12}(z_0)$. Therefore, we can follow the same arguments as in the estimates of $J_1$ and $J_2$, to derive that the desired estimate \eqref{goal0.h/phi} when $z_1\in \mathcal{H}^{+}_{\frac12}(z_0)$.
\end{itemize}
Before completing the proof, we note that for any $z_1,z_2\in \overline{\mathcal{H}}_{\frac12}(z_0)$,
\begin{align}\label{fin1.f/phi}
     \left|\frac{f}{\Phi_{a_1}}(z_2)-\frac{f}{\Phi_{a_2}}(z_2)\right|\leq c|f(z_2)||X_2|^{-1}|\Phi_{a_1}(z_2)-\Phi_{a_2}(z_2)|\leq c|f(z_2)||X_2|^{-\frac12}|z_1-z_2|,
\end{align}
where we have used \eqref{ineq6.h/phi} with $r_1=|X_2|$ and \eqref{ineq4.h/phi} with $\rho=|X_2|$. On the other hand, by \autoref{thm.hol12}, we derive
\begin{align}\label{fin2.f/phi}
    |f(z_2)|=|f(z_2)-f(\widetilde{z}_2)|\leq c|z_2-\widetilde{z}_2|\leq c|X_2|^{\frac12},
\end{align}
where $\widetilde{z}_2=(t_2,x_2',0,v_2',0)$ and we have also used the fact that $z_2\in \overline{\mathcal{H}}_{\frac12}(z_0)$ for the last inequality. Thus combining \eqref{fin1.f/phi} and \eqref{fin2.f/phi} with $z_2$ replaced by $z$ yields
\begin{align*}
    \dashint_{\mathcal{H}_r(z_1)\cap\overline{ \mathcal{H}}_{\frac12}(z_0)}\left|\frac{f}{\Phi_a}(z)-\frac{f}{\Phi_{a_1}}(z)\right|\,dz\leq c\dashint_{\mathcal{H}_r(z_1)\cap\overline{ \mathcal{H}}_{\frac12}(z_0)}|z-z_1|\,dz\leq cr,
\end{align*}
where $a=\sqrt{A_{n,n}(z)}$.
Therefore, using this and \eqref{goal0.h/phi} gives 
\begin{align}\label{last.h/phi}
    \dashint_{\mathcal{H}_r(z_1)\cap\overline{ \mathcal{H}}_{\frac12}(z_0)}\left|\frac{f}{\Phi_a}(z)-c_{z_1}\right|\,dz\leq cr.
\end{align}
Now by the classical theory of Campanato's space, we deduce the desired estimate from \eqref{last.h/phi}. This completes the proof.
\end{proof}

In order to discuss the sharpness of our result, we briefly explain how to derive $C^{0,1}$ estimates of $f/{\pmb{\Phi}_{0,a}}$ when the right-hand is given by divergence type data.  To this end, let us consider a weak solution to
   \begin{equation}\label{ex.h/phi}
\left\{
\begin{alignedat}{3}
(\partial_t+v\cdot\nabla_x)f-\ddiv(A\nabla_vf)&=B\cdot\nabla_vf+F-\ddiv(G)&&\qquad \mbox{in  $\mathcal{H}_1(z_0)$}, \\
f&=0&&\qquad  \mbox{in $\gamma_-\cap \mathcal{Q}_1(z_0)$},
\end{alignedat} \right.
\end{equation}
where $(x_{0})_n=(v_{0})_n=0$,  $A\in C^{\frac12+\epsilon}(\mathcal{H}_1(z_0))$, $B\in C^{ \epsilon}(\mathcal{H}_1(z_0))$, $ F\in L^{\infty}(\mathcal{H}_1(z_0))$, $G\in C^{\frac12+\eps}(\mathcal{H}_1(z_0))$. We may assume
\begin{align*}
    \|f\|_{L^\infty(\mathcal{H}_1(z_0)}+\|F\|_{L^\infty(\mathcal{H}_1(z_0)}+[G]_{C^{\frac12+\eps}(\mathcal{H}_1(z_0)} \le 1.
\end{align*}
Then by following the same lines as in the proof of \autoref{lem.exp.nd.a} together with the fact that $-\ddiv(G)=-\ddiv(G-G(z_0))$, which allows us to use $\|G-G(z_0)\|_{L^\infty(\mathcal{H}_r(z_0)}\leq cr^{\frac12+\eps}[G]_{C^{\frac12+\eps}(\mathcal{H}_1(z_0)}$, we deduce 
\eqref{ineq111.h/phi} with $z_1$ replaced by $z_0$, as the expansion order is less than 2 so that we only use \autoref{lem.bdry.hol}. Furthermore, we know interior $C^{1,\eps}$-estimates of the solution to \eqref{ex.h/phi} (see \autoref{lem.sch.divergence}). Therefore, by repeating the argument used in \autoref{thm.h/phi}, we derive that the weak solution $f$ to \eqref{ex.h/phi} satisfies for any 
$z,z'\in\overline{\mathcal{H}}_{\frac12}(z_0)$,
\begin{align*}
   \left|\frac{f}{{\pmb{\Phi}}_{0,a}}(z)-\frac{f}{{\pmb{\Phi}}_{0,a'}}(z')\right| &\leq c{|z-z'|}.
\end{align*}

\begin{example}\label{ex.sharp.f/phi}
    Here, we explain the sharpness of \autoref{thm:f/phi}. Let us consider $n=1$ and $f\coloneqq \phi_0+v^2\partial_v\phi_0$. Then by \eqref{basic.formula.liou.inflow2}, we have that $f$ is a solution to
    \begin{equation*}
\left\{
\begin{alignedat}{3}
(\partial_t+v\partial_x)f-\partial_{vv}f&=-\partial_v(-3\phi_0+5v\partial_v\phi_0)&&\qquad \mbox{in  $\mathcal{H}_1$}, \\
f&=0&&\qquad  \mbox{in $\gamma_-\cap \mathcal{Q}_1$},
\end{alignedat} \right.
\end{equation*}
i.e. it solves \eqref{ex.h/phi} with $G = -3\phi_0 + 5v \partial_v \phi_0$ and $B=F=0$.
 Moreover, by \eqref{defn.gm1.trico} and \eqref{appen.defn.j1} we have that  $\varphi\coloneqq \frac{f}{\phi_0} = 1 + \frac{v^2\partial_v\phi_0}{\phi_0}= 1 + \frac{v^4}{3x}\frac{U(-\frac56,\frac53,-\frac{v^3}{9x})}{U(\frac16,\frac23,-\frac{v^3}{9x})}$. Using the homogeneity as in the proof of \cite[Proposition 3.4]{RoWe25}, we deduce that $\varphi \in C^{0,1} \setminus C^1$. Therefore, in order to get the $C^{0,1}$ regularity, we have to assume $G\in C^{\frac12+\eps}$ for any $\eps\geq0$ in \eqref{ex.h/phi}. In this sense, our result is sharp.
\end{example}

\begin{remark}
\label{rem:higher-quotient-reg}
    We now argue that when $A\equiv1$ and $B\equiv0$ in \eqref{eq.h/phi}, then the regularity of $f/\phi_0$ is better than $C^{0,1}$. Let $f$ be a solution to 
    \begin{equation*}
\left\{
\begin{alignedat}{3}
(\partial_t+v\cdot\nabla_x)f-\Delta_vf&=F&&\qquad \mbox{in  $\mathcal{H}_1(z_0)$}, \\
f&=0&&\qquad  \mbox{in $\gamma_-\cap \mathcal{Q}_1(z_0)$},
\end{alignedat} \right.
\end{equation*}
where $x_{0,n}=v_{0,n}=0$ and $F\in C^{\eps}(\mathcal{H}_1(z_0))$. To show this, first, we note from \autoref{thm:Liou-intro} that we only need two elements $\phi_0$ and $\psi_0$ to prove the expansion estimates of order $2+\frac12-\eps$ given in \autoref{lem.exp.nd.a}. This implies that by repeating the same arguments as in the proof of \autoref{lem.exp.nd.a}, we deduce
\begin{align*}
    \|f-(c_0+\mathbf{c}_1\cdot(v-v_0)+(v-v_0)^T\mathbf{C}_2(v-v_j))\phi_0-c_3\psi_0\|_{L^\infty(\mathcal{H}_\rho(z_0))}\leq c\rho^{2+\frac12-\eps}
\end{align*}
for some constants $\mathbf{c}_1\in\bbR^n, \mathbf{C}_2\in\bbR^{n\times n}, c\in\bbR$,
where we have assumed $\|f\|_{L^\infty(\mathcal{H}_1(z_0))}+[F]_{C^\eps(\mathcal{H}_1(z_0))}\leq 1$. Therefore, as in the proof of \autoref{lem.exp.nd.a} with a few modifications and due to the fact that $\psi_0$ has homogeneity 2, one could show $f/\phi_0\in C^{\frac32}(\overline{\mathcal{H}}_{\frac12}(z_0))$ (see \eqref{defn.overline.h} for the definition of the set $\overline{\mathcal{H}}_{\frac12}(z_0)$). Clearly, this result is sharp, since $f = \psi_0$ is a solution. 

On the other hand, when $F=0$, there is no $\psi_0$ in the expansion, so that we can deduce
\begin{align*}
    \|f-(c_0+\mathbf{c}_1\cdot(v-v_0)+(v-v_0)^T\mathbf{C}_2(v-v_j))\phi_0-c_3\partial_{v_n}\phi_1\|_{L^\infty(\mathcal{H}_\rho(z_0))}\leq c\rho^{2+\frac12+\eps}.
\end{align*}
Since $\partial_{v_n}\phi_1$ has homogeneity $\frac52$, one could prove $f/\phi_0\in C^{1,1}(\overline{\mathcal{H}}_{\frac12}(z_0))$. Clearly, also this result is sharp, since $f = c v_1^2 \phi_0 - \partial_{v_n} \phi_1$ is a solution for some $c \in \R$.
\end{remark}

\subsection{Higher regularity in $\mathcal{R}^-$}

Now we prove the optimal regularity near $\gamma_0\cup\gamma_-$. 
More precisely, for any $\epsilon\in(0,\frac13)$ and $z_0$ with $(x_0)_n=(v_0)_n=0$, we define the set $\mathcal{H}^{-}_{\rho,\epsilon}(z_0)$ by
\begin{align*}
    \mathcal{H}^{-}_{\rho,\epsilon}(z_0)\coloneqq \left\{(t,x',x_n,v',v_n)\in\mathcal{H}_\rho(z_0) \mathcal\,:\, v_n^{\frac{1}\epsilon}\geq x_n \right\},
\end{align*}
and we show that the following fine regularity result in this set. The main point here is that the expansion from \autoref{lem.exp.nd.a} is improved in the sense that the complicated function in $\overline{\mathcal{T}}_{\lambda,\eps,z_0,a}$ can be replaced by a polynomial since all the elements in the expansion are sufficiently smooth themselves, in $\mathcal{H}^{-}_{\rho,\epsilon}(z_0)$.

\begin{theorem}\label{thm.3reggamma-}
   Let $f$ be a weak solution to 
    \begin{equation*}
\left\{
\begin{alignedat}{3}
(\partial_t+v\cdot\nabla_x)f-\ddiv(A\nabla_vf)&=B\cdot\nabla_vf+F&&\qquad \mbox{in  $\mathcal{H}_R(z_0)$}, \\
f&=g&&\qquad  \mbox{in $\gamma_-\cap \mathcal{Q}_R(z_0)$}
\end{alignedat} \right.
\end{equation*}
with $(x_0)_n=(v_0)_n=0$.
Let $A\in C^{2,\frac12}(\mathcal{H}_R(z_0))$, $F\in C^{1,\frac12}(\mathcal{H}_R(z_0))$, $B\in C^{1,\frac12}(\mathcal{H}_R(z_0))$, and $g\in C^{2,1}(\gamma_-\cap \mathcal{Q}_R(z_0))$. For any $\eps\in(0,1)$ and $z_1\in \mathcal{H}^{-}_{\frac{R}2,\epsilon}(z_0)$, whenever $\mathcal{H}_\rho(z_1) \subset  \mathcal{H}^{-}_{\frac{R}2,\epsilon}(z_0)$, then there is a polynomial $p_{z_1}$ such that 
\begin{align*}
   \|f-p_{z_1}\|_{L^\infty(\mathcal{H}_\rho(z_1))}\leq c\left(\frac{\rho}{R}\right)^{3(1-\epsilon)}\left(\|f\|_{L^\infty(\mathcal{H}_R(z_0))}+R^{3}[F]_{C^{2,\frac12}(\mathcal{H}_R(z_0))}+R^3[g]_{C^{2,1}(\gamma_-\cap \mathcal{Q}_R(z_0))}\right),
\end{align*}
for some constant $c=c(n,\epsilon,\Lambda,\|A\|_{C^{2,\frac12}(\mathcal{H}_R(z_0))},\|B\|_{C^{1,\frac12}(\mathcal{H}_R(z_0))},|v_0|)$.
\end{theorem}
\begin{proof}
    We may assume $R=1$ and 
    \begin{align*}
        \|f\|_{L^\infty(\mathcal{H}_1(z_0))}+[F]_{C^{2,\frac12}(\mathcal{H}_1(z_0))}+[g]_{C^{3}(\gamma_-\cap \mathcal{Q}_1(z_0))} \leq 1.
    \end{align*}
    For simplicity of notation, we omit the dependence of the constant when it depends only on the data such as $n,\epsilon,\Lambda,\|A\|_{C^{2,\frac12}(\mathcal{H}_1(z_0))},\|B\|_{C^{1,\frac12}(\mathcal{H}_1(z_0))}$.

    We fix  $z_1=(t_1,x_1',x_{1,n},v_1',v_{1,n})\in \mathcal{H}^{-}_{\frac12,\epsilon}(z_0)$, which implies $x_{1,n}\leq v_{1,n}^{\frac1\epsilon}$, and assume $\mathcal{H}_{\rho}(z_1)\subset \mathcal{H}^{-}_{\frac12,\eps}(z_0)$. 

    First, write $ r_1\coloneqq\frac1{2}\left(\frac{x_{1,n}}{v_{1,n}}\right)^{\frac12}$ to observe that
    \begin{align*}
    \mathcal{H}_{r_1}(z_1)\cap\gamma =\emptyset,\quad \mathcal{H}_{2r_1}(z_1)\subset \mathcal{H}_{8r_1}(\widehat{z}_1)\subset \mathcal{H}_{\frac{v_{1,n}}4}(\widehat{z}_1),\quad \mathcal{H}_{\frac{v_{1,n}}4}(\widehat{z}_1)\subset \mathcal{H}_{v_{1,n}}(\widetilde{z}_1),
\end{align*}
where we write 
\begin{align*}
    \widehat{z}_1\coloneqq \left(t_1-\frac{x_{1,n}}{v_{1,n}},x_1'-\frac{x_{1,n}}{v_{1,n}}v_1',0,v_1\right)\quad\text{and}\quad \widetilde{z}_1\coloneqq \left(t_1-\frac{x_{1,n}}{v_{1,n}},x_1'-\frac{x_{1,n}}{v_{1,n}}v_1',0,v_1',0\right).
\end{align*}
    By \autoref{lem.exp.nd.a}, there is $P_{\widetilde{z}_1}\in \mathcal{\overline{T}}_{2,(1-\frac{3\eps}2),\widetilde{z}_1,\widetilde{a}_1}$ with $\widetilde{a}_1\coloneqq\sqrt{(A(\widetilde{z}_1))_{n,n}} $ such that for any $r\leq\frac1{32}$
\begin{align}\label{ineq1.3reggamma-}
    \|f-P_{\widetilde{z}_1}\|_{L^\infty(\mathcal{H}_r(\widetilde{z}_1))}\leq cr^{3(1-\frac{\eps}{2})}.
\end{align}
Indeed, $P_{\widetilde{z}_1}$ is of the form 
\begin{align*}
    P_{\widetilde{z}_1}\coloneqq p_{\widetilde{z}_1}+\sum_{i=0}^2p_i{\pmb{\Phi}}_{\widetilde{a}_1,i}+b_3{\pmb{\Phi}}_{\widetilde{a}_1,3}+b_5{\pmb{\Phi}}_{\widetilde{a}_1,5},
\end{align*}
where $p_{\widetilde{z}_1},p_i\in \cP_2$ and $p_{1}(\widetilde{z}_1)=p_{2}(\widetilde{z}_1)=Dp_{2}(\widetilde{z}_1)=0$.

Next, let us choose a constant $\delta\in(\frac{\eps}3,1)$, which will be determined later (see \eqref{choi.delta.3reggamma-}). Then we observe for $\widehat{r}_1\coloneqq\frac{1}{2}({v_{1,n}})^{\frac12(\frac{3}{\delta}-1)}$, that there is a constant small $c_\delta=c_\delta(\delta)$ such that when $v_{1,n}\leq c_\delta$, then
\begin{align}\label{v1n.ass.3reggamma-}
    v_{1,n}-4\widehat{r}_1>\frac{v_{1,n}}{2}.
\end{align}
We now assume \eqref{v1n.ass.3reggamma-} so that we have
    \begin{align}\label{delta.cond.3reggamma-}
       (4 \widehat{r}_1)^2v_{1,n}+(4\widehat{r}_1)^3\leq \left(16(v_{1,n}-\widehat{r}_1)\right)^{\frac3{\delta}}\quad\text{and}\quad \mathcal{H}_{{4\widehat{r}_1}}(\widehat{z}_1)\subset \{z\in\mathcal{R}^-\,:\,x_n\leq (16v_n)^{\frac3{\delta}}\}.
    \end{align}
     Then, by \eqref{est.smooth.phim} together with  \eqref{delta.cond.3reggamma-}, we get 
\begin{align}\label{ineq.phi00.3reggamma-}
    \|D^2(\pmb{\Phi}_{\widetilde{a}_1,i})\|_{L^\infty(\mathcal{H}_{4\widehat{r}_1}(\widehat{z}_1))}\leq c
\end{align}
for any $i\in\{0,1,2,5\}$. It remains to estimate $\pmb{\Phi}_{\widetilde{a}_1,i}$. To do so, we use that since $\psi_0$ is a solution to 
 \begin{equation*}
\left\{
\begin{alignedat}{3}
(\partial_t+v\cdot\nabla_x)\psi_0-\Delta_v\psi_0&=0&&\qquad \mbox{in  $\bbR\times \{x_n>0\}\times \bbR^n$}, \\
\psi_0&=c_0v_n^2&&\qquad  \mbox{in $\bbR\times \{x_n=0\}\times \{v_n>0\}$},
\end{alignedat} \right.
\end{equation*}
by \autoref{lem.reg.gamma+}, we get for $\delta'\coloneqq 1-3\eps$,
\begin{align*}
    [\psi_0]_{C^{2,\delta'}(\mathcal{H}_{8\widehat{r}_1}(\widehat{z}_1))}\leq [\psi_0]_{C^{2,\delta'}(\mathcal{H}_{\frac{v_{1,n}}{4}}(\widehat{z}_1))}\leq c({v_{1,n}})^{-(2+\delta')}\|\psi_0\|_{L^\infty(\mathcal{H}_{\frac{v_{1,n}}2}(\widehat{z}_1))}.
\end{align*}

Since $\mathcal{H}_{\frac{v_{1,n}}2}(\widehat{z}_1)\subset  \{x_{n}\leq (2v_n)^3\}$, using \eqref{psi0.asy.v}, we have $\|\psi_0\|_{L^\infty(\mathcal{H}_{\frac{v_{1,n}}2}(\widehat{z}_1))}\leq cv_{1,n}^2$. Thus we deduce
\begin{align*}
    [\psi_0]_{C^{2,\delta'}(\mathcal{H}_{4\widehat{r}_1}(\widetilde{z}_1))}\leq c(v_{1,n})^{-\delta'},
\end{align*}
which implies 
\begin{align}\label{ineq.psi.3reggamm-}
    [{\pmb{\Phi}}_{\widetilde{a}_1,3}]_{C^{2,\delta'}(\mathcal{H}_{4\widehat{r}_1}(\widetilde{z}_1))}\leq c(v_{1,n})^{-\delta'},
\end{align}
Furthermore, we note 
\begin{align*}
    &(\partial_t+v\cdot\nabla_x){\pmb{\Phi}_{\widetilde{a}_1,3}}-\ddiv(A\nabla_v{\pmb{\Phi}_{\widetilde{a}_1,3}})-B\nabla_v{\pmb{\Phi}_{\widetilde{a}_1,3}}\\
    &=\partial_{v_n}(A_{n,n}(z)-A_{n,n}(\widetilde{z}_1))\partial_{v_n}{\pmb{\Phi}_{\widetilde{a}_1,3}}+(A_{n,n}(z)-A_{n,n}(\widetilde{z}_1))\partial^2_{v_n}{\pmb{\Phi}_{\widetilde{a}_1,3}}-B\partial_{v_n}{\pmb{\Phi}_{\widetilde{a}_1,3}},
\end{align*}
and 
\begin{align*}
    \|A_{n,n}(z)-A_{n,n}(\widetilde{z}_1)\|_{L^\infty(\mathcal{H}_{\widehat{r}_1}(4\widehat{z}_1))}\leq \|A_{n,n}(z)-A_{n,n}(\widetilde{z}_1)\|_{L^\infty(\mathcal{H}_{v_{1,n}}(\widetilde{z}_1))} \leq cv_{1,n},
\end{align*}
where we have used the Lipschitz regularity of $A$ and $\mathcal{H}_{v_{1,n}}(\widetilde{z}_1)\supset \mathcal{H}_{8r_1}(\widehat{z}_1)$.
Therefore, using these together with \eqref{ineq.psi.3reggamm-}, we have 
\begin{align*}
    [(\partial_t+v\cdot\nabla_x){\pmb{\Phi}_{\widetilde{a}_1,3}}-\ddiv(A\nabla_v{\pmb{\Phi}_{\widetilde{a}_1,3}})-B\nabla_v{\pmb{\Phi}_{\widetilde{a}_1,3}}]_{C^{1-3\eps}(\mathcal{H}_{4\widehat{r}_1}(\widehat{z}_1))}\leq c(\eps),
\end{align*}
Thus, with this, \eqref{ineq.phi00.3reggamma-} and  $\|p_{\widetilde{z}_1}\|_{L^\infty(\mathcal{H}_{\frac1{32}}(\widetilde{z}_1))}\leq c$ by \eqref{ineq1.3reggamma-}, we have 
\begin{align}\label{ineq.Pz1.3reggamma-}
     [(\partial_t+v\cdot\nabla_x)P_{\widetilde{z}_1}-\ddiv(A\nabla_vP_{\widetilde{z}_1})-B\nabla_vP_{\widetilde{z}_1}]_{C^{1-3\eps}(\mathcal{H}_{4\widehat{r}_1}(\widehat{z}_1))}\leq c(\eps).
\end{align}
Since $\widetilde{f}\coloneqq f-P_{\widetilde{z}_1}$ satisfies
\begin{align*}
    (\partial_t+v\cdot\nabla_x)\widetilde{f}-\ddiv(A\nabla_v\widetilde{f})=B\cdot \nabla_v\widetilde{f}+F+\widetilde{F}\quad\text{in }\mathcal{H}_{4\widehat{r}_1}(\widehat{z}_1)
\end{align*}
with $\widetilde{F}\coloneqq (\partial_t+v\cdot\nabla_x)P_{\widetilde{z}_1}-\ddiv(A\nabla_vP_{\widetilde{z}_1})-B\nabla_vP_{\widetilde{z}_1}$ and $\widetilde{f}=g-\widetilde{p}$ on $\gamma_-\cap\mathcal{Q}_{4\widehat{r}_1}(\widehat{z}_1)$ for some $\widetilde{p}\in\cP_2$, by applying \autoref{lem.reg.gamma+}, there is a polynomial $p_{\widehat{z}_1}\in\cP_2$ such that for any $r\leq 4\widehat{r}_1$,
\begin{align}\label{exp1.3reggamma-}
    \|f-P_{\widetilde{z}_1}-p_{\widehat{z}_1}\|_{L^\infty(\mathcal{H}_{r}(\widehat{z}_1))}\leq c\left(\left(\frac{r}{\widehat{r}_1}\right)^{3-3\eps}\|f-P_{\widetilde{z}_1}\|_{L^\infty(\mathcal{H}_{4\widehat{r}_1}(\widetilde{z}_1))}+r^{3-3\eps}\right),
\end{align}
where we have also used \eqref{ineq.Pz1.3reggamma-}. This gives 
\begin{align*}
    \|p_{\widehat{z}_1}\|_{L^\infty(\mathcal{H}_{4\widehat{r}_1}(\widehat{z}_1))}\leq c(\|f-P_{\widetilde{z}_1}\|_{L^\infty(\mathcal{H}_{4\widehat{r}_1}(\widetilde{z}_1))}+r_1^{3-3\eps}),
\end{align*}
which implies 
\begin{align*}
    [p_{\widehat{z}_1}]_{C^{k,\eps}(\mathcal{H}_{4\widehat{r}_1}(\widetilde{z}_1))}\leq \frac{c}{(\widehat{r}_1)^{k+\eps}}(\|f-P_{\widetilde{z}_1}\|_{L^\infty(\mathcal{H}_{4\widehat{r}_1}(\widetilde{z}_1))}+\widehat{r_1}^{3-3\eps}).
\end{align*}
On the other hand, we observe that $\widehat{f}=f-P_{\widetilde{z}_1}-p_{\widehat{z}_1}$ satisfies
\begin{align*}
    (\partial_t+v\cdot\nabla_x)\widehat{f}-\ddiv(A\nabla_v\widehat{f})=B\cdot \nabla_v\widehat{f}+F+\widetilde{F}+\widehat{G},
\end{align*}
where $\widehat{G}=-\partial_tp_{\widehat{z}_1}+\ddiv(A\nabla_vp_{\widehat{z}_1})+B\cdot\nabla_vp_{\widehat{z}_1}$. Thus, by \autoref{lem.reg.gamma+}, we have for any $\rho\leq r_1$,
\begin{equation}\label{exp2.3reggamma-}
\begin{aligned}
    \|f-P_{\widetilde{z}_1}-p_{\widehat{z}_1}-p_{z_1}\|_{L^\infty(\mathcal{H}_{\rho}(z_1))}&\leq c\left(\frac{\rho}{{r}_1}\right)^{3-3\eps}\|f-P_{\widetilde{z}_1}-p_{\widehat{z}_1}\|_{L^\infty(\mathcal{H}_{r_1}(z_1))}\\
    &\quad+\rho^{3-3\eps}\left[1+\frac{c}{(\widehat{r}_1)^{3-3\eps}}\|f-P_{\widetilde{z}_1}\|_{L^\infty(\mathcal{H}_{4\widehat{r}_1}(z_1))}\right],
\end{aligned}
\end{equation}
where we have also used \eqref{ineq.Pz1.3reggamma-} and  $[\widehat{G}]_{C^{1-3\eps}(\mathcal{H}_{r_1}(z_1))}\leq [\widehat{G}]_{C^{1-3\eps}(\mathcal{H}_{4\widehat{r}_1}(\widehat{z}_1))}\leq \|p_{\widehat{z}_1}\|_{C^{2,1-3\eps}(\mathcal{H}_{4\widehat{r}_1}(\widehat{z}_1))}$. Thus, combining \eqref{exp1.3reggamma-} with $r=4r_1$, \eqref{exp2.3reggamma-} yields
\begin{align}\label{exp3.3reggamma-}
    \|f-P_{\widetilde{z}_1}-p_{\widehat{z}_1}-p_{z_1}\|_{L^\infty(\mathcal{H}_{\rho}(z_1))}\leq c\left( \left(\frac{\rho}{\widehat{r}_1}\right)^{3-3\eps}\|f-P_{\widetilde{z}_1}\|_{L^\infty(\mathcal{H}_{4\widehat{r}_1}(z_1))}+\rho^{3-3\eps}\right),
\end{align}
where we have also used the fact that $\mathcal{H}_{r_1}(z_1)\subset \mathcal{H}_{4{r}_1}(\widehat{z}_1)$. Lastly, combining this, \eqref{ineq1.3reggamma-} with $r=v_{1,n}$ and the fact that $\mathcal{H}_{4\widehat{r}_1}(\widehat{z}_1)\subset \mathcal{H}_{4v_{1,n}}(\widetilde{z}_1)$ leads to 
\begin{align*}
     \|f-P_{\widetilde{z}_1}-p_{\widehat{z}_1}-p_{z_1}\|_{L^\infty(\mathcal{H}_{\varrho}(z_1))}\leq c\left(\left(\frac{\rho}{\widehat{r}_1}\right)^{3(1-\eps)}v_{1,n}^{3(1-\frac\eps2)}+\rho^{3(1-\eps)}\right),
\end{align*}
where $c=c(\delta)$. Now by taking $\delta=\delta(\eps)$ very close to 1 so that 
\begin{align}\label{choi.delta.3reggamma-}
-\frac12(\frac3{\delta}-1)(3-3\eps)+(3-\frac32\eps)\geq0,    
\end{align}
we have $\widehat{r}_1^{-3(1-\eps)}v_{1,n}^{3-\frac32\eps}\leq c$, 
as $\widehat{r}_1=\frac{1}{2}({v_{1,n}})^{\frac12(\frac{3}{\delta}-1)}$. This implies that 
\begin{align*}
     \|f-P_{\widetilde{z}_1}-p_{\widehat{z}_1}-p_{z_1}\|_{L^\infty(\mathcal{H}_{\varrho}(z_1))}\leq c\rho^{3(1-\eps)}.
\end{align*}
Now using \eqref{est.smooth.phim} and \eqref{taylor.psi0} together with the fact that $\mathcal{H}_{\rho}(z_1)\subset \mathcal{H}^-_{1,\eps}(z_0)$, there is the second order Taylor polynomial of $P_{\widetilde{z_1}}$, $\widetilde{p}_{\widetilde{z}_1}$ such that 
\begin{align*}
    \|P_{\widetilde{z}_1}-\widetilde{p}_{\widetilde{z}_1}\|_{L^\infty(\mathcal{H}_\rho(z_1))}\leq c\rho^{3-3\epsilon}.
\end{align*}

Therefore, we have 
\begin{align*}
    \|f-\widetilde{p}_{\widetilde{z}_1}-p_{\widehat{z}_1}-p_{z_1}\|_{L^\infty(\mathcal{H}_\rho(z_1))}\leq c\rho^{3(1-\eps)},
\end{align*}
whenever $v_{1,n}\leq c_\delta$ by \eqref{v1n.ass.3reggamma-}. However,  if $v_{1,n}> c_\delta$, then by \autoref{lem.reg.gamma+}, we get the desired estimate by a combination of interior regularity estimates and the boundary estimates at $\gamma_-$, as $z_1$ is away from $\gamma_0$ and $c_\delta>0$ only depends on the constant $\eps$. More precisely, as in \eqref{exp1.3reggamma-} and \eqref{exp2.3reggamma-} with $P_{\widetilde{z}_1}=0$ and $\widehat{r}_1=\frac{v_{1,n}}{8}$, we get \eqref{exp3.3reggamma-} with $P_{\widetilde{z}_1}=0$ and $\widehat{r}_1=\frac{v_{1,n}}{8}$. Then the fact that $v_{1,n}\geq c_\delta$ yields the desired estimate. This completes the proof.
\end{proof}
We end this section with an explicit example verifying the optimality of \autoref{thm.3reggamma-} (and therefore of \autoref{thm:C3}).
\begin{example}\label{ex.sharp.3eps}
    Let $n=1$ and $f\coloneqq\psi_0$. Then we have 
    \begin{equation*}
\left\{
\begin{alignedat}{3}
(\partial_t+v\partial_x)f-\partial_{vv}f&=0&&\qquad \mbox{in  $\mathcal{H}_1$}, \\
f&=c_0v^2&&\qquad  \mbox{in $\gamma_-\cap\mathcal{Q}_1$}.
\end{alignedat} \right.
\end{equation*}
By \eqref{psi.sharp}, we have to assume $D\subset \{x\leq v^{\frac1\eps}\}$ to prove $f\in C^{3-3\eps}(D)$. This implies that our result given in \autoref{thm.3reggamma-} is optimal.
\end{example}

\section{Optimal boundary regularity with diffuse reflection condition}
\label{sec:diffuse}
In this section, we prove global regularity with diffuse reflection condition in the halfspace $\Omega = \{ x_n > 0 \}$. Throughout this section, we denote
\begin{align}\label{defn.nf}
    \mathcal{N}f(t,x,v) \coloneqq \mathcal{M}(z)g(t,x') \coloneqq \mathcal{M}(z) \int_{\bbR^n\cap\{w_n<0\}}f(t,x',0,w',w_n)(w_n)_-\,dw
\end{align}
and assume that $\mathcal{M}(z)$ has fast decay as $|v|\to\infty$ and is sufficiently smooth. Later, by flattening the boundary and a covering argument, we can derive global estimates for a general domain. Moreover, \autoref{example:diffuse-optimal} shows that the $C^{\frac{1}{2}}$ regularity is optimal.

The following is the main result of this section and it is our most general result on diffuse reflection and yields \autoref{thm:main-diffuse}. We define 
\begin{align*}
    \mathcal{H}_{R,\infty} = (-R^2,R^2) \times \big(B_{R^3}'\times (0,R^3)\big) \times \bbR^n.
\end{align*}

\begin{theorem}
\label{thm.diff}
Let $f$ be a weak solution to
    \begin{equation}\label{sub.eq.diffuse}
\left\{
\begin{alignedat}{3}
(\partial_t+v\cdot\nabla_x)f-\ddiv(A\nabla_vf)&=B\cdot\nabla_vf+F&&\qquad \mbox{in  $\mathcal{H}_{2,\infty}$}, \\
f&=\mathcal{N}f&&\qquad  \mbox{in $\gamma_-\cap \overline{\mathcal{H}_{2,\infty}}$},
\end{alignedat} \right.
\end{equation}
where $\mathcal{N}f$ is defined as in \eqref{defn.nf}. Let us fix $p\geq\frac12$ and write for any $v_0\in\bbR$, $r_0\coloneqq\frac{\langle v_0\rangle ^{-p}}{64}$. Let $\eps \in (0,1)$, $\lambda \in \N$, and $M >0$. Then we have the following.
\begin{itemize}
    \item Let us fix $\eps \in (0,1)$ and assume that for some $c_B, M > 0$,
    \begin{align}
    \label{b.infty.nf}
           \sup_{z_0\in\mathcal{H}_{2,\infty}}\left( \langle v_0\rangle^{M}\|\mathcal{M}\|_{C^{\frac12+\eps}(\mathcal{H}_{r_0}(z_0))} +  \|A\|_{C^{\eps}(\mathcal{H}_{r_0}(z_0))}+\langle v_0\rangle^{-2}\|B\|_{L^\infty(\mathcal{H}_{r_0}(z_0))}\right)\leq c_B.
        \end{align}
        Then there is $M_0=M_0(n,p) > 0$ such that if $M > M_0$, then for any $z_0\in\mathcal{H}_{1,\infty}$, we have 
\begin{align*}
    \|f\|_{C^{\frac12}(\mathcal{H}_{r_0}(z_0))}&\leq c\langle v_0\rangle(\|f\|_{L^{\infty}(\mathcal{H}_{2r_0}(z_0))}+r_0^2\|F\|_{L^\infty(\mathcal{H}_{2r_0}(z_0))})\\
    &\quad+c\langle v_0\rangle\|\mathcal{M}\|_{C^{\frac12+\eps}(\mathcal{H}_{r_0}(z_0))}\left(\|\langle \cdot\rangle^{q}f(t,x,\cdot)\|_{L^\infty(\mathcal{H}_{2,\infty})}+\left\|\langle\cdot\rangle^{q}F(t,x,\cdot)\right\|_{L^{\infty}(\mathcal{H}_{2,\infty})}\right),
\end{align*}
for some constant $c=c(n,\Lambda,c_B,p)$ and $q=q(n,p)$.
\item Let us fix $\lambda\in\N$, $\eps\in(0,1)$ and assume that for some $c_B, M > 0$,
\begin{align}
\label{b.hol.nf}
            \sup_{z_0\in\mathcal{H}_{2,\infty}}\left( \langle v_0\rangle^{M}\|\mathcal{M}\|_{C^{6\lambda+\frac76}(\mathcal{H}_{r_0}(z_0))} +  \|A\|_{C^{6\lambda+1}(\mathcal{H}_{r_0}(z_0))}+\langle v_0\rangle^{-2}\|B\|_{C^{6\lambda}(\mathcal{H}_{r_0}(v_0))}\right)\leq c_B.
    \end{align}
        Then, there is $M_1=M_1(n,\lambda,\eps,p)$, such that if $M > M_1$, then for any $z_0\in\mathcal{H}_{1,\infty}$
        and $\mathcal{H}_{r_0}(z_0)\cap \gamma_0=\emptyset$, we have 
        \begin{align*}
            \|f\|_{C^{\lambda,\eps}(\mathcal{H}_{r_0}(z_0))}\leq c\left(\|\langle \cdot\rangle^{q}f(t,x,\cdot)\|_{L^\infty(\mathcal{H}_{2,\infty})}+\left\|\langle\cdot\rangle^{q}F(t,x,\cdot)\right\|_{C^{6\lambda}(\mathcal{H}_{2,\infty})}\right),
        \end{align*}
        for some constant $c=c(n,\Lambda,\lambda,\eps,c_B,p)$ and $q=q(n,\lambda,\eps,p)$, where we write
        \begin{align}\label{space.hol.weight}
            \left\|\langle\cdot\rangle^{q}F(t,x,\cdot)\right\|_{C^{6\lambda}(\mathcal{H}_{2,\infty})}\coloneqq \sup_{z_0\in\mathcal{H}_{2,\infty}}\langle v_0\rangle^q\|F\|_{C^{6\lambda}(\mathcal{H}_{1}(z_0)\cap\mathcal{H}_{2,\infty})}.
        \end{align}
\end{itemize}
\end{theorem}
Note that we assume $p\geq \frac12$ to ensure that $\mathcal{H}_{4r_0}(z_0)\subset \mathcal{H}_{2,\infty}$ for any $z_0\in\mathcal{H}_{1,\infty}$. In addition, we will select the constant $p\geq\frac12$ in the proof of \autoref{thm:main-diffuse} given in Section \ref{sec:main-proof}, to satisfy assumptions such as \eqref{b.infty.nf} and \eqref{b.hol.nf} after flattening arguments (see the proof of \autoref{thm:main-diffuse} for more details).
           
The key difficulty of the proof is to obtain regularity of the boundary datum $\mathcal{N}f$, which will then allow us to apply known regularity results about in-flow (see \cite{RoWe25}). The following is the main result in this direction.

\begin{proposition}\label{lem.reg.nf.diff}
    Let us fix $p\geq\frac12$ and write $r_0\coloneqq \frac{\langle v_0\rangle^{-p}}{64}$ for any $v_0\in \mathcal{H}_{1,\infty}$. Let $f$ be a weak solution to \eqref{sub.eq.diffuse}. Then we have the following:
    \begin{itemize}
        \item There is $M_0 = M_0(n,p) > 0$ such that if \eqref{b.infty.nf} holds for some $\eps\in(\frac12,1)$ and $M > M_0$, then for any $z_0\in \mathcal{H}_{1,\infty}$,
\begin{align}\label{second.goal.diff}
    [\cN f]_{C^{\eps}(\gamma_-\cap\mathcal{Q}_{r_0}(z_0))}\leq c\left(\|\langle \cdot\rangle^{q}f(t,x,\cdot)\|_{L^\infty(\mathcal{H}_{2,\infty})}+\left\|\langle\cdot\rangle^{q}F(t,x,\cdot)\right\|_{L^{\infty}(\mathcal{H}_{2,\infty})}\right),
\end{align}
where $c=c(n,\Lambda,\eps,c_B,p,\mathcal{M})$, and $q$ depend only on $n$ and $\eps$.
\item There is $M_1 = M_1(n,\lambda,\eps,p) > 0$ such that that if \eqref{b.hol.nf} holds for some $\lambda\geq1$, $\eps\in(0,1)$, and $M \ge M_1$, then for any $z_0\in \mathcal{H}_{1,\infty}$,
\begin{align}\label{first.goal.diff}
    [\cN f]_{C^{\lambda,\eps}(\gamma_-\cap\mathcal{Q}_{r_0}(z_0))}\leq c\left(\|\langle \cdot\rangle^{q}f(t,x,\cdot)\|_{L^\infty(\mathcal{H}_{2,\infty})}+\left\|\langle\cdot\rangle^{q}F(t,x,\cdot)\right\|_{C^{6\lambda}(\mathcal{H}_{2,\infty})}\right),
\end{align}
where $c=c(n,\Lambda,\lambda,\eps,c_B,p,\mathcal{M})$, and $q$ depend only on $n,\lambda$, and $\eps$.
    \end{itemize}
\end{proposition}

The regularity assumptions on the data $A,B,F$ and $\mathcal{M}$ in \eqref{first.goal.diff} are not sharp. In this paper, we focus on proving the $C^\infty$-regularity away from $\gamma_0$ when the given data is smooth. Therefore, we do not proceed to find a sharp relation between the regularity of the solution and the regularity of data.

Since the proof involves a lot of steps, we here briefly explain the strategy of the proof. We will use a difference quotient technique and a bootstrap argument to derive H\"older estimates of 
\begin{align*}
\widehat{g}_k\coloneqq\frac{\delta_{h,k}\mathcal{N}f}{|h|^k} \quad \text{ and } \quad \widehat{f}_k\coloneqq\frac{\delta_{h,k}f}{|h|^k}.
\end{align*}
See \eqref{defn.diff.quot} for the definition of $\delta_{h,k}$. First, we can observe that $L^\infty$ estimates of $\widehat{f}_k$ and $C^{1,\eps}$ estimates of $\widehat{f}_k$ near $\gamma_+$ are enough to get uniform bounds of $\widehat{g}_k$ in H\"older spaces $C^{\frac43-\eps}$ for any $\delta>0$ (see \eqref{widehatg.hol} below). Next, we prove $C^{1,\eps}$ estimates of $\widehat{f}_k$ away from $\gamma_+$ in terms of the supremum norm of $\widehat{f}_k$ (see \eqref{c2eps.ind.diff}). Based on these observations, we will inductively prove $C^{\frac12-\eps}$ estimates of $\widehat{f}_k$ (see \eqref{ind.level} below). 

There are three steps to derive uniform $C^{\frac12-\eps}$ estimates of $\widehat{f}_{k+1}$ from the estimates of $\widehat{f}_{k}$.
\begin{itemize}
    \item[1.] Using \eqref{hol.rel.kin}, we get that $\frac{\delta_{h,k+1}f}{|h|^{k+\alpha}}$ is bounded for any $\alpha<\frac16$. Thus, using this and $C^{2-\eps}$ estimates, we can derive  $\frac{\delta_{h,k+1}\mathcal{N}f}{|h|^{k+\alpha}}\in C^{\frac76}$.

    \item[2.] Note that $\frac{\delta_{h,k+1}f}{|h|^{k+\alpha}}$ satisfies an equation with an unbounded right-hand side, such as for example $-\ddiv\left(\frac{\delta_{h}A}{|h|^\alpha}\nabla_v\widehat{f}_k\right)$, which prevents us from applying global $C^{\frac12}$ regularity estimates. Instead, we prove that $\nabla_v\widehat{f}_k\in L^q$ for a sufficiently large $q$, by obtaining gradient estimates of $\widehat{f}_k$ away from $\gamma_0$, and combining this with $C^{\frac12-\eps}$ estimates of $\widehat{f}_k$ at $\gamma_0$.

    \item[3.] Now we use a perturbation argument together with a higher integrability result of $\nabla_v\widehat{f}_k$ to derive $\frac{\delta_{h,k+1}f}{|h|^{k+\alpha}}\in C^{\frac12-\eps}$, which implies that $\frac{\delta_{h,k+1}f}{|h|^{k+\alpha}}$ is bounded for any $\alpha<\frac26$. Now we iterate this procedures 7-times so that $\alpha$ becomes 1 and we get uniform $C^{\frac12-\eps}$ estimates of $\widehat{f}_{k+1}$.
\end{itemize}

    Therefore, by \eqref{widehatg.hol}, we prove $\widehat{g}_{k}\in C^{\frac43-\eps}$, which gives $\partial^kg\in C^{\frac43-\eps}$. Now using \eqref{fund.lambda.eps}, we can find a number $k$ such that $g\in C^{\lambda,\delta}$. As we mentioned before, we just find a large number $k$, which is not sharp, but is enough to prove $g\in C^{\lambda,\delta}$.

\subsection{Preliminary results: boundary estimates with in-flow}

Before we can prove \autoref{lem.reg.nf.diff}, we need to provide several results on regularity estimates away from $\gamma_0$. They were basically shown in \cite{RoWe25} (see \cite{KLN25} for the interior case), but for the application to diffuse reflection, we need to track the explicit dependence of the constants on the drift and include source terms in divergence form of the form $\mathrm{div}_v(G)$ into our study. Moreover, in \autoref{lem.bdry.hol.bdep} we prove $C^{\frac{1}{2}-\delta}$ estimates up to $\gamma_0$ for unbounded source terms, thereby refining \autoref{thm.hol12}.

\begin{lemma}\label{lem.gra.bdep}
    Let $f$ be a weak solution to 
\begin{equation}
\left\{
\begin{alignedat}{3}
(\partial_t+v\cdot\nabla_x)f-\ddiv(A\nabla_vf)&=B\cdot\nabla_vf-\ddiv(G)+F&&\qquad \mbox{in  $\mathcal{H}_R(z_0)$}, \\
f&=g&&\qquad  \mbox{in $\gamma_-\cap\mathcal{Q}_R(z_0)$}.
\end{alignedat} \right.
\end{equation}
 Suppose $\mathcal{H}_{2R}(z_0)\cap\gamma_0=\emptyset$. For any $\eps\in(0,1)$, we have 
\begin{equation}\label{est.c1eps.awaygamma}
\begin{aligned}
        &R\|\nabla_vf\|_{L^\infty(\mathcal{H}_{R/2}(z_0))}+R^{1+\eps}[\nabla_vf]_{C^{\eps}(\mathcal{H}_{R/2}(z_0))}+R^{1+\eps}[f]_{C^{1,\eps}(\mathcal{H}_{R/2}(z_0))}\\
        &\leq c\mathbf{B}(z_0,R)^{1+\eps}\Bigg(\|f\|_{L^\infty(\mathcal{H}_R(z_0))}+R^{1+\eps}[G]_{C^{\eps}(\mathcal{H}_{R}(z_0))}\\
        &\qquad\qquad\qquad\quad+R^2\|F\|_{L^\infty(\mathcal{H}_R(z_0))}+R^{1+\eps}[g]_{C^{1,\eps}(\gamma_-\cap\mathcal{Q}_R(z_0))}\Bigg),
    \end{aligned}
    \end{equation}
    where $\mathbf{B}(z_0,R)\coloneqq (1+\|B\|_{L^\infty(\mathcal{H}_R(z_0))})$ and
    $c=c(n,\Lambda,\eps,\|A\|_{C^\eps(\mathcal{H}_R(z_0))})$.
\end{lemma}

\begin{proof}
The explicit dependence of the constant on $B$ can be obtained by the same scaling argument as in the proof of \eqref{estw/oB.exp.nd}. Hence, it suffices to prove \eqref{est.c1eps.awaygamma}, when $\|{B}\|_{L^\infty(\mathcal{H}_R(z_0))} \leq 1$. 

In case $G = 0$, interior estimates and boundary expansions have been obtained in \cite{KLN25} and \autoref{lem.reg.gamma+} in case $G = 0$. To treat divergence form source terms, first by \autoref{lem.sch.divergence}, we have 
\begin{equation}\label{gra.est.gra.bdep}
\begin{aligned}
        R^{1+\eps}[f]_{C^{1,\eps}(\mathcal{H}_{R/2}(z_0))}&\leq c\left(\|f\|_{L^\infty(\mathcal{H}_R(z_0))}+R^{\eps}[G]_{C^{\eps}(\mathcal{H}_{R}(z_0))})+cR^2\|F\|_{L^\infty(\mathcal{H}_R(z_0))}\right),
    \end{aligned}
    \end{equation}
    where $c=c(n,\Lambda,\eps,\|A\|_{C^\eps(\mathcal{H}_R(z_0))})$, whenever $\mathcal{H}_{2R}(z_0)\cap \gamma=\emptyset$. In addition, for any $z_1\in\mathcal{H}_{2R}(z_0)\cap\gamma$, by following the same lines as in the proof of \autoref{lem.bdry.hol} together with the following observation $-\ddiv(G)=\ddiv(G-G(z_0))$, which allows us to estimate $\|G-G(z_1)\|_{L^\infty(\mathcal{H}_r(z_1))}\lesssim r^\eps[G]_{C^\eps(\mathcal{H}_r(z_1))}$, we get that there is a small constant $\alpha=\alpha(n,\Lambda)$ such that
    \begin{align}\label{boundary.gra.bdep}
        r^\alpha[f]_{C^{\alpha}(\mathcal{H}_{r/2}(z_1))}\leq c(\|f\|_{L^\infty(\mathcal{H}_r(z_1))}+r^\eps[G]_{C^\eps(\mathcal{H}_r(z_1))}+r^2\|F\|_{L^\infty(\mathcal{H}_r(z_1))}+r^\alpha[g]_{C^\alpha(\gamma_-\cap\mathcal{H}_{r}(z_1))}),
    \end{align}
    for some constant $c=c(n,\Lambda,\eps,\|A\|_{C^\eps(\mathcal{H}_r(z_1))})$, whenever $\mathcal{H}_{r}(z_1)\subset \mathcal{H}_R(z_0)$.
    In order to prove expansions at boundary points, when { $\mathcal{H}_{R}(z_0)\cap\gamma_0\neq\emptyset$ }, we use the same arguments as in the proof of \cite[Lemma 4.1 and Lemma 4.4]{RoWe25} (see also \autoref{lem.exp.nd.a}) together with interior $C^{1,\eps}$-estimates \eqref{gra.est.gra.bdep} and boundary $C^\alpha$-estimates \eqref{boundary.gra.bdep},  which will be used in order to guarantee convergence of the blow-up sequence as in \eqref{conv.grad.exp.nd}. It is easy to see that the assumption on $G$ has the right scaling and hence guarantees convergence of the blow-up sequence to a suitable limit.  Altogether, we get
     \begin{equation*}
     \begin{aligned}
       \|f-p_{z_0}\|_{L^\infty(\mathcal{H}_r(z_0))}&\leq c\left(\frac{r}{R}\right)^{1+\eps}\Bigg(\|f\|_{L^\infty(\mathcal{H}_R(z_0))}+R^{1+\eps}[G]_{C^{\eps}(\mathcal{H}_{R}(z_0))}\\
        &\qquad\qquad\qquad\quad+R^2\|F\|_{L^\infty(\mathcal{H}_R(z_0))}+R^{1+\eps}[g]_{C^{1,\eps}(\gamma_-\cap\mathcal{Q}_R(z_0))}\Bigg)
    \end{aligned}
    \end{equation*}
    for some polynomials $p_{z_0}\in\cP_1$, 
    which implies the following estimate upon combination with the interior estimates from \eqref{gra.est.gra.bdep}
    \begin{align*}
        R^{1+\eps}[f]_{C^{1+\eps}(\mathcal{H}_{R/2}(z_0))}&\leq c\Bigg(\|f\|_{L^\infty(\mathcal{H}_R(z_0))}+R^{\eps}[G]_{C^{\eps}(\mathcal{H}_{R}(z_0))}\\
        &\qquad\quad+R^2\|F\|_{L^\infty(\mathcal{H}_R(z_0))}+R^{1+\eps}[g]_{C^{1,\eps}(\gamma_-\cap\mathcal{Q}_R(z_0))}\Bigg),
    \end{align*}for some constant $c=c(n,\Lambda,\eps,\|A\|_{C^\eps(\mathcal{H}_R(z_0))})$. This completes the proof.
\end{proof}
\begin{remark}\label{rmk.c12.bdep}
    Note that that by repeating the same arguments as in the proof \autoref{lem.gra.bdep}, we derive that whenever $\mathcal{H}_R(z_0)\cap\gamma_0=\emptyset$,
    \begin{align*}
        R^{\frac12}[f]_{C^{\frac12}(\mathcal{H}_{R/2}(z_0))}&\leq c\mathbf{B}(z_0,R)^{\frac12}\Bigg(\|f\|_{L^\infty(\mathcal{H}_R(z_0))}+R\|G\|_{L^{\infty}(\mathcal{H}_{R}(z_0))}\\
        &\qquad\qquad\qquad\quad+R^2\|F\|_{L^\infty(\mathcal{H}_R(z_0))}+R^{\frac12}[g]_{C^{\frac12}(\gamma_-\cap\mathcal{Q}_R(z_0))}\Bigg),
    \end{align*}
    for some constant $c=c(n,\Lambda,\|A\|_{C^\eps(\mathcal{H}_R(z_0))})$. 
\end{remark}

Next, we prove H\"older estimates of order $\frac{1}{2} - \delta$ near the boundary and up to $\gamma_0$. 

\begin{lemma}\label{lem.bdry.hol.bdep}
Let $R\leq1$, $z_0\in\gamma_0$ with $x_{0,n}=v_{0,n}=0$ and let $f$ be a weak solution to 
\begin{equation}
\left\{
\begin{alignedat}{3}
(\partial_t+v\cdot\nabla_x)f-\ddiv(A\nabla_vf)&=B\cdot\nabla_vf-\ddiv(G)+F&&\qquad \mbox{in  $\mathcal{H}_R(z_0)$}, \\
f&=g&&\qquad  \mbox{in $\gamma_-\cap\mathcal{Q}_R(z_0)$},
\end{alignedat} \right.
\end{equation}
where $A\in C^\eps(\mathcal{H}_R(z_0))$ for some $\eps>0$.
Let us fix $p\geq\frac12$, $C_1\geq1$ and write $r_1\coloneqq \frac{\langle v_1\rangle^{-p}}{64}$ and $\widehat{r}_1\coloneqq  \frac{\langle v_1\rangle^{-p}}{C_1}$ for any $z_1$. We now fix $\delta<\frac12$. Suppose 
\begin{equation}\label{ass.bdep.bdry.hol}
\begin{aligned}
&\|B\|_{L^\infty(\mathcal{H}_{R}(z_0))}\leq c_B\langle v_0\rangle^2\\
    &\sup_{z_1\in\gamma_0\cap\mathcal{Q}_R(z_0)}\sup_{(t,x',v')\in\mathcal{H}'_{\widehat{r}_1}(z_1')}\|G(t,x',\cdot,v',\cdot)\|_{L^{8-\delta_0}(H_{\widehat{r}_1})}+\|F(t,x',\cdot,v',\cdot)\|_{L^{8-\delta_0}(H_{\widehat{r}_1})}\leq C_0,\\
    &\sup_{z_2\in\mathcal{Q}_R(z_0)}\sup_{\substack{
    \mathcal{Q}_{\rho}(z_2)\cap\gamma_0=\emptyset,\\
    \rho\leq r_2}} \rho^{-\frac12-\delta}(\|G\|_{L^{\infty}(\mathcal{H}_{\rho}(z_2))}+\|F\|_{L^{\infty}(\mathcal{H}_{\rho}(z_2))})\leq C_0
\end{aligned}
\end{equation}
for some constants $c_B$ and $C_0$, where $\delta_0\coloneqq\frac{16\delta}{1+2\delta}$ and we write $\mathcal{H}'_{r}(z_1')$ as the kinetic cylinder defined in $\bbR^{2(n-1)+1}$ and ${H}_{r}$ as the stationary cylinder determined in \eqref{defn.staty.cylinder}. If $2r_0= {R}$, we have 
\begin{align}\label{est.bdep.bdry.hol}
    [f]_{C^{\frac12-\delta}(\mathcal{H}_{r_0}(z_0))}\leq c\langle v_0\rangle^{c_p}(\|f\|_{L^\infty(\mathcal{H}_{2r_0}(z_0))}+[g]_{C^{\frac34}(\gamma_-\cap\mathcal{Q}_{2r_0}(z_0))}+C_0),
\end{align}
where $c_p=c(p)$ and $c=c(n,\Lambda,\eps,\|A\|_{C^\eps(\mathcal{H}_R(z_0))},c_B,\delta,C_1,p)$.
\end{lemma}

To prove this, first we provide the following lemma explaining that the $L^\infty$-norm of the boundary datum can be estimated by its H\"older semi-norm.
\begin{lemma}\label{lem.rep.hol2}
Let $f$ be a weak solution to 
\begin{equation}
\left\{
\begin{alignedat}{3}
(\partial_t+v\cdot\nabla_x)f-\ddiv(A\nabla_vf)&=B\cdot\nabla_vf-\ddiv(G)+F&&\qquad \mbox{in  $\mathcal{H}_R(z_0)$}, \\
f&=g&&\qquad  \mbox{in $\gamma_-\cap\mathcal{Q}_R(z_0)$}
\end{alignedat} \right.
\end{equation}
with $z_0\in\gamma_-\cup\gamma_0$.
If $g\in C^{\eps}(\gamma_-\cap\mathcal{Q}_R(z_0))$ for some $\eps > 0$, then we have
 \begin{align*}
     \|g\|_{L^\infty(\gamma_-\cap\mathcal{Q}_R(z_0))}&\leq c(R^{-(4n+2)}\|f\|_{L^1(\mathcal{H}_R(z_0))}+\|RG\|_{L^\infty(\mathcal{H}_R(z_0))}+\|R^2F\|_{L^\infty(\mathcal{H}_R(z_0))})\\
     &\quad+cR^{\eps}[g]_{C^\eps(\gamma_-\cap\mathcal{Q}_R(z_0))},
 \end{align*}
 where $c=c(n,\Lambda,\eps,\|B\|_{L^\infty(\mathcal{H}_R(z_0))})$.
\end{lemma}
\begin{proof}
First, by \cite{Sil22}, $f$ is continuous up to the boundary, which implies $f(z_0)=g(z_0)$.
Thus, it suffices to show that 
\begin{align*}
    |f(z_0)|\leq c(R^{-(4n+2)}\|f\|_{L^1(\mathcal{H}_R(z_0))}+\|RG\|_{L^\infty(\mathcal{H}_R(z_0))}+\|R^2F\|_{L^\infty(\mathcal{H}_R(z_0))} + R^{\eps}[g]_{C^{\eps}(\gamma_-\cap\mathcal{Q}_{R}(z_0))}),
\end{align*}
where $c(n,\Lambda,\eps,\|B\|_{L^\infty(\mathcal{H}_R(z_0))})$, as  
\begin{align*}
    \|g-f(z_0)\|_{L^\infty(\gamma_-\cap\mathcal{Q}_R(z_0))}=\|g-g(z_0)\|_{L^\infty(\gamma_-\cap\mathcal{Q}_R(z_0))}\leq  cR^{\eps}[g]_{C^\eps(\gamma_-\cap\mathcal{Q}_R(z_0))}.
\end{align*}
    Suppose not, then there are sequences $(f_k)_k$, $(A_k)_k$, $(B_k)_k$, $(G_k)$, $(F_k)$, $(R_k)_k$, $(z_k)_k$, and $(g_k)_k$ with
    \begin{equation}
\left\{
\begin{alignedat}{3}
(\partial_t+v\cdot\nabla_x)f_k-\ddiv(A_k\nabla_vf_k)&=B_k\cdot\nabla_vf_k-\ddiv(G_k)+F_k&&\qquad \mbox{in  $\mathcal{H}_{R_k}(z_k)$}, \\
f_k&=g_k&&\qquad  \mbox{in $\gamma_-\cap\mathcal{Q}_{R_k}(z_k)$},
\end{alignedat} \right.
\end{equation}
and $\Lambda^{-1}I\leq A_k\leq \Lambda I$ and $\|B_k\|_{L^\infty(\mathcal{H}_{R_k}(z_k))}\leq \Lambda$,
but $|f_k(z_k)|=1$ and
\begin{align*}
   R_k^{-(4n+2)}\|f_k\|_{L^1(\mathcal{H}_{R_k}(z_k))}+\|R_kG_k\|_{L^\infty(\mathcal{H}_{R_k}(z_k))}+\|R_k^2F_k\|_{L^\infty(\mathcal{H}_{R_k}(z_k))}+[R_k^{\eps} g_k]_{C^{\eps}(\gamma_-\cap\mathcal{Q}_{R_k}(z_k))}\to 0.
\end{align*}
After simple modifications of the proof given in \cite[Lemma 2.25]{RoWe25} (see also \autoref{lem.bdry.hol}) to treat the right-hand side $-\ddiv(G_k)$, we deduce that
\begin{align*}
    R_k^\eps[f_k]_{C^\eps(\overline{\mathcal{H}_{R_k/2} (z_k)})}\leq c.
\end{align*}
We now consider the function $\widehat{f}_k\coloneqq f_k(z_k\circ S_{R_k}z)$ to see that 
\begin{align*}
    [\widehat{f}_k]_{C^\eps(\mathcal{H}_{1,k})}\leq c\quad\text{and}\quad \|\widehat{f}_k\|_{L^2(\mathcal{H}_{1,k})}\to0,
\end{align*}
where 
\begin{align*}
    \mathcal{H}_{1,k}\coloneqq \left\{z\,:\, z_k\circ S_{R_k}z\in \mathcal{H}_{R_k}(z_k)\right\}\supset  [0,1)\times B_1'\times[0,1]\times B_1.
\end{align*}
Indeed, we have used the fact that 
\begin{align*}
    |\mathcal{H}_{R_k}(z_k)|\eqsim R_k^{-(4n+2)}\quad\text{and}\quad |\mathcal{H}_{1,k}|\eqsim 1.
\end{align*}
By Arzel\`a-Ascoli's theorem, $\widehat{f}_k $ converges uniformly to $ \widehat{f}_\infty$ in $[0,1)\times (B_1'\times[0,1])\times B_1$ with 
\begin{align*}
    \|\widehat{f}_\infty\|_{L^2([0,1)\times (B_1'\times[0,1])\times B_1)}=0.
\end{align*}
This gives the contradiction $0 = |\widehat{f}_\infty(0)|=\lim\limits_{k\to\infty}|\widehat{f}_k(0)|=\lim\limits_{k\to\infty}|{f}_k(z_k)|=1$, completing the proof.
\end{proof}
Now we are ready to prove \autoref{lem.bdry.hol.bdep}.
\begin{proof}[Proof of \autoref{lem.bdry.hol.bdep}.]
The proof is similar to \autoref{lem.bdry.hol} except that here we want to get the optimal H\"older exponent. We may assume $z_0=0$. Let us fix $\delta<\frac12$.

First, we want to show that for any $\widetilde{z}_0\in\gamma_0\cap\mathcal{Q}_{r_0}(z_0)$ and $r\leq r_0$,
\begin{align}\label{first.goal.hol.bdep}
    \dashint_{\mathcal{H}_r(\widetilde{z}_0)}|f-(f)_{\mathcal{H}_r(\widetilde{z}_0)}|\,dz\leq c\langle v_0\rangle^{c_p} r^{\frac12-\delta} \big(\|f\|_{L^\infty(\mathcal{H}_{2r_0}(z_0))}+[g]_{C^{\frac34}(\gamma_-\cap\mathcal{Q}_{2r_0}(z_0))}+C_0 \big),
\end{align}
where $c=c(n,\Lambda,\eps,\|A\|_{C^\eps(\mathcal{H}_{2r_0}(z_0))},c_B,\delta,C_1,p)$. Let us fix $M_0=M_0(n,\Lambda,c_B)$, which will be determined later (see \eqref{cond.M0.bdep.bdry.hol}). Note that when $r>\frac{\langle v_0\rangle^{-(2+p)}}{M_0}$, then we observe
\begin{align*}
    \dashint_{\mathcal{H}_r(\widetilde{z}_0)}|f-(f)_{\mathcal{H}_r(\widetilde{z}_0)}|\,dz\leq cr^{-(4n+2)-\frac12+\delta}r^{\frac12-\delta}\|f\|_{L^\infty(\mathcal{H}_{2_{r_0}}(z_0))}\leq c\langle v_0\rangle^{-(4n+3)(2+p)}\|f\|_{L^\infty(\mathcal{H}_{2r_0}(z_0))},
\end{align*}
which gives \eqref{first.goal.hol.bdep} with $c_p>(4n+3)(2+p)$ and $c=c(n,M_0)$. Thus, it suffices to prove \eqref{first.goal.hol.bdep} with $r\leq \frac{\langle v_0\rangle^{-(2+p)}}{M_0}$. 

Now let $h$ be a unique weak solution to
\begin{equation*}
\left\{
\begin{alignedat}{3}
(\partial_t+v\cdot\nabla_x)h-\ddiv(A\nabla_vh)&=B\cdot\nabla_vf&&\qquad \mbox{in  $\mathcal{H}_r(\widetilde{z}_0)$}, \\
h&=f&&\qquad  \mbox{in $\partial_{\mathrm{kin}}\mathcal{Q}_R(\widetilde{z}_0)$.}
\end{alignedat} \right.
\end{equation*}
By following the same lines as in the proof of \autoref{lem.bdry.hol} together with the second condition given in \eqref{ass.bdep.bdry.hol}, we deduce
\begin{align*}
    \dashint_{\mathcal{H}_r(\widetilde{z}_0)}|f-h|^2\,dz\leq c \Big(\|B\|^2_{L^\infty(\mathcal{H}_r(\widetilde{z}_0)}r^2\dashint_{\mathcal{H}_r(\widetilde{z}_0)}|f-h|^2\,dz+C_0^2r^{2-2\delta} \Big).
\end{align*}
By taking $M_0=M_0(n,\Lambda,c_B)$ large enough, we have
\begin{align}\label{cond.M0.bdep.bdry.hol}
    r(1+\|B\|_{L^\infty(\mathcal{H}_{r}(\widetilde{z}_0))})\leq \frac{\langle v_0\rangle^{-2}}{M_0}(c_B\langle v_0\rangle^2+1)\leq\frac1{2c},
\end{align}
which implies
\begin{align*}
    \dashint_{\mathcal{H}_r(\widetilde{z}_0)}|f-h|^2\,dz\leq cC_0^2r^{2-2\delta}.
\end{align*}
By a combination of \autoref{thm.hol12}, the second estimates of (2.33) in \cite[Lemma 2.25]{RoWe25} and \autoref{lem.rep.hol2} together with the scaling argument as in \eqref{scaling.argument}, we obtain
\begin{align*}
    r^{\frac12}[h]_{C^{\frac12-\frac{\delta}2}(\mathcal{H}_{\frac{r}2}(\widetilde{z}_0))}&\leq c \big( r^{-(2n+1)}\|h\|_{L^2(\mathcal{H}_{r}(\widetilde{z}_0))}+[r^{\frac34}g]_{C^{\frac34}(\gamma_-\cap\mathcal{Q}_r(\widetilde{z}_0))} \big),
\end{align*}
where $c=c(n,\Lambda,\eps,\|A\|_{C^\eps(\mathcal{H}_{r_0}(z_0))})$.
In particular, for the scaling argument, we have used \eqref{scaling.gamma0} by \eqref{cond.M0.bdep.bdry.hol}.
Thus, as in the proof of \autoref{lem.bdry.hol}, we deduce
\begin{equation*}
\begin{aligned}
    \dashint_{\mathcal{H}_{\rho r}(\widetilde{z}_0)}|f-(f)_{\mathcal{H}_{\rho r}(\widetilde{z}_0)}|^2\,dz&\leq c\rho^{1-{\delta}}\dashint_{\mathcal{H}_{r}(\widetilde{z}_0)}|f-(f)_{\mathcal{H}_{r}(\widetilde{z}_0)}|^2\,dz+cr^{\frac34}[g]_{C^{\frac34}(\gamma_-\cap\mathcal{Q}_r(\widetilde{z}_0))}+cr^{1-2\delta}C_0, 
\end{aligned}
\end{equation*}
where $c=c(n,\Lambda,\eps,\|A\|_{C^\eps(\mathcal{H}_{r_0}(z_0))},\delta)$. Now taking $\rho$ sufficiently small so that $\rho^{{\delta}}\leq \frac{1}{16c}$ and following the standard iterative argument as in \cite[Lemma 5.6]{KLN25}, we have \eqref{first.goal.hol.bdep}. 

Now we prove \eqref{first.goal.hol.bdep} with $\widetilde{z}_0$ replaced by $z_1\in \mathcal{H}_{r_0}(z_0)$. Similarly, we assume $r\leq \frac{\langle v_0\rangle^{-2-p}}{M_0}$. First, suppose  $\mathcal{H}_{2r}(z_1)\cap\gamma_0\neq\emptyset$. Then there is a point $\widetilde{z}_1\in \mathcal{H}_{2r}(z_1)\cap\gamma_0$ such that $\mathcal{H}_{2r}(z_1)\subset\mathcal{H}_{8r}(\widetilde{z}_1)$. Thus, using this, we also get \eqref{first.goal.hol.bdep}. Now we assume $\mathcal{H}_{2r}(z_1)\cap\gamma_0=\emptyset$. Then we have for $\rho_1\coloneqq \frac1{64}\max\{{(x_{1,n})^{\frac13}},|v_{1,n}|\}$, $r\leq 64\rho_1$ and $\mathcal{H}_{\rho_1}(z_1)\cap\gamma_0=\emptyset$. By \autoref{rmk.c12.bdep} together with the scaling arguments as in \eqref{scaling.gamma0} by \eqref{cond.M0.bdep.bdry.hol}, which allows us to remove the dependence of $\|B\|_{L^\infty}$, we have 
\begin{align*}
    J\coloneqq\dashint_{\mathcal{H}_r(z_1)}|f-(f)_{\mathcal{H}_r(z_1)}|^2\,dz&\leq c\left(\frac{r}{\rho_1}\right)^{1-2\delta}\Bigg(\dashint_{\mathcal{H}_{\rho_1}(z_1)}|f-(f)_{\mathcal{H}_{\rho_1}(z_1)}|^2\,dz+\rho_1^2\|G\|^2_{L^\infty(\mathcal{H}_{\rho_1}(z_1)}\\
    &\quad\qquad\qquad\qquad+\rho_1^4\|F\|^2_{L^\infty(\mathcal{H}_{\rho_1}(z_1)}+\rho_1[g]^2_{C^{\frac12}(\gamma_-\cap\mathcal{H}_{\rho_1}(z_1))}\Bigg),
\end{align*}
where $c=c(n,\Lambda,\eps,\|A\|_{C^{\eps}(\mathcal{H}_{r_0}(z_0))})$, whenever $r\leq \rho_1/2$. Now using the third condition given in \eqref{ass.bdep.bdry.hol} together with $\mathcal{H}_{\rho_1}(z_1)\cap\gamma_0=\emptyset$, we further estimate $J$ as
\begin{equation*}
    \begin{aligned}
        J\leq c\left(\left(\frac{r}{\rho_1}\right)^{1-2\delta}\dashint_{\mathcal{H}_{\rho_1}(z_1)}|f-(f)_{\mathcal{H}_{\rho_1}(z_1)}|^2\,dz+r^{1-2\delta}(C_0+[g]^2_{C^{\frac12}(\gamma_-\cap\mathcal{H}_{r_0}(z_0))})\right).
    \end{aligned}
\end{equation*}
When $r\in[\rho_1/2,64\rho_1]$, we have $J\leq c\dashint_{\mathcal{H}_{64\rho_1}(z_1)}|f-(f)_{\mathcal{H}_{64\rho_1}(z_1)}|^2\,dz$, which gives
\begin{equation*}
    \begin{aligned}
        J\leq c\left(\left(\frac{r}{\rho_1}\right)^{1-2\delta}\dashint_{\mathcal{H}_{64\rho_1}(z_1)}|f-(f)_{\mathcal{H}_{64\rho_1}(z_1)}|^2\,dz+r^{1-2\delta}((C_0)^2+[g]^2_{C^{\frac12}(\gamma_-\cap\mathcal{H}_{r_0}(z_0))})\right).
    \end{aligned}
\end{equation*}
Since $\mathcal{H}_{64\rho_1}(z_1)\cap\gamma_0\neq\emptyset$, there is a point $\widetilde{z}_1\in\gamma_0$ such that $\mathcal{H}_{64\rho_1}(z_1)\subset \mathcal{H}_{2^{9}\rho_1}(\widetilde{z}_1)$. Now we plug \eqref{first.goal.hol.bdep} with $\widetilde{z}_0$ and $r$ replaced by $\widetilde{z}_1$ and $64\rho_1$, respectively, into the estimate of $J$ so that we derive
\begin{equation*}
    \begin{aligned}
        J\leq c\langle v_0\rangle^{2c_p} r^{1-2\delta}(\|f\|^2_{L^\infty(\mathcal{H}_{2r_0}(z_0))}+[g]^2_{C^{\frac34}(\gamma_-\cap\mathcal{Q}_{2r_0}(z_0))}+(C_0)^2).
    \end{aligned}
\end{equation*}
Therefore, we have proved \eqref{first.goal.hol.bdep} for any $z\in\gamma_0\cap\mathcal{Q}_{r_0}(z_0)$, which implies the desired $C^{\frac12}$-estimate.
\end{proof}

\subsection{Regularity of the diffuse reflection boundary datum}

In this subsection, we prove \autoref{lem.reg.nf.diff}.
First, we introduce some notation, which will be used frequently.
\begin{align}\label{defn.r0rho0}
    \langle v_0\rangle\coloneqq 1+|v_0|,\quad r_0\coloneqq \frac{\langle v_0\rangle^{-p} }{64},\quad \rho_0\coloneqq \max\left\{\frac{|v_{0,n}|}{2},\left(\frac{|x_{0,n}|}{2}\right)^{\frac13}\right\},
\end{align}
where the constant $p\geq \frac12$ is fixed in \autoref{lem.reg.nf.diff}.

\begin{proof}[Proof of \autoref{lem.reg.nf.diff}.]
First, we prove regularity of the time derivative of $\cN f$. Recall the definition of $g = g(t,x')$ in \eqref{defn.nf}.
For any $h\in \bbR$ and $k\in \N\cup\{0\}$, we write 
\begin{align}\label{defn.diff.quot}
\begin{split}
    \delta_{h,1}f(z) &\coloneqq \delta_hf(z) \coloneqq f_h(z)-f(z),\quad f_h(z) \coloneqq f( (h,0,0)\circ z) = f(t + h , x,v), \\
    \delta_{h,k}f &\coloneqq  \delta_{h,k-1}\circ \delta_hf,\qquad\qquad\qquad\quad \widehat{g}_k\coloneqq \frac{\delta_{h,k}g}{|h|^{k}}.
    \end{split}
\end{align}
It is straightforward to check that $f_h$ is also a weak solution. We assume 
\begin{align}\label{m.ass.diff}
    \sup_{z_0\in\mathcal{H}_{1,\infty}}\langle v_0\rangle^{M}\sum_{i=0}^k\left\|\partial_t^i\mathcal{M}\right\|_{C^{\frac76}(\gamma_-\cap\mathcal{Q}_{32r_0}(z_0))}\leq 1
\end{align}
for some large constant $M$. Let us now fix $k \in \N \cup \{0\}$. We will prove that for any $z_0\in\mathcal{H}_{1,\infty}$, $r\in(0,r_0]$, $\eps\in(0,1)$, and $|h|\leq 2^{-100k}\eqqcolon h_0$, if $M$ is sufficiently large, then it holds
\begin{align}\label{goal.hol.diffuse}
    \|\widehat{g}_k-a_r\|_{L^\infty(\gamma_-\cap\mathcal{Q}_r(z_0))}&\leq cr^{\frac76}\langle v_0\rangle^{c_k}D_k,
\end{align}
where $c_k=c_k(n,k,p)$ and $c=c(n,\Lambda,k,C_k)$, 
\begin{equation}\label{defn.Mk.diff}
\begin{aligned}
    &a_r\coloneqq \int_{\bbR^{n-1}}\int_{w_n<-16r-r_{w'}}\frac{\delta_{h,k}f}{|h|^{k}}(t_0,x_0',0,w',w_n)(w_n)_-\,dw_n\,dw',\\
    &C_k\coloneqq \sup_{z_0\in \mathcal{H}_{\frac32,\infty}}\left\|{A}\right\|_{C^{3k+1}(\mathcal{H}_{32r_0}(z_0))}+\frac{1}{\langle v_0\rangle^2}\left\|B\right\|_{C^{3k}(\mathcal{H}_{32r_0}(z_0))},\\
    &D_k\coloneqq \|\langle \cdot\rangle^{c_k}f(t,x,\cdot)\|_{L^\infty(\mathcal{H}_{2,\infty})}+\left\|\langle\cdot\rangle^{c_k}{F(t,x,\cdot)}\right\|_{C^{3k}(\mathcal{H}_{2,\infty})}
\end{aligned}
\end{equation}
with $r_{w'}\coloneqq 16(r^2|w'-v_0'|)^{\frac13}$, where the space $\|\langle \cdot\rangle^{c_k}F(t,x,\cdot)\|_{C^{3k}}$ is defined in \eqref{space.hol.weight}. The choice of $r_{w'}$ will be required in order to have \eqref{rel.diff.cylinder} and we note for any $j\leq k$ and $\eps\in[0,1]$,
\begin{align}\label{partialt}
    \|{\partial_t^jf}\|_{C^{\eps}(\mathcal{H}_{r}(z_0))}\leq c(k,j,\eps)\|f\|_{C^{3k+\eps}(\mathcal{H}_{r}(z_0))}.
\end{align}

Note that once \eqref{goal.hol.diffuse} is established, the proof can be concluded easily (see Step 4). We now divide the proof of \eqref{goal.hol.diffuse} into three steps.

\textbf{Step 1. Estimates of higher order different quotients of the boundary data.} First, we are going to prove
\begin{equation}\label{widehatg.hol}
\begin{aligned}
    &|h|^{k}\|\widehat{g}_k-a_r\|_{L^\infty(\gamma_-\cap\mathcal{Q}_{r}(z_0))}\\
    &\leq cr^{\frac43}\langle v_0\rangle \|\langle\cdot\rangle^4\delta_{h,k}f(t,x',0,\cdot)\|_{L^\infty(\bbR^n)}\\
    &\quad+cr^{\frac{2(2-\eps)}{3}}\langle v_0\rangle\sup_{(t,x')\in \mathbf{B}_{\frac32}}\underbrace{\int_{ \frac{1}{{\langle w\rangle}^{p}}\leq |w_n|}\int_{{w_n<-64r
   }}\langle w\rangle^5[\delta_{h,k}f]_{C^{2-\eps}(\mathcal{H}_{\frac{1}{64{\langle w\rangle}^p}}(t,x',0,w))}(w_n)_-\,dw}_{\eqqcolon L_1}\\
    &\quad+ cr^{\frac{2(2-\eps)}{3}}\langle v_0\rangle\sup_{(t,x')\in \mathbf{B}_{\frac32}}\underbrace{\int_{\frac{1}{{\langle w\rangle}^{p}}\geq |w_n|}\int_{{w_n<-64r
    }}\langle w\rangle^5[\delta_{h,k}f]_{C^{2-\eps}(\mathcal{H}_{\frac{|w_n|}{64}}(t,x',0,w))}(w_n)_-\,dw}_{\eqqcolon L_2},
\end{aligned}
\end{equation}
where $ \mathbf{B}_{\frac32}\coloneqq I_{(\frac32)^2}\times B'_{(\frac32)^3}$ and $c=c(n,\eps,p)$. Let us fix $(t,x',0,v)\in \gamma_-\cap\mathcal{H}_r(z_0)\subset \mathcal{H}_{1,\infty}$ to observe
\begin{align*}
    |h|^k|\widehat{g}_k(t,x',0,v)-a_r|&\leq \left|\int_{w_n<-64r-r_{w'}}\left(\delta_{h,k}f(t,x',0,w',w_n)-\delta_{h,k}f(t_0,x_0',0,w',w_n)\right)(w_n)_-\,dw\right|\\
    &\quad+\left|\int_{0>w_n>-64r-r_{w'}}\delta_{h,k}f(t,x',0,w',w_n)(w_n)_-\,dw\right|\eqqcolon J_1+J_2.
\end{align*}
To estimate $J_1$, we investigate the term $J$ defined by
\begin{align*}
    J\coloneqq |\delta_{h,k}f(t,x',0,w',w_n)-\delta_{h,k}f(t_0,x_0',0,w',w_n)|.
\end{align*}
Now we split into two cases, as $\mathcal{H}_{\frac{|w_n|}{2}}(t_0,x_0',0,w)\not\subset \mathcal{H}_{\frac32,\infty} $ when $|w_n|\geq \frac{1}{\sqrt{\langle w\rangle}}$ and we want to get H\"older estimates on the cylinder $\mathcal{H}_{r}(z_1)$, where $r\lesssim \langle v_1\rangle ^{-p}$.
\begin{itemize}
\item Assume $\frac{1}{{\langle w\rangle}^p}\geq{|w_n|}$ and note
\begin{align*}
    J
    &\leq |\delta_{h,k}f(t,x',0,w',w_n)-\delta_{h,k}f(t_0,x'+(t_0-t)w',(t_0-t)w_n,w',w_n)|\\
    &\quad+ |\delta_{h,k}f(t_0,x'+(t_0-t)w',(t_0-t)w_n,w',w_n)-\delta_{h,k}f(t_0,x_0',0,w',w_n)|.
\end{align*}
By the choice of $r_{w'}=16(r^2|w'-v_0'|)^{\frac13}$ and $\frac{1}{{\langle w\rangle}^p}\geq{|w_n|}$, we have 
\begin{align}\label{rel.diff.cylinder}
    &(t_0,x'+(t_0-t)w',(t_0-t)w_n,w)\in \mathcal{H}_{R_0}(t,x',0,w)\subset \mathcal{H}_{\frac32,\infty},\nonumber\\
&(t_0,x'+(t_0-t)w',(t_0-t)w_n,w)\in \mathcal{H}_{R_0}(t_0,x_0',0,w)\subset \mathcal{H}_{\frac32,\infty},
\end{align}
where $R_0\coloneqq \frac{|w_n|}{64}$. Therefore, we estimate $J$ as
\begin{equation}\label{estj1.diff}
\begin{aligned}
    J&\leq c(r^{2-\eps}+|x'-x_0'+(t_0-t)w'|^{\frac{2-\eps}3}+|(t_0-t)w_n|^{\frac{2-\eps}3})\sup_{(t,x')\in\mathbf{B}_{\frac32}}[\delta_{h,k}f]_{C^{2-\eps}(\mathcal{H}_{R_0}(t,x',0,w))}\\
    &\leq cr^{\frac{2(2-\eps)}3}\langle v_0\rangle\langle w\rangle\sup_{(t,x')\in\mathbf{B}_{\frac32}}[\delta_{h,k}f]_{C^{2-\eps}(\mathcal{H}_{R_0}(t,x',0,w))},
\end{aligned}
\end{equation}
where we have also used the fact that $(t,x',0,v)\in \mathcal{H}_r(z_0)$, which implies
\begin{equation*}
\begin{aligned}
    |x'-(x_0'+(t_0-t)w')|\leq |x'-x_0'+(t_0-t)v_0'|+|(t_0-t)(v_0'-w')|&\leq r^3+r^2(\langle v_0\rangle+\langle w\rangle)\\
    &\leq c r^2\langle v_0\rangle\langle w\rangle,
\end{aligned}
\end{equation*}
    \item Assume $\frac{1}{{\langle w\rangle}^p}\leq{|w_n|}$. Then we write
    \begin{align*}
    J&\leq |\delta_{h,k}f(t,x',0,w',w_n)-\delta_{h,k}f(t_0,x',0,w',w_n)|+ |\delta_{h,k}f(t_0,x',0,w',w_n)-\delta_{h,k}f(t_0,x_0',0,w',w_n)|\\
    &\eqqcolon J_{1,1}+J_{1,2}.
\end{align*}
We may assume $t_0<t$. Now let us choose $R_0\coloneqq \frac{1}{64{\langle w\rangle}^p}$ and define a sequence $t_i\coloneqq t_0+(16)^2(i-1)(R_0)^5$ to see that
\begin{align*}
    (R_0)^3>(16)^2(R_0)^5\langle w\rangle\quad\text{and}\quad R_0^2\langle w\rangle\leq \frac1{64}.
\end{align*}
Thus, there is a positive integer $m$ such that $t\in(t_{m-1},t_m]$ and
\begin{align*}
    (t_{i-1},x',0,w',w_n)\in \mathcal{H}_{R_0}(t_i,x',0,w)\subset \mathcal{H}_{\frac32,\infty},\quad (16)^2(m-1)(R_0)^5<|t-t_0|.
\end{align*}
Using these, we can estimate $J_{1,1}$ as
\begin{align*}
    J_{1,1}&\leq \sum_{i=1}^{m}|\delta_{h,k}f(t_i,x',0,w)-\delta_{h,k}f(t_{i-1},x',0,w)|\\
    &\leq c\sum_{i=1}^{m}\left(|t_i-t_{i-1}|^{\frac{2-\eps}2}+|(t_i-t_{i-1})w|^{\frac{2-\eps}{3}}\right)[\delta_{h,k}f]_{C^{2-\eps}(\mathcal{H}_{R_0}(t_{i},x',0,w))}\\
    &\leq c\sum_{i=1}^{m}|r^2\langle w\rangle|^{\frac{2(1-\eps)}{3}}[\delta_{h,k}f]_{C^{2-\eps}(\mathcal{H}_{R_0}(t_{i},x',0,w))}\\
    &\leq cm\sup_{i\in[1,m]}r^{\frac{2(2-\eps)}3}\langle w\rangle^{\frac23}[\delta_{h,k}f]_{C^{2-\eps}(\mathcal{H}_{R_0}(t_{i},x',0,w))}\\
    &\leq c\sup_{(t,x')\in\mathbf{B}_{\frac32}}r^{\frac{2(2-\eps)}3}\langle w\rangle^{1+5p}[\delta_{h,k}f]_{C^{2-\eps}(\mathcal{H}_{R_0}(t,x',0,w))}.
\end{align*}
Similarly, by choosing $x_{0,i}'=x_{0}'+(i-1)(R_0)^3$, we estimate $J_{2,2}$ as 
\begin{align*}
    J_{1,2}&\leq \sum_{i=1}^{m}|\delta_{h,k}f(t_0,x_i',0,w)-\delta_{h,k}f(t_0,x_{i-1}',0,w)|\\
     &\leq c\sum_{i=1}^{m}|x'-x'_{0}|^{\frac{2(1-\eps)}3}[\delta_{h,k}f]_{C^{2-\eps}(\mathcal{H}_{R_0}(t_0,x'_{i-1},0,w))}\\
    &\leq cr^{\frac{2(2-\eps)}3}\langle v_0\rangle^{\frac23}\langle w\rangle^{3p}\sup_{(t,x')\in \mathbf{B}_{\frac32}}[\delta_{h,k}f]_{C^{2-\eps}(\mathcal{H}_{R_0}(t,x',0,w))},
\end{align*}
where $m$ is the positive integer such that $x'\in [x_{m-1}',x_{m}']$ and we have used the fact that $|x'-x_0'|\leq r^3+|(t-t_0)v_0'|$ by $(t,x',0,v)\in \mathcal{H}_r(z_0)$. Combining the estimates $J_{1,1}$ and $J_{1,2}$ gives
\begin{align}\label{estj2.diff}
    J\leq cr^{\frac{2(2-\eps)}3}\langle v_0\rangle \langle w\rangle^{10p}\sup_{(t,x')\in \mathbf{B}_{\frac32}}[\delta_{h,k}f]_{C^{2-\eps}(\mathcal{H}_{R_0}(t,x',0,w))}.
\end{align}
\end{itemize}
Next, we estimate $J_1$ as
\begin{align*}
    J_2&\leq \int\int^0_{-16r-r_{w'}}\sup_{w_n<0}\left|\delta_{h,k}f(t,x',0,w',w_n)\right||w_n|\,dw_n\,dw'\\
    &\leq cr^{\frac43}\int_{\bbR^{n-1}}\sup_{w_n<0}\left|\delta_{h,k}f(t,x',0,w',w_n)\right|(1+|w'-v_0'|^{\frac23})\,dw'\\
    &\leq cr^{\frac43}\langle v_0'\rangle ^{\frac23}\int_{\bbR^{n-1}}\sup_{w_n<0}\left|\delta_{h,k}f(t,x',0,w',w_n)\right|\langle w'\rangle^{4} \langle w'\rangle^{-\frac{10}3}\,dw'\\
    &\leq cr^{\frac43}\langle v_0\rangle\|\langle\cdot\rangle^4\delta_{h,k}f(t,x',0,\cdot)\|_{L^\infty(\bbR^n)}.
\end{align*}
Combining the previous estimate with \eqref{estj1.diff} and \eqref{estj2.diff} yields \eqref{widehatg.hol}, as desired.

\textbf{Step 2. Regularity of higher order different quotients of the solution.} 
By taking the different quotient operator to \eqref{sub.eq.diffuse}, we get for any $l\leq k$,
\begin{equation}\label{diff1.eq.diffuse.ind}
\left\{
\begin{alignedat}{3}
(\partial_t+v\cdot\nabla_x)\delta_{h,l}f-\ddiv(A_{lh}\nabla_v\delta_{h,l}f)&=B_{lh}\cdot\nabla_v\delta_{h,l}f+\delta_{h,l}F+G_l&&\quad \mbox{in  $\{x_n>0\}\cap \mathcal{H}_{\frac74,\infty}$}, \\
\delta_{h,l}f&=\delta_{h,l}\mathcal{N}f&&\quad  \mbox{in $\gamma_-$},
\end{alignedat} \right.
\end{equation}
where 
\begin{align*}
    G_l\coloneqq \sum_{i=0}^{l}\ddiv(\delta_{h,l-i}A_{ih}\nabla_v\delta_{h,l}f)+\delta_{h,l-i}B_{ih}\cdot\nabla_v\delta_{h,i}f.
\end{align*}
Then by \autoref{lem.gra.bdep}, we have for any $\mathcal{H}_{2\rho}(z_1)\cap\gamma_0=\emptyset$ with $\rho\leq r_1 =\frac{\langle v_1\rangle^p}{16}$, $x_{1,n}=0$, and $v_{1,n}<0$, and $l\in \{0,\dots,k\}$,
\begin{equation*}
\begin{aligned}
    &\rho^{\eps-1}\|\nabla_v\delta_{h,l}f\|_{L^\infty(\mathcal{H}_\rho(z_1))}+[\delta_{h,l}f]_{C^{2-\eps}(\mathcal{H}_\rho(z_1))}\\
    &\leq 
    c\langle v_1\rangle^{2}(\rho^{-2+\eps}\|\delta_{h,l}f\|_{L^\infty(\mathcal{H}_{2\rho}(z_1))}+\|\delta_{h,l}F\|_{L^\infty(\mathcal{H}_{2\rho}(z_1))})\\
    &\quad+c\langle v_1\rangle^{2}\sum_{i=0}^{l-1}(\|\delta_{h,l-i} A_{ih}\|_{C^{1}(\mathcal{H}_{2\rho}(z_1))}+\|\delta_{h,l-i} B_{ih}\|_{L^{\infty}(\mathcal{H}_{2\rho}(z_1))})\|\nabla_v\delta_{h,i}f\|_{C^{1-\eps}(\mathcal{H}_{2\rho}(z_1))},
\end{aligned}
\end{equation*}
where $c=c(n,\eps,\|A_{lh}\|_{C^{1-\eps}(\mathcal{H}_{2\rho}(z_1))},c_B)$. Therefore, by iterating this estimate $l-1$ times, we deduce
\begin{align}\label{c2eps.ind.diff}
    \left[\frac{\delta_{h,l}f}{|h|^l}\right]_{C^{2-\eps}(\mathcal{H}_\rho(z_1))}&\leq c\langle v_1\rangle^{4(l+1)}\sum_{i=0}^{l}\left(\rho^{-2+\eps} \left\|\frac{\delta_{h,i}f}{|h|^i}\right\|_{L^\infty(\mathcal{H}_{2\rho}(z_1))}+\left\|\frac{\delta_{h,i}F}{|h|^i}\right\|_{L^\infty(\mathcal{H}_{2\rho}(z_1))}\right),
\end{align}
where $c=c(n,l,\eps,c_B,C_l)$ and the constant $C_l$ is determined in \eqref{defn.Mk.diff}. We will use this estimate on the terms $L_1,L_2$ in \eqref{widehatg.hol}. 

\textbf{Step 3. Induction.}  Now, based on an inductive argument, we will prove that for any $j\in\{0,1,\ldots , k\}$ and $\delta\in(0,\frac12)$, there is a constant $N_j=N_j(n,j,p)$ such that for any $z_2\in \mathcal{H}_{\frac32,\infty}$ and $\rho\leq {r_2}=\frac{1}{64{\langle v_2\rangle}^p}$,
\begin{equation}\label{ind.level}
\begin{aligned}
&\left[\frac{\delta_{h,j}f}{|h|^j}\right]_{C^{\frac12-\delta}(\mathcal{H}_{2\rho}(z_2))}+\left\|\frac{\delta_{h,j}f}{|h|^j}\right\|_{L^\infty(\mathcal{H}_{2\rho}(z_2))}\\
&\leq c\langle v_2\rangle^{N_j}\left(\|f\|_{L^\infty(\mathcal{H}_{4r_2}(t,x_2,v_2))}+\left\|{F}\right\|_{C^{3j}(\mathcal{H}_{4r_2}(t,x_2,v_2))}+D_j\left\|{\mathcal{M}}\right\|_{C^{3_j+\frac76}(\gamma_-\cap\mathcal{Q}_{4r_2}(t,x_2,v_2))}\right),
\end{aligned}
\end{equation}
where $c=c(n,\delta,j,C_j,p)$, and $C_j,D_j$ are determined in \eqref{defn.Mk.diff}, provided that the constant $M$ given in \eqref{m.ass.diff} is sufficiently large, which will be used in \eqref{l1.mcon.used} and \eqref{l2.mcon.used}.  Note that we just choose $\frac76\in(1,\frac43)$ for convenience. Once \eqref{ind.level} is proved, we can use it to estimate the first term in \eqref{widehatg.hol}.

Clearly, for $j=0$, combining \eqref{c2eps.ind.diff} and \eqref{widehatg.hol}, we have that when $A\in C^{2(\eps-\frac12)}$ for $\eps>\frac12$, then $g\in C^{{\frac{4\eps}3}-\delta}$ for any $\delta>0$. Therefore, we get \eqref{second.goal.diff} and then apply \autoref{thm.hol12} to get \eqref{ind.level} with $j=0$.

Now, under the assumption \eqref{ind.level} for any $j=0,\ldots, l$, we will prove \eqref{ind.level} with $j=l+1$. 

\textbf{Step 3-1. Higher regularity for \mbox{$\frac{\delta_{h,{l+1}f}}{|h|^{l+\alpha}}$ for $\alpha<\frac16$ away from $\gamma_0$}.} Let us fix $\alpha<\frac16$. Then using \eqref{ind.level}, \eqref{hol.rel.kin}, and \eqref{hol.diff.rel}, we derive for any $z_2\in \mathcal{H}_{1,\infty}$
\begin{equation}\label{ind.level.alpha}
\begin{aligned}
    \left\|\frac{\delta_{h,l+1}f}{|h|^{l+\alpha}}\right\|_{L^\infty(\mathcal{H}_{2r_2}(z_2))}&\leq c\langle v_2\rangle^{c_l}\left(\|f\|_{L^\infty(\mathcal{H}_{4r_2}(t,x_2,v_2))}+\left\|{F}\right\|_{C^{3l}(\mathcal{H}_{4r_2}(t,x_2,v_2))}\right)\\
&\quad+cD_l\langle v_2\rangle^{c_l}\left\|{\mathcal{M}}\right\|_{C^{3l+\frac76}(\gamma_-\cap\mathcal{Q}_{4r_2}(t,x_2,v_2))},
\end{aligned}
\end{equation}
 where $c=c(n,\alpha,l,C_l,p)$, $c_l=c_l(n,l,p)$.

Using this and \eqref{c2eps.ind.diff} with $\frac{\delta_{h,l}f}{|h|^l}$ replaced by $\frac{\delta_{h,l+1}f}{|h|^{l+\alpha}}$, we get for any $z_1\in \mathcal{H}_{1,\infty}$ with $\mathcal{H}_{2\rho}(z_1)\cap\gamma_0=\emptyset$ and $\rho\leq r_1$, $x_{1,n}=0$ and $v_{1,n}<0$,
\begin{equation}\label{ind.c2eps}
\begin{aligned}
    \left[\frac{\delta_{h,l+1}f}{|h|^{l+\alpha}}\right]_{C^{2-\eps}(\mathcal{H}_\rho(z_1))}&\leq c\rho^{-2+\eps}\langle v_1\rangle^{c_l}\left(\|f\|_{L^\infty(\mathcal{H}_{4r_1}(t,x_1,v_1))}+\left\|{F}\right\|_{C^{3(l+1)}(\mathcal{H}_{4r_1}(t,x_1,v_1))}\right)\\
    &\quad+c\rho^{-2+\eps}D_l\langle v_1\rangle^{c_l}\left\|{\mathcal{M}}\right\|_{C^{3l+\frac76}(\gamma_-\cap\mathcal{Q}_{4r_1}(t,x_1,v_1))},
\end{aligned}
\end{equation}
where  $c=c(n,\alpha,l,\eps,c_B,C_l,p)$ and $r_1=\frac{\langle v_1\rangle^{-p}}{64}$. Indeed, we have also used some covering argument to get the radius $4r_1$ for the cylinders given in the right-hand side of \eqref{ind.c2eps}. Now plugging this with $z_1$ and $\rho$ replaced by $(t,x',0,w)$ and $\frac{1}{64{\langle w\rangle }^p}\eqqcolon \widetilde{r}$ into ${L_1}$ with $k=l+1$ given in \eqref{widehatg.hol} leads to 
\begin{equation}\label{l1.mcon.used}
    \begin{aligned}
        \frac{L_1}{|h|^{l+\alpha}}&\leq \int_{\frac1{\sqrt{w}} \leq|w_n|}\int_{w_n<-64r}\langle w\rangle^{c_l+12p}\left(\|f\|_{L^\infty(\mathcal{H}_{4\widetilde{r}}(\widetilde{z}))}+\left\|{F}\right\|_{C^{3(l+1)}(\mathcal{H}_{4\widetilde{r}}(\widetilde{z}))}\right)\,dw\\
        &\quad+cD_l\int_{\frac1{\sqrt{w}} \leq|w_n|}\int_{w_n<-64r}\langle w\rangle^{c_l+12p}\left\|{\mathcal{M}}\right\|_{C^{3l+\frac76}(\gamma_-\cap\mathcal{Q}_{4\widetilde{r}}(\widetilde{z}))}\,dw\\
        &\leq cD_{l+1},
    \end{aligned}
\end{equation}
where $\widetilde{z}=(t,x',0,w)$, we choose $M=M(n,l,p)$ large enough, recalling the definition of $D_l$ given in \eqref{m.ass.diff} and using \eqref{defn.Mk.diff}. 
Similarly, using \eqref{ind.c2eps} with $\rho$ and $z_1$ replaced by $\frac{|w_n|}{64}$ and $\widetilde{z}\coloneqq(t,x',0,w)$, we estimate $\frac{L_2}{|h|^{l+\alpha}}$ as 
\begin{equation}\label{l2.mcon.used}
    \begin{aligned}
        \frac{L_2}{|h|^{l+\alpha}}&\leq \int_{\frac1{\sqrt{w}} \geq|w_n|}\int_{w_n<-64r}\langle w\rangle^{10p+c_l}\left(\|f\|_{L^\infty(\mathcal{H}_{4\widetilde{r}}(\widetilde{z}))}+\left\|{F}\right\|_{C^{3(l+1)}(\mathcal{H}_{4\widetilde{r}}(\widetilde{z}))}\right)(w_n)_-^{-1+\eps}\,dw\\
        &\quad+cD_l\int_{\frac1{\sqrt{w}} \geq|w_n|}\int_{w_n<-64r}\langle w\rangle^{10p+c_l}\left\|{\mathcal{M}}\right\|_{C^{3l+\frac76}(\gamma_-\cap\mathcal{Q}_{4\widetilde{r}}(\widetilde{z}))}(w_n)_-^{-1+\eps}\,dw\\
        &\leq cD_{l+1}\int_{w_n<-64r}\langle w'\rangle^{-2}(w_n)_{-}^{-1-\eps}\,dw_n\,dw'\\
        &\leq cD_{l+1}r^{-\eps},
    \end{aligned}
\end{equation}
where $c=c(n,\alpha,l,\eps,c_B,C_l,p,\eps)$. Since \eqref{l1.mcon.used} and \eqref{l2.mcon.used} are true for any $(t,x')\in\mathbf{B}_{\frac32}$, combining these estimates, \eqref{ind.level.alpha} and  \eqref{widehatg.hol} with $k=l+1$ and $\eps=\frac1{10}$ leads to
\begin{equation}\label{boundary.ind0}
    \begin{aligned}
        \left\|\frac{\delta_{h,l+1}g}{|h|^{l+\alpha}}\right\|_{C^{\frac76}(\gamma_-\cap\mathcal{Q}_{r_0}(z_0))}\leq cD_{l+1},
    \end{aligned}
\end{equation}
where $c=c(n,\alpha,l,\eps,c_B,C_l,p)$.
Note that inductively, we also get 
\begin{equation}\label{boundary.ind}
    \begin{aligned}
        \left\|\frac{\delta_{h,j}g}{|h|^{j}}\right\|_{C^{\frac76}(\gamma_-\cap\mathcal{Q}_{r_0}(z_0))}\leq D_j
    \end{aligned}
\end{equation}
for any $j\in\{0,\ldots,l\}$.

\textbf{Step 3-2. Explosion of gradients of \mbox{$\frac{\delta_{h,l}f}{|h|^l}$} near $\gamma_0$.} The results in this step are required in order to treat the source terms $G_l$ in divergence form in \eqref{diff1.eq.diffuse.ind}. We combine $C^{1/2-\delta}$ estimates near $\gamma_0$ with higher order estimates away from $\gamma_0$. First, using \autoref{lem.gra.bdep} with
$\mathcal{H}_{2\rho}(z_1)\cap\gamma_0=\emptyset$ with $\rho\leq1$, and following the same iterative argument as in \eqref{c2eps.ind.diff} with $\delta_{h,l}f$ replaced by $\delta_{h,l}f-\delta_{h,l}f(z_1)$, we deduce
\begin{align*}
   & \rho^{-\frac16}\left\|\frac{\nabla_v\delta_{h,l}f}{|h|^{l}}\right\|_{L^{\infty}(\mathcal{H}_\rho(z_1))}+ \left[\frac{\nabla_v\delta_{h,l}f}{|h|^{l}}\right]_{C^{\frac16}(\mathcal{H}_\rho(z_1))}\\
    &\leq c\langle v_1\rangle^{4(l+1)}\sum_{i=0}^l\rho^{-\frac76} \left\|\frac{\delta_{h,i}f-\delta_{h,i}f_{}(z_1)}{|h|^i}\right\|_{L^\infty(\mathcal{H}_{2\rho}(z_1))}\\
    &\quad+c\langle v_1\rangle^{4(l+1)}\sum_{i=0}^l\left(\left\|\frac{\delta_{h,i}F}{|h|^i}\right\|_{L^\infty(\mathcal{H}_{2\rho}(z_1))}+\left[\frac{\delta_{h,i}\mathcal{N}f}{|h|^i}\right]_{C^{\frac76}(\mathcal{H}_{2\rho}(z_1))}\right),
\end{align*} 
where $c=c(n,l,c_B,C_l)$. Now plugging in the $C^{\frac12-\delta}$ estimates of $\frac{\delta_{h,i}f_{(l-i)h}}{|h|^i}$ given in \eqref{ind.level} with $j\in\{0,\ldots,l\}$ and $z_2$ replaced by $z_1$, and using \eqref{boundary.ind}, we deduce for any $\delta>0$,
\begin{equation}\label{gra.point.diff.las0t}
\begin{aligned}
    \left\|\frac{\delta_{h,l}\nabla_vf}{|h|^l}\right\|_{L^{\infty}(\mathcal{H}_\rho(z_1))}&\leq c\langle v_1\rangle^{c_l}\rho^{-\frac12-\delta}\left(\|f\|_{L^\infty(\mathcal{H}_{4r_1}(t,x_1,v_1))}+\left\|{F}\right\|_{C^{3l}(\mathcal{H}_{4r_1}(t,x_1,v_1))}\right)\\
    &\quad+cD_l\langle v_1\rangle^{c_l}\rho^{-\frac12-\delta}\left\|\mathcal{M}\right\|_{C^{3l+\frac76}(\gamma_-\cap\mathcal{Q}_{4r_1}(t,x_1,v_1))},
\end{aligned}
\end{equation}
where $c_l=c_l(n,l,p)$ and $c=c(n,\alpha,\delta,l,c_B,C_l,p)$.

\textbf{Step 3-3. \mbox{$C^{\frac12-\delta}$ estimates of $\frac{\delta_{h,l+1}f}{|h|^{l+\alpha}}$ up to $\gamma_0$}.} Now we fix $z_2\in\gamma_0$ and $\widehat{r}_2\coloneqq\frac{\langle v_2\rangle^{-p}}{64}(1-2^{-\frac1p})$ to see that for any $z_1\in\mathcal{H}_{\widehat{r}_2}(z_2)$, $r_1=\frac{\langle v_1\rangle^{-p}}{64}\leq 2r_2=\frac{\langle v_2\rangle^{-p}}{32}$, which implies $\mathcal{H}_{4r_1}(z_1)\subset \mathcal{H}_{32r_2}(z_2)$.
First, let us write $\widetilde{G}\coloneqq \sum\limits_{i=0}^{l}\left(\frac{\delta_{h,l-i+1}A}{|h|^{l-i+\alpha}}\frac{\delta_{h,i}\nabla_vf_{(l+1-i)h}}{|h|^{i}}\right)$ to see that
\begin{align*}
   \sup_{(t,x',v')\in\mathcal{H}'_{\widehat{r}_2}(z'_2)} \|\widetilde{G}\|_{L^{8-\delta_0}(H_{\widehat{r}_2})}&\leq c\int_{H_{\widehat{r}_2}}\sup_{z_1\in\mathcal{H}_{\widehat{r}_2}(z_2)}\sum_{i=0}^l\left\|\frac{\delta_{h,l}\nabla_vf}{|h|^l}\right\|_{L^{\infty}(\mathcal{H}_\rho(z_1))}\,dx_{1,n}\,dv_{1,n}
   \\&\leq c\int_{H_r}\max\{|x_{1,n}|^{\frac13},|v_{1,n}|\}^{(-\frac12-\sigma)(8-\delta_0)}\\
   &\times\Bigg(\langle v_2\rangle^{c_l}\left(\|f\|_{L^\infty(\mathcal{H}_{32r_2}(z_2))}+\left\|{F}\right\|_{C^{3l}(\mathcal{H}_{32r_2}(z_2))}\right)\\
    &\quad+cD_l\langle v_2\rangle^{c_l}\left\|{\mathcal{M}}\right\|_{C^{3l+\frac76}(\gamma_-\cap\mathcal{Q}_{32r_2}(z_2))}\Bigg)\,dx_ndv_n,
\end{align*}
where $c_l=c_l(n,l,p)$, $c=c(n,\alpha,\sigma,l,c_B,C_l,p,\delta_0)$ and we have used \eqref{gra.point.diff.las0t} with $\delta=\sigma$ and $\rho=\frac1{2}\max\{|x_{2,n}|^{\frac13},|v_{2,n}|\}$. Then note that for any small $\delta_0$, there is a small constant $\sigma=\sigma(\delta_0)$ such that $(\frac12+\sigma)(8-\delta_0)<4$ so that we have 
\begin{align*}
    \sup_{(t,x',v')\in\mathcal{H}'_{\widehat{r}_2}(z'_2)} \|\widetilde{G}\|_{L^{8-\delta_0}(H_{\widehat{r}_2})}&\leq c\langle v_2\rangle^{c_l}\left(\|f\|_{L^\infty(\mathcal{H}_{32r_2}(z_2))}+\left\|{F}\right\|_{C^{3l}(\mathcal{H}_{32r_2}(z_2))}\right)\\
    &\quad+D_l\langle v_1\rangle^{c_l}\left\|{\mathcal{M}}\right\|_{C^{3l+\frac76}(\gamma_-\cap\mathcal{Q}_{32r_2}(z_2))},
\end{align*}
where $c_l=c_l(n,l,p)$, $c=c(n,\alpha,\sigma,l,c_B,C_l,p,\delta_0)$.
Similarly, we get the same estimates for $\widetilde{F}\coloneqq \frac{\delta_{h,l+1}F}{|h|^{l+\alpha}}+\sum\limits_{i=0}^{l}\frac{\delta_{h,l-i+1}B}{|h|^{l-i+\alpha}}\cdot\frac{\delta_{h,i}\nabla_vf_{(k-i)h}}{|h|^i}$.
Therefore, we can apply \autoref{lem.bdry.hol.bdep} with $f=\frac{\delta_{h,l+1}f}{|h|^{l+\alpha}}$, $G=\widetilde{G}$, and $F=\widetilde{F}$ together with \eqref{gra.point.diff.las0t} and \eqref{boundary.ind0} to derive that for any $z_2\in\gamma_0$ and $\rho\leq r_2$, we have
\begin{equation}\label{gamma0.est.diff.0}
\begin{aligned}
&\left[\frac{\delta_{h,l+1}f}{|h|^{l+\alpha}}\right]_{C^{\frac12-\delta}(\mathcal{H}_{2\rho}(z_2))}+\left\|\frac{\delta_{h,l+1}f}{|h|^{l+\alpha}}\right\|_{L^\infty(\mathcal{H}_{2\rho}(z_2))}\\
&\leq c\langle v_2\rangle^{c_l}\left(\|f\|_{L^\infty(\mathcal{H}_{32r_2}(z_2))}+\left\|{F}\right\|_{C^{3(l+1)}(\mathcal{H}_{32r_2}(z_2))}+D_{l+1}\left\|{\mathcal{M}}\right\|_{C^{3(l+1)+\frac76}(\gamma_-\cap\mathcal{Q}_{32r_2}(t,x_2,v_2))}\right),
\end{aligned}
\end{equation}
where $c_l=c_l(n,l,p)$ and $c=c(n,\alpha,\sigma,l,\|A\|_{C^{1}(\mathcal{H}_{2\rho}(z_1))},c_B,C_l)$, whenever $M=M(n,l,p)$ given in \eqref{m.ass.diff} is sufficiently large (this was needed for \eqref{l1.mcon.used} and \eqref{l2.mcon.used}). Moreover, when $z_2$ is away from $\gamma_0$, then by \eqref{gra.point.diff.las0t} with $\rho=r_2$, we observe that $\widetilde{F}$ and $\widetilde{G}$ are bounded. Now by \autoref{rmk.c12.bdep}, we obtain the same estimate as in \eqref{gamma0.est.diff.0}. Therefore, we have proved \eqref{gamma0.est.diff.0} for any $z_2\in\mathcal{H}_{2,\infty}$.

Thus, taking $\delta=\frac12-3\alpha$ and using some covering argument to reduce the radius of the cylinders given in the right-hand side, we have for any $z_2\in\mathcal{H}_{1,\infty}$,
\begin{align*}
   \left\|\frac{\delta_{h,l+1}f}{|h|^{l+2\alpha}}\right\|_{L^\infty(\mathcal{H}_{2\rho}(z_2))}&\leq c\langle v_1\rangle^{c_l}\left(\|f\|_{L^\infty(\mathcal{H}_{4r_2}(t,x_2,v_2))}+\left\|{F}\right\|_{C^{3(l+1)}(\mathcal{H}_{4r_2}(t,x_2,v_2))}\right)\\
&\quad+cD_{l+1}\langle v_2\rangle^{c_l}\left\|{\mathcal{M}}\right\|_{C^{3(l+1)+\frac76}(\gamma_-\cap\mathcal{Q}_{4r_2}(t,x_2,v_2))}.
\end{align*}
This estimate is a higher order version of \eqref{ind.level.alpha}. Since $\alpha \in (0,\frac16)$ was arbitrary, we can iterate repeat the previous procedure, starting from \eqref{ind.level.alpha} with $\alpha$ replaced by $2\alpha$, six times and finally derive \eqref{ind.level} with $j=l+1$, which implies \eqref{ind.level} for any $j\in \{0,\ldots,k\}$ by the induction.

\textbf{Step 4. Conclusion.}
Combining \eqref{ind.level} with $j=k$, \eqref{c2eps.ind.diff} and \eqref{widehatg.hol} immediately yields \eqref{goal.hol.diffuse}. Then, by \eqref{goal.hol.diffuse}, we have 
\begin{align*}
    \|\partial^k_tg\|_{C^{\frac76}(\gamma_-\cap\mathcal{Q}_{r_0}(z_0))}\leq c\langle v_0\rangle^{c_k}D_k,
\end{align*}
where $c=c(n,\Lambda,k,C_k)$ and $c_k=c_k(n,k,p)$. On the other hand, by considering $\delta^{x_i}_{h}g(z)=g(t,x+he_i,v)-g(z)$ with $i\in\{1,\ldots, k-1\}$ instead of $\delta_hg(z)=g(t+h,x,v)-g(z)$ and observing that \eqref{partialt} with $\partial_t$ replaced by $\partial_{x_i}$ holds, we also deduce \eqref{goal.hol.diffuse} with $\delta_h$ replaced by $\delta^{x_i}_h$, as $\delta^{x_i}_{h}f$ satisfies the equation \eqref{diff1.eq.diffuse.ind} with $\delta_{h,k}$ replaced by $\delta^{x_i}_{h,k}$. Thus, we obtain that for any $i,j$ with $i+j\leq{k}$,
\begin{equation}\label{partial.last}
\begin{aligned}
   \|\partial_x^{i}\partial^{j}_t\mathcal{N}f\|_{C^{\frac76}(\gamma_-\cap\mathcal{Q}_{r_0}(z_0))}&\leq c\|\mathcal{M}\|_{C^{3k+\frac76}(\gamma_-\cap\mathcal{Q}_{r_0}(z_0))}\langle v_0\rangle^{c_k}\bigg(\|\langle \cdot\rangle^{c_k}f(t,x,\cdot)\|_{L^\infty(\mathcal{H}_{2,\infty})}\\
   &\qquad\qquad\qquad\qquad\quad\qquad\qquad\qquad+\|\langle \cdot\rangle^{c_k}F(t,x,\cdot)\|_{C^{3k}(\mathcal{H}_{2,\infty})}\bigg),
\end{aligned}
\end{equation}
where $c=c(n,\Lambda,k,c_B)$ and $c_k=c_k(n,k,p)$, whenever $\langle v_0\rangle^{M}\|\mathcal{M}\|_{C^{3k}(\gamma_-\cap\mathcal{Q}_{r_0}(z_0))}\leq 1$ for large $M=M(n,k,p)$ and $\|B\|_{C^{3k}(\mathcal{H}_{r_0}(z_0))}\leq c_B \langle v_0\rangle^2$. Therefore, by \eqref{fund.lambda.eps}, we directly check that when $k=2\lambda$, 
\begin{align}\label{partial.last2}
    [\mathcal{N}f]_{C^{\lambda,\eps}(\gamma_-\cap\mathcal{Q}_{r_0}(z_0))}\leq c\sum_{i+j\leq 2\lambda}\|\partial_x^{i}\partial^{j}_t\mathcal{N}f\|_{C^{\frac76}(\gamma_-\cap\mathcal{Q}_{r_0}(z_0))}.
\end{align}
Thus, combining \eqref{partial.last} and \eqref{partial.last2} yields \eqref{first.goal.diff}, which completes the proof.
Note that one can slightly improve the value of $k$ here, but we decided to work in a suboptimal framework in order make the proof easier to read. 
\end{proof}

\begin{proof}[Proof of \autoref{thm.diff}]
    Since we have proved the regularity of the boundary data $\mathcal{N}f$ with in-flow, we apply \autoref{thm.hol12} and \cite[Lemma 4.4]{RoWe25} together with \autoref{lem.reg.nf.diff} to deduce the desired result.
\end{proof}

We end this section by proving that the $C^{\frac{1}{2}}$ regularity obtained in \autoref{thm.diff} is optimal. We provide a counterexample in 1D for the canonical choice of $\cM(v) = 2e^{-|v|^2}$ being the appropriately normalized boundary Maxwellian\footnote{Note that $\int_{\R} \cM(w) w_- \d w = 1$}, but our construction also works for more general choices of $\cM$.

\begin{example}
\label{example:diffuse-optimal}
   Let $n = 1$. Take a cutoff function $\psi=\psi(v)\in C_c^\infty(B_4)$ such that $\psi\equiv 1$ in $(-2,2)$.
   Then, $f\coloneqq \phi_0\psi$ is a weak solution to 
    \begin{align*}
        v\partial_x f-\partial_{vv}f=-2\partial_{v}\phi_0\partial_{v}\psi-\phi_0\partial_{vv}\psi\eqqcolon F\quad\text{in }\{x>0\}
    \end{align*}
    Since $\phi_0 \in C^{\infty}_{loc}([0,\infty) \times \{|v| \ge 2\})$, we have $F \in C^{\infty}_{loc}([0,\infty)\times\R)$. Next we select a cut off function $\xi\in C^\infty(\bbR)$ such that $\xi(v)\equiv 2$ on $v\geq -1$ and $\xi(v)\equiv 0$ on $v\leq -2$, and $\int_{\bbR}(\xi(w)e^{-|w|^2}) w_- \d w \neq 1$. Now, we choose $a \not= 0$ such that 
    \begin{align*}
        a  = \left(\int_{\R}f(0,w) w_-\,d w \right)^{-1} \left(1 - \int_{\R} (\xi(w)e^{-|w|^2}) w_- \d w \right),
    \end{align*}
    and define $\widetilde{f}=af+\xi(v)e^{-|v|^2}$. Then, setting $\mathcal{M}(v)={2e^{-|v|^2}}$, by construction, it holds
    \begin{align*}
        \mathcal{N} \widetilde{f}(0,v) = \mathcal{M}(v) \int_{\R}(af(0,w)+\xi(w)e^{-|w|^2})w_-\,d w = \mathcal{M}(v) = \xi(v)e^{-|v|^2}  ~~ \text{ for } v \in  (0,\infty),
    \end{align*}
    and therefore
    \begin{equation*}
\left\{
\begin{alignedat}{3}
v\partial_{x} \widetilde{f}-\partial_{vv}\widetilde{f}&=\widetilde{F}&&\qquad \mbox{in  $\{x>0\}$}, \\
\widetilde{f}&=\mathcal{N}\widetilde{f}&&\qquad  \mbox{in $\{0\}\times \{v>0\}$},
\end{alignedat} \right.
\end{equation*}
where $\widetilde{F}= F - \partial_{vv}(\xi(v)e^{-|v|^2}) \in C^{\infty}_{loc}([0,\infty) \times \R)$, and we used that $\phi_0(0,v) = 0$ for $v > 0$. However,  \eqref{phim.asymptotic} implies that the optimal regularity of the solution is $\widetilde{f} \in C^{\frac12}_{loc}([0,\infty) \times \R)$ is optimal.
\end{example}

\section{Proof of main results in general domains}
\label{sec:main-proof}
In this section, we will generalize our previous results to general domains via a flattening argument (similar to  \cite{DGY22,GHJO20,Sil22,RoWe25}) and prove \autoref{thm:main-diffuse}, \autoref{thm:in-flow}, \autoref{thm:f/phi}, and \autoref{thm:C3}. Note that we also establish \autoref{thm:hol12.gen}, \autoref{thm:hol12.gen.diff}, \autoref{thm:C3-coeff}, and \autoref{thm:f/phi-coeff}, which are generalizations and sharpened versions of \autoref{thm:main-diffuse}, \autoref{thm:in-flow}, \autoref{thm:C3}, and \autoref{thm:f/phi}, respectively, to kinetic Fokker Planck equations with sufficiently smooth coefficients. Since the flattening argument is quite well understood, we will omit most parts and give suitable references instead.

Let us assume that the domain $\Omega$ is of class $C^{\beta}$ with $\beta\geq2$. We now recall several useful facts which are already proved in \cite[Section 2.4]{RoWe25} and \cite[Chapter 3]{Sil22}.
\begin{itemize}
    \item There are constants $R_0,R_1\leq1$ depending only on $\Omega$ such that for any ${x}_1\in\partial\Omega$, there is a $C^{\beta-2}$
    diffeomorphism 
    \begin{equation}\label{defn.Psi}
    \Psi^{-1}_{{x}_1}(t,y,w)\equiv \Psi^{-1}(t,y,w)=(t,\varphi^{-1}(y),\nabla\varphi^{-1}(y)w),    
    \end{equation}
     where $\varphi^{-1}$ is a $C^{\beta} $ diffeomorphism and 
    \begin{equation}\label{after.trans1}
    \begin{aligned}
        ( I\times (B_{{R_0}}({x}_0',0))\times \bbR^n)\cap\{x_n>0\}&\subset \Psi_{}(I\times (\Omega\cap B_{R_1}({x}_1))\times \bbR^n )\\
        &\subset( I\times (B_{2R_0}({x}_0',0))\times \bbR^n)\cap\{x_n>0\},\\
        ( I\times B_{{R_0}}({x}_0',0)\times \bbR^n)\cap\Gamma_-&\subset\Psi_{}((I\times (\Omega\cap B_{R_1}({x}_1))\times \bbR^n)\cap\gamma_- )\\
        &\subset ( I\times B_{2R_0}({x}_0',0)\times \bbR^n)\cap\Gamma_-
    \end{aligned}
    \end{equation}
    with $I$ being a time interval $(a,b)$ and $x_0=(x_0',0)=\Psi(x_1) $. To avoid confusion of the notation on the incoming part, we write $\Gamma_-=\{x_n=0\}\times \{v_n>0\}$.
    \item If $f$ is a weak solution to \eqref{eq:FPK}, then $\widetilde{f}=f\circ\Psi^{-1}_{}(t,y,w)$ is a weak solution to 
    \begin{equation}\label{after.trans2}
\left\{
\begin{alignedat}{3}
(\partial_t+w\cdot\nabla_y)\widetilde{f}-\ddiw(\widetilde{A}\nabla_w\widetilde{f})&=\widetilde{B}\cdot\nabla_w\widetilde{f}+\widetilde{F}&&\qquad \mbox{in  $I\times  (B_{{R_0}}({x}_0',0)\cap\{x_n>0\})\times \bbR^n$}, \\
\widetilde{f}&=\widetilde{g}&&\qquad  \mbox{in $(I\times B_{{R_0}}({x}_0',0)\times \bbR^n)\cap \Gamma_-$},
\end{alignedat} \right.
\end{equation}
where we write $\Gamma_-\coloneqq (\{x_n=0\}\times\{v_n>0\})$ and
\begin{align*}
    &\widetilde{A}_{i,j}(t,y,w)\coloneqq \widehat{A}_{i,l}(y)A_{l,k}(\Psi_{}^{-1}(t,y,w))\widehat{A}_{j,k}(y)\quad\text{with}\quad\widehat{A}\coloneqq D\varphi(\varphi^{-1}(y)),\\
    &\widetilde{B}_i\coloneqq \widehat{A}_{i,j}(y)B_j(\Psi_{}^{-1}(t,y,w))+\widehat{A}^{-1}_{j,s}(y)w_s\widehat{A}^{-1}_{l,k}w_k\partial_{x_j,x_l}\varphi(x), \\
    &\widetilde{F}(t,y,w)\coloneqq F(\Psi_{}^{-1}(t,y,w)),\quad \widetilde{g}(t,y,w)\coloneqq g(\Psi_{}^{-1}(t,y,w)).
\end{align*}

\end{itemize}

We note that when the domain is smooth, then the distance function $d_\Omega(x)=:\mathrm{dist}(x,\partial\Omega)\in C^\infty(\overline{\Omega})$. This allows us to follow every argument given in \cite[Section 2.4]{RoWe25} when a regularized distance is replaced by the distance function. In particular, by using the same notation as in \cite[Section 2.4]{RoWe25}, we observe $\psi^{-1}(x)\coloneqq (x',d_\Omega(x))$ and $\varphi^{-1}(y)\coloneqq \psi(y',0)+y_n(\nabla d_\Omega)(\psi(y))$, where the function $\varphi$ given in \eqref{defn.Psi}. Thus, we have the following.
\begin{equation}\label{property.varphi}
\begin{aligned}
    &y_n=(\varphi(x))_n=\big(\varphi\big(\psi(y',0)+d_\Omega(x)\nabla d_\Omega(\psi(y))\big)_n=d_\Omega(x), \\
    &w_n=(\nabla\varphi^{-1}(y))^{-1}v)_n=(\nabla\varphi(x)v)_n=\nabla d_\Omega(x)\cdot v = n_x\cdot v,\\
    &\widetilde{A}_{n,n}(t,y,w)=(\nabla\varphi(x))_{n,k}A_{k,l}(z)(\nabla\varphi(x))_{n,l},
\end{aligned}
\end{equation}
where $n_x\coloneqq \nabla d_\Omega(x)$ and we have used the fact that for any $x\in \Omega$, there is a boundary point $\psi(y',0)\in \partial\Omega$ so that $\psi(y',d_\Omega(x))=x$ and $x=\psi(y',0)+d_\Omega(x)\nabla d_\Omega(\psi(y))$ with $y=(y',d_\Omega(x))$.

For a more detailed discussion, we refer to \cite[Section 2.4]{RoWe25}.

As we will deal with general domains, we frequently use the following notation: 
\begin{align*}
    \mathbf{H}_r(z_1)\coloneqq \{z\,:\,t\in I_{r^2}(t_1),\,v\in B_r(v_1),\,x\in B_{r^3}(x_1+(t-t_1)v_1)\cap\Omega\}.
\end{align*}

\subsection{Optimal regularity for in-flow and proof of  \texorpdfstring{\autoref{thm:in-flow}}{Theorem 1.3}}

Now we give the global H\"older regularity with in-flow condition for general domains.
\begin{theorem}\label{thm:hol12.gen}
    Let $\Omega$ be a bounded domain of class $C^{2,\eps}$ for some $\eps>0$. Let $f$ be a weak solution to 
     \begin{equation*}
\left\{
\begin{alignedat}{3}
(\partial_t+v\cdot\nabla_x){f}-\ddiv({A}\nabla_vf)&={B}\cdot\nabla_vf+{F}&&\qquad \mbox{in  $(-1,1)\times  \Omega\times \bbR^n$}, \\
{f}&={g}&&\qquad  \mbox{in $\gamma_-$},
\end{alignedat} \right.
\end{equation*}
where \begin{align*}
    \sup_{z_1\in(-1,1)\times  \Omega\times \bbR^n }\|A\|_{C^{\eps}(\mathbf{H}_1(z_1))}+\|B\|_{L^{\infty}(\mathbf{H}_1(z_1))})\leq 1.
\end{align*}
Then we have 
\begin{align*}
    \|f\|_{C^{\frac12}([-1/8,1/8]\times\Omega\times \bbR^n) }\leq c\left(\|\langle v\rangle^qf\|_{L^1((-1,1)\times \Omega\times \bbR^n)}+\|\langle v\rangle^qG\|_{L^{\infty}((-1,1)\times \Omega\times \bbR^n)}+[g]_{C^{\frac12+\eps}_q(\gamma_-)}\right)
\end{align*}
for some constants $c=c(n,\Lambda,\eps)$ and $q=q(n)$, where we write 
\begin{align*}
    [g]_{C^{\frac12+\eps}_q(\gamma_-)}\coloneqq \sup_{z_0\in\gamma_-}\langle v_0\rangle^q[g]_{C^{\frac12+\eps}(\gamma_-\cap{\mathcal{Q}_1(z_0)})}.
\end{align*}
\end{theorem}
\begin{proof} First, let us fix $z_1\in \gamma$. By \eqref{after.trans1} and \eqref{after.trans2}, we have $\widetilde{f}=f\circ \Psi_{x_1}^{-1}$ is a solution to \eqref{after.trans2} for some $R_1,R_2\leq 1$. By \cite[Lemma 2.17]{RoWe25} and since the domain is of  class $C^{2,\eps}$, we have 
\begin{equation}\label{coeff.thm12}
\begin{aligned}
    \|\widetilde{A}\|_{C^{\frac{\eps}{2}}(\mathcal{H}_{r}(z_0))}&\leq c\|{A}\|_{C^{{\eps}}(\Psi^{-1}(\mathcal{H}_{r}(z_0)))},\\
     \|\widetilde{B}\|_{L^{\infty}(\mathcal{H}_{r}(z_0))}&\leq c(\|B\|_{L^{{\infty}}(\Psi^{-1}(\mathcal{H}_{r}(z_0)))}+\langle v_0\rangle^{2}),\\
     \|F\|_{L^{\infty}(\mathcal{H}_{r}(z_0))}&\leq c \|F\|_{L^{{\infty}}(\Psi^{-1}(\mathcal{H}_{r}(z_0)))} \\
     [\widetilde{g}]_{C^{\frac12+\frac{\eps}{2}}(\gamma_-\cap\mathcal{Q}_{r}(z_0))}&\leq c[g]_{C^{\frac12+{\eps}}(\Psi^{-1}((\gamma_-\cap\mathcal{Q}_{r}(z_0))))},
\end{aligned}
\end{equation}
where $\Psi\equiv \Psi_{x_1}$, $z_0\coloneqq \Psi(z_1)$ and $c=c(n,\eps,\Lambda)$, whenever $r\leq \langle v_0\rangle^{-N}$ for some large constant $N=N(n,\eps)$. In particular, we choose $N$ large enough so that $N\geq \frac2{\eps}$. Thus, by combining \autoref{thm.hol12}, \cite[Lemma 2.25]{RoWe25} and \autoref{lem.rep.hol2} to estimate the $L^\infty$-norm of $f$ by its $L^1$-norm, we have 
\begin{align*}
    \|\widetilde{f}\|_{C^{\frac12}(\mathcal{H}_{\rho_0}(z_0))}&\leq c\rho_0^{-\frac12}\langle v_0\rangle\left(\rho_0^{-(4n+2)}\|\widetilde{f}\|_{L^1(\mathcal{H}_{2\rho_0}(z_0))}+\rho_0^2\|\widetilde{F}\|_{L^\infty(\mathcal{H}_{2\rho_0}(z_0))}+\rho_0^{\frac12+\frac{\eps}2}[\widetilde{g}]_{C^{\frac12+\frac{\eps}{2}}(\gamma_-\cap\mathcal{Q}_{2\rho_0}(z_0))}\right)\\
    &\leq c\left(\langle v_0\rangle^{p}\|\widetilde{f}\|_{L^1(\mathcal{H}_{2\rho_0}(z_0))}+\|\widetilde{F}\|_{L^\infty(\mathcal{H}_{2\rho_0}(z_0))}+[\widetilde{g}]_{C^{\frac12+\frac{\eps}{2}}(\gamma_-\cap\mathcal{Q}_{2\rho_0}(z_0))}\right),
\end{align*}
where $\rho_0=\frac{R_1}{2}{\langle v_0\rangle^{-N}}$ and $c=c(n,\eps,\Lambda)$. 
Now, by scaling back together with \cite[Lemma 2.17]{RoWe25} and \eqref{after.trans1}, there is a constant $r_1= \frac{\langle v_1\rangle^{-N}}{c_1}$ for some constant $c_1=c_1(\Omega)$ such that
\begin{align*}
    \|f\|_{C^{\frac12}({\mathbf{H}}_{r_1}(z_1))}\leq c\langle v_1\rangle^{q}\left(\|{f}\|_{L^1({\mathbf{H}}_{2r_1}(z_1)}+\|{F}\|_{L^\infty({\mathbf{H}}_{2r_1}(z_1))}+[{g}]_{C^{\frac12+{\eps}}(\gamma_-\cap{{Q}}_{2r_1}(z_1))}\right),
\end{align*}
where $c=c(n,\eps,\Lambda)$ and $q=q(n,\eps)$. Thus, by a covering argument as in \cite[Theorem 5.4]{RoWe25}, we derive the desired estimate.
\end{proof}

\begin{proof}[Proof of \autoref{thm:in-flow}]
    The result follows immediately from \autoref{thm:hol12.gen} and the fact that the smoothness of the solution away from the grazing set is already proved in \cite[Theorem 5.6]{RoWe25}.
\end{proof}

\subsection{Optimal regularity for diffuse reflection and proof of  \texorpdfstring{\autoref{thm:main-diffuse}}{Theorem 1.1}}

Next we provide the proof of  global H\"older regularity with diffuse reflection condition for general domains.
\begin{theorem}\label{thm:hol12.gen.diff}
    Let $\Omega$ be a bounded domain of class $C^{2,\eps}$ for some $\eps>0$. Let $f$ be a weak solution to 
     \begin{equation*}
\left\{
\begin{alignedat}{3}
(\partial_t+v\cdot\nabla_x){f}-\ddiv({A}\nabla_vf)&={B}\cdot\nabla_vf+{F}&&\qquad \mbox{in  $(-1,1)\times  \Omega\times \bbR^n$}, \\
{f}&=\mathcal{N}f&&\qquad  \mbox{in $\gamma_-$},
\end{alignedat} \right.
\end{equation*}
where
\begin{align*}
    \mathcal{N}f(t,x,v)\coloneqq\mathcal{M}(t,x,v)\int_{\bbR^n}f(t,x,w)(w\cdot n_x)_-\,dw.
\end{align*}
Assume 
\begin{align}\label{ass.m.diff.gen}
    \sup_{z_1\in(-1,1)\times  \Omega\times \bbR^n } \big(\langle v_1\rangle^{M}\|\mathcal{M}\|_{C^{\frac12+\eps}(\gamma_-\cap \mathcal{Q}_1(z_1))}+\|A\|_{C^{\eps}(\mathbf{H}_1(z_1))}+\|B\|_{L^{\infty}(\mathbf{H}_1(z_1))} \big) \leq 1
\end{align}
for some $M>0$. There is $M_0=M_0(n,\eps)$ such that if $M > M_0$, then we have 
\begin{align*}
    \|f\|_{C^{\frac12}([-1/8,1/8]\times\Omega\times \bbR^n) }\leq c\left(\|\langle v\rangle^qf\|_{L^1((-1,1)\times \Omega\times \bbR^n)}+\|\langle v\rangle^qG\|_{L^{\infty}((-1,1)\times \Omega\times \bbR^n)}\right)
\end{align*}
for some constants $c=c(n,\Lambda,\eps)$ and $q=q(n)$. 
\end{theorem}
\begin{proof}
    First, let us fix $z_1\in\gamma$. Then by \eqref{after.trans2}, we have that $\widetilde{f}$ is a weak solution to 
    \begin{equation}\label{eq.m.}
\left\{
\begin{alignedat}{3}
(\partial_t+w\cdot\nabla_y)\widetilde{f}-\ddiw(\widetilde{A}\nabla_w\widetilde{f})&=\widetilde{B}\cdot\nabla_w\widetilde{f}+\widetilde{F}&&\qquad \mbox{in  $I\times  (B_{\frac{R_0}{2}}({x}_0',0)\cap\{x_n>0\})\times \bbR^n$}, \\
\widetilde{f}&=\widetilde{\cN}\widetilde{f}&&\qquad  \mbox{in $\Gamma_-\cap(I\times B_{\frac{R_0}{2}}(\overline{x}_0',0)\times \bbR^n)$},
\end{alignedat} \right.
\end{equation}
where $I=(t_0-(R_0/2)^2,t_0+(R_0/2)^2)$ with  $t_0\in I_{\frac12}$, $\Gamma_-=(\{x_n=0\}\times\{v_n>0\})$ and
\begin{align*}
    \widetilde{\cN}\widetilde{f}(t,y,w)\coloneqq \mathcal{M}(t,\varphi^{-1}(y),\nabla\varphi^{-1}(y)w)|\nabla\varphi^{-1}(y)|\int_{\bbR^n}\widetilde{f}(t,y',0,\xi)(\xi_n)_-\,d\xi,
\end{align*}
which follows from the change of variables and \eqref{property.varphi}. Next, as in \eqref{coeff.thm12}, there is a large constant $p=p(n,\eps)$ for which the first three conditions in \eqref{coeff.thm12} are satisfied whenever $r\leq \langle v_0\rangle ^{-p}$ with $z_0\coloneqq \Psi(z_1)$. Moreover, by taking $M=M(n,\eps)\geq1$ large so that we observe from \eqref{ass.m.diff.gen} that
\begin{align}\label{pr.m}
    \langle v_2\rangle^M\|\widetilde{\mathcal{M}}\|_{C^{\frac12+\eps}(\gamma_-\cap \mathcal{Q}_1(z_0))}\leq c(\eps,\Omega),
\end{align}
where we write
\begin{align}\label{tildem.defn}
    \widetilde{\mathcal{M}}(t,y,w)\coloneqq \mathcal{M}(t,\varphi^{-1}(y),\nabla\varphi^{-1}(y)w)|\nabla\varphi^{-1}(y)|.
\end{align}
This together with \eqref{coeff.thm12} and \eqref{ass.m.diff.gen} allows us to apply \autoref{thm.diff}. Hence, we have
\begin{equation}\label{ineq0.diff.thm}
\begin{aligned}
    \|\widetilde{f}\|_{C^{\frac12}(\mathcal{H}_{r_0}(z_0))}&\leq cr_0^{-\frac12}\langle v_0\rangle(\|\widetilde{f}\|_{L^{\infty}(\mathcal{H}_{2r_0}(z_0))}+r_0^2\|\widetilde{F}\|_{L^\infty(\mathcal{H}_{2r_0}(z_0))})\\
    &\quad+c\langle v_0\rangle\|\widetilde{\mathcal{M}}\|_{C^{\frac12+\eps}(\mathcal{H}_{r_0}(z_0))}\left(\|\langle \cdot\rangle^{q}\widetilde{f}(t,x,\cdot)\|_{L^\infty(\widetilde{\mathcal{H}}_{\infty})}+\left\|\langle\cdot\rangle^{q}\widetilde{F}(t,x,\cdot)\right\|_{L^{\infty}(\widetilde{\mathcal{H}}_{\infty})}\right),
\end{aligned}
\end{equation}
where $c=c(n,\eps,\Omega)$, $q=q(n,\eps)$ and we write $\widetilde{\mathcal{H}}_{\infty}\coloneqq I\times  (B_{\frac{R_1}{2}}({x}_1',0)\cap\{x_n>0\})\times \bbR^n$. Using the change of variables along with \cite[Lemma 2.17]{RoWe25}, we can rewrite \eqref{ineq0.diff.thm} as 
\begin{equation}\label{ineq1.diff.thm}
\begin{aligned}
    \|{f}\|_{C^{\frac12}(\mathbf{H}_{\rho_1}(z_1))}&\leq c\langle v_1\rangle(\|{f}\|_{L^{\infty}(\mathbf{H}_{2\rho_1}(z_1))}+\|{F}\|_{L^\infty(\mathbf{H}_{2\rho_1}(z_1))})\\
    &+c\langle v_1\rangle\|{\mathcal{M}}\|_{C^{\frac12+\eps}(\gamma_-\cap\mathcal{Q}_{2\rho_1}(z_1))}\left(\|\langle \cdot\rangle^{q}{f}(t,x,\cdot)\|_{L^\infty(I\times \Omega\times \bbR^n)}+\left\|\langle\cdot\rangle^{q}{F}(t,x,\cdot)\right\|_{L^{\infty}(I\times \Omega\times \bbR^n)}\right)
\end{aligned}
\end{equation}
for some constants $c=c(n,\eps,\Omega)$ and $q=q(n,\eps)$, where $\rho_1=\frac{\langle v_1\rangle^{-p}}{c_1}$ for some constant $c_1=c_1(\Omega)\geq1$. Now using the covering argument as in \cite[Theorem 5.4]{RoWe25} together with \eqref{pr.m}, we derive the desired global $C^{\frac12}$ estimate from \eqref{ineq1.diff.thm}. This completes the proof.
\end{proof}

\begin{remark}
Note that by the same proof, we could establish $C^{\lambda,\eps}$ estimates for diffuse reflection near $\gamma$ (but away from $\gamma_0$), whenever \eqref{b.hol.nf} and \eqref{space.hol.weight} hold for $A,B,F,\cM$, using the second claim in \autoref{thm.diff} together with the flattening argument. Below, we give the proof only for smooth coefficients.
\end{remark}

We now give the proof of \autoref{thm:main-diffuse}.
\begin{proof}[Proof of \autoref{thm:main-diffuse}.]Since the $C^{\frac12}$-regularity follows from \autoref{thm:hol12.gen.diff}, we only give the proof of smoothness of the solution away from $\gamma_0$. For any $z_1\in\gamma_-$, we get that $\widetilde{f}=f\circ \Psi^{-1}$ is a weak solution to \eqref{eq.m.}. Now using the definition \eqref{tildem.defn}, we see that for any $M>0$, $k\in\N\cup\{0\}$, $\eps\in(0,1)$, and $\mathcal{H}_{r_2}(z_2)\subset I\times  (B_{\frac{R_1}{2}}({x}_1',0)\cap\{x_n>0\})\times \bbR^n $,
\begin{align}\label{ineq2.diff.calm}
    \langle v_2\rangle^{M}\|\widetilde{\mathcal{M}}\|_{C^{k,\eps}(\gamma_-\cap\mathcal{Q}_{{r_2}}(z_2))}\leq c(p,k,\eps,\Omega,M),
\end{align}
as $\mathcal{M}$ has fast decay as $|v|\to\infty$. Next, for any $k\in\N\cup\{0\}$ and $\eps\in(0,1)$, there is a large constant $p=p(n,k,\eps,\Omega)$ such that if $r_2\leq \frac{\langle v_2\rangle^{-p}}{c_2}$ for some constant $c_2=c_2(n,k,\eps,\Omega)$, then
\begin{align*}
    \|\widetilde{A}\|_{C^{k,\eps}(\mathcal{H}_{r_2}(z_2))}+\frac{1}{\langle v_2\rangle^2}\|\widetilde{B}\|_{C^{k,\eps}(\mathcal{H}_{r_2}(z_2))}\leq c(k,\eps,\Omega),
\end{align*}
where we have used \cite[Lemma 2.17]{RoWe25}. Now using this, \eqref{ineq2.diff.calm}, and \autoref{thm.diff}, we get $f\in C^{k,\eps}(\mathcal{H}_{r_2}(z_2))$, whenever $\mathcal{H}_{r_2}(z_2)\subset \widetilde{H}_\infty$. This implies that $f$ is smooth away from the grazing set, which completes the proof. 
\end{proof}

\subsection{Higher order estimates near $\gamma_0$ and proofs of \texorpdfstring{\autoref{thm:C3}}{Theorem 1.6} and \texorpdfstring{\autoref{thm:f/phi}}{Theorem 1.7}}

Now we are ready to prove our two remaining main theorems.
\begin{proof}[Proof of \autoref{thm:C3} and \autoref{thm:f/phi}.] 
Let us fix $z_1\in\gamma_0$. We have \eqref{after.trans1} and that $\widetilde{f}$ is a weak solution to \eqref{after.trans2} with $A=I_n$, $B=0$, and $z_0=\Psi(z_1)$. In addition, from \eqref{after.trans1} and \eqref{property.varphi}, we have 
\begin{align*}
    \Psi^{-1}({\mathcal{H}}_{\frac{R_0}{2}}(z_0)\cap\{y_n\leq w_n^{\frac3\eps}\})\subset \mathbf{H}_{R_1}(z_1)\cap\mathcal{R}^{\eps}_-
\end{align*}
for some constants $R_0,R_1\leq 1$ depending on $n,|z_0|,\Omega$.
Therefore, for the proof of \autoref{thm:C3}, it suffices to show that 
\begin{align}\label{last.thm16}
    \|\widetilde{f}\|_{C^{3-\eps}(\mathcal{H}_{\frac{R_0}2}(z_0)\cap\{y_n\leq w_n^{\frac3\eps}\})}\leq c\left( \|\widetilde{f}\|_{L^1(\mathcal{H}_{R_0}(z_0))}+[\widetilde{F}]_{C^{\frac52}(\mathcal{H}_{R_0}(z_0))}+[\widetilde{g}]_{C^{3-\eps}(\mathcal{H}_{R_0}(z_0))}\right)
\end{align}
for some constant $c=c(n,\eps,|z_0|,\Omega)$, as the desired estimate follows by scaling back together with \cite[Lemma 2.17]{RoWe25}. By \autoref{thm.3reggamma-}, for any $\mathcal{H}_r(z_2)\subset \mathcal{H}_{\frac{R_0}2}(z_0)\cap\{y_n\leq w_n^{\frac3\eps}\}$,
\begin{equation}\label{ineq.linftytol1}
\begin{aligned}
    \|\widetilde{f}\|_{C^{3-\eps}(\mathcal{H}_{r}(z_2))}&\leq c\left( \|\widetilde{f}\|_{L^{\infty}(\mathcal{H}_{\frac34R_0}(z_0))}+[\widetilde{F}]_{C^{\frac52}(\mathcal{H}_{\frac34R_0}(z_0))}+[\widetilde{g}]_{C^{3-\eps}(\mathcal{H}_{\frac34R_0}(z_0))}\right)\\
    &\leq c\left( \|\widetilde{f}\|_{L^{1}(\mathcal{H}_{R_0}(z_0))}+[\widetilde{F}]_{C^{\frac52}(\mathcal{H}_{R_0}(z_0))}+[\widetilde{g}]_{C^{3-\eps}(\mathcal{H}_{R_0}(z_0))}\right),
\end{aligned}
\end{equation}
where $c=c(n,\eps,|z_0|,\Omega)$, as $\|\widetilde{A}\|_{C^{\frac52}}$, $\|\widetilde{B}\|_{C^{\frac32}}$ and $R_1$ depend only on $n,|z_0|$ and $\Omega$. In particular, we have also used \cite[Lemma 2.25]{RoWe25} and \autoref{lem.rep.hol2} for the last inequality. Now using the covering argument given in \cite[Lemma 2.9]{RoWe25}, we derive \eqref{last.thm16}, which completes the proof of \autoref{thm:C3}.

Now we are going to prove \autoref{thm:f/phi}. Note from \eqref{property.varphi} that $\widetilde{A}_{n,n}=(\nabla\varphi(x))_{n,k}(\nabla\varphi(x))_{n,k}=|\nabla\varphi(x)|^2\equiv 1$. Therefore, by using \autoref{thm.h/phi} and applying the same argument as in \eqref{ineq.linftytol1}, which allows us to replace the $L^\infty$-norm of $\widetilde{f}$ by its $L^1$-norm, we have 
\begin{align}\label{last.f/phi}
    \left\|\frac{\widetilde{f}}{\phi_0}\right\|_{C^{0,1}(\mathcal{H}_{\frac{R_{0}}2}(z_0)\setminus\{(v_n)^3\geq x_n\})}\leq c\left(\|\widetilde{f}\|_{L^{1}(\mathcal{H}_{R_0}(z_0))}+\|\widetilde{F}\|_{L^{\infty}(\mathcal{H}_{R_0}(z_0))}\right)
\end{align}
for some constant $c=c(n,\Lambda,\Omega,|z_0|)$. On the other hand, we observe from \eqref{property.varphi} that $\varphi(y_n,w_n)=\varphi(d_\Omega(x),(n_x\cdot v))$. Thus using this and \eqref{last.f/phi} together with the change of variables yields the desired estimate. This completes the proof.
\end{proof}

\subsection{Generalization of \texorpdfstring{\autoref{thm:C3}}{Theorem 1.6} and \texorpdfstring{\autoref{thm:f/phi}}{Theorem 1.7} to equations with coefficients}

We end this section by providing a general version of \autoref{thm:C3} and \autoref{thm:f/phi}. We consider the broad class of kinetic Fokker-Planck equations as in \eqref{eq:FPK} with sufficiently smooth coefficients. 

First, we state a comprehensive version of \autoref{thm:C3}. It follows directly from \autoref{thm.3reggamma-} by repeating the same argument as in the proof of \autoref{thm:C3}, and thus we skip the details.
\begin{theorem}
\label{thm:C3-coeff}
     Let $\Omega \subset \R^n$ be a bounded smooth domain and let $z_1 \in \gamma_0$. Let $F \in C^{0,1}(\mathbf{H}_1(z_1))$, $A\in C^{\frac52}(\mathbf{H}_1(z_1))$, $B\in C^{\frac32}(\mathbf{H}_1(z_1))$ and $g\in C^{2,1}(\gamma_-\cap \mathcal{Q}_1(z_1))$. Let $f$ be a solution to 
\begin{equation*}
\left\{
\begin{alignedat}{3}
(\partial_t+v\cdot\nabla_x){f}-\ddiv({A}\nabla_v{f})&={B}\cdot\nabla_v{f}+{F}&&\qquad \mbox{in  $\mathbf{H}_1(z_1)$}, \\
f&=g&&\qquad  \mbox{in $\gamma_-\cap\mathcal{Q}_1(z_1)$}.
\end{alignedat} \right.
\end{equation*}    
Then for any $\eps>0$, we have 
\begin{align*}
    [f]_{C^{3-\eps}(\mathcal{H}_{1/2}(z_1)\cap\mathcal{R}^{\eps}_-)}\leq c(\|f\|_{L^1(\mathbf{H}_1(z_1))}+\|F\|_{C^{0,1}(\mathbf{H}_1(z_1))}+[g]_{C^{2,1}(\mathcal{Q}_1(z_1)\cap\gamma_-)})
\end{align*}
for some constant $c=c(n,\eps,\|A\|_{C^{\frac52}(\mathbf{H}_1(z_1))},\|B\|_{C^{\frac32}(\mathbf{H}_1(z_1))},|z_1|,\Omega)$, where $ \mathcal{R}^{\eps}_-\coloneqq\{d_\Omega(x)^{\eps}\leq (n_x\cdot v)^3\}$.
\end{theorem}

Next, we give a general version of \autoref{thm:f/phi}. To this end, we write for any $z\in I_2\times \Omega\times \bbR^n$
\begin{align}
\label{eq:kinetic-dist}
    \mathrm{dist}_{\mathrm{kin}}(z_1)\coloneqq \max\left\{r\,:\,\mathbf{H}_r(z_1)\cap \gamma_0=\emptyset\right\}\eqsim \min\{\mathrm{dist}(z_1,z)\,:\,z\in\gamma_0\}.
\end{align}
We recall that we have assumed that the projection onto the boundary $\Omega$ is unique in $\mathbf{H}_1(z_0)$. If not, then we can just make the radius smaller. Thus, we may write $n_x$ as the outer unit normal vector at the projection $\overline{x}\in\partial\Omega$ of $x\in\Omega$. 

First, we provide a lemma that describes the behavior of the kinetic distance function near $\gamma_0$.
\begin{lemma}\label{lem.dist.asymp}
    Let $\Omega$ be a bounded convex domain of class $C^2$. We fix $z_0\in (I_1\times \Omega\times \bbR^n)\cap \gamma_0$. Then we have for any $z\in\mathbf{H}_1(z_0)$, 
    \begin{align*}
      \frac1c\max\{d_\Omega(x)^{\frac13},\langle v_0\rangle^{-2}|n_x\cdot v|\}\leq \mathrm{dist}_{kin}(z)\leq c\max\{d_\Omega(x)^{\frac13},|n_x\cdot v|\}
    \end{align*}
    for some constant $c=c(n,\Omega)$.   
\end{lemma}
\begin{proof}
First, we note that for any $z_1\in \mathbf{H}_1(z_0)$,  
\begin{align}\label{proj.gamma0}
        \overline{z}_1\coloneqq(t_1,x_1-(n_{x_1}\cdot x_1)n_{x_1},v_1-(n_{x_1}\cdot v_1)n_{x_1})\in\gamma_0,\quad \mathrm{dist}(z_1,\overline{z}_1)\eqsim
        \max\{d_\Omega(x_1)^{\frac13},|n_{x_1}\cdot v_1|\},
    \end{align}
where we have used the fact that $n_{x_1}\cdot x_1=d_\Omega(x_1)$. This implies that 
\begin{align*}
    \mathrm{dist}_{\mathrm{kin}}(z_1)\lesssim  \max\{d_\Omega(x_1)^{\frac13},|n_{x_1}\cdot v_1|\}.
\end{align*}
For the lower bound, first we use the convexity condition of the domain. In particular, we are going to prove that for any $z_2\in \gamma_0$, when $z_1\in \mathbf{H}_{r}(z_2)$, then $d_\Omega(x_1)\leq r^3$. To do this, we observe  
    \begin{align*}
        r^3>|x_1-(x_2+(t_1-t_2)v_2)|>d_{\Omega}(x_1),
    \end{align*}
    as $x_2+(t_1-t_2)v_2\in\bbR^n\setminus\Omega$ by the fact that  $n_{x_2}\cdot v_2=0$ and since the domain is convex. Therefore, if $r<\frac{d_\Omega(x_1)^{\frac13}}{16}$, then $\mathcal{H}_{r}(z_1)\cap\gamma_0=\emptyset$, which implies $\mathrm{dist}_{\mathrm{kin}}(z_1)\gtrsim {d_\Omega(x_1)^{\frac13}}$. 
    
    On the other hand, since the domain $\Omega$ is smooth, there is a universal constant $c_0=c_0(\Omega)\geq1$ such that $|n_{x_1}-n_{x_2}|\leq c_0|x_1-x_2|$. Using this, we can prove for any $r=\frac{|n_{x_1}\cdot v_1|}{100c_0\langle v_0\rangle^2}$,
        \begin{align}\label{R+-gamm0cap}
            \mathcal{H}_r(z_1)\cap\gamma_0=\emptyset.
        \end{align}
         To prove it, we assume by contradiction that there is $z_2\in \mathcal{H}_r(z_1)\cap\gamma_0$. Now note that
        \begin{align*}
            n_{x_2}\cdot v_2=(n_{x_2}-n_{x_1})\cdot v_2+n_{x_1}\cdot (v_2-v_1)+n_{x_1}\cdot v_1.
        \end{align*}
        We next estimate
        \begin{align*}
            |(n_{x_2}-n_{x_1})\cdot v_2+n_{x_1}\cdot(v_2-v_1)|&\leq c_0|x_2-x_1||v_2|+r\\
            &\leq c_0|x_2-(x_1+(t_1-t_2)v_1)||v_2|+|(t_1-t_2)v_1||v_2|+r\\
            &\leq c_0(r^3\langle v_2\rangle +r^2\langle v_1\rangle ^2+r),
        \end{align*}
        where we have used the fact that $\frac{\langle v_0\rangle}2\leq \langle v_i\rangle\leq 2 \langle v_0\rangle$ for each $i=1,2$. Therefore, using this and the fact that $r\leq 1$ and $r\langle v_1\rangle \leq \frac{|n_{x_1}\cdot v_1|}{50c_0\langle v_0\rangle}\leq1$, we have
        \begin{align*}
              |n_{x_2}\cdot v_2|\geq |n_{x_1}\cdot v_1|-c(r^3\langle v_2\rangle +r^2\langle v_1\rangle ^2+r)\geq |n_{x_1}\cdot v_1|-\frac12\frac{|n_{x_1}\cdot v_1|}{\langle v_1\rangle}>0,
        \end{align*}
        This proves \eqref{R+-gamm0cap} and hence $\mathrm{dist}_{\mathrm{kin}}(z_1)\gtrsim \frac{|n_{x_1}\cdot v_1|}{\langle v_0\rangle^2}$. Therefore, we conclude 
        \begin{equation*}
            \mathrm{dist}_{\mathrm{kin}}(z_1)\gtrsim\max\left\{d_\Omega(x_1)^{\frac13},\frac{|n_{x_1}\cdot v_1|}{\langle v_0\rangle^2}\right\},
        \end{equation*}
        which completes the proof.
\end{proof}

Now, we are in a position to state and prove the main result of this section. It says that a variant of \autoref{thm:f/phi} remains true for equations with coefficients. In this case, the function $\Phi$ does not have an elegant explicit formula, however, we can prove that it behaves like $\mathrm{dist}_{\mathrm{kin}}^{1/2}$ up to higher order if $\Omega$ is a convex domain.

\begin{theorem}
\label{thm:f/phi-coeff}
    Let $\Omega \subset \R^n$ be a bounded smooth and convex domain, and let $z_1 \in \gamma_0$. Let $F \in L^{\infty}(\mathbf{H}_1(z_1))$, $A\in C^{k+\eps}(\mathbf{H}_1(z_1))$, and $B\in C^{\eps}(\mathbf{H}_1(z_1))$ for some $k\in\N$ and $\eps>0$. Let $f$ be a solution to 
\begin{equation*}
\left\{
\begin{alignedat}{3}
(\partial_t+v\cdot\nabla_x){f}-\ddiv({A}\nabla_v{f})&={B}\cdot\nabla_v{f}+{F}&&\qquad \mbox{in  $\mathbf{H}_1(z_1)$}, \\
f&=0&&\qquad  \mbox{in $\gamma_-\cap\mathcal{Q}_1(z_1)$}.
\end{alignedat} \right.
\end{equation*}    
Then, there exists a function $\Phi \in C^{1/2}(\mathbf{H}_1(z_1))$ such that
    \begin{align}
    \label{eq:Phi-prop}
        \Phi \asymp_{\langle v_1\rangle} \mathrm{dist}_{\text{kin}}^{1/2}, \qquad D^{\beta} \Phi \lesssim_{|v_1|}\mathrm{dist}_{\text{kin}}^{1/2 - |\beta|}\quad\text{in }\mathbf{H}_1(z_1)\setminus \mathcal{R}^1_-
    \end{align}
    for any $\beta\in \N$ with $|\beta|\leq k$
    and it holds
    \begin{align*}
        \frac{f}{\Phi} \in C^{0,1}(\mathbf{H}_{1/2}(z_1) \setminus \mathcal{R}^1_-)),
    \end{align*}
    where $\mathcal{R}^1_-\coloneqq\{d_\Omega(x)\leq (n_x\cdot v)^3\}$.
    Moreover, we have the estimate
    \begin{align*}
        \Vert f / \Phi \Vert_{C^{0,1}(\mathbf{H}_{1/2}(z_1) \setminus \mathcal{R}_-^{1})} \le C \big( \Vert f \Vert_{L^1(\mathbf{H}_1(z_1))} +  \Vert F \Vert_{L^{\infty}(\mathbf{H}_1(z_1))} \big) ,
    \end{align*}
    where $C$ depends only on $n,|z_1|,\Omega,\eps,\|A\|_{C^{1,\eps}(\mathbf{H}_1(z_1))},\|B\|_{C^{\eps}(\mathbf{H}_1(z_1))}$.
\end{theorem}
\begin{proof}
    As in the proof of \eqref{last.f/phi}, we have for $z_0=\Psi(z_1)$,
    \begin{align*}
        \left\|\frac{\widetilde{f}}{{{\Phi}}}\right\|_{C^{0,1}(\mathcal{H}_{\frac{R_{0}}2}(z_0)\setminus\{(v_n)^3\geq x_n\})}\leq c\left(\|\widetilde{f}\|_{L^{1}(\mathcal{H}_{R_0}(z_0))}+\|\widetilde{F}\|_{L^{\infty}(\mathcal{H}_{R_0}(z_0))}\right),
    \end{align*}
    for some constants $c=c(n,|z_1|,\eps,\|A\|_{C^{1+\eps}(\mathcal{H}_{R_0}(z_0))},\|B\|_{C^{\eps}(\mathcal{H}_{R_0}(z_0))})$ and $R_0=R_0(|z_1|,\Omega)$, where we write  ${{\Phi}}(t,y,w)\coloneqq\phi_0\left(y_n/\sqrt{\widetilde{A}_{n,n}(t,y,w)},w_n/\sqrt{\widetilde{A}_{n,n}(t,y,w)}\right)$. Thus, using the observation $\widetilde{A}_{n,n}(t,y,w)=(\nabla\varphi(x))_{n,k}A_{k,l}(z)(\nabla\varphi(x))_{n,l}$ from \eqref{property.varphi} and the change of variables, we have
\begin{align*}
        \Vert f / \Phi \Vert_{C^{0,1}(\mathbf{H}_{1/2}(z_1) \setminus \mathcal{R}_-^{1})} \le C \big( \Vert f \Vert_{L^1(\mathbf{H}_1(z_1))} +  \Vert F \Vert_{L^{\infty}(\mathbf{H}_1(z_1))} \big),
    \end{align*}
    for some constants $c=c(n,|z_1|,\|A\|_{C^{1,\eps}(\mathbf{H}_{1}(z_1))},\|B\|_{C^{\eps}(\mathbf{H}_{1}(z_1))}),\Omega$. To investigate the behavior $\Phi$, we rewrite $\Phi$ as
    \begin{align*}
        {{\Phi}}(t,x,v)=\phi_0\left(d_\Omega(x)G(z),(n_x\cdot v)G(z)\right)\quad\text{with}\quad G(z)\coloneqq \frac{1}{\sqrt{(\nabla\varphi(x))_{n,k}A_{k,l}(z)(\nabla\varphi(x))_{n,l}}}.
    \end{align*}
    Note that $G(z)\in C^{\infty}(\mathbf{H}_1(z_1))$ with 
    \begin{align}\label{ineq0.kin}
        c(k,\Omega,\Lambda,\|A\|_{C^k(\mathcal{H}_1(z_1))})^{-1}\leq \|G\|_{C^k(\mathbf{H}_1(z_1))}\leq c(k,\Omega,\Lambda,\|A\|_{C^k(\mathcal{H}_1(z_1))}),
    \end{align}
    as $A$ is positive symmetric and $|\nabla\varphi |^2\geq\frac1c$ for some constant $c=c(\Omega)$.
    First, from \eqref{phim.asymptotic} with $m=0$ and \autoref{lem.dist.asymp}, we derive
    \begin{align*}
        {\Phi}(z)\asymp\max\{d_\Omega(x)^{\frac13},|n_x\cdot v|\}^{\frac12} \asymp_{\langle v_1\rangle} \mathrm{dist}^{\frac12}_{\mathrm{kin}}(z)\quad\text{for any }z\in\mathbf{H}_1(z_1)\setminus\mathcal{R}^1_-.
    \end{align*}
    For the second relation given in \eqref{eq:Phi-prop}, we observe that since $\phi_0$ is a weak solution to $(\partial_t+v_n\partial_{x_n})\phi_0-\Delta_{v_n}\phi_0=0$, by \autoref{lem.int.grad}, we have for any $|v_{1,n}|\leq 2(x_{1,n})^{\frac13}$, 
    \begin{align*}
        |D^\beta\phi_0(x_{1,n},v_{1,n})|\leq c\rho^{-\beta}\|\phi_0\|_{L^\infty(\mathcal{H}_\rho(0,x_{1,n},v_{1,n}))},
    \end{align*}
    where $c=c(\beta)$ and $\rho\coloneqq\frac{(x_{1,n})^{\frac13}}{64}$. By \eqref{phi.asymptotic}, we have 
    \begin{align}\label{ineq1.kin}
        |D^\beta\phi_0(x_{1,n},v_{1,n})|\leq c(x_{1,n})^{-\frac{\beta}{3}+\frac16}.
    \end{align}
    Next for any $(-v_{1,n})\geq 2(x_{1,n})^{\frac13}$, by \autoref{lemma.1D.grax.est2} with $v_0$ and $R$ replaced by $v_{1,n}$ and $R=\frac{v_{1,n}}{2}$, respectively, we obtain
    \begin{align}\label{ineq2.kin}
        |D^\beta\phi_0(x_{1,n},v_{1,n})|\leq \|D^\beta\phi_0\|_{L^\infty(H_R(0,v_{1,n}))}\leq c|v_{1,n}|^{-\beta+\frac12},
    \end{align}
    where we have used \eqref{phi.asymptotic}. By \eqref{ineq1.kin} and \eqref{ineq2.kin}, we get 
    \begin{align*}
        |D^\beta\phi_0(x_{1,n},v_{1,n})|\leq c\max\{(x_{1,n})^{\frac13},|v_{1,n}|\}^{-\beta+\frac12}
    \end{align*}
    for some constant $c=c(\beta)$. Now using this and \eqref{ineq0.kin} together with the fact that $d_\Omega(x)$ and $n_x\cdot v$ are smooth, we have 
    \begin{align*}
        |D^\beta \Phi(z)|\leq c \max\{d_\Omega(x)^{\frac13},|(n_x\cdot v)|\}^{-\beta+\frac12},
    \end{align*}
    where $c=c(n,\beta,A,\Omega)$. Combining this and \autoref{lem.dist.asymp} leads to the desired estimate
    \begin{align*}
        D^\beta {\pmb{\Phi}}\leq c\mathrm{dist}_{kin}^{\frac12-\beta}\quad\text{in }\mathbf{H}_1(z_1)\setminus \mathcal{R}_-^1
    \end{align*}
    for some constant $c=c(\beta,n,\Omega,A,\langle v_1\rangle)$,
    which completes the proof.
\end{proof}

\appendix

\section{Functions depending only on time and spatial variables}
Here, we provide some basic facts about functions  $g=g(t,x')$, i.e. which depend only on $t,x'$ (whenever $\Omega = \{x_n > 0\}$), with respect to kinetic H\"older spaces. These results are required for the proof of \autoref{lem.reg.nf.diff}.

\begin{itemize}
    \item For any $\lambda\in\N\cup\{0\}$ and $\eps\in(0,1)$,
\begin{align}\label{fund.lambda.eps}
    [g]_{C^{\lambda,\eps}(\gamma_-\cap\mathcal{Q}_r(z_0))}\leq c\left([\partial_x^lg]_{C_x^{\frac{\lambda+\eps-3l}{3}}(\gamma_-\cap\mathcal{Q}_r(z_0))}+\sum_{j=0}^{l-1}[\partial_x^j\mathcal{D}_t^{k_j}g]_{C^{\lambda+\eps-(2k_j+3j)}(\gamma_-\cap\mathcal{Q}_r(z_0))}\right),
\end{align}
where $\mathcal{D}\coloneqq \partial_t+v'\cdot\nabla_{x'}$, $c=c(n,\lambda,\eps)$, $\lambda-3<3l\leq \lambda$ and $\lambda-2<2k_j+3j\leq \lambda$.

\item For any $\eps\in(0,1)$ and $r\leq1$, 
\begin{align}\label{hol.rel.kin}
    [g]_{C_t^{\frac\eps3}(\mathcal{H}_r(z_0))}+[g]_{C_x^{\frac\eps3}(\mathcal{H}_r(z_0))}\leq c\langle v_0\rangle^{\frac{\eps}3}[g]_{C^{\eps}(\mathcal{H}_r(z_0))},
\end{align}
where $c=c(n,\eps)$ and 
\begin{align*}
    &[g]_{C_t^{\frac\eps3}(\mathcal{H}_r(z_0))}=\sup_{(t,x,v),(\widehat{t},x,v)\in\mathcal{H}_r(z_0)}\frac{|g(t,x',v)-g(\widehat{t},x',v)|}{|t-\widehat{t}|^{\frac\eps3}},\\
    &[g]_{C_x^{\frac\eps3}(\mathcal{H}_r(z_0))}=\sup_{(t,x',v),(t,\widehat{x},v)\in\mathcal{H}_r(z_0)}\frac{|g(t,x',v)-g(t',x',v)|}{|x-\widehat{x}|^{\frac\eps3}}.
\end{align*}
\item For any $\eps\in(0,1)$, we have 
\begin{align}\label{hol.diff.rel}
    \sup_{0<|h|<\frac{r}2}\left\|\frac{g(t+h,x',v)-g(t,x',v)}{|h|^{\frac\eps3}}\right\|_{L^\infty(\mathcal{H}_r(z_0))}\leq c[g]_{C_t^{\frac\eps3}(\mathcal{H}_{2r}(z_0))}.
\end{align}
\end{itemize}
\begin{proof}
To prove \eqref{fund.lambda.eps},  for $v_1'\in\bbR^{n-1}$, $X'=(t,x',v_1')$ and $X_1'=(t_1,x_1',v_1')$, we observe
\begin{align}\label{taylor1.g}
    |(g-p_{X_1'})(X')|\leq |(g-p_{X_1'})(X')-(g-p_{\overline{X}_1'})(X')|+|(g-p_{\overline{X}'_1})(X')|\eqqcolon J_1+J_2,
\end{align}
where $\overline{X}_1'=(t,\overline{x}_1',v_1')$ with $\overline{x}'_1\coloneqq x_1'+(t-t_1)v_1'$, and $p_{X_1'}$ and $p_{\overline{X}_1'}$ are $k$-order Taylor polynomial of $g$ at $X_1'$ and $\overline{X}_1'$, respectively. Note from \cite[(2.3) in Definition 2.3]{ImSi21} that 
\begin{align*}
    p_{(t,x',v')}((t,x',v')\circ (t_2,x_2',0))=\sum_{2i+3j\leq k}\mathcal{D}^i\nabla_{x'}^jg(t,x',v')(t_2)^i\cdot (x_2')^j,
\end{align*}
where $p_{(t,x',v')}$ is $k$-order Taylor polynomial of $g$ at $(t',x',v')$. Using this and the following observations
 $X'=X_1'\circ (t-t_1,x'-\overline{x}_1',0)$ and $X'=\overline{X}_1'\circ (0,x'-\overline{x}_1',0) $, we have 
\begin{equation}\label{taylor2.g}
\begin{aligned}
    J_1&\leq \sum_{j=0}^{l}|x'-\overline{x}_1'|^j|\partial_x^jg(t_1,x_1')-\partial_x^jg(t,\overline{x}_1')-(\sum_{i=1}^{k_j}\partial_x^j\mathcal{D}^ig(t_1,x_1')(t-t_1)^i|,\\
    J_2&\leq |g(t,x')-g(t,\overline{x}_1')-\sum_{j=1}^{m}\partial_x^jg(t,\overline{x}_1')(x-\overline{x}_1')^j|,
\end{aligned}
\end{equation}
where $k-2<2k_j+3j\leq k$ and $k-3< 3m\leq k$. By the fundamental theorem of calculus, we have
\begin{align*}
    J_1+J_2&\leq \sum_{j=0}^{l-1}|t-t_1|^{k_j}|x'-x_1'-(t-t_1')v_1'|^{j+\frac{(k+\eps)-(2k_j+3l)}{2}}[\mathcal{D}^{k_j}\partial_x^jg]_{C^{(k+\eps)-(2k_j+3j)}(\gamma_-\cap\mathcal{Q}_r(z_0))},\\
    &\quad+|x'-x_1'-(t-t_1')v_1'|^{\frac{k+\eps}{3}}[\partial_x^lg]_{C_x^{\frac{k+\eps-3m}{3}}(\gamma_-\cap\mathcal{Q}_r(z_0))},
\end{align*}
which implies \eqref{fund.lambda.eps}.

On the other hand, \eqref{hol.rel.kin} follows from 
\begin{align*}
    \frac{\left(|t-t'|^{\frac12}+|x-x'-(t-t')v|^{\frac13}\right)^{\eps}}{|t-t'|^{\frac{\eps}3}+|x-x'|^{\frac{\eps}3}}\leq c\langle v\rangle^{\frac{\eps}3}.
\end{align*}
This completes the proof.
\end{proof}
\section{Perturbation argument with unbounded right-hand side}\label{rmk.l8delta}
Here, we give the proof of \autoref{lem.sch.divergence} via the perturbation argument based on \cite{KLN25}.
\begin{proof}[Proof of \autoref{lem.sch.divergence}.] We assume that $\mathcal{Q}_1$ is a one-sided cylinder, as the estimates on two-sided cylinders directly follow as in \cite[Lemma 2.23]{RoWe25}. Let us fix $r\leq \frac14$. Let $f_1$ be a weak solution to 
\begin{equation*}
\left\{
\begin{alignedat}{3}
(\partial_t+v\cdot\nabla_x)f_1-\ddiv(A\nabla_vf_1)&=0&&\qquad \mbox{in  $\mathcal{Q}_{{r}}$}, \\
f_1&=f&&\qquad  \mbox{in $\partial_{\mathrm{kin}}\mathcal{Q}_{{r}}$},
\end{alignedat} \right.
\end{equation*}
and let $f_2$ be a weak solution to 
\begin{equation*}
\left\{
\begin{alignedat}{3}
(\partial_t+v\cdot\nabla_x)f_2-\ddiv(A(0)\nabla_vf_2)&=0&&\qquad \mbox{in  $\mathcal{Q}_{\frac{r}{2}}$}, \\
f_2&=f_1&&\qquad  \mbox{in $\partial_{\mathrm{kin}}\mathcal{Q}_{\frac{r}2}$}.
\end{alignedat} \right.
\end{equation*}
Note that $f-f_1$ solves
\begin{equation*}
\begin{aligned}
    (\partial_t+v\cdot\nabla_x)(f-f_1)-\ddiv(A\nabla_v(f-f_1))=B\nabla_vf+F-\ddiv(G)\quad\text{in }\mathcal{Q}_r,
\end{aligned}
\end{equation*}
and satisfies $f-f_1 = 0$ on $\partial_{\mathrm{kin}}\mathcal{Q}_r$. Therefore we get by testing the equation for $f-f_1$ with $f-f_1$ itself, 
\begin{align}\label{firs.comp.est1}
    \dashint_{\mathcal{Q}_r}|\nabla_v(f-f_1)|^2\,dz\leq c\left(r^2\|B\|^2_{L^\infty(\mathcal{Q}_r)}\int_{\mathcal{Q}_r}|\nabla_vf|^2\,dz+r^4\|F\|_{L^\infty(\mathcal{Q}_r)}+r^{2+2\eps}[G]_{C^\eps(\mathcal{Q}_r)}\right),
\end{align}
where we have used $-\ddiv(G)=-\ddiv(G-G(0))$, as well as Poincar\'e's inequality, which yields
\begin{align*}
    \dashint_{\mathcal{Q}_r}|f-f_1|^2\,dz\leq cr^2\dashint_{\mathcal{Q}_r}|\nabla_v(f-f_1)|^2\,dz
\end{align*}
and allows to absorb the term on the right-hand side to the left-hand side in \eqref{firs.comp.est1}.
Next by \cite[Lemma 4.6]{KLN25} and \eqref{firs.comp.est1}, we have 
\begin{equation}\label{sec.comp.est2}
\begin{aligned}
     \dashint_{\mathcal{Q}_{\frac{r}2}}|\nabla_v(f_1-f_2)|^2\,dz\leq cr^{2\eps}\dashint_{\mathcal{Q}_{r}}|\nabla_vf_1|^2\,dz&\leq c(r^{2\eps}+r^2\|B\|_{L^\infty(\mathcal{Q}_r)})\dashint_{\mathcal{Q}_{r}}|\nabla_vf|^2\,dz\\
     &\quad+c\left(r^4\|F\|_{L^\infty(\mathcal{Q}_r)}+r^{2+2\eps}[G]_{C^\eps(\mathcal{Q}_r)}\right),
\end{aligned}
\end{equation}
where $c=c(n,\Lambda,\eps,\|A\|_{C^\eps(\mathcal{Q}_1)})$.
Now using \eqref{firs.comp.est1} and \eqref{sec.comp.est2} and following the same lines as in the proof of \cite[(5.33) in Lemma 5.7]{KLN25} with $w=f_1$, $g=f_2$, and $\alpha=\beta=\eps$, we derive 
\begin{equation*}
\begin{aligned}
    \dashint_{\mathcal{Q}_\rho}|\nabla_vf-(\nabla_vf)_{\mathcal{Q}_\rho}|^2\,dz&\leq c(\rho/r)^{4n+2+2\eps}\dashint_{\mathcal{Q}_r}|\nabla_vf-(\nabla_vf)_{\mathcal{Q}_r}|^2\,dz+c(r^{2\eps}+r^2\|B\|_{L^\infty(\mathcal{Q}_r)})\int_{\mathcal{Q}_r}|\nabla_vf|^2\,dz\\
    &\quad+c\left(r^{4n+2+2\eps}[G]_{C^\eps(\mathcal{Q}_r)}+r^{4n+2+2}\|F\|_{L^\infty(\mathcal{Q}_r)}\right)
\end{aligned}
\end{equation*}
for some constant $c=c(n,\Lambda,\eps,\|A\|_{C^\eps(\mathcal{Q}_1)})$. Now by applying a technical iteration lemma as in \cite[Lemma 5.7]{KLN25}, we get 
\begin{align}\label{gra.bdd.est.sch}
    \|\nabla_vf\|_{L^\infty(\mathcal{H}_{r})}\leq c(\|\nabla_vf\|_{L^2(\mathcal{Q}_{\frac12})}+\|F\|_{L^\infty(\mathcal{Q}_{\frac12})}+[G]_{C^\eps(\mathcal{Q}_{\frac12})}).
\end{align}
On the other hand, by \autoref{lem.int.grad} and \cite{GIMV19}, we have for any $\delta\in(0,1)$ and any affine function $l=a+b\cdot v$,
\begin{align*}
    [f_2]_{C_{\ell}^{1+\delta}(\mathcal{Q}_{\frac{r}{4}})}\leq cr^{-(1+\delta)}\|f_2-l\|_{L^\infty(\mathcal{Q}_{\frac{3r}{8}})}\leq cr^{-(1+\delta)}\dashint_{\mathcal{Q}_{\frac{r}{2}}}|(f_2-l)|^2\,dz,
\end{align*}
where $c=c(n,\Lambda,\delta)$, as $f_2-l$ is also a weak solution. Therefore, by taking $l=f_2(0)+\nabla_vf_2(0)\cdot v$ and using the definition of the kinetic H\"older space $C^{1+\delta}$, we derive for any $\rho\in(0,\frac12)$,
\begin{align}\label{c1eps.est}
    \dashint_{\mathcal{Q}_{\rho r}}|f_2-l|^2\,dz\leq c\rho^{2(1+\delta)}\dashint_{\mathcal{Q}_{\frac{r}{2}}}|f_2-l|^2\,dz.
\end{align}
Now we want to estimate the following terms.
\begin{align*}
    \dashint_{\mathcal{Q}_{\rho r}}|f-l|^2\,dz\leq c\dashint_{\mathcal{Q}_{\rho r}}|f-f_1|^2\,dz+c\dashint_{\mathcal{Q}_{\rho r}}|f_1-f_2|^2\,dz+c\dashint_{\mathcal{Q}_{\rho r}}|f_2-l|^2\,dz\eqqcolon \sum_{i=1}^{3}J_i.
\end{align*}
Using Poincar\'e's inequality and \eqref{firs.comp.est1}, we estimate $J_1$ as 
\begin{align*}
    J_1&\leq \rho^{-(4n+2)}\dashint_{\mathcal{Q}_r}|f-f_1|^2\,dz\\
    &\leq \rho^{-(4n+2)}r^2\dashint_{\mathcal{Q}_r}|\nabla_v(f-f_1)|^2\,dz\\
    &\leq c\rho^{-(4n+2)}r^2\left(r^2\|B\|^2_{L^\infty(\mathcal{Q}_r)}\int_{\mathcal{Q}_r}|\nabla_vf|^2\,dz+r^4\|F\|_{L^\infty(\mathcal{Q}_r)}+r^{2+2\eps}[G]_{C^\eps(\mathcal{Q}_r)}\right).
\end{align*}
The estimate of $J_2$ follows directly from \eqref{sec.comp.est2}. For the term $J_3$, we use \eqref{c1eps.est} to see that
\begin{align*}
    J_3\leq c\rho^{2(1+\delta)}\dashint_{\mathcal{Q}_r}|f_2-l|^2\,dz\leq c\rho^{2(1+\delta)}\left(J_1+J_2+\dashint_{\mathcal{Q}_r}|f-l|^2\,dz\right).
\end{align*}
Altogether, we deduce
\begin{align*}
    \dashint_{\mathcal{Q}_{\rho r}}|f-l|^2\,dz&\leq c\rho^{2(1+\delta)}\dashint_{\mathcal{Q}_r}|f-l|^2\,dz+c\rho^{-(4n+2)}(r^4\|B\|^2_{L^\infty(\mathcal{Q}_r)}+r^{2+2\eps})\|\nabla_vf\|_{L^\infty(\mathcal{Q}_r)}\\
    &\quad+c\rho^{-(4n+2)}r^4\|F\|_{L^\infty(\mathcal{Q}_r)}+r^{2+2\eps}[G]_{C^\eps(\mathcal{Q}_r)}.
\end{align*}
Now we choose $\delta=\frac{1+\eps}{2}$ and $\rho=\rho(n,\Lambda,\eps)$ such that
\begin{align}\label{rho.cho.sch}
    \rho^{2+2\delta-(2+2\eps)}=\rho^{\frac{1-\eps}{2}}\leq \frac{1}{4c}.
\end{align}
Then, as in the proof of \cite[Lemma 5.6]{KLN25} together with \eqref{gra.bdd.est.sch}, we get for any $r\leq \frac14$
\begin{align}\label{l2.c12}
    \dashint_{\mathcal{Q}_r}\frac{|f-l|^2}{r^{2(1+\eps)}}\,dz\leq c(\|\nabla_vf\|_{L^2(\mathcal{Q}_{\frac12})}+\|F\|_{L^\infty(\mathcal{Q}_{\frac12})}+[G]_{C^\eps(\mathcal{Q}_{\frac12})}),
\end{align}
where $c=c(n,\Lambda,\eps,\|A\|_{C^\eps(\mathcal{Q}_1)},\|B\|_{L^\infty(\mathcal{Q}_1)})$. Lastly, note that $\widetilde{f}\coloneqq f-l$ solves
\begin{align*}
    (\partial_t+v\cdot\nabla_v)\widetilde{f}-\ddiv(A\nabla_v\widetilde{f})=B\cdot\nabla_v\widetilde{f}+\widetilde{F}-\ddiv(\widetilde{G}),
\end{align*}
where $\widetilde{F}\coloneqq F+B\cdot\nabla_vf(0)$ and $\widetilde{G}\coloneqq G-A\nabla_vf(0)$. Thus, using \cite{GIMV19}, we derive
\begin{align*}
    \|\widetilde{f}\|^2_{L^\infty(\mathcal{Q}_{\frac{r}2})}&\leq c(\dashint_{\mathcal{Q}_r}|\widetilde{f}|^2\,dz+r^{4}\|\widetilde{F}\|_{L^\infty(\mathcal{Q}_r)}+r^{2+2\eps}[\widetilde{G}]_{C^{\eps}(\mathcal{Q}_r)}).
\end{align*}
Now combining this, \eqref{gra.bdd.est.sch}, and \eqref{l2.c12} together with energy estimates implies that for any $r\leq \frac14$,
\begin{align*}
    \|{f}-l\|_{L^\infty(\mathcal{Q}_{\frac{r}2})}\leq cr^{1+\eps}(\|f\|_{L^1(\mathcal{Q}_{\frac12})}+\|F\|_{L^\infty(\mathcal{Q}_{\frac12})}+[G]_{C^\eps(\mathcal{Q}_{\frac12})}).
\end{align*}
This yields the desired estimate.
\end{proof}

\section{Properties of 1D solutions with in-flow condition}
\label{sec:appendix}

The goal of this section is to obtain an infinite countable family of homogeneous solutions to Kolmogorov's equation in 1D in the half-space with absorbing and in-flow boundary condition and to establish some fine properties of these solutions. For our main theorem, we only rely on these results for $\phi_0,\phi_{\pm 1}, \psi_0$. Although some of the results might have already been established in the literature, we decided to give proofs that are as self-contained as possible since we believe that such a presentation might be of independent interest.

Our analysis is based on the representation of solutions in terms of Tricomi confluent hypergeometric functions $U(a,b,z)$. First, let us recall several basic properties of $U(a,b,z)$, which can be found in \cite{DLMF} and \cite{Abramowitz}:
\begin{align}
\label{eq:multiple-U}
    U'(a,b,z)=-aU(a+1,b+1,z),
\end{align}
\begin{equation}\label{eq.U}
    z(a+1)U(a+2,b+2,z)+(z-b)U(a+1,b+1,z)-U(a,b,z)=0,
\end{equation}
\begin{align}
\label{eq:der-U-formula}
    z U'(a,b,z) = a \big[ (1+a-b) U(a+1,b,z) - U(a,b,z) \big],
\end{align}
\begin{equation}\label{eq2.U}
    a(a-b+1)U(a+1,b,z)+(b-2a-z)U(a,b,z)+U(a-1,b,z)=0,
\end{equation}
\begin{equation}\label{asymp1.U}
    U(a,b,z)=\frac{\Gamma(1-b)}{\Gamma(a-b+1)}+O(z^{1-b})\quad \text{ as }  z \to 0, \quad \text{if }0<b<1, \quad
\end{equation}
\begin{equation}\label{asymp2.U}
    U(a,b,z)=\frac{\Gamma(b-1)}{\Gamma(a)}z^{1-b}+\frac{\Gamma(1-b)}{\Gamma(a-b+1)}+O(z^{2-b}),\quad \text{ as }  z \to 0, \quad \text{if }1 < b<2,
\end{equation}
\begin{equation}\label{asymp3.U}
    U(a,b,z)=\frac{\Gamma(b-1)}{\Gamma(a)}z^{1-b}+O(z^{2-b}),\quad \text{ as }  z \to 0, \quad \text{if }2 < b,
\end{equation}
\begin{equation}\label{asymp4.U}
U(a,b,z)\leq  c(a,b)z^{-a}\quad\text{when }z\gg1,
\end{equation}
\begin{equation}\label{explicit.U}
    U(a,b,z)=\frac{1}{\Gamma(a)}\int_{0}^\infty e^{-z\xi}\xi^{a-1}(1+\xi)^{b-a-1}\,d\xi\quad\text{for any }a > 0, \quad z>0,
\end{equation}
\begin{equation}\label{aplus.U}
    U(a,b,z)=z^{1-b}U(a-b+1,2-b,z),
\end{equation}

Moreover, note that $U(a,b,z)$ is analytic on the complex plane cut along $(-\infty,0]$.

\subsection{1D solutions with absorbing boundary condition}

Using these formulas, we find a family of solutions to Kolmogorov's equation in 1D with absorbing boundary condition. Note that the function $\phi_0$ has already appeared in the literature; see, for instance, \cite{GJW99}.

\begin{lemma}
\label{lemma:1D-sol-Tricomi}
    For any integer $m\geq-1$, let us define 
    \begin{equation}\label{rel.phiandU}
    \begin{aligned}
        \phi_m(x,v)\coloneqq \begin{cases}
            c_mx^{\lambda_m}U\left(-\frac16-m,\frac23,-\frac{v^3}{9x}\right)&\quad\text{if }v<0,\\
            c_mM_mx^{\lambda_m}e^{-\frac{v^3}{9x}}U\left(\frac56+m,\frac23,\frac{v^3}{9x}\right)&\quad\text{if }v\geq0,
        \end{cases}
    \end{aligned}
    \end{equation}
    where we write 
    \begin{equation}\label{cm.defn}
    \begin{aligned}
        c_m\coloneqq\begin{cases}
            1&\quad\text{if }m=-1,\\
            (-1)^m&\quad\text{if }m\geq0,
        \end{cases}
    \end{aligned}
    \end{equation}
    and
\begin{equation*}
    \lambda_m\coloneqq 1/6+m \quad\text{and}\quad M_m=\frac{\Gamma(7/6+m)}{\Gamma(1/6-m)}.
\end{equation*}
Then $\phi_m$ is a classical solution to
\begin{equation}\label{eq.1d.absorb}
\left\{
\begin{alignedat}{3}
v\partial_x\phi_m-\partial_{vv}\phi_m&=0&&\qquad \mbox{in  $\{x>0\}\times \bbR$}, \\
\phi_m&=0&&\qquad  \mbox{in $ \{x=0\}\times \{v>0\}$}.
\end{alignedat} \right.
\end{equation}
In addition, for any $m\geq0$, $\phi_m$ are weak solutions to \eqref{eq.1d.absorb}
 and $\phi_{-1}$ is a weak solution to 
\begin{align}\label{eq.1d.phi-1}
    v\partial_x\phi_{-1}-\partial_{vv}\phi_{-1}=0\quad\text{in }((1/R^3,\infty)\times \bbR)\cup ((0,\infty)\times (-\infty,-1/R))
\end{align}
for any $R>0$.
\end{lemma}

\begin{remark}
Note that because of the constant $c_m$, it holds $\phi_m\geq0$ when $v\geq0$.
\end{remark}

\begin{remark}
The reader might wonder, whether it is possible to find solutions to \eqref{eq.1d.absorb} with another homogeneity than $\lambda_m = \frac{1}{6} + m$. We show in \autoref{lem.liou2} that this is impossible. However, note that by replacing the values $\lambda_m$ by $a$, as well as $-\frac{1}{6} - m$ and $\frac{5}{6} + m$ in \eqref{rel.phiandU} by $-a < 0$ and $\frac{2}{3} + a$, respectively, the corresponding function $\phi$ will still solve \eqref{eq.1d.absorb} except for $\{ v = 0 \}$. However, it is impossible to find $M$ such that $\phi$ is $C^1$ across $\{ v = 0 \}$. More precisely, if we consider a function 
    \begin{align}
    \label{phi.remark}
        \phi(x,v)\coloneqq\begin{cases}
            x^{a}U\left(-a,\frac23,-\frac{v^3}{9x}\right)& \quad\text{if }v\leq0,\\
           M x^{a}e^{-\frac{v^3}{9x}}U\left(\frac23+a,\frac23,\frac{v^3}{9x}\right)&\quad\text{if }v>0
        \end{cases}
    \end{align}
    for some constant $M>0$, then $\phi$ satisfies \eqref{eq.1d.absorb} except for $\{ v = 0 \}$. However, to ensure $\phi\in C^1$ at $\{v=0\}$, the constant $M$ has to satisfy
    \begin{align*}
        \lim_{v\nearrow 0}\phi(x,v)=\lim_{v\nearrow 0}x^a\frac{\Gamma(1/3)}{\Gamma(-a+1/3)}=\lim_{v\searrow 0}Mx^a\frac{\Gamma(1/3)}{\Gamma(a+1)}=\lim_{v\searrow 0}\phi(x,v),
    \end{align*}
    and
    \begin{align*}
        \lim_{v\nearrow 0}\partial_v\phi(x,v)&=\lim_{v\nearrow 0}ax^a\left(-\frac{v^2}{3x}\right)\left(\frac{(-v)^3}{9x}\right)^{-\frac23}\frac{\Gamma(2/3)}{\Gamma(-a+1)}\\
        &=\lim_{v\searrow 0}-Mx^a\left(\frac23+a\right)\left(\frac{v^2}{3x}\right)\left(\frac{v^3}{9x}\right)^{-\frac23}\frac{\Gamma(2/3)}{\Gamma(a+5/3)}=\lim_{v\searrow 0}\partial_v\phi(x,v),
    \end{align*}
    where we have used \eqref{asymp1.U} and \eqref{asymp2.U}.
    This is equivalent to 
    \begin{align}\label{ineq.equiva}
        \frac{\Gamma(a+1)}{\Gamma(-a+1/3)}=\frac{a}{\frac23+a}\frac{\Gamma(a+5/3)}{\Gamma(-a+1)}=-\frac{\Gamma(a+2/3)}{\Gamma(-a)}.
    \end{align}
    We now consider the following three cases.
    \begin{itemize}
        \item $a,a+\frac23\notin \Z$. Then by the fact that 
        \begin{equation}\label{gamma.multiple}
            \mbox{$\Gamma(\xi)\Gamma(1-\xi)=\frac{\pi}{\sin(\pi \xi)}$ when $\xi\notin\Z$,}
        \end{equation} we have 
        \begin{align*}
            -\Gamma(1-(-a))\Gamma(-a)=-\frac{\pi}{\sin(\pi(-a))}\quad\text{and}\quad \Gamma(1-(a+2/3))\Gamma(a+2/3)=\frac{\pi}{\sin(\pi(a+2/3))}.
        \end{align*}
        To get \eqref{ineq.equiva}, it should be $(a+a+2/3)=1+2k$ for any $k\geq0$, which gives $a=k+1/6$.
        \item $a\in\Z$ and $a+\frac23\notin \Z$. Then we get
        \begin{align*}
            \Gamma(1-(-a))\Gamma(-a)\in \Z\quad\text{and}\quad \Gamma(1-(a+2/3))\Gamma(a+2/3)=\frac{\pi}{\sin(\pi(a+2/3))}=\frac{\pi}{\pm \sqrt{3}/2}\notin \Z,
        \end{align*}
        which implies that \eqref{ineq.equiva} never happens.
        \item $a\notin\Z$ and $a+\frac23\in \Z$. Similarly, we get that \eqref{ineq.equiva} is impossible.
    \end{itemize}
    Hence, we find that for all solutions to \eqref{eq.1d.absorb} of the form \eqref{phi.remark}, it must be  $a=\frac16+m$.
\end{remark}

 Special care is required in the proof of \autoref{lemma:1D-sol-Tricomi} due to the branch cut of $U$ in $(-\infty,0]$.

\begin{proof}[Proof of \autoref{lemma:1D-sol-Tricomi}]
First, we claim that for any $m\geq-1$, the functions
\begin{align*}
    U\left(-\frac16-m,\frac23,z\right)\quad\text{and}\quad  e^zU\left(\frac56+m,\frac23,-z\right)
\end{align*}
solve Kummer's equation 
\begin{equation}\label{eq.kummer}
    \left(\frac16+m\right)\Psi_m(z)+\left(\frac23-z\right)\Psi_m'(z)+z\Psi_m''(z)=0.
\end{equation}
in $z > 0$ and $z < 0$, respectively.
This seems to be a standard fact for the first function (see \cite[13.2.6]{DLMF}) but is less standard for the second one. Hence, we compute it by hand. First, we observe that by \eqref{eq:multiple-U}
\begin{align*}
    \frac{d}{dz}\left(e^zU\left(a,b,-z\right)\right)&=e^zU\left(a,b,-z\right)-e^zU'\left(a,b,-z\right)\\
    &=e^z\big(U\left(a,b,-z\right)+aU\left(a+1,b+1,-z\right)\big),
\end{align*}
and
\begin{align*}
     \frac{d^2}{dz^2}\big(e^zU\left(a,b,-z\right)\big)&=e^z\big(U\left(a,b,-z\right)+2aU\left(a+1,b+1,-z\right)-aU'(a+1,b+1,-z)\big)\\
     &=e^z\big(U\left(a,b,-z\right)+2aU\left(a+1,b+1,-z\right)+a(a+1)U(a+2,b+2,-z)\big).
\end{align*}
We now set $a_m=\frac56+m$ and $b=\frac23$ and compute the left-hand side in \eqref{eq.kummer} with $\Psi_m(z)=e^zU\left(a_m,b,-z\right)$. This yields,
\begin{align*}
&\left(\frac16+m\right)e^zU\left(a_m,b,-z\right)+\left(\frac23-z\right)\left(e^zU\left(a_m,b,-z\right)\right)'+z\left(e^zU\left(a_m,b,-z\right)\right)''\\
    &=e^z\left[\left(\frac16+m\right)U\left(a_m,b,-z\right)+\left(\frac23-z\right)\big(U\left(a_m,b,-z\right)+a_mU\left(a_{m+1},b+1,-z\right)\big)\right]\\
    &\quad+z e^z \big[ U\left(a_m,b,-z\right)+2a_mU\left(a_{m+1},b+1,-z\right)+a_ma_{m+1}U\left(a_{m+2},b+2,-z\right)\big]\\
    &=e^z\left(a_mU\left(a_m,b,-z\right)+a_m\left(\frac23+z\right)U\left(a_{m+1},b+1,-z\right)+za_ma_{m+1}U\left(a_{m+2},b+2,-z\right)\right)\\
    &= - a_m e^{-z}\left( - U\left(a_m,b,-z\right) + \left(-z - \frac23\right)U\left(a_m+1,b+1,-z\right)+(-z) (a_m+1)U\left(a_{m}+2,b+2,-z\right)\right) \\
    & = 0,
\end{align*}
where we have applied \eqref{eq.U} in the last line.
This implies that $e^zU\left(a_m,b,-z\right)$ is a solution to \eqref{eq.kummer} for $z \in (-\infty,0)$. Similarly, one can show that also $U(-\frac16-m,\frac23,z)$ is a solution to \eqref{eq.kummer} for $z \in (0,\infty)$.

It follows from an easy computation (see e.g. \cite{RoWe25}) that $x^{\frac16+m}\Psi_m(-v^3/(9x))$ is a solution to $v\partial_xg-\partial_{vv}g=0$, whenever $\Psi_m$ solves \eqref{eq.kummer}. Hence, we get that
\begin{equation}\label{defn.gm1.trico}
    g_{m,1}(x,v)\coloneqq x^{\frac16+m}U\left(-\frac16-m,\frac23,-\frac{v^3}{9x}\right) \quad \text{for } (x,v) \in D_1 := \{x>0\}\times \{v<0\},
\end{equation}
and 
\begin{equation}\label{defn.gm2.trico}
    g_{m,2}(x,v)\coloneqq M_mx^{\frac16+m}e^{-\frac{v^3}{9x}}U\left(\frac56+m,\frac23,\frac{v^3}{9x}\right)\quad \text{for } (x,v) \in D_2 := \{x>0\}\times \{v>0\}
\end{equation}
are solutions to
\begin{equation*}
    v\partial_xg_{m,i}-\partial_{vv}g_{m,i}=0 \quad\text{in }D_i
\end{equation*}
for each $i \in \{1,2\}$.

In order to conclude that $\phi_m$ defined in \eqref{rel.phiandU} are classical solutions, it remains to check that
\begin{align}\label{derative.gi}
    g_{m,1}(x,0)=g_{m,2}(x,0), \quad \partial_xg_{m,1}(x,0)=\partial_xg_{m,2}(x,0), \quad \partial_vg_{m,1}(x,0)=\partial_vg_{m,2}(x,0),
\end{align}
and
\begin{align}\label{dderative.gi}
    \partial_{vv}g_{m,1}(x,0)=\partial_{vv}g_{m,2}(x,0)=0.
\end{align}
\begin{itemize}
    \item For the first claim in \eqref{derative.gi} we recall $\lambda_m=1/6+m$ and $M_m = \Gamma(7/6 + m)/\Gamma(1/6-m)$. By \eqref{asymp1.U} and $\Gamma(z+1) = z \Gamma(z)$,
\begin{align*}
    g_{m,1}(x,0)=x^{\lambda_m}\frac{\Gamma(1/3)}{\Gamma(1/6-m)}=x^{\lambda_m}\frac{\Gamma(1/3)}{\Gamma(7/6+m)}M_m=g_{m,2}(x,0).
\end{align*}
\item For the second claim in \eqref{derative.gi} we recall $a_m=5/6+m$ and $b=2/3$.
Next, using \eqref{asymp1.U}, \eqref{asymp2.U}, and \eqref{eq:multiple-U}, we derive
\begin{align*}
    \partial_x g_{m,1}(x,v)&=\lambda_mx^{\lambda_{m}-1}U\left(-\lambda_m,b,-\frac{v^3}{9x}\right)+x^{\lambda_m}\frac{v^3}{9x^2}U'\left(-\lambda_m,b,-\frac{v^3}{9x}\right)\\
    &=\lambda_mx^{\lambda_m-1}\frac{\Gamma(1/3)}{\Gamma(-\lambda_m-b+1)}+O(v).
\end{align*}
Similarly, we get
\begin{align*}
    \partial_x g_{m,2}(x,v)=M_m\lambda_m x^{\lambda_m-1}e^{-\frac{v^3}{9x}}U\left(a_m,b,-z\right)+O(v),
\end{align*}
which implies 
\begin{align*}
    \partial_xg_{m,1}(x,0)=\frac{\Gamma(1/3)}{\Gamma(1/6-m)}\lambda_mx^{\lambda_m-1}=\partial_xg_{m,2}(x,0).
\end{align*}
\item For the third claim in \eqref{derative.gi} we observe that by \eqref{asymp1.U}, \eqref{asymp2.U}, and \eqref{eq:multiple-U},
\begin{equation}\label{appen.defn.j1}
\begin{aligned}
    \partial_v g_{m,1}(x,v)=-x^{\lambda_m}\frac{v^2}{3x}U'\left(-\lambda_m,b,-\frac{v^3}{9x}\right)&={\underbrace{-\lambda_mx^{\lambda_m}\frac{v^2}{3x}U\left(-\lambda_m+1,b+1,-\frac{v^3}{9x}\right)}_{J_1}}\\
    &=-\lambda_m\frac{\Gamma(b)}{\Gamma(-\lambda_m+1)}{x^{\lambda_m}}\frac{v^2}{3x}\left(-\frac{v^3}{9x}\right)^{-\frac23}+O(v),
\end{aligned}
\end{equation}
and
\begin{equation}\label{appen.defn.j2}
\begin{aligned}
    \partial_v g_{m,2}(x,v)&={\underbrace{M_mx^{\lambda_m}\left[\left(-\frac{v^2}{3x}\right)e^{-\frac{v^3}{9x}}U\left(a_m,b,\frac{v^3}{9x}\right)+e^{-\frac{v^3}{9x}}\frac{v^2}{3x}U'\left(a_m,b,\frac{v^3}{9x}\right)\right]}_{J_2}}\\
    &=(-a_m)M_mx^{\lambda_m}\frac{v^2}{3x}e^{-\frac{v^3}{9x}}U\left(a_m+1,b+1,\frac{v^3}{9x}\right) + O(v)\\
    &=-M_ma_m\frac{\Gamma(b)}{\Gamma(a_m+1)}x^{\lambda_m}\frac{v^2}{3x}\left(\frac{v^3}{9x}\right)^{-\frac23}e^{-\frac{v^3}{9x}}+O(v).
\end{aligned}
\end{equation}
Before we compare $\partial_v g_{m,1}(x,0)$ and $\partial_v g_{m,2}(x,0)$, we note 
\begin{align*}
    -a_l=(-5/6-l)=(1/6-(l+1))=\lambda_{-(l+1)}\quad\text{for any integer }l.
\end{align*}
In addition, we observe 
\begin{align*}
    \Gamma(a_{m} + 1)=a_m\Gamma(a_m)=\cdots= \prod_{i=-m}^{m}a_i\Gamma(a_{-m})=\prod_{i=-m}^{m}a_i\Gamma(-\lambda_{m-1})=\prod_{i=-m}^{m}a_i\Gamma(-\lambda_{m}+1),
\end{align*}
and
\begin{align*}
    \Gamma(7/6+m)=\Gamma(\lambda_{m+1})=\prod_{i=-m}^{m}\lambda_i\Gamma(\lambda_{-m})=\prod_{i=-m}^{m}\lambda_i\Gamma(1/6-m),
\end{align*}
which implies
\begin{align*}
    \lambda_m \frac{\Gamma(b)}{\Gamma(-\lambda_m + 1)} = \left( \prod_{i = -m}^m \lambda_i \right) a_m \frac{\Gamma(b)}{\Gamma(-\lambda_m + 1)} \left( \prod_{i = -m}^m a_i \right)^{-1}  = M_m a_m \frac{\Gamma(b)}{\Gamma(a_m+1)}.
\end{align*}
Thus, we deduce
\begin{align*}
    \partial_vg_{m,1}(x,0)=-\lambda_mx^{\lambda_m-\frac13}3^{\frac13}\frac{\Gamma(b)}{\Gamma({-\lambda_m+1})} = - M_m a_m x^{\lambda_m-\frac13}3^{\frac13}\frac{\Gamma(b)}{\Gamma({a_m+1})} =\partial_vg_{m,2}(x,0).
\end{align*}
\item To prove \eqref{dderative.gi} we use \eqref{asymp1.U}, \eqref{asymp2.U}, and \eqref{eq:multiple-U}. For $v < 0$, we recall $J_2$ and observe that
\begin{align*}
    \partial_{vv}g_{m,1}(x,v)&=-\lambda_mx^{\lambda_m}\frac{2v}{3x}U\left(-\lambda_m+1,b+1,-\frac{v^3}{9x}\right)+\lambda_mx^{\lambda_m}\frac{v^2}{3x}\frac{v^2}{3x}U'\left(-\lambda_m+1,b+1,-\frac{v^3}{9x}\right)\\
    &=-\lambda_mx^{\lambda_m}\frac{2v}{3x}\frac{\Gamma(b)}{\Gamma(-\lambda_m+1)}\left(-\frac{v^3}{9x}\right)^{-\frac23}+O(v)\\
    &\quad+\lambda_m(\lambda_m-1)x^{\lambda_m}\frac{v^4}{(3x)^2}U\left(-\lambda_m+2,b+2,-\frac{v^3}{9x}\right)\\
    &=-\lambda_mx^{\lambda_m}\frac{2v}{3x}\frac{\Gamma(b)}{\Gamma(-\lambda_m+1)}\left(-\frac{v^3}{9x}\right)^{-\frac23}+O(v)\\
    &\quad+\lambda_m(\lambda_m-1)x^{\lambda_m}\frac{v^4}{(3x)^2}\frac{\Gamma(b+1)}{\Gamma(-\lambda_m+2)}\left(-\frac{v^3}{9x}\right)^{-\frac53}\\
    &=O(v).
\end{align*} 
For $v > 0$, by using $J_1$ and similar arguments,
\begin{align*}
    \partial_{vv}g_{m,2}(x,v)&=O(v)+M_mx^{\lambda_m}e^{-\frac{v^3}{9x}}\left[\frac{2v}{3x}U'\left(a_m,b,\frac{v^3}{9x}\right)+\frac{v^2}{3x}\frac{v^2}{3x}U''\left(a_m,b,\frac{v^3}{9x}\right)\right]\\
    &=O(v)-M_mx^{\lambda_m}e^{-\frac{v^3}{9x}}a_m\frac{2v}{3x}U\left(a_m+1,b+1,\frac{v^3}{9x}\right)+\\
    &\quad+M_mx^{\lambda_m}e^{-\frac{v^3}{9x}}a_m(a_m+1)\frac{v^4}{9x^2}U\left(a_m+2,b+2,\frac{v^3}{9x}\right)\\
    &=O(v)+M_mx^{\lambda_m}e^{-\frac{v^3}{9x}}a_m\left[-\frac{2v}{3x}\frac{\Gamma(b)}{\Gamma(a_m+1)}\left(\frac{v^3}{9x}\right)^{-\frac23}+a_{m+1}\frac{v^4}{9x^2}\frac{\Gamma(b+1)}{\Gamma(a_m+2)}\left(\frac{v^3}{9x}\right)^{-1-\frac23}\right]\\
    &=O(v).
\end{align*}
Therefore, we have verified \eqref{dderative.gi}.
\end{itemize}
Hence, we conclude that the function $\phi_m$ is a classical solution to 
\begin{equation}\label{pde.phim}
    v\partial_x\phi_m-\partial_{vv}\phi_m=0\quad\text{in }\{x>0\}\times \bbR.
\end{equation}
Moreover, we are able to show that the $\phi_m$ satisfy the absorbing boundary condition, i.e. we prove
\begin{align}\label{int.phim}
    \phi_{m}(0,v)=0\quad\text{if }v>0.
\end{align}
To do this, first, we note from \eqref{explicit.U} that for $m \ge 0$
     \begin{equation*}
    \begin{aligned}
        \phi_m(x,v)=(-1)^m\frac{M_m}{\Gamma(a_m)}x^{\lambda_m}e^{-\frac{v^3}{9x}}\int_{0}^{\infty}e^{-\frac{v^3}{9x}\xi}\xi^{a_m-1}(1+\xi)^{b-a_m-1}\,d\xi\quad\text{if }v\geq0.
    \end{aligned}
    \end{equation*}
    After a change of variables, we deduce 
    \begin{equation}\label{explic.phim.pos}
        \begin{aligned}
            \phi_m(x,v)=(-1)^m\frac{M_m}{\Gamma(a_m)}x^{\lambda_m}e^{-\frac{v^3}{9x}}\left(\frac{v^3}{9x}\right)^{\frac13}\int_{0}^{\infty}e^{-\xi}\xi^{a_m-1}\left(\frac{v^3}{9x}+\xi\right)^{b-a_m-1}\,d\xi.
        \end{aligned}
    \end{equation}
Therefore, when $v^3\geq x$, we get since $b-a_m-1 < 0$ for any $m \ge 0$,
\begin{align}\label{ineq.upper.phim}
    |\phi_m(x,v)| \leq cx^{\lambda_m}e^{-\frac{v^3}{9x}}\left(\frac{v^3}{9x}\right)^{\frac13+b-a_m-1}\int_{0}^{\infty}e^{-\xi}\xi^{a_m-1}\,d\xi\leq c{x^{\lambda_m}}\left(\frac{v^3}{x}\right)^{-a_m}e^{-\frac{v^3}{9x}}.
\end{align}
Similarly, for $m= -1$, we deduce from \eqref{rel.phiandU}, \eqref{explicit.U}, and \eqref{aplus.U}
\begin{align*}
    \phi_{-1}(x,v) &= M_{-1}x^{-\frac56}e^{-\frac{v^3}{9x}}\left(\frac{v^3}{9x}\right)^{\frac13}U\left(\frac16,\frac43,\frac{v^3}{9x}\right) \\
    &=\frac{M_{-1}}{\Gamma(1/6)}x^{-\frac56}e^{-\frac{v^3}{9x}}\int_{0}^\infty e^{-\xi}\xi^{-\frac56}\left(\frac{v^3}{9x}+\xi\right)^{\frac16}\,d\xi \quad \text{if } v > 0.
\end{align*}

For $v^3 \ge x$, we get
\begin{align}
\label{eq:phi-1-bdry}
\phi_{-1}(x,v)\leq cx^{-\frac56}e^{-\frac{v^3}{9x}}\left[\left(\frac{v^3}{x}\right)^{\frac16}\int_0^{\frac{v^3}{9x}}e^{-\xi}\xi^{-\frac56}\,d\xi+\int_{\frac{v^3}{9x}}^{\infty}e^{-\xi}\xi^{-\frac46}\right]\leq cx^{-\frac56}e^{-\frac{v^3}{9x}}\left(\frac{v^3}{x}\right)^{\frac16}.
\end{align}

Thus, we have $\phi(x,v)\to 0$ as $x\to 0$ for any $v>0$. Hence, by \eqref{pde.phim} and \eqref{int.phim}, the functions $\phi_m$ are classical solutions to \eqref{eq.1d.absorb} for any $m \ge -1$. 

Lastly, we want to prove that $\phi_m$ solves \eqref{eq.1d.absorb} for any $m\geq0$ and $\phi_{-1}$ solves \eqref{eq.1d.phi-1} in the weak sense, respectively.
\begin{itemize}
    \item Assume $m\geq0$. First, we prove
    \begin{align}\label{l2dervphim}
        \partial^k_v\phi_m \in L^2(H_R)
    \end{align}
    for any $R>0$ and $k=1,2$. To this end, by applying the operator $\partial^2_v$ to the formulas \eqref{defn.gm1.trico} and \eqref{defn.gm2.trico}, and using \eqref{eq:multiple-U}, we deduce
    \begin{equation}\label{vvgm1.trico}
    \begin{aligned}
        |\partial_v^2g_{m,1}|&\leq cx^{\lambda_m}\left(\frac{|v|}{x}  \Bigg| U'\left(-\lambda_m,\frac23,-\frac{v^3}{9x}\right)   \Bigg| +\frac{|v|^4}{x^2}  \Bigg| U''\left(-\lambda_m,\frac23,-\frac{v^3}{9x}\right)  \Bigg|  \right)\\
        &\leq cx^{\lambda_m}\left(\frac{|v|}{x}U\left(-\lambda_{m-1},\frac53,-\frac{v^3}{9x}\right)+\frac{|v|^4}{x^2}U\left(-\lambda_{m-2},\frac83,-\frac{v^3}{9x}\right)\right),
    \end{aligned}
    \end{equation}
    and
    \begin{align*}
        |\partial_v^2g_{m,2}|&\leq cx^{\lambda_m}e^{-\frac{v^3}{18x}}\frac{v}{x}\left( \Bigg| U\left(a_m,\frac23,\frac{v^3}{9x}\right) \Bigg| +  \Bigg|  U'\left(a_m,\frac23,\frac{v^3}{9x}\right) \Bigg| +  \Bigg|  U''\left(a_m,\frac23,\frac{v^3}{9x}\right)  \Bigg| \right)\\
        &\leq cx^{\lambda_m}e^{-\frac{v^3}{18x}}\frac{v}{x}\left(U\left(a_m,\frac23,\frac{v^3}{9x}\right)+U\left(a_{m+1},\frac53,\frac{v^3}{9x}\right)+U\left(a_{m+2},\frac83,\frac{v^3}{9x}\right)\right),
    \end{align*}
    where we have used the fact that $(v^2/x)^{2}e^{-\frac{v^3}{9x}}\leq c \frac{v}{x}e^{-\frac{v^3}{18x}}$.
    Note 
    \begin{equation}\label{asym.U.lange}
        \begin{aligned}
            U(a,b,\pm v^3/9x)\leq\begin{cases}c&\quad\text{if }\pm v^3\leq x,\\
            c(|v|^3/9x)^{-a}&\quad\text{if }\pm v^3\geq x
            \end{cases}
        \end{aligned}
    \end{equation}
    for some constant $c=c(a,b)$, where we have used the fact that $U(a,b,z)$ is analytic with respect to the $z$-variable and \eqref{asymp4.U}. Using these properties, we get 
    \begin{equation}\label{partialvvphim2}
    \begin{aligned}
        &\int_{H_R\cap \{x\leq -v^3\}}|\partial_v^2\phi_m|^2+\int_{H_R\cap \{x\geq |v|^3\}}|\partial_v^2\phi_{m}|^2+\int_{H_R\cap \{x\leq v^3\}}|\partial_v^2\phi_{m}|^2\\
        &\leq c\left(\int_{-R}^{0}\int_{x\leq |v|^3} |v|^{-4+6\lambda_m}\,dx\,dv+\int_{0}^{R^3}\int_{|v|^3\leq x}x^{2(\lambda_m-1)}v^2\,dv\,dx+\int_{0}^{R}\int_{x\leq v^3}|v|^{-4+6\lambda_m}\,dx\,dv\right)\\
        &\leq c\left(\int_{-R}^{0}|v|^{-1+6\lambda_m}\,dv+\int_{0}^{R^3}x^{2\lambda_m-1}\,dx+\int_{0}^{R}|v|^{-1+6\lambda_m}\,dx\,dv\right)\leq c,
    \end{aligned}
    \end{equation}
    where we have also used the fact that $(v^3/x)^ke^{-\frac{v^3}{18x}}\leq c(k)$. Similarly, we can prove $\partial_v\phi_m\in L^{2+}(H_R)$. Then, by the equation, we get $v\partial_x\phi_m=\partial_{vv}\phi_m$, which gives $v\partial_x\phi_m\in L^2(H_R)$.
    Therefore, by Remark \ref{rmk.equ.weak.str.par}, $\phi_m$ is a weak solution to \eqref{eq.1d.absorb}.
    \item Assume $m=-1$. Similarly, we prove
    \begin{align*}
        \partial^k_v\phi_{-1}\in L^2(D_R)
    \end{align*}
    for any $R>1$ and $k = 1,2$, where we write $D_R\coloneq (0,R^3)\times(-R,-1/R)$. To prove it for $k=2$, we employ \eqref{vvgm1.trico} with $\lambda_m=\lambda_{-1}$ to get 
    \begin{align*}
        &\int_{D_R\cap \mathcal{R}^+}|\partial^2_{v}\phi_{-1}|^2+\int_{(1/R)^3}^{R^3}\int_{-R}^{R}|\partial^2_{v}\phi_{-1}|^2\leq c\left(\int_{-R}^{-1/R}\int_{x<|v|^3}|v|^{-4+6\lambda_{-1}}\,dx\,dv+1\right)\leq c,
    \end{align*}
    where we have also used the fact that $\phi_{-1}$ is smooth away from the boundary. Similarly, we can prove $\partial_v\phi_m\in L^2(D_R)$. In addition, $\phi_{-1}$ is a classical solution to \eqref{eq.1d.absorb}, $v\partial_x\phi_{-1}=\partial_{vv}\phi_{-1}\in L^2(D_R)$, which gives that it is a weak solution to \eqref{eq.1d.phi-1}. 
\end{itemize}
This completes the proof.
\end{proof}

\begin{remark}\label{rmk.integrable.derv.phim}
Here, we remark that as in \eqref{l2dervphim}, we can get more general results, which are used in \autoref{lem.bdry.hol}. For any $\eps>0$ and $k=0,1,2$, we have
\begin{equation}\label{claim.integrable}
\begin{aligned}
    v^k\partial^{k+1}_{v}\phi_0\in L^{8-\eps}(H_R)\quad\text{and}\quad v^k\partial^{k+2}_v\phi_0\in L^{\frac83-\eps}(H_R).
\end{aligned}
\end{equation}
For $k=0$, using \eqref{appen.defn.j1}, \eqref{appen.defn.j2} and \eqref{asym.U.lange}, we have 
\begin{align*}
     &\int_{H_R\cap \{x\leq -v^3\}}|\partial_v\phi_0|^{8-\eps}+\int_{H_R\cap \{x\geq |v|^3\}}|\partial_v\phi_{0}|^{8-\eps}+\int_{H_R\cap \{x\leq v^3\}}|\partial_v\phi_{0}|^{8-\eps}\\
        &\leq c\left(\int_{-R}^{0}\int_{x\leq |v|^3} |v|^{\frac{-8+\eps}2}\,dx\,dv+\int_{0}^{R^3}\int_{|v|^3\leq x}(x^{-\frac56}v^{2})^{8-\eps}\,dv\,dx+\int_{0}^{R}\int_{x\leq v^3}|v|^{\frac{-8+\eps}2}\,dx\,dv\right)\\
        &\leq c(R).
\end{align*}
Similarly, we can prove $v^2\phi_{-1}, v^{3}\partial_v\phi_{-1}\in L^{8-\eps}(H_R)$. Therefore, we prove the first claim in \eqref{claim.integrable}, as $ v^k\partial^{k+1}_{v}\phi_0$ is a sum of $\partial_v\phi_0,\, v^2\phi_{-1},\, v^3\partial_v\phi_{-1}$ by \eqref{lemma:derx.phim}.

On the other hand, by considering $|\partial_v^2\phi_m|^{\frac83-\eps}$ instead of $|\partial_v^2\phi_m|^2$ in \eqref{partialvvphim2}, we directly derive $\partial^{2}_v\phi_0\in L^{\frac83-\eps}(H_R)$. 
Next, we get
\begin{equation*}
\begin{aligned}
    &\int_{H_R\cap \{x\leq -v^3\}}|v^3\partial^{2}_v\phi_{-1}|^{\frac83-\eps}+\int_{H_R\cap \{x\geq |v|^3\}}|v^3\partial^{2}_v\phi_{-1}|^{\frac83-\eps}+\int_{H_R\cap \{x\leq v^3\}}|v^3\partial^{2}_v\phi_{-1}|^{\frac83-\eps}\\
        &\leq c\left(\int_{-R}^{0}\int_{x\leq |v|^3} |v|^{-\frac32(\frac{8}{3}-\eps)}\,dx\,dv+\int_{0}^{R^3}\int_{|v|^3\leq x}(x^{-\frac{11}6}v^{4})^{\frac83-\eps}\,dv\,dx+\int_{0}^{R}\int_{x\leq v^3}|v|^{-\frac32(\frac{8}{3}-\eps)}\,dx\,dv\right)\\
        &\leq c(R).
\end{aligned}
\end{equation*}
Similarly, we also derive that $v^2\partial_v\phi_{-1}\in L^{\frac83-\eps}(H_R)$ and $v^k\partial_v^{k+2}$ can be represented as a sum of such functions. Thus, we can deduce the second claim in \eqref{claim.integrable}.
\end{remark}

Now we discuss several properties of the functions $\phi_m$ such as the precise asymptotic behavior of $\phi_m$ close to the grazing point $(0,0)$ and at infinity, the regularity of $\phi_m$ and the relation between $\phi_m$ and $\phi_{m-1}$. To do so, it is convenient to split 
\begin{align*}
    (0,\infty) \times \R = \mathcal{R}^- \cup \mathcal{R}^0 \cup \mathcal{R}^+
\end{align*}
into the following three characteristic regions:
\begin{equation}\label{defn.r-r0r+}
\begin{aligned}
    &\mathcal{R}^-\coloneqq\{(x,v)\,:\,0<x\leq v^{3}\quad\text{with }v\geq0\},\\
    &\mathcal{R}^0\coloneqq\{(x,v)\,:\, x\geq v^3\geq -x\quad\text{with }x>0\},\\
    &\mathcal{R}^+\coloneqq\{(x,v)\,:\,0<x\leq-v^{3}\quad\text{with }v\leq0\}.
\end{aligned}
\end{equation}

We say that $f_1(x,v) \eqsim f_2(x,v)$ in a set $\cA \subset \R^2$, whenever there exists $C > 0$ such that for all $x,v \in \cA$ it holds $C^{-1} f_2(x,v) \le f_1(x,v) \le C f_2(x,v)$.

\begin{lemma}
Let the functions $\phi_m$ be defined as in \eqref{rel.phiandU} for $m \ge -1$. Then, the following hold:
\begin{itemize}
\item $\phi_0$ satisfies
\begin{equation}\label{phi.asymptotic}
\begin{aligned}
    \phi_0(x,v)\eqsim\begin{cases}
       x^{\frac16}\left(\frac{v^3}x\right)^{-\frac56}e^{-\frac{v^3}{9x}}\quad&\text{if }(x,v)\in\mathcal{R}^-,\\
        x^{\frac16}\quad&\text{if }(x,v)\in\mathcal{R}^0,\\
        |v|^{\frac12}\quad&\text{if }(x,v)\in \mathcal{R}^+.
    \end{cases}
\end{aligned}
\end{equation}
\item For any $m\geq 1$, $\phi_m$ satisfies
\begin{equation}\label{phim.asymptotic}
\begin{aligned}
    \phi_m(x,v)\eqsim\begin{cases}
        x^{\frac{1}{6} + m}\left(\frac{v^3}x\right)^{-\frac{5}{6} - m} e^{-\frac{v^3}{9x}}\quad&\text{if }(x,v)\in\mathcal{R}^-,\\
        x^{\frac{1}{6} + m}\quad&\text{if }(x,v)\in\mathcal{R}^0\cap(\{x>0\}\times \{v\geq0\}).
    \end{cases}
\end{aligned}
\end{equation}
\item $\phi_{-1}$ satisfies
\begin{equation}\label{phi-1.asymptotic}
    \phi_{-1}(x,v)\eqsim \begin{cases}
    x^{-\frac{5}{6}}\left(\frac{v^3}x\right)^{\frac{1}{6}} e^{-\frac{v^3}{9x}}\quad&\text{if }(x,v)\in\mathcal{R}^-,\\
    x^{-\frac56}\quad&\text{if }(x,v)\in\mathcal{R}^0,\\
        |v|^{-\frac52}\quad&\text{if }(x,v)\in \mathcal{R}^+.
    \end{cases}
\end{equation}
    \item The functions $\phi_m$ are smooth in $H_1\cap \mathcal{R}^-$ in the sense that for $(x,v)\in H_1\cap \mathcal{R}^{-}$ with $x\leq v^{3+\epsilon}$ for some $\epsilon>0$, 
\begin{align}\label{est.smooth.phim}
    |D^k\phi_m(x,v)|\leq c(k,m,\epsilon).
\end{align}
\item For any $m\geq0$,
\begin{align}\label{lemma:derx.phim}
        c_m\partial_x \phi_m (x,v) = c_{m-1}\left( \frac{1}{36} - m^2 \right) \phi_{m-1}(x,v),
    \end{align}
    where the constants $c_m$ are given in \eqref{cm.defn}.
 
\end{itemize}
\end{lemma}

We first establish the asymptotic behavior of the functions $\phi_m$. To do so, we crucially use the integral representation formula of $U$ from \eqref{explicit.U}. Note that the formula \eqref{phim.asymptotic} cannot cover all points $(x,v) \in (0,\infty) \times \R$, since the functions $\phi_m$ change sign for $m \ge 1$. This follows from \eqref{phi.asymptotic} and \eqref{lemma:derx.phim}.

\begin{proof}[Proof of \eqref{phi.asymptotic}-\eqref{phi-1.asymptotic}.]
Let us recall $b=\frac{2}{3}$,  $a_m = \frac{5}{6} + m$, and $\lambda_m = \frac{1}{6} + m$. We consider the two cases $m\geq0$ and $m=-1$ separately.

First, we assume $m\geq0$. We split the proof into four parts.
\begin{itemize}
    \item Suppose $(x,v)\in \mathcal{R}^-$. Using the explicit formula of $\phi_m$ given by \eqref{explic.phim.pos} when $v>0$, we derive since $b - a_m - 1 < 0$
    \begin{align*}
        \phi_m(x,v) &\geq \frac1cx^{\lambda_m}e^{-\frac{v^3}{9x}}\left(\frac{v^3}{x}\right)^{\frac13+b-a_m-1}\int_{0}^{\frac{v^3}{9x}}e^{-\xi}\xi^{a_m-1}\,d\xi \\
        &\geq \frac1cx^{\lambda_m}e^{-\frac{v^3}{9x}}\left(\frac{v^3}{x}\right)^{-a_m}\int_{0}^{\frac19}e^{-\xi}\xi^{a_m-1}\,d\xi \geq \frac1cx^{\lambda_m}\left(\frac{v^3}{x}\right)^{-a_m}e^{-\frac{v^3}{9x}}
    \end{align*}
    for some constant $c=c(m)$, where we have used the fact that $v^3/x\geq1$. Using this and \eqref{ineq.upper.phim},
    \begin{align}\label{res1.asymphi}
        \phi_m(x,v)\eqsim x^{\lambda_m}\left(\frac{v^3}{x}\right)^{-a_m}e^{-\frac{v^3}{x}}\quad\text{in }\mathcal{R}^0\cap(\{x>0\}\times\{v\geq0\}).
    \end{align}
    \item Suppose $(x,v)\in \mathcal{R}^0\cap(\{x>0\}\times\{v\geq0\})$. Proceeding as before and using the fact that $v^3\leq x$ and $a_m > 0$ for $m \ge 0$, we observe
    \begin{align*}
        \phi_m(x,v)&\ge \frac1cx^{\lambda_m}e^{-\frac{v^3}{9x}}\left(\frac{v^3}{x}\right)^{-a_m}\int_{0}^{\frac{v^3}{9x}}e^{-\xi}\xi^{a_m-1}\,d\xi\geq\frac{x^{\lambda_m}}{c}\left(\frac{v^3}{x}\right)^{-a_m}\int_0^{\frac{v^3}{9x}} \xi^{a_m-1}\,d\xi\geq cx^{\lambda_m}.
    \end{align*}
    Moreover, 
    \begin{align*}
        \phi_m(x,v)&\leq cx^{\lambda_m}\left(\frac{v^3}{x}\right)^{\frac13} \int_{0}^{\infty}e^{-\xi}\xi^{a_m-1}\left(\frac{v^3}{9x}+\xi\right)^{b-a_m-1}\,d\xi \\
        &\leq cx^{\lambda_m}\left(\frac{v^3}{x}\right)^{\frac13}\left[\int_{0}^{\frac{v^3}{9x}} \xi^{a_m-1} \left(\frac{v^3}{x}\right)^{b-a_m-1}\,d\xi+\int_{\frac{v^3}{9x}}^{\infty}\xi^{b-2}\,d\xi\right]\leq cx^{\lambda_m},
    \end{align*}
    where we have also used the fact that $b-2=-5/3<-1$ and $a_m > 0$. Thus, we get 
    \begin{align}\label{res2.asymphi}
        \phi_m(x,v)\eqsim x^{\lambda_m}\quad\text{in }\mathcal{R}^0\cap(\{x>0\}\times\{v\geq0\}).
    \end{align}
    \end{itemize}
    For the remaining regions, i.e. when $v < 0$, we only consider the case $m=0$. Using \eqref{rel.phiandU}, \eqref{aplus.U}, \eqref{explicit.U}, and a change of variables, we deduce the following formula for $v < 0$,
\begin{equation*}
    \begin{aligned}
        \phi_0(x,v)&= x^{\frac16}\left(-\frac{v^3}{9x}\right)^{\frac13} U\left(\frac{1}{6},\frac{4}{3},-\frac{v^3}{9x}\right)\\
        &=\frac{1}{\Gamma(1/6)}x^{\frac16}\left(-\frac{v^3}{9x}\right)^{\frac13}\int_{0}^{\infty}e^{\frac{v^3}{9x}\xi}\xi^{-\frac56}(1+\xi)^{\frac16}\,d\xi =\frac{1}{\Gamma(1/6)}x^{\frac16}\int_{0}^{\infty}e^{-\xi}\xi^{-\frac56}\left(-\frac{v^3}{9x}+\xi\right)^{\frac16}\,d\xi.
    \end{aligned}
    \end{equation*} 
    We stress that this formula is not valid for $m \ge 1$ since it requires $-\frac{1}{6} - m - b + 1 = \frac{1}{6} - m >0$.
    \begin{itemize}
    \item Suppose $(x,v)\in \mathcal{R}^0\cap(\{x>0\}\times\{v<0\})$. Then we deduce 
    \begin{align*}
        \phi_0(x,v)\geq \frac1c x^{\frac16}\int_{0}^{1}\xi^{-\frac56}\,d\xi\geq \frac1cx^{\frac16}
    \end{align*}
    and
    \begin{align*}
        \phi_0(x,v)\leq cx^{\frac16}\left[\int_{0}^{1}e^{-\xi}\xi^{-\frac56}\,d\xi+\int_{1}^\infty e^{-\xi}\xi^{-\frac46}\,d\xi\right]\leq cx^{\frac16},
    \end{align*}
    where we have used the fact that $|v|^3\leq x$. Altogether, we have shown
     \begin{align}\label{res3.asymphi}
        \phi_0(x,v)\eqsim x^{\frac16}\quad\text{in }\mathcal{R}^0\cap(\{x>0\}\times\{v<0\}).
    \end{align}
    \item  Suppose $(x,v)\in \mathcal{R}^+$. Observe 
    \begin{align*}
        \phi_0(x,v)\geq \frac1cx^{\frac16}\left(-\frac{v^3}{x}\right)^{\frac16}\int_{0}^{-\frac{v^3}{9x}}e^{-\xi}\xi^{-\frac56}\,d\xi\geq \frac1c|v|^{\frac12}\int_{0}^{\frac19}e^{-\xi}\xi^{-\frac56}\,d\xi\geq  \frac1c|v|^{\frac12}
    \end{align*}
    and
    \begin{align*}
        \phi_0(x,v)\leq cx^{\frac16}\left[\left(-\frac{v^3}{x}\right)^{\frac16}\int_{0}^{-\frac{v^3}{9x}}e^{-\xi}\xi^{-\frac56}\,d\xi+\int_{-\frac{v^3}{9x}}^\infty e^{-\xi}\xi^{-\frac46}\,d\xi\right]\leq cx^{\frac16}\left(\left(-\frac{v^3}{x}\right)^{\frac16}+1\right)\leq c|v|^{\frac12},
    \end{align*}
    where we have used the fact that $x\leq |v|^3$. This gives
     \begin{align}\label{res4.asymphi}
        \phi_0(x,v)\eqsim |v|^{\frac12}\quad\text{in }\mathcal{R}^{-}.
    \end{align}
\end{itemize}
By \eqref{res1.asymphi} \eqref{res2.asymphi}, \eqref{res3.asymphi}, and \eqref{res4.asymphi}, we obtain \eqref{phi.asymptotic} and \eqref{phim.asymptotic}. This completes the proof for $m\geq0$.

Finally, we consider the case $m=-1$. By \eqref{rel.phiandU}, \eqref{explicit.U}, and \eqref{aplus.U}, we get 
\begin{align*}
    \phi_{-1}(x,v)=\begin{cases}
        x^{-\frac56}U\left(\frac56,\frac23,-\frac{v^3}{9x}\right)=\frac{1}{\Gamma(5/6)}x^{-\frac56}\left(-\frac{v^3}{9x}\right)^{\frac13}\int_{0}^\infty e^{-\xi}\xi^{-\frac16}\left(-\frac{v^3}{9x}+\xi\right)^{-\frac76}\,d\xi&\quad\text{if }v<0,\\
        M_{-1}x^{-\frac56}e^{-\frac{v^3}{9x}}\left(\frac{v^3}{9x}\right)^{\frac13}U\left(\frac16,\frac43,\frac{v^3}{9x}\right)=\frac{M_{-1}}{\Gamma(1/6)}x^{-\frac56}e^{-\frac{v^3}{9x}}\int_{0}^\infty e^{-\xi}\xi^{-\frac56}\left(\frac{v^3}{9x}+\xi\right)^{\frac16}\,d\xi&\quad\text{if }v\geq0.
    \end{cases}
\end{align*}

We follow similar arguments as in the proof for $m\geq0$.
\begin{itemize}
    \item Suppose $(x,v)\in \mathcal{R}^-$. Then we have 
    \begin{align*}
        \phi_{-1}(x,v)\geq\frac{x^{-\frac56}}{c}e^{-\frac{v^3}{9x}}\left(\frac{v^3}{x}\right)^{\frac16}\int_0^{\frac{v^3}{9x}}e^{-\xi}\xi^{-\frac56}\,d\xi\geq \frac{x^{-\frac56}}ce^{-\frac{v^3}{9x}}\left(\frac{v^3}{x}\right)^{\frac16}\int_{0}^{\frac19}e^{-\xi}\xi^{-\frac56}\,d\xi\geq \frac{x^{-\frac56}}ce^{-\frac{v^3}{9x}}\left(\frac{v^3}{x}\right)^{\frac16},
    \end{align*}
    and the corresponding upper bound was already deduced in \eqref{eq:phi-1-bdry}.
    \item Suppose $(x,v)\in \mathcal{R}^0\cap (\{x>0\}\times\{v\geq0\})$. Then we get 
    \begin{align*}
        \phi_{-1}(x,v)\geq \frac{x^{-\frac56}}{c}e^{-\frac{v^3}{9x}}\left(\frac{v^3}{x}\right)^{\frac16}\int_0^{\frac{9x}{v^3}}e^{-\xi}\xi^{-\frac56}\,d\xi\geq \frac{x^{-\frac56}}{c}\left(\frac{v^3}{x}\right)^{\frac16}\int_0^{\frac{9x}{v^3}}\xi^{-\frac56}\,d\xi\geq \frac{x^{-\frac56}}c,
    \end{align*}
    and
    \begin{align*}
        \phi_{-1}(x,v)\leq cx^{-\frac56}e^{-\frac{v^3}{9x}}\left[\left(\frac{v^3}{x}\right)^{\frac16}\int_0^{\frac{9x}{v^3}}e^{-\xi}\xi^{-\frac56}\,d\xi+\int_{\frac{9x}{v^3}}^\infty e^{-\xi}\xi^{-\frac46}\,d\xi\right]\leq cx^{-\frac56}.
    \end{align*}
    \item Suppose $(x,v)\in \mathcal{R}^0\cap (\{x>0\}\times\{v<0\})$. Then we have 
    \begin{align*}
        \phi_{-1}(x,v)\geq \frac{x^{-\frac56}}{c}\left(-\frac{v^3}{9x}\right)^{\frac13-\frac76}\int_{0}^{-\frac{v^3}{9x}}e^{-\xi}\xi^{-\frac16}\,d\xi\geq \frac{x^{-\frac56}}{c}\left(-\frac{v^3}{9x}\right)^{-\frac56}\int_{0}^{-\frac{v^3}{9x}}\xi^{-\frac16}\,d\xi\geq \frac{x^{-\frac56}}{c},
    \end{align*}
    and
    \begin{align*}
        \phi_{-1}(x,v)\leq c{x^{-\frac56}}\left(-\frac{v^3}{9x}\right)^{\frac13}\left[\left(-\frac{v^3}{9x}\right)^{-\frac76}\int_{0}^{-\frac{v^3}{9x}}e^{-\xi}\xi^{-\frac16}\,d\xi+\int_{-\frac{v^3}{9x}}^{\infty}e^{-\xi}\xi^{-\frac43}\,d\xi\right]\leq cx^{-\frac56}.
    \end{align*}
    \item Suppose $(x,v)\in \mathcal{R}^+$. Then we obtain 
    \begin{align*}
        \phi_{-1}(x,v)\geq \frac{x^{-\frac56}}{c}\left(-\frac{v^3}{9x}\right)^{\frac13-\frac76}\int_{0}^{-\frac{v^3}{9x}}e^{-\xi}\xi^{-\frac16}\,d\xi\geq \frac{x^{-\frac56}}{c}\left(-\frac{v^3}{9x}\right)^{\frac13-\frac76}\int_{0}^{\frac{1}{9}}e^{-\xi}\xi^{-\frac16}\,d\xi\geq \frac{|v|^{-\frac52}}{c},
    \end{align*}
    and
    \begin{align*}
        \phi_{-1}(x,v)\leq c{x^{-\frac56}}\left(-\frac{v^3}{9x}\right)^{\frac13}\left[\left(-\frac{v^3}{9x}\right)^{-\frac76}\int_{0}^{-\frac{v^3}{9x}}e^{-\xi}\xi^{-\frac16}\,d\xi+\int_{-\frac{v^3}{9x}}^{\infty}e^{-\xi}\xi^{-\frac43}\,d\xi\right]\leq c |v|^{-\frac52}.
    \end{align*}
\end{itemize}
\end{proof}

\begin{proof}[Proof of \eqref{est.smooth.phim}.]
Fix two points $(x,v)\in H_1\cap \mathcal{R}^-$ such that $v^{3+\epsilon}\geq x$ for some $\epsilon>0$. Using \eqref{rel.phiandU} with $v\geq0$, we have for some constant $c \in \R$,
\begin{align*}
    \phi_m(x,v)=c x^{\lambda_m}\widetilde{U}\left(a_m,\frac23,\frac{v^3}{9x}\right),
\end{align*}
where we write $\lambda_m\coloneqq\frac16+m$, $a_m\coloneqq\frac56+m$, and $\widetilde{U}(a,b,z)\coloneqq e^{-z}U(a,b,z)$. Since by \cite[13.3.27]{DLMF},
\begin{align*}
    \widetilde{U}^{(j)}(a,b,z)=(-1)^j\widetilde{U}(a,b+j,z),
\end{align*}
we observe 
\begin{align*}
    D^k\phi_m(x,v)&=\sum_{\substack{j\in J,\\
    3|A_j|+|B_j|\leq 3(\lambda_m+2k)}}C_jx^{A_j}v^{B_j} \widetilde{U}^{(j)}\left(a_m,\frac23,\frac{v^3}{9x}\right)\\
    &=\sum_{\substack{j\in J,\\
    3|A_j|+|B_j|\leq 3(\lambda_m+2k)}}C_j' x^{A_j}v^{B_j}e^{-\frac{v^3}{9x}}U\left(a_m,\frac23+j,\frac{v^3}{9x}\right)
\end{align*}
for some constants $C_j, C_j' \in\bbR$, where $J=\{0,1,\ldots, k\}$.

Using this together with the fact that $|z|^K e^{-|z|}\leq c(K)e^{-|z|/2}$ for any $K > 0$, $v^3/x \geq v^{-\epsilon} $, and \eqref{asymp4.U}, we derive
\begin{equation}\label{ineq.derk.phim.app}
\begin{aligned}
    |D^k\phi_m(x,v)|\leq cx^{-(\lambda_m+2k)}v^{-3(\lambda_m+2k-a_m)}e^{-\frac{v^3}{9x}}&\leq c(k,m)v^{-6(\lambda_m+2k)-a_m}e^{-\frac{v^3}{18x}}\\
    &\leq c(k,m)v^{-6(\lambda_m+2k)-a_m}e^{-\frac{1}{18v^{\epsilon}}}\\
    &\leq c(k,m,\epsilon).
\end{aligned}
\end{equation}
This completes the proof of \eqref{est.smooth.phim}.
\end{proof}

\begin{proof}[Proof of \eqref{lemma:derx.phim}.]
    First, given $\lambda > 0$ and $c \in \R$, we deduce from \eqref{eq:der-U-formula} and the chain rule,
\begin{align*}
    \partial_x U(a,b,c/x) = - x^{-1} (c/x) U'(a,b,c/x) = - x^{-1} a \big[  (1+a-b) U(a+1,b,c/x) - U(a,b,c/x)  \big].
\end{align*}
Hence, we obtain the following formula:
\begin{align}
\label{eq:der-U-formula-2}
\begin{split}
    \partial_x [x^{\lambda} U(a,b,c/x)] &= \lambda x^{\lambda-1} U(a,b,c/x) + x^{\lambda} \partial_x U(a,b,c/x) \\
    &= x^{\lambda-1} \big[(\lambda + a) U(a,b,c/x) - a (1+a-b) U(a+1,b,c/x) \big].
\end{split}
\end{align}
Similarly, we derive
\begin{equation}\label{eq:der-U-formula-3}
\begin{aligned}
    \partial_x [x^\lambda e^{c/x} U(a,b,-c/x)]&=\lambda x^{\lambda-1}e^{c/x}U(a,b,-c/x)+x^{\lambda-1}(-{c}/{x})e^{c/x}U(a,b,-c/x)\\
    &\quad +x^{\lambda} e^{c/x} \partial_x U(a,b,-c/x)\\
    &=\lambda x^{\lambda-1}e^{c/x}U(a,b,-c/x)+x^{\lambda-1}(-{c}/{x})e^{c/x}U(a,b,-c/x)\\
    &\quad-x^{\lambda-1}e^{c/x}a\big[(1+a-b)U(a+1,b,-c/x)-U(a,b,-c/x)\big]\\
    &=x^{\lambda-1}e^{c/x}\big[-a(a-b+1)U(a+1,b,-c/x)+(a+\lambda-c/x)U(a,b,-c/x)\big].
\end{aligned}
\end{equation}

Recall from \eqref{rel.phiandU} that
\begin{align*}
    \phi_m(x,v) =\begin{cases} c_mx^{\frac{1}{6}+m} U\left(-\frac{1}{6} - m , \frac{2}{3} , -\frac{v^3}{9x} \right)&\quad\text{if }v<0,\\
    c_mM_mx^{\frac16+m}e^{-\frac{v^3}{9x}}U\left(\frac56+m,\frac23,\frac{v^3}{9x}\right)&\quad\text{if }v\geq0.
    \end{cases}
\end{align*}
Let us take $m \geq0$ and apply the formula \eqref{eq:der-U-formula-2} with
\begin{align*}
    a = -\frac{1}{6} - m, ~~ \lambda = \frac{1}{6} + m, ~~ b = \frac{2}{3}, ~~ c = -\frac{v^3}{9}.
\end{align*}
Then, we deduce for $v < 0$,
\begin{align*}
    c_m\partial_x \phi_{m}(x,v) &= x^{\frac{1}{6} + (m - 1) } \left(\frac{1}{6} + m \right)\left( \frac{1}{6} - m \right) U \left( -\frac{1}{6} - (m - 1) , \frac{2}{3} , -\frac{v^3}{9x} \right) \\
    &= c_{m-1}\left(\frac{1}{36} - m^2 \right) \phi_{m-1}(x,v).
\end{align*}
 On the other hand, when $v\geq0$, using the formula \eqref{eq:der-U-formula-3} with 
\begin{align*}
    a = \frac{5}{6} + m, ~~ \lambda = \frac{1}{6} + m, ~~ b = \frac{2}{3}, ~~ c = -\frac{v^3}{9}
\end{align*}
and \eqref{eq2.U}, we get
\begin{align*}
    c_m\partial_x\phi_m(x,v) &= M_m x^{\frac16+m-1}e^{-\frac{v^3}{9x}} \Bigg[ \left(-\frac{5}{6} - m \right) \left(\frac{7}{6} +m \right) U \left( \frac{11}{6} + m , \frac{2}{3} , \frac{v^3}{9x} \right) \\
    &\qquad\qquad\qquad\qquad \qquad \qquad\qquad + (1 + 2m + v^3/(9x)) U \left( \frac{5}{6} + m , \frac{2}{3} , \frac{v^3}{9x} \right) \Bigg]\\
    &= M_mx^{\frac16+m-1}e^{-\frac{v^3}{9x}}U\left(\frac56+m-1,\frac23,\frac{v^3}{9x}\right)=c_{m-1}\left(\frac{1}{36}-m^2\right)\phi_{m-1}(x,v),
\end{align*}
where we have also used the fact that 
\begin{align*}
    M_m=\left(\frac{1}{36}-m^2\right)M_{m-1}.
\end{align*}

\end{proof}

\subsection{1D solutions with in-flow boundary condition}

We now classify all $\lambda$-homogeneous solutions to kinetic Fokker-Planck equations with source term and data being homogeneous polynomials. Obviously, it suffices to consider $\lambda \in \N \cup \{ 0 \}$.
\begin{lemma}
\label{lemma:1D-classification-inflow}
Let $\lambda \in\N\cup\{0\}$ and $h$ be a $\lambda$-homogeneous solution to
    \begin{align}\label{sol.inflow.psi}
    \begin{cases}
        v \partial_x h - \partial_{vv} h &= p ~~ \text{ in } (0,\infty) \times \R, \\
        h(x,v) &= a v^{\lambda} ~~ \text{ on } \{ 0 \} \times \{ v > 0 \},
    \end{cases}
    \end{align}
where $p \in \cP_{\lambda-2}$ is a homogeneous polynomial of degree $\lambda-2$ and $a \in \R$. In addition, we assume that either $p\neq0$ or $a\neq0$, and that $h$ does not grow exponentially as $|(x,v)|\to\infty$. Then, there are a polynomial $P \in \cP_{\lambda}$ and a unique $\mu \in \R$ such that
\begin{align}
    h = P + \mu \psi_{[\lambda/3]},
\end{align}
where $\psi_{l}$ is defined as
\begin{align}\label{defn.psilambda}
    \psi_l(x,v) \coloneqq x^{\frac{3l+2}{3}}\left( M\left( - \frac{3l+2}{3} , \frac{2}{3} , -\frac{v^3}{9x} \right) +\frac{\Gamma\left(\frac23\right)}{\Gamma\left(-\frac{3l+2}3\right)}\frac{\Gamma\left(-\frac{3l+1}3\right)}{\Gamma\left(\frac43\right)} \frac{v}{(9x)^{1/3}} M\left( - \frac{3l+1}{3} , \frac{4}{3} ,-\frac{v^3}{9x} \right)\right),
\end{align}
where $M$ denotes a hypergeometric function of the first kind.
Moreover, when $\lambda=3l$ or $3l+1$ for some nonnegative integer $l$, then $\mu=0$.
\end{lemma}
\begin{remark}
From \eqref{eq.reg.psilambda}, we observe that $\psi_0$ solves the equation \eqref{sol.inflow.psi} with $p=0$ and $\lambda=2$. However in \eqref{psi0.sol}, we denote by the function $\psi_0(x,v)$ as $\psi_0(x,v)-a{v^2}$ to ensure that it solves the equation with the constant right-hand side and zero-boundary data.
\end{remark}

\begin{remark}
Note that by \cite[13.2.42]{DLMF} we have 
\begin{align}
\label{eq:psi-lambda-U}
    \psi_{l}(x,v) = c(l) x^{\frac{3l+2}{3}} U\left(-\frac{3l+2}3,\frac23,-\frac{v^3}{9x}\right), \qquad \text{ where } \qquad c(l) = 3 \frac{\Gamma(\frac{4}{3})}{\Gamma(-\frac{3l+1}{3})}
\end{align}
for $v < 0$. For $v > 0$ it extends analytically via the same expression, taking the principal branch of the function $U$ for $z \in (-\infty,0)$. 
\end{remark}

\begin{proof}[Proof of \autoref{lemma:1D-classification-inflow}]
If $a \not=0$, we consider $\tilde{h} = h - a v^{\lambda}$, which solves
\begin{align*}
       \begin{cases}
        v \partial_x \tilde{h} - \partial_{vv} \tilde{h} &= p + a\lambda(\lambda-1) v^{\lambda-2} ~~ \text{ in } (0,\infty) \times \R, \\
        \tilde{h}(x,v) &= 0~~ \text{ on } \{ 0 \} \times \{ v > 0 \}.
    \end{cases} 
\end{align*}
Hence, it suffices to consider the case $a = 0$.

By Step 2 in the proof of \cite[Lemma 3.7]{RoWe25}, we know that for any $p \in \cP_{\lambda-2}$, there exists $P \in \cP_{\lambda}$ such that
\begin{align*}
    v \partial_x P - \partial_{vv} P = p ~~ \text{ in } (0,\infty) \times \R.
\end{align*}
Moreover, there must be $b \in \R$ such that 
\begin{align}
    P(0,v) = b v^{\lambda} ~~ \text{ on } \{ 0 \} \times \{ v > 0 \}.
\end{align}
Hence, it suffices to find all $\lambda$-homogeneous solutions $g$ such that
\begin{align}
\label{eq:inhom-1D-problem}
        \begin{cases}
        v \partial_x g - \partial_{vv} g &= 0 ~~ \text{ in } (0,\infty) \times \R, \\
        g(x,v) &= b v^{\lambda} ~~ \text{ on } \{ 0 \} \times \{ v > 0 \},
    \end{cases}
\end{align}
since then, all solutions to our problem are of the form $h = g + P$.

From \cite{RoWe25}, we know that all functions $g : [0,\infty) \times \R \to \R$ that are homogeneous of degree $\lambda \ge 0$ and solve the homogeneous equation
\begin{align}
\label{eq:1D-hom-PDE}
v \partial_x g - \partial_{vv} g = 0 ~~ \text{ in } (0,\infty) \times \R
\end{align}
are of the form
\begin{align*}
    g(x,v) = x^{\frac{\lambda}{3}} \Psi(\tau), \qquad \tau := \frac{-v^3}{9x},
\end{align*}
where $\Psi$ solves the so-called Kummer's equation
\begin{align*}
    \frac{\lambda}{3} \Psi(\tau) + \left( \frac{2}{3} - \tau \right) \Psi'(\tau) + \tau \Psi''(\tau) = 0.
\end{align*}
By solving this equation explicitly, we obtain that any solution $g$ of \eqref{eq:1D-hom-PDE} is of the form
\begin{align}
\label{eq:h-formula-1D-2}
    g(x,v) = C_1 x^{\frac{\lambda}{3}} M\left( - \frac{\lambda}{3} , \frac{2}{3} , -\frac{v^3}{9x} \right) - C_2 9^{-\frac{1}{3}} v x^{\frac{\lambda-1}{3}} M\left( - \frac{\lambda-1}{3} , \frac{4}{3} ,-\frac{v^3}{9x} \right)
\end{align}
for some $C_1,C_2 \in \R$. Recall that if $-a \in \N$, then $M(a,b,z)$ is a polynomial in $z$ of degree $-a$. Otherwise, 
\begin{align}
\label{eq:asymp-1F1}
    M(a;b;z) \asymp 
    \begin{cases}
    \frac{\Gamma(b)}{\Gamma(a)} e^z z^{a-b} \qquad &\text{ as } z \to +\infty,\\
    \frac{\Gamma(b)}{\Gamma(b-a)}(-z)^{-a} \qquad &\text{ as } z \to -\infty.
    \end{cases}
\end{align}

Now, we consider the case $\lambda \in 3(\N\cup\{0\})$. Then, we must have $C_2 = 0$, since otherwise $g(1,v)$ grows exponentially as $v \to -\infty$ by \eqref{eq:asymp-1F1}, since then $\tau = -\frac{v^3}{9x} = -\frac{v^3}{9} \to +\infty$. Hence, in that case, $M \left( - \frac{\lambda}{3} , \frac{2}{3} , \tau \right)$ is a polynomial in $\tau = -\frac{v^3}{9x}$ of degree $\frac{\lambda}{3}$. By an analogous argument, we get that $C_1=0$ and $M \left( - \frac{\lambda-1}{3} , \frac{2}{3} , \tau \right)$ is polynomial when $\lambda - 1 \in 3(\N\cup\{0\})$.
Therefore $h \in \cP_{\lambda}$, and we are done. 

Hence, we can exclude the cases $\lambda \in 3(\N\cup\{0\})$ and $\lambda-1 \in 3(\N\cup\{0\})$. Therefore, by application of \eqref{eq:asymp-1F1} to $h$, and setting $\tau = -\frac{v^3}{9x}$ we obtain
\begin{align}
\label{eq:tau+infty}
g(x,v) &\asymp 
C_1 x^{\frac{\lambda}{3}} \frac{\Gamma(\frac{2}{3})}{\Gamma(-\frac{\lambda}{3})} e^{\tau} \tau^{-\frac{\lambda+2}{3}} + C_2 x^{\frac{\lambda}{3}} \frac{\Gamma(\frac{4}{3})}{\Gamma(\frac{1-\lambda}{3})} e^{\tau} \tau^{-\frac{\lambda+2}{3}} \qquad \text{ as } \tau \to +\infty,\\
\label{eq:tau-infty}
g(x,v) &\asymp C_1 x^{\frac{\lambda}{3}} \frac{\Gamma(\frac{2}{3})}{\Gamma(\frac{\lambda+2}{3})} (-\tau)^{\frac{\lambda}{3}} - C_2 x^{\frac{\lambda}{3}} \frac{\Gamma(\frac{4}{3})}{\Gamma(\frac{\lambda + 3}{3})} (-\tau)^\frac{\lambda}{3} \qquad \text{ as } \tau \to -\infty.
\end{align}

Thus, the parameters $C_1,C_2$ must satisfy
\begin{align}
\label{eq:first-cond}
    C_1 \frac{\Gamma(\frac{2}{3})}{\Gamma(-\frac{\lambda}{3})} + C_2 \frac{\Gamma(\frac{4}{3})}{\Gamma(\frac{1-\lambda}{3})} = 0,
\end{align}
since otherwise \eqref{eq:tau+infty} would imply again that $h(1,v)$ grows exponentially as $v \to -\infty$.  An equivalent way to guarantee \eqref{eq:first-cond}, is by taking $m \in \R$ and choosing
\begin{align*}
    C_1 = m, \qquad C_2 = - \frac{\Gamma\left(\frac23\right)}{\Gamma\left(-\frac\lambda3\right)}\frac{\Gamma\left(\frac{1-\lambda}3\right)}{\Gamma\left(\frac43\right)} m.
\end{align*}
Moreover, for any $v > 0$ we use \eqref{eq:tau-infty} and that by sending $x \to 0$ it holds $\tau := -\frac{v^3}{9x} \to - \infty$, to see that when $\lambda=3l+2$,
\begin{align}
\label{eq:values-gamma_-1D-absorbing}
    g(0,v) = m 9^{-\frac\lambda3}\Gamma\left(\frac23\right)  \left( \frac{1}{\Gamma(\frac{\lambda+2}{3})} + \frac{\Gamma(\frac{1-\lambda}{3})}{\Gamma(-\frac{\lambda}{3})\Gamma(\frac{\lambda + 3}{3})} \right) v^{\lambda}  = m 9^{-\frac\lambda3}\Gamma\left(\frac23\right)\Gamma\left(\frac{1-\lambda}{3}\right) (-1)^{l+1}\frac{4\pi}{\sqrt{3}} v^{\lambda},
\end{align}
where we used elementary properties of the Gamma function as in \eqref{gamma.multiple} to prove the second equality.

Clearly, we can now find a unique $m = m(\lambda,b) \in \R$ such that $g(0,v) = b v^{\lambda}$ for any $v > 0$, as $\lambda=3l+2$ for some $l\in\N\cup\{0\}$. In addition, such a choice of the constant $m$ makes the corresponding function $g$ a solution to \eqref{eq:inhom-1D-problem}. 

Hence, altogether, we have proved that all homogeneous solutions $h$ to \eqref{eq:1D-hom-PDE} are of the form
\begin{align*}
    g(x,v)=\mu_1P+\mu_2 x^{\frac{\lambda}{3}}\left( M\left( - \frac{\lambda}{3} , \frac{2}{3} , -\frac{v^3}{9x} \right) +\frac{\Gamma\left(\frac23\right)}{\Gamma\left(-\frac\lambda3\right)}\frac{\Gamma\left(\frac{1-\lambda}3\right)}{\Gamma\left(\frac43\right)} \frac{v}{(9x)^{1/3}} M\left( - \frac{\lambda-1}{3} , \frac{4}{3} ,-\frac{v^3}{9x} \right)\right)
\end{align*}
for some $P\in \cP_\lambda$ and $\mu_1,\mu_2 \in \R$, where $\mu_1=0$ when $\lambda-2\in 3(\N\cup\{0\})$ and $\mu_2=0$ when $\lambda-2\notin 3(\N\cup\{0\})$. Therefore, using this together with \eqref{defn.psilambda}, we get the desired result.

\end{proof}

We now prove several properties of the functions $\psi_\lambda$. Note that \eqref{eq.der.psilambda} parallelizes the corresponding relation for the $\phi_m$ \eqref{lemma:derx.phim}. Remarkably, the functions $\psi_{\lambda}$ are not smooth in $\mathcal{R}^-$, which stands in stark contrast to the behavior of the functions $\phi_m$ (see \eqref{est.smooth.phim}).

\begin{lemma}
\label{lemma:psi-reg}
Let $\psi_{l}$ be defined as in \eqref{defn.psilambda}, with a nonnegative integer $l$. Then, the following hold.
\begin{itemize}
\item For any $l \ge 1$,
\begin{align}\label{eq.der.psilambda}
    \partial_x\psi_{l}=c_l\psi_{l-1}.
\end{align}
\item $\psi_{l}$ is a weak solution to 
    \begin{equation}\label{eq.reg.psilambda}
\left\{
\begin{alignedat}{3}
v\partial_x\psi_{l}-\partial_{vv}\psi_{l}&=0&&\qquad \mbox{in  $\{x>0\}\times \bbR$}, \\
\psi_{l}&=c_\lambda v^{3l+2}&&\qquad  \mbox{in $ \{x=0\}\times \{v>0\}$}.
\end{alignedat} \right.
\end{equation}
for some constant $c_\lambda\in\bbR\setminus\{0\}$.
    \item For any $k\in\N$, 
\begin{align}\label{goal.vder.psilambda}
    |\partial_v^k\psi_l(x,v)|\leq c(k,j)\quad\text{for any }(x,v)\in H_1\cap \mathcal{R}^-\text{ with } x\leq v^{k+3-(3l+2)}.
\end{align}
\item For any $\epsilon\in(0,1)$,
\begin{align}\label{goal.xder.psilambda}
    \frac{|\psi_l(x_1,v)-\psi_l(x_2,v)|}{|x_1-x_2|^{1-\epsilon}}\leq c(\epsilon,j)\quad\text{for any }(x_1,v), (x_2,v) \in H_1 \cap \mathcal{R}^- \text{ with } x_1,x_2\leq v^{\frac1{\epsilon}}.
\end{align}
In particular, we have $\psi_0\in C^{3-3\eps}(\mathcal{H}_1\cap\{x\leq v^{\frac1\eps}\})$, but $\psi_0\notin C^{3-3\eps}(\mathcal{H}_1\cap\{x\leq v^{\frac1{\eps'}}\})$ if $\eps'>\eps$, where $\mathcal{H}_1$ is the 1-dimensional two-sided kinetic cylinder determined in Section \ref{sec:prelim}.
\end{itemize}
\end{lemma}

\begin{proof}
First, we aim to prove \eqref{eq.der.psilambda}. The proof is very similar to \eqref{lemma:derx.phim}. By \cite[13.4.10]{Abramowitz}, we observe 
\begin{align*}
    zM'(a,b,z)=aM(a+1,b,z)-aM(a,b,z).
\end{align*}
Using this, we have 
\begin{align*}
    \partial_x(M\left(a,b,c/x\right))=-x^{-1}(c/x)M'(a,b,c/x)=-x^{-1}\left(aM(a+1,b,c/x)-aM(a,b,c/x)\right).
\end{align*}
With this, we derive
\begin{equation}\label{eq1.der.psilambda}
\begin{aligned}
    \partial_x\left(x^{-a}M(a,b,c/x)\right)&=-ax^{-a-1}M(a,b,-c/x)-x^{-a-1}\left(aM(a+1,b,c/x)-aM(a,b,c/x)\right)\\
    &=-ax^{-a-1}M(a+1,b,c/x).
\end{aligned}
\end{equation}
Therefore, we have 
\begin{align*}
    \partial_x\psi_{l}&=\partial_x\left(x^{-a}M(a,b,c/x)+Cv9^{-\frac13}x^{-a-\frac13}M\left(a+\frac13,b+\frac23,c/x\right)\right)\\
    &=-ax^{-a-1}M(a+1,b,c/x)-Cv9^{-\frac13}\left(a+\frac13\right)x^{-a-\frac13-1}M\left(a+\frac13+1,b+\frac23,c/x\right)\\
    &=\frac{3l+2}3x^{\frac{3l-1}3}\left(M\left(\frac{1-3l}{3},\frac23,-\frac{v^3}{9x}\right)+C\frac{\frac{3l+1}3}{\frac{3l+2}3}\frac{v}{(9x)^{\frac13}}M\left(\frac{2-3l}{3},\frac43,-\frac{v^3}{9x}\right)\right)\\
    &=\frac{3l+2}3x^{\frac{3l-1}3}\left(M\left(\frac{1-3l}{3},\frac23,-\frac{v^3}{9x}\right)+\frac{\Gamma\left(\frac23\right)}{\Gamma\left(\frac{1-3l}3\right)}\frac{\Gamma\left(\frac{2-3l}3\right)}{\Gamma\left(\frac43\right)}\frac{v}{(9x)^{\frac13}}M\left(\frac{2-3l}{3},\frac43,-\frac{v^3}{9x}\right)\right)\\
    &=\frac{3l+2}3\psi_{j-1}
\end{align*}
where $a=-\frac{3l+2}3$, $b=\frac23$, $c=-v^3$ and $C=\frac{\Gamma\left(\frac23\right)}{\Gamma\left(-\frac{3l+2}3\right)}\frac{\Gamma\left(-\frac{3l+1}3\right)}{\Gamma\left(\frac43\right)} $. In particular, we have used 
\begin{align*}
    \Gamma\left(-\frac{3l+2}3\right)\frac{-(3l+2)}3= \Gamma\left(\frac{1-3l}3\right)\quad\text{and}\quad \Gamma\left(-\frac{3l+1}3\right)\frac{-(3l+1)}3=\Gamma\left(\frac{2-3l}3\right).
\end{align*}
Therefore, we have proved \eqref{eq.der.psilambda}.

To prove that $\psi_l$ is a weak solution, first we show
\begin{align}\label{secder.psilambda}
    \partial_{vv}\psi_l\in L^\infty(H_R)\quad\text{for any }R>0.
\end{align}
To this end, let us write
\begin{align}\label{C.defn.appen}
    \widetilde{ {M}}(a,b,z)\coloneqq z^{b-1}M\left(a,b,-z\right),\quad \lambda=3l+2\quad\text{and}\quad C\coloneqq \frac{\Gamma\left(\frac23\right)}{\Gamma\left(\frac43\right)}\frac{\Gamma\left(\frac{1-\lambda}{3}\right)}{\Gamma\left(-\frac{\lambda}3\right)}
\end{align}
to see that 
\begin{align}\label{redefn.psilambda}
    \psi_l(x,v)=x^{\frac{\lambda}3}\left( M\left(-\frac{\lambda}{3},\frac23,-\frac{v^3}{9x}\right)+C\widetilde{M}\left(-\frac{\lambda-1}{3},\frac43,\frac{v^3}{9x}\right)\right),
\end{align}
by recalling \eqref{defn.psilambda}. Before we differentiate the function $\psi_l$, based on \cite[13.4.13]{Abramowitz} with $z$ replaced by $-z$, we claim that for any $z>0$,
\begin{align}\label{ind.psi1}
     \frac{d^j}{dz^j}\widetilde{M}(a,b,z)=\frac{d^j}{dz^j}\left(z^{b-1}M(a,b,-z)\right)=(b-j)_jz^{b-j-1}M(a,b-j,-z),
\end{align}
where we write 
\begin{equation*}
    (a)_0=1\quad\text{and}\quad (a)_j\coloneqq \prod_{k=0}^{j-1}(a+k)\text{ for any }j\geq1.
\end{equation*}
When $j=1$,
\begin{align*}
    \frac{d}{dz}\left(z^{b-1}M(a,b,-z)\right)&=(b-1)z^{b-2}M(a,b,-z)-z^{b-1}M'(a,b,-z)\\
    &=z^{b-2}((b-1)M(a,b,-z)+(-z)M'(a,b,-z))\\
    &=z^{b-2}(b-1)M(a,b-1,-z),
\end{align*}
where we have used \cite[13.4.13]{Abramowitz} with $z$ replaced by $-z$ for the last equality. Suppose \eqref{ind.psi1} is true for $j=m-1$. Similarly, we deduce 
\begin{align*}
    \frac{d^{m}}{dz^{m}}\left(z^{b-1}M(a,b,-z)\right)&=\frac{d}{dz}\left((b-m+1)_{m-1}z^{b-m}M(a,b-m+1,-z)\right)\\
    &=(b-m+1)_{m-1}(b-m)z^{b-m-1}M(a,b-m+1,-z)\\
    &\quad- (b-m+1)_{m-1}z^{b-m}M'(a,b-m+1,-z)\\
    &=(b-m)_mz^{b-m-1}M(a,b-m,-z).
\end{align*}
This implies \eqref{ind.psi1}, by induction. 

\begin{itemize}
    \item When $v<0$, we recall the formula \eqref{eq:psi-lambda-U}. Then by differentiating the function $\psi_l$ with respect to the $v$-variable and using \eqref{eq:multiple-U}, we get
    \begin{align*}
        |\partial^2_{v}\psi_l(x,v)|&\leq c\left(x^{\frac{\lambda}3-1}|v| \Bigg| U'\left(-\frac{\lambda}3,\frac23,-\frac{v^3}{9x}\right) \Bigg| +x^{\frac{\lambda}3-2}|v|^4 \Bigg| U''\left(-\frac{\lambda}3,\frac23,-\frac{v^3}{9x}\right) \Bigg| \right)\\
        &\leq c\left(x^{\frac{\lambda}3-1}vU\left(-\frac{\lambda}3+1,\frac53,-\frac{v^3}{9x}\right)+x^{\frac{\lambda}3-2}v^4U\left(-\frac{\lambda}3+2,\frac83,-\frac{v^3}{9x}\right)\right).
    \end{align*}
    Thus we now use \eqref{asymp4.U} and the fact that $U(a,b,z)$ is analytic, in order to further estimate
    \begin{align*}
        |\partial^2_{v}\psi_l(x,v)|\leq \begin{cases}
             cv^{{\lambda}-2}\leq cv^{3j}\leq c&\quad\text{if }-v^3\geq x,\\
            cx^{\frac{\lambda}3-\frac{2}3}\leq cx^{j}\leq c&\quad\text{if }-v^3\leq x.
        \end{cases}
    \end{align*}
    \item When $v\geq0$, we differentiate the equation given in \eqref{redefn.psilambda} to see that 
    \begin{equation}\label{ineq.psidd}
    \begin{aligned}
        \partial^2_{v}\psi_l(x,v)&=  x^{\frac{\lambda}3}\frac{2v}{3x}\left[(-1)M^{'}\left(-\frac{\lambda}3,\frac23,-\frac{v^3}{9x}\right)+C\widetilde{M}^{'}\left(-\frac{\lambda-1}3,\frac43,\frac{v^3}{9x}\right)\right]\\
        &\quad+x^{\frac{\lambda}3}\frac{v^4}{9x^2}\left[(-1)^2M^{''}\left(-\frac{\lambda}3,\frac23,-\frac{v^3}{9x}\right)+C\widetilde{M}^{''}\left(\frac{\lambda-1}3,\frac43,-\frac{v^3}{9x}\right)\right]\\
        &=-cx^{\frac{\lambda}3}\frac{v}{x}M\left(-\frac{\lambda}3+1,\frac53,-\frac{v^3}{9x}\right)+cx^{\frac{\lambda}3}\frac{v^4}{x^2}M\left(-\frac{\lambda}3+2,\frac83,-\frac{v^3}{9x}\right)\\
        &\quad+Cx^{\frac{\lambda}3}\left(\frac{v^3}{9x}\right)^{\frac13}\frac{2v}{3x}M'\left(-\frac{\lambda-1}3,\frac13,-\frac{v^3}{9x}\right),
    \end{aligned}
    \end{equation}
    where we have also used \cite[13.3.16]{DLMF}, \eqref{ind.psi1}, and  \cite[13.4.13]{Abramowitz}. Using \cite[13.3.16]{DLMF} again, we derive
    \begin{align*}
        |\partial^2_{v}\psi_l(x,v)|&\leq c\left(x^{\frac{\lambda-3}3}|v|M\left(-\frac{\lambda}3+1,\frac53,-\frac{v^3}{9x}\right)+x^{\frac{\lambda-6}3}|v|^4M\left(-\frac{\lambda}3+2,\frac83,-\frac{v^3}{9x}\right)\right)\\
        &\quad+cx^{\frac{\lambda-4}{3}}|v|^{2}M\left(-\frac{\lambda-1}3+1,\frac43,-\frac{v^3}{9x}\right).
    \end{align*}
    If $v^3\gg x$, then by using \eqref{eq:asymp-1F1}, we deduce
    \begin{equation}\label{ineq.pside}
    \begin{aligned}
        |\partial^2_{v}\psi_l(x,v)|
        &\leq c\left(x^{\frac{\lambda-3}3}|v|(v^3/9x)^{\frac{\lambda-3}3}+x^{\frac{\lambda-6}3}|v|^4(v^3/9x)^{\frac{\lambda-6}3}+x^{\frac{\lambda-4}{3}}|v|^{2}(v^3/9x)^{\frac{\lambda-4}3}\right)\leq cv^{\lambda-2}\leq c.
    \end{aligned}
    \end{equation}
    On the other hand, for the case that $v^3\leq cx$ for some large constant $c\geq1$, then we have
     \begin{align*}
         |\partial^2_{v}\psi_\lambda(x,v)|\leq c(x^{\frac{\lambda-3}{3}}x^{\frac13}+x^{\frac{\lambda-6}{3}}x^{\frac43}+x^{\frac{\lambda-4}{3}}x^{\frac23})\leq cx^{\frac{\lambda-2}{3}}\leq c.
     \end{align*}
\end{itemize}
Indeed, using \eqref{eq:values-gamma_-1D-absorbing} and \eqref{redefn.psilambda}, we get that $\psi_l$ is continuous up to the boundary $\{x=0\}\times \{v>0\}$ with $\psi_l=cv^{3l+2}$ for some constant $c\in\bbR$. Altogether, we have proved \eqref{secder.psilambda}, namely that $\partial_{vv}\psi_l\in L^\infty_{\mathrm{loc}}((0,\infty)\times \bbR)$. This gives $v\partial_x\psi_l\in L^\infty_{\mathrm{loc}}((0,\infty)\times \bbR)$, as $\psi_l$ solves \eqref{eq.reg.psilambda} in the classical sense. Therefore, by \autoref{rmk.equ.weak.str.par}, $\psi_l$ is a weak solution to \eqref{eq.reg.psilambda}.

Next, we are going to prove \eqref{goal.vder.psilambda}. 
To do so, first, we show that if $v^3\gg x$ with $v,x\geq0$ and $k>\lambda$,
\begin{align}\label{goal2.reg.psilambda}
    \partial_v^k\psi_l(x,v)=cv^{3l+2-k}O\left(\frac{x}{v^3}\right).
\end{align}
First, let us write 
\begin{align*}
    F(z)\coloneqq M\left(-\frac{\lambda}{3},\frac23,-z\right)+C\widetilde{M}\left(-\frac{\lambda-1}{3},\frac43,z\right)\quad\text{and}\quad G(v)\coloneqq \frac{v^3}{9x}
\end{align*}
and recall \eqref{C.defn.appen} to see that 
\begin{align}\label{redefn2.psilambda}
    x^{\frac\lambda3}(F\circ G)(v)=\psi_l(x,v).
\end{align}

Then by this and Fa\'a di Bruno's formula together with the fact that $G^{(j)}(v)=0$ when $j\geq4$, we have 
\begin{align}\label{formula.comp}
    \partial_{v}^k(F\circ G)(v)&=\sum_{i+2j+3m=k}\frac{k!}{i!l!2^{j}m!6^{m}}F^{(i+j+m)}(G(v))\left(\frac{v^2}{3x}\right)^{i}\left(\frac{2v}{3x}\right)^{j}\left(\frac{2}{3x}\right)^{m}.
\end{align}
Before proceeding with further estimates, we prove that if $\frac{v^3}{9x}\gg1$, then 
\begin{equation}\label{ineq.goal.derpsi}
\begin{aligned}
    F^{(j)}\left(\frac{v^3}{9x}\right)&=(-1)^j{M}^{(j)}\left(-\frac{\lambda}{3},\frac23,-\frac{v^3}{9x}\right)+C\widetilde{M}^{(j)}\left(-\frac{\lambda-1}{3},\frac43,\frac{v^3}{9x}\right)\\
    &=(-1)^jA_\lambda \left(\frac{v^3}{9x}\right)^{-j+\frac{\lambda}3}\left(\left(-\lambda/3\right)_{j}+O\left(\frac{9x}{v^3}\right)\right),
\end{aligned}
\end{equation} 
where
\begin{align*}
    A_\lambda\coloneq \Gamma\left(\frac23\right)\Gamma\left(\frac{1-\lambda}3\right)\left(\frac{1}{\Gamma\left(\frac{2+\lambda}3\right)\Gamma\left(\frac{1-\lambda}3\right)}+\frac{1}{\Gamma\left(-\frac{\lambda}3\right)\Gamma\left(\frac{3+\lambda}3\right)}\right).
\end{align*}
To prove this, we use \cite[13.5.1]{Abramowitz} together with the fact that $e^{ i\pi a}z^{-a}=(-z)^{-a}$ (note that we can always take the principle branch of $\log z = \log |z| + i \pi$ by the analyticity of $M$) for $z\in\bbR$ to see that when $v^3\gg x$, then for any $j$,
\begin{equation}\label{ineq1.mj1}
\begin{aligned}
    M\left(j-\frac{\lambda}{3},j+\frac23,-\frac{v^3}{9x}\right)=\frac{\Gamma(l+2/3)}{\Gamma(2/3+\lambda/3)}\left(\frac{v^3}{9x}\right)^{\frac{\lambda}3-j}\left(1+O\left(\frac{9x}{v^3}\right)\right),
\end{aligned}
\end{equation}
and
\begin{equation}\label{ineq2.mj1}
\begin{aligned}
    M\left(-\frac{\lambda-1}{3},-j+\frac43,-\frac{v^3}{9x}\right)=\frac{\Gamma(-j+4/3)}{\Gamma(-j+\lambda/3+1)}\left(\frac{v^3}{9x}\right)^{\frac{\lambda-1}3}\left(1+O\left(\frac{9x}{v^3}\right)\right).
\end{aligned}
\end{equation}
Recall from \cite[13.3.16]{DLMF} and \eqref{ind.psi1} that
\begin{align}\label{der.mj}
    {M}^{(l)}\left(-\frac{\lambda}{3},\frac23,-\frac{v^3}{9x}\right)=\frac{\left(-{\lambda}/{3}\right)_j}{(2/3)_j}M\left(j-\frac{\lambda}{3},j+\frac23,-\frac{v^3}{9x}\right),
\end{align}
and
\begin{align}\label{der.tilmj}
    C\widetilde{M}^{(j)}\left(-\frac{\lambda-1}{3},\frac43,\frac{v^3}{9x}\right)=C(4/3-j)_j\left(\frac{v^3}{9x}\right)^{\frac43-j-1}M\left(-\frac{\lambda-1}{3},-j+\frac43,-\frac{v^3}{9x}\right).
\end{align} 
By plugging \eqref{ineq1.mj1} and \eqref{ineq2.mj1} into \eqref{der.mj} and \eqref{der.tilmj}, respectively, we derive 
\begin{equation}\label{ineq3.mj1}
\begin{aligned}
    {M}^{(j)}\left(-\frac{\lambda}{3},\frac23,-\frac{v^3}{9x}\right)=(-\lambda/3)_j\frac{\Gamma(2/3)}{\Gamma(2/3+\lambda/3)}\left(\frac{v^3}{9x}\right)^{-j+\frac{\lambda}3}\left(1+O\left(\frac{9x}{v^3}\right)\right),
\end{aligned}
\end{equation}
and
\begin{equation}\label{ineq4.mj1}
\begin{aligned}
    &C\widetilde{M}^{(j)}\left(-\frac{\lambda-1}{3},\frac43,\frac{v^3}{9x}\right)=C\frac{\Gamma(4/3)}{\Gamma(-j+\lambda/3+1)}\left(\frac{v^3}{9x}\right)^{-j+\frac{\lambda}3}\left(1+O\left(\frac{9x}{v^3}\right)\right),
\end{aligned}
\end{equation}
where we have also used the fact that $\Gamma(a+j)=(a)_j\Gamma(a)$. Note from \eqref{C.defn.appen} that
\begin{align*}
    C\frac{\Gamma(4/3)}{\Gamma(-j+\lambda/3+1)}=C\frac{\Gamma(4/3)(-\lambda/3)_j}{(-1)^j\Gamma(\lambda/3+1)}=(-1)^j\frac{\Gamma((1-\lambda)/3)\Gamma(2/3)(-\lambda/3)_j}{\Gamma(-\lambda/3)\Gamma(\lambda/3+1)}.
\end{align*}
Therefore, using this, \eqref{ineq3.mj1}, and \eqref{ineq4.mj1}, we get \eqref{ineq.goal.derpsi}.

Now plugging \eqref{ineq.goal.derpsi} into \eqref{formula.comp} yields
\begin{equation}\label{der.comppsilambda}
\begin{aligned}
    &\frac{\partial_{v}^k(F\circ G)(v)}{A_\lambda}\\
    &=\sum_{i+2j+3m=k}\frac{k!}{i!j!2^{j}m!6^{m}}(-1)^{i+j+m}\left(\frac{v^3}{9x}\right)^{-(i+j+m)+\frac\lambda3}\left(\left(-\frac\lambda3\right)_{i+j+m}+O\left(\frac{9x}{v^3}\right)\right)\left(\frac{v^2}{3x}\right)^{i}\left(\frac{2v}{3x}\right)^{j}\left(\frac{2}{3x}\right)^{m}\\
    &=\sum_{i+2j+3m=k}\frac{k!}{i!j!m!}(-1)^{i+j+m}v^{\lambda-k}x^{-\frac\lambda3}\left(\left(-\frac\lambda3\right)_{i+j+m}+O\left(\frac{9x}{v^3}\right)\right)3^{i+j-\frac{2\lambda}3}.
\end{aligned}
\end{equation}
Now we want to prove for any $k>\lambda$,
\begin{align*}
    a_k\coloneqq\sum_{i+2j+3m=k}\frac{1}{i!j!m!}(-1)^{i+j+m}3^{i+j}\left(-\frac{\lambda}{3}\right)_{i+j+m}=0.
\end{align*}
This follows from the following observation. For a very small $y$ such as $|y|\ll1$, we get 
\begin{align*}
    (1+3y+3y^2+y^3)^{\frac{\lambda}3}&=\sum_{k=0}^{\infty}{\binom{\lambda/3}k}(3y+3y^2+y^3)^k\\
    &=\sum_{k=0}^{\infty}\frac{\lambda/3(\lambda/3-1)\cdots(\lambda/3-(k-1))}{k!}(3y+3y^2+y^3)^k\\
    &=\sum_{k=0}^{\infty}(-1)^k\frac{(-\lambda/3)_k}{k!}(3y+3y^2+y^3)^k\\
    &=\sum_{k=0}^{\infty}(-1)^k\frac{(-\lambda/3)_k}{k!}\sum_{i+j+m=k}\frac{(3y)^i(3y^2)^l(y^3)^m}{i!j!m!}k!\\
    &=\sum_{\substack{k=0\\
    i,l,m\geq0}}^{\infty}(-1)^{i+j+m}\frac{(-\lambda/3)_{i+j+m}}{i!j!m!}3^{i+j}y^{i+2j+3m},
\end{align*}
where we have used the generalized binomial theorem and the multinomial theorem. In addition, since $|y|\ll1$ and $(3y+3y^2+y^3)\ll1$, the above sums all converge. Moreover, we can rewrite
\begin{align*}
    (1+y)^\lambda=(1+3y+3y^2+y^3)^{\frac{\lambda}3}=\sum_{{k=0}}^{\infty}\sum_{i+2j+3m=k}(-1)^{i+j+m}\frac{(-\lambda/3)_{i+j+m}}{i!j!m!}3^{i+j}y^{i+2j+3m}=\sum_{k=0}^{\infty}a_ky^{k}.
\end{align*}
Therefore, we get $a_k=\binom{\lambda}{k}$ when $k\leq \lambda=3l+2$ and $a_k=0$ when $k>\lambda=3l+2$. 

By combining this result with \eqref{der.comppsilambda} and \eqref{redefn2.psilambda}, we arrive at \eqref{goal2.reg.psilambda}. Therefore, when $k>\lambda=3l+2$, we have for any $x \le v^{k+3-\lambda}$,
\begin{align}\label{ineq00.psilambda}
    |\partial_v^k\psi(x,v)|\leq cv^{\lambda-k}\frac{v^{k+3-\lambda}}{v^3}\leq c.
\end{align}
Indeed, when $v\leq v_0$ for some small constant $v_0$, then $v^3\gg v^{k+3-\lambda}\geq x$ and \eqref{goal2.reg.psilambda} implies the desired result. If $v\geq v_0$, then by applying \eqref{eq:asymp-1F1} with $x\ll1$ in the first equality given in \eqref{ineq.goal.derpsi} together with \eqref{der.mj}, and \eqref{der.tilmj}, we get \eqref{ineq00.psilambda}.

In addition, when $k\leq \lambda$, by \eqref{der.comppsilambda} and \eqref{redefn2.psilambda}, we have 
\begin{align*}
    |\partial_v^k\psi(x,v)|\leq cv^{{\lambda}-k}\leq c.
\end{align*}
Altogether, we have proved \eqref{goal.vder.psilambda}.

Now, we aim to prove \eqref{goal.xder.psilambda}. 
\begin{itemize}
    \item Assume $|x_1-x_2|\geq v^{\frac{1+k}{\eps}}$. We observe from \eqref{ineq.psidd}, \eqref{redefn.psilambda}, \eqref{ineq1.mj1} and \eqref{ineq2.mj1} that
    \begin{align}\label{psi0.asy.v}
        \partial^{k}_{v}\psi_0(x,v)=cv^{2-k}+v^{2-k}O(9x/v^3)=\partial^{k}_{v}\psi_0(0,v)+v^{2-k}O(9x/v^3)
    \end{align}
    when $k\in\{0,1,2\}$ $v^3\gg x$. Since the case $k=1$ similarly follows by a few elementary computations together with \eqref{redefn.psilambda}, \eqref{ineq1.mj1} and \eqref{ineq2.mj1}, we omit the proof.  Therefore, we get whenever $v\leq v_0$ for some small constant $v_0$,
    \begin{equation}\label{ineq1.xder.psilambda}
    \begin{aligned}
        \frac{|\partial_{v}^{k}\psi_0(x_1,v)-\partial_{v}^{k}\psi_0(x_2,v)|}{|x_1-x_2|^{1-\epsilon}}\leq c\sum_{j=1}^{2}\frac{|\partial_{v}^{k}\psi_0(x_j,v)-\partial_{v}^{k}\psi_0(0,v)|}{v^{\frac{(1+k)(1-\eps)}{\eps}}}&\leq \frac{1}{v^{\frac{(1+k)(1-\eps)}{\eps}}}\sum_{j=1}^{2}v^{2-k}O({9x_j}/{v^3})\leq c,
    \end{aligned}
    \end{equation}
    where we have used the fact that $x_j\leq v^{\frac{1+k}\epsilon} \ll v^3$ since $v \le v_0$. When $v \ge v_0$, then we use the boundary regularity result from \autoref{lem.reg.gamma+} to deduce \eqref{ineq1.xder.psilambda}, which is applicable since $(x,v)$ is away from $\gamma_0$ when $v\geq v_0$.
    \item Assume $|x_1-x_2|<v^{\frac{1+k}{\eps}}$. Then 
    \begin{equation}\label{ineq2.xder.psilambda}
    \begin{aligned}
        \frac{|\partial^{k}_{v}\psi_0(x_1,v)-\partial^{k}_{v}\psi_0(x_2,v)|}{|x_1-x_2|^{1-\epsilon}}&\leq |x_1-x_2|^\epsilon\int_{0}^{1}\left|\partial^{k}_{v}\partial_x\psi_0(tx_1+(1-t)x_2,v)\right|\,dt\\
        &\leq {|x_1-x_2|^\epsilon}\int_{0}^{1}\left|\partial^{k}_{v}(v^{-1}\partial_{vv})\psi_0(tx_1+(1-t)x_2,v)\right|\,dt\\
        &\leq c\frac{|x_1-x_2|^\epsilon}{v^{1+k}}\leq c ,
    \end{aligned}
    \end{equation}
    where we have used the fact that $v\partial_x\psi_0=\partial_{vv}\psi_0$ and \eqref{goal.vder.psilambda}. 
\end{itemize}
Combining \eqref{ineq1.xder.psilambda} and \eqref{ineq2.xder.psilambda}, we have 
\begin{equation}\label{x.dir.psi0}
\begin{aligned}
    &\frac{|\psi_0(x_1,v)-\psi_0(x_2,v)|}{|x_1-x_2|^{1-\eps}}\leq c&\quad\text{for any }x_1,x_2\leq v^{\frac1{\eps}},\\
    &\frac{|\partial^2_{v}\psi_0(x_1,v)-\partial_v^2\psi_0(x_2,v)|}{|x_1-x_2|^{1-3\eps}}\leq c&\quad\text{for any }x_1,x_2\leq v^{\frac1\eps},\\
\end{aligned}
\end{equation}
Now we use \eqref{goal.vder.psilambda}and \eqref{x.dir.psi0}, in order to derive $\psi_0\in C^{3-3\eps}(\{x\leq v^{\frac1\eps}\})$. Consider two points $z_1,z_2\in \mathcal{H}_1\cap  \{x\leq v^{\frac1\eps}\}$ and write for any $z\in \mathcal{H}_1\cap\{x\leq v^{\frac1\eps}\}$, $p_{z}$ as the second order Taylor polynomial of $\psi_0$ at $z$. Then we have 
\begin{align*}
    |(\psi_0-p_{z_1})(z_2)|\leq |(p_{(t_1,x_1,v_2)}-p_{z_1})(z_2)|+ |(\psi_0-p_{(t_1,x_1,v_2)})(z_2)|\eqqcolon J_1+J_2,
\end{align*}
where $p_{(t_1,x_1,v_2)}$ is  the second order Taylor polynomial of $\psi_0$ at $(t_1,x_1,v_2)$.
By the fundamental theorem of the calculus, \eqref{goal.vder.psilambda} and the fact that $v\partial_x\psi_0=\partial^2_v\psi_0$, we estimate $J_1$ as 
\begin{align*}
    J_1&\leq|v_2\partial_x\psi_0(x_1,v_2)-v_1\partial_x\psi_0(x_1,v_1)||t_2-t_1|\\
    &+|\psi_0(x_1,v_2)-(\psi_0(X_1)+\partial_v\psi_0(X_1)(v_2-v_1)+\partial^2_{v}\psi_0(X_1)(v_2-v_1)^2|\\
    &\leq |t_2-t_1||\partial^2_{v}\psi_0(x_1,v_2)-\partial^2_{v}\psi_0(x_1,v_1)|+c\|\partial_v^3\psi_0\|_{L^\infty(\{x_n\leq v_n^{\frac1\eps}\})}(v_2-v_1)^3\\
    &\leq c|t_2-t_1||v_1-v_2|,
\end{align*}
where $X_1\coloneqq (x_1,v_1)$.
For the estimate $J_2$, we follow the same lines as in \eqref{taylor1.g} and \eqref{taylor2.g} to get that
\begin{align*}
    J_2&\leq |p_{(t_2,x_1+(t_2-t_1)v_2,v_2)}(z_2)-p_{(t_1,x_1,v_2)}(z_2)|+|\psi_0(z_2)-p_{(t_2,x_1+(t_2-t_1)v_2,v_2)}(z_2)|\\
    &\leq \left|\int_{0}^1(\mathcal{D}\psi_0(t_1+\xi(t_2-t_1),x_1+\xi(t_2-t_1)v_2,v_2)-\mathcal{D}\psi_0(t_1,x_1,v_2))\,d\xi(t_2-t_1)\right|\\
    &\quad+|\psi_0(x_2,v_2)-\psi_0{(x_1+(t_2-t_1)v_2,v_2)}|\\
    &\leq \left|\int_{0}^1(v_2\partial_x\psi_0(x_1+\xi(t_2-t_1)v_2,v_2)-v_2\partial_x\psi_0(x_1,v_2))\,d\xi(t_2-t_1)\right|\\\
    &\quad+c|x_2-(x_1+(t_2-t_1)v_2)|^{1-\eps}\\
    &\leq c(|(t_2-t_1)v_2|^{1-3\eps}|t_2-t_1|+|x_2-(x_1+(t_2-t_1)v_2)|^{1-\eps}),
\end{align*}
where $\mathcal{D}=\partial_t+v\partial_x$, and we have used the fact that $\psi_0=\psi_0(x,v)$, $v\partial_x\psi_0=\partial_v^2\psi_0$, and \eqref{x.dir.psi0}. Combining estimates $J_1$ and $J_2$ leads to 
\begin{align}\label{taylor.psi0}
    |(\psi_0-p_{z_1})(z_2)|\leq c(|t_2-t_1||v_1-v_2|+|t_2-t_1|^{2-3\eps}+|x_2-(x_1+(t_2-t_1)v_2|^{1-\eps}),
\end{align}
which implies that $\psi_0\in C^{3-3\eps}(\mathcal{H}_1\cap\{x\leq v^{\frac1\eps}\})$ when $\eps\leq \frac13$. Note that when $\eps>\frac13$, then $\psi_0\in C^{3-3\eps}(\mathcal{H}_1\cap\{x\leq v^{3}\})$ follows from the first inequality given in \eqref{x.dir.psi0} and \eqref{secder.psilambda}.
In addition, for very small $\delta$ and $v$, we write  $x_1=\frac{v^{\frac{1}\eps-\delta}}2,\,x_2=v^{\frac{1}\eps-\delta}$. Then using \eqref{psi0.asy.v}, we have 
\begin{equation}\label{psi.sharp}
\begin{aligned}
    \frac{|\psi_0(x_1,v)-\psi_0(x_2,v)|}{|x_1-x_2|^{1-\epsilon}}\eqsim \frac{v^2O(v^{\frac{1}\eps-\delta-3})}{v^{(\frac{1}\eps-\delta)(1-\eps)}}\eqsim v^{-\delta\eps},
\end{aligned}
\end{equation}
which implies that $\psi_0\notin C^{3-3\eps}(\mathcal{H}_1\cap\{x\leq v^{\frac1\eps-\delta}\})$.

For $l\geq1$, by following the same lines, we deduce 
\begin{align*}
        \frac{|\psi_l(x_1,v)-\psi_l(x_2,v)|}{|x_1-x_2|}\leq c(j)\quad\text{for any }(x_1,v),(x_2,v)\in H_1\cap\mathcal{R}^{-}.
    \end{align*}

\end{proof}

\section{A stability result for weak solutions}
In this section, we prove that the limit of a sequence of weak solutions is also a weak solution. We give a result adapted to the setting of the proof of \autoref{lem.exp.nd.a}.

\begin{lemma}\label{lem.limit.app}
    Let $(f_m)$ be a sequence of weak solutions to 
    \begin{equation}\label{eq.fm.limit.app}
\left\{
\begin{alignedat}{3}
(\partial_t+v\cdot\nabla_x){f}_{m}-\ddiv({A}_m\nabla_v{f}_m)&=B_m\cdot\nabla_vf_m+F_m+\sum_{k=0}^{k_0}F_{m,k}\widetilde{\pmb{\Phi}}_{k}&&\quad \mbox{in  $\mathcal{H}_{\frac1{r_m}}$}, \\
{f}_m&=g_m&&\quad  \mbox{in $\gamma_-\cap \mathcal{Q}_{\frac1{r_m}}$},
\end{alignedat} \right.
\end{equation}
where $A_m\in C^{\alpha}(\mathcal{H}_{1/r_m})$, $F\in C^{\lambda}(\mathcal{H}_{1/r_m})$ $F_{m,k}\in C^{\lambda,k}(\mathcal{H}_{1/r_m})$, $g_m\in C^{\alpha}(\gamma_-\cap \mathcal{Q}_{1/r_m})$, $r_m\to0$ and $\widetilde{\pmb{\Phi}}_{k}$ is one of the elements given in \eqref{defn.mathcalg3}. We assume the following:
\begin{itemize}
    \item For any $R\leq \frac1{r_m}$,
    \begin{align}\label{ass1.limit.app}
    [F_m]_{C^{\lambda}(\mathcal{H}_{R})}, [F_{m,k}]_{C^{\lambda_1}(\mathcal{H}_{R})},\|g_m\|_{L^\infty(\gamma_-\cap \mathcal{Q}_{R})}\leq \frac{c(R)}{\theta(r_m)},\quad \|A_m\|_{C^{\alpha}(\mathcal{H}_{R})}+\|r_m^{-1}B_m\|_{L^\infty(\mathcal{H}_R)}\leq \Lambda,
\end{align}
where $\Lambda>0$, $\theta(r_m)\to 0$ as $m\to\infty$.
\item For any $\mathcal{Q}_1(z_0)\Subset \{x_n>0\}$,
\begin{align}\label{ass2.limit.app}
    f_m\to f_\infty\quad\text{in }C^{1,\alpha}(\mathcal{Q}_1(z_0)).
\end{align}
\item For any $R>0$, 
\begin{align}\label{ass3.limit.app}
    \|f_m\|_{C^\alpha(\mathcal{H}_R)}+\|\nabla_vf_m\|_{L^2(\mathcal{H}_R)}\leq c(R).
\end{align}
\end{itemize}

Then there is a weak solution $f_\infty=\lim\limits_{m\to\infty}f_m$ to
\begin{equation*}
\left\{
\begin{alignedat}{3}
(\partial_t+v\cdot\nabla_x){f}_{\infty}-\ddiv({A}_\infty\nabla_v{f}_\infty)&=F_\infty+F_{\infty,k}\widetilde{\pmb{\Phi}}_k&&\quad \mbox{in  $\mathcal{H}_{\frac1{r_m}}$}, \\
{f}_\infty&=0&&\quad  \mbox{in $\{x_n=0\}\times \{v_n>0\}$},
\end{alignedat} \right.
\end{equation*}
where $A_\infty$ is a constant matrix satisfying $\Lambda^{-1}I\leq A\leq \Lambda I$, $F_\infty\in \cP_{[\lambda]}$ and $F_{\infty,k}\in \cP_{[\lambda+k]}$.
\end{lemma}
\begin{proof}
    By Lemma \ref{lem.bdry.hol.rep} together with \eqref{ass1.limit.app} and \eqref{ass3.limit.app} , we have for any $\psi\in L^2(\mathcal{H}_R)$, 
\begin{align*}
    \left|\int_{\mathcal{H}_R}F_m \psi\,dz\right|\leq \|F_m\|_{L^\infty(\mathcal{H}_R)}\|\psi\|_{L^2(\mathcal{H}_R)}\leq c(\theta(r_m)^{-1}+\|f_m\|_{L^\infty(\mathcal{H}_R)})\|\psi\|_{L^2(\mathcal{H}_R)}\leq c(R)\|\psi\|_{L^2(\mathcal{H}_R)}.
\end{align*}
Thus, $F_m\to F_\infty$ for some $F_\infty\in L^2(\mathcal{H}_R)$. In particular, $[F_m]_{C^{\lambda}(\mathcal{H}_R)}\to0$. Thus $\|F_\infty\|_{C^{\lambda}(\mathcal{H}_R)}\leq c(R)$ by the above $L^\infty$-bound of $F$ together with the interpolation argument, which implies $F_\infty\in \cP_{[\lambda]}$. Similarly, we prove 
\begin{align*}
    \int_{\mathcal{H}_R}F_{m,k}\widetilde{\pmb{\Phi}}_k\psi \to \int_{\mathcal{H}_R}F_{\infty,k}\widetilde{\pmb{\Phi}}_k\psi, \qquad
    \int_{\mathcal{H}_R}B_m\cdot\nabla_vf_m\psi \to0.
\end{align*}
Moreover, by \eqref{ass2.limit.app} and the above observations, $f_\infty$ is a weak solution to 
\begin{align*}
    (\partial_t+v\cdot\nabla_x){f}_{\infty}-\ddiv({A}_\infty\nabla_v{f}_\infty)=F_\infty+F_{\infty,k}\widetilde{\pmb{\Phi}}_k\quad\text{in }\mathcal{H}_R,
\end{align*}
where $A_\infty$ is a constant matrix satisfying the uniformly ellipticity condition. This completes the proof.
\end{proof}


\bibliographystyle{alpha}
\bibliography{literature}

\end{document}